\documentclass{article}
\usepackage[utf8]{inputenc}
\usepackage{amsmath,amssymb,amsthm,graphicx,mathrsfs,url}
\usepackage[usenames,dvipsnames]{color}
\usepackage[colorlinks=true,linkcolor=blue,citecolor=Green,urlcolor=Blue]{hyperref}
\usepackage[super]{nth}
\usepackage{bm}
\usepackage[open, openlevel=2, depth=3, atend]{bookmark}
\hypersetup{pdfstartview=XYZ}
\usepackage{graphicx} 
\usepackage{float} 
\usepackage{subfigure} 

\def\?[#1]{\textbf{[#1]}\marginpar{\Large{\textbf{??}}}}

\numberwithin{equation}{section}

\newtheorem{theorem}{Theorem}[section]
\newtheorem{lemma}[theorem]{Lemma}
\newtheorem{proposition}[theorem]{Proposition}
\newtheorem{corollary}[theorem]{Corollary}

\newtheorem{remark}[theorem]{Remark}
\newtheorem{definition}[theorem]{Definition}

\newcommand{\mc}{\mathcal}

\newcommand{\nn}{\mathbb{N}}

\newcommand{\la}{\lambda}
\newcommand{\eps}{\epsilon}

\newcommand{\pl}{\partial}

\newcommand{\bbar}{\overline}

\newcommand{\cjd}{\rangle}
\newcommand{\cjg}{\langle}

\newcommand{\caO}{{\mathcal O}}
\newcommand{\caT}{{\mathcal T}}

\def\caC{\mathcal{C}}
\def\caH{\mathcal{H}}

\let\Im=\Imag

\let\Re=\Real

\def\indic{\operatorname{1\hskip-2.75pt\relax l}}

\renewcommand{\tilde}{\widetilde}          
\DeclareMathSymbol{\leqslant}{\mathalpha}{AMSa}{"36} 
\DeclareMathSymbol{\geqslant}{\mathalpha}{AMSa}{"3E} 
\DeclareMathSymbol{\eset}{\mathalpha}{AMSb}{"3F}     
\renewcommand{\leq}{\;\leqslant\;}                   
\renewcommand{\geq}{\;\geqslant\;}                   
\newcommand{\dd}{\text{\rm d}}             

\newcommand{\A}{\mathbb{A}}
\newcommand{\C}{\mathbb{C}}
\newcommand{\D}{\mathbb{D}}
\newcommand{\R}{\mathbb{R}}
\newcommand{\Z}{\mathbb{Z}}
\renewcommand{\H}{\mathbb{H}}
\newcommand{\N}{\mathbb{N}}
\newcommand{\Q}{\mathbb{Q}}

\newcommand{\E}{\mathbb{E}}
\renewcommand{\P}{\mathbb{P}}

\newcommand{\caA}{\mathcal{A}}
\newcommand{\caD}{\mathcal{D}}

\newcommand{\caB}{\mathcal{B}}

\def\eps{\varepsilon}

\def\T{\mathbb{T}}

\def\bi{\begin{itemize}}
\def\ei{\end{itemize}}

\def\bnum{\begin{enumerate}} 

\def\enum{\end{enumerate}}
\def\<#1{\langle #1 \rangle}

\begin{document}
\title{Conformal Bootstrap on the Annulus in Liouville CFT}
\author{Baojun Wu}

\maketitle
\begin{abstract}
This paper is the first of a series of works on the conformal bootstrap in Liouville conformal field theory (CFT) with boundaries. We focus here on the case of the annulus with two boundary insertions, each of which lies on the different connected components of the boundary. In the course of proving the bootstrap formula, we established several properties on the corresponding annulus conformal blocks: 

1) we show that they converge everywhere on the spectral line and they are continuous with respect to the spectrum and the primary weights.

2) we relate them to their torus counterparts by rigorously implementing Cardy's doubling trick for boundary CFT, 

3) we solve a conjecture of Martinec on the annulus partition function,

4) we also extend the bootstrap formula to the one-point case. 

As an application of our bootstrap result, we give an exact formula for the bosonic LQG partition function of the annulus when $\gamma\in (0,2)$. Our paper serves as a key ingredient in the recent derivation of the random moduli for the Brownian annulus by Ang, Remy, and Sun (2022). We also solve several other conjectures related to torus conformal blocks which arise from physics literature.
\end{abstract}

\footnotesize
\tableofcontents


\normalsize




\section{Introduction}
The  Liouville conformal field theory (LCFT) was introduced by A. Polyakov in his path integral formulation of String Theory \cite{Pol} and it served as a motivation for Belavin, Polyakov, and Zamolodchickov in their work on conformal field theory \cite{BPZ}. It also plays a fundamental role in the study of random surfaces, statistical physics, $4\dd$ supersymmetric Yang-Mills theory, and many other fields of physics and mathematics. It corresponds to taking the particular action functional, called Liouville action, defined for $\phi:\Sigma\to\R$  on closed Riemann surface $(\Sigma,g)$ by
\begin{equation}\label{introaction}
S_{\Sigma}(g,\phi):= \frac{1}{4\pi}\int_{\Sigma}\big(|d\phi|_g^2+QK_g \phi  + 4\pi \mu e^{\gamma \phi  }\big)\,{\rm dv}_g.
\end{equation}
where $K_g$ is the scalar curvature and ${\rm v}_g$ the volume form on $\Sigma$ determined by the metric $g$. The parameters of LCFT are cosmology constant $\mu>0$,  $\gamma\in(0,2)$ and background charge $Q=\frac{\gamma}{2}+ \frac{2}{\gamma}$. The LCFT is described in terms of positive measure on a set  $\mathcal{D}(\Sigma)$  of (generalized) real-valued functions $\phi$ on $\Sigma$.   Expectation (denoted by $\langle \cdot\rangle^{\Sigma}_g$ in what follows) under this measure is formally given as a {\it path integral}  
\begin{align}\label{pathi}
\langle F\rangle^{\Sigma}_g=\int_{\mathcal{D}(\Sigma)} F(\phi)e^{-S_{\Sigma}(g,\phi)}D\phi
\end{align}
for suitable observables $F:\mathcal{D}(\Sigma)\to \C$ and $D\phi$ a formal Lebesgue measure on $\mathcal{D}(\Sigma)$. The basic observables in LCFT are the vertex operators, which are formally defined by $e^{\alpha\phi(z)}$, $\alpha\in\C$ and $z\in\Sigma$. The fundamental objects in the LCFT are correlation functions, which can be described by path integral as follows. For $z_1, z_2,...,z_n\in\Sigma$ and $\alpha_1, \alpha_2,..,\alpha_n\in \C$
\begin{align}\label{definitionofcorrelation}
    \langle \prod_{i=1}^n e^{\alpha_i\phi(z_i)}\rangle^{\Sigma}_g
\end{align}
This theory, with central charge $c_{{\rm L}}:=1+6Q^2$, has been extensively studied in theoretical physics. 
Its resolution, i.e. finding explicit formulae for the correlation functions, can be summarized into $4$ fundamental steps when $\Sigma$ is a closed Riemann surface: 
\begin{itemize}
\item Step 1: Giving a probabilistic construction of the path integral \eqref{pathi} for the correlation function, and show it is a conformal field theory (means the correlation function behaves well under diffeomorphism coordinates change and conformal metrics change), 
\item Step 2: Finding a formula for the $3$-point functions $ \langle  e^{\alpha_1\phi(0)}e^{\alpha_2\phi(1)}e^{\alpha_3\phi(\infty)}\rangle^{\hat\C}_g$, called \emph{structure constant}, on the Riemann sphere,
\item Step 3:  A spectral resolution of the Hamiltonian of the theory in terms of primary fields and other fields called descendant fields, which is the fundamental tool for the so-called \emph{conformal bootstrap} method,
\item Step 4: Expression of the correlations functions in terms of amplitudes of pairs of pants, annuli, or discs. Roughly speaking, the {\bf amplitude} is described as an integral kernel (possibly with some vertex operator insertions)
\begin{align}\label{ampli}
\caA_{\Sigma,g}(\tilde{\bm{\varphi}})=\int_{\phi|_{\boldsymbol{\mc{C}}}=\tilde{\bm{\varphi}}}e^{-S_\Sigma (g,\phi)}D\phi
\end{align}
where $\boldsymbol{\varphi}:=(\tilde{\varphi}_1,\dots,\tilde{\varphi}_n)$ and $\tilde{\varphi}_i$ are  (generalized) functions defined on the boundary circles $\boldsymbol{\mc{C}}:=(\mc{C}_1,\dots,\mc{C}_n)$ of $\partial\Sigma$. We can associate to each component $\mc{C}_i$ a Hilbert space $\mathcal{H}$. 

Let us briefly introduce what is {\bf Segal's axiom}. In 1987 Graeme Segal \cite{segal} gave a functorial definition of Conformal Field Theory (CFT) that was
designed to capture the mathematical essence of the Conformal Bootstrap formalism pioneered in physics
by A.A. Belavin, A.M Polyakov and A.B. Zamolodchikov. For Riemann surface $\Sigma$, we separate the boundary components into $\partial^{in}\Sigma$ and $\partial^{out}\Sigma$ due to the orientation. The above integral kernel produces an operator $\mathcal{A}_{\Sigma,g}$ (we use the same notation for an operator and its integral kernel) from the Hilbert space of function on $\partial^{in}\Sigma$ to the Hilbert space of function on $\partial^{out}\Sigma$.
When we glue two surfaces $\Sigma_1$ and $\Sigma_2$ by identifying $\partial^{out}\Sigma_1=\partial^{in}\Sigma_2$, we have $\caA_{\Sigma_1\#\Sigma_2,g}=\caA_{\Sigma_2,g}\circ \caA_{\Sigma_1,g}$. By using this geometric picture and Steps 2 and 3, we can get arbitrary correlation functions for Riemann surfaces. 
\end{itemize}

It has been a major challenge to make a mathematical sense of these four steps. First, a rigorous probabilistic construction of the path integral was given in \cite{DKRV}, opening the 
possibilities for a mathematical resolution of the LCFT. The correlation function is defined non trivially in the Seiberg bound i.e. $\alpha_i<Q$, $\sum_{i=1}^n\alpha_i>\chi(\Sigma)Q$ and $\gamma\in (0,2)$, here $\chi(\Sigma)$ is the Euler characteristic of $\Sigma$. For genus $0$ case, the minimal number of insertions is $n=3$. This approach was then extended to the genus $1$ case if $n\geq 1$ in \cite{DRV} and in genus ${\bf g}\geq 2$ for all $n\geq 0$ in \cite{GRV19}, these works solve step 1.

Step 2 has then been solved in \cite{KRV}, where they proved the so-called DOZZ formula \eqref{theDOZZformula} for the structure constant; this explicit formula was proposed in physics by Dorn-Otto and Zamolodchikov-Zamolodchikov in the nineties.

Step 3 was recently mathematically settled in \cite{GKRV} using tools from scattering theory, semigroups, and probability, allowing them to prove the conformal bootstrap for the 4-point sphere. The Hamiltonian $\mathbf{H}$ for the Liouville theory is determined by the annulus amplitude $\sqrt{2}e^{c_Lt/12}e^{-t\mathbf{H}}=\mathcal{A}_{\A_{e^{-t}},g_\A}$, where $\A^1_{e^{-t}}=\{e^{-t}\leq |z|\leq 1\}$ and $g_\A=\frac{|dz|^2}{|z|^2}$. The associated Hilbert space is
 \begin{align*}
 \caH=L^2(H^{s}(\T),\dd\mu_0)
\end{align*}
and the Sobolev space $H^{s}(\T)$ with $s<0$ is equipped with a cylinder sigma-algebra and a Gaussian cylinder measure $\mu_0$ coming from the restriction of the Gaussian Free Field to $\T$ (see the exact definition in Subsection \ref{sub:hilbert}). This Hilbert space $\mathcal{H}$ can be diagonalized by the generalized eigenstates of $\mathbf{H}$.

Step 4 was finally solved in  \cite{GKRV21}. They defined amplitudes by rigorously implementing the path integral method, and showed amplitudes are some nice operators which are called Hilbert-Schmidt operators. They proved Segal's axioms for LCFT on closed Riemann surfaces, which allows to represent $n$-point correlation functions on Riemann surface with genus $\mathrm{g}$ by the structure constants and conformal blocks, i.e. the conformal bootstrap equation
   \begin{align*}
\langle \prod_{j=1}^m V_{\alpha_j}(x_{j})\rangle_{\Sigma_{\bf q},g_{\bf q}}=\int_{\R_+^{3{\bf g}-3+m}}|\mathcal{F}_{\bf P}(\boldsymbol{\alpha},{\bf q})|^2\rho({\bf P},\boldsymbol{\alpha})\dd{\bf P},
\end{align*}
Here $\rho({\bf P},\boldsymbol{\alpha})$ is the product of DOZZ formula, conformal blocks $\mathcal{F}_{\bf P}(\boldsymbol{\alpha},{\bf q})$ are power series of moduli parameter $\bf q$ whose coefficients rely on the weight of primary fields $\Delta_{\boldsymbol{\alpha}}$, the weight of intermediate spectrum $\Delta_{Q+i\bf P}$ $(\bf P\in\R^{3g-3+m}_+)$, Young diagrams. The surface $(\Sigma_{\bf q},g_{\bf q})$ corresponds a point in moduli space parameterized by certain coordinate named plumbing coordinates. The reader may temporarily think of conformal blocks as some nice function satisfying the conformal bootstrap equation, their explicit definition given in \cite{GKRV21} is much more involved. 

Now we come to the story of the {\bf boundary Liouville field theory} (BLCFT). The action functional of boundary LCFT is given by 
\begin{equation}\label{boundaryaction}
S_{\Sigma}(g,\phi):= \frac{1}{4\pi}\int_{\Sigma}\big(|d\phi|_g^2+QK_g \phi  + 4\pi \mu e^{\gamma \phi  }\big)\,{\rm dv}_g+\frac{1}{2\pi}\int_{\partial\Sigma}\big(Qk_g \phi  + 2\pi \mu_{\partial} e^{\frac{\gamma}{2} \phi  }\big)\,{\rm d\lambda}_g.
\end{equation}
Here $k_g$ is the geodesic curvature, $\lambda_g$ is the boundary measure with respect to the metric $g$, and $\mu_\partial$ is called boundary cosmology constant (we will sometimes let $\mu_\partial$ be piecewise-constant in this paper). In the boundary case, the vertex operators can be located both in the interior of $\Sigma$ by $e^{\alpha\phi(z)}$ and on the boundary of $\Sigma$ by $e^{\frac{\beta}{2}\phi(s)}$, which are called bulk insertions and boundary insertions respectively. The correlation function for bulk insertions $(z_i,\alpha_i)_{1\leq i\leq n}$ and boundary insertions $(s_j,\beta_j)_{1\leq j\leq m}$ is formally given by
\begin{align}\label{boundarycorrelation function}
  \langle \prod_{i=1}^n e^{\alpha_i\phi(z_i)}\prod_{j=1}^m e^{\frac{\beta_j}{2}\phi(s_j)}\rangle^{\Sigma}_g   
\end{align}
The mathematical definition of path integral with action functional \eqref{boundaryaction} was first treated in \cite{HRV} for the disk, and then \cite{Remy} for the annulus. Recently, the author has provided a unified approach to define the LCFT on the Riemann surface with boundaries \cite{Wu1}. We can naturally ask whether the last three steps above can be implemented for boundary LCFT.

For step 2, the basic topology in the boundary LCFT is the half upper-plane $\H$. There are four basic structure constants in the boundary LCFT (for Segal's proposal): bulk 1-point correlator $U_{\rm FZZ}(\alpha)$, bulk-boundary correlator $G(\alpha,\beta)$, boundary 2-point $R(\beta, \mu_1, \mu_2)$, and boundary 3-point  $H^{(\beta_1, \beta_2, \beta_3)}_{(\mu_1, \mu_2, \mu_3)}$ on the upper half-plane. Here we only give the explicit definition of the first two kinds (see section \eqref{definition of disk}), these are the only cases we need to deal with in this paper \footnote{Here we use a renormalization strategy as in \cite{RZ20} to coincide with the physicists' notation and introduce these to readers neatly. Their explicit definition will be given in \eqref{DOZZ2point} and \eqref{onepointfunction}.}
\begin{itemize}
\item \textbf{Bulk one-point function.} For $z \in \mathbb{H}$:
\begin{equation}\label{c1}
\left \langle e^{\alpha \phi(z)} \right \rangle^{\H} = \frac{U_{\rm FZZ}(\alpha)}{\vert z - \overline{z} \vert^{2 \Delta_{\alpha}}}.
\end{equation}
\item \textbf{Bulk-boundary correlator.} For $ z \in \mathbb{H}$, $s \in \mathbb{R}$:
\begin{equation}\label{c2}
\left \langle e^{\alpha \phi(z)} e^{\frac{\beta}{2} \phi(s)} \right \rangle^{\H} = \frac{G(\alpha, \beta)}{\vert z - \overline{z} \vert^{2 \Delta_{\alpha} - \Delta_{\beta}}\vert z -s \vert^{2 \Delta_{\beta}}}.
\end{equation}
\end{itemize}
Here $\Delta_{\alpha} = \frac{\alpha}{2}(Q - \frac{\alpha}{2})$, $\Delta_{\beta} = \frac{\beta}{2}(Q - \frac{\beta}{2})$ are called conformal weights of bulk insertion and boundary insertion respectively.
For the bulk cosmology constant $\mu=0$ case, these four constants are characterized in \cite{RZ20} by using the BPZ equation. For the $\mu>0$ case, Ang, Remy, and Sun solved the bulk 1-point constant, so-called the FZZ formula \eqref{fzz} by combining LCFT techniques and mating of trees strategies \cite{ARS21}. In this paper, we show how to represent the Liouville correlation function on annulus by the above structure constants. We need the explicit FZZ formula as an input.

For step 3, when considering Riemann surfaces with boundaries, we can choose either to cut a full circle within the interior of the surface as in \cite{GKRV21} or a half-circle with its two ends on the boundary of the surface to determine the spectrum resolution of different types of Hamiltonian. In the latter case, the Hamiltonian of boundary LCFT $\mathbf{H}^B$ describes the dynamic evolution in the upper half-plane and should correspond to the semi-group generated by half-annulus amplitude. We will treat the half-cutting in another paper. In this paper, we mainly focus on separating an annulus into two annuli by a concentric circle inside. In this case, the Hamiltonian is still $\mathbf{H}$.

For the forthcoming work, we will solve the spectrum resolution corresponding to boundary Liouville Hamiltonian $\mathbf{H}^B$ and give its geometric interpretation. This can, for example, represent the correlation of $n$-point on disk as products of structure constants in boundary LCFT and conformal blocks. We will give a full answer for step 4, i.e. Segal's axiom for Riemann surfaces with boundaries.

To state our main theorems, we need the notion of bulk one-point correlator \eqref{fzz} and bulk-boundary correlator.
At this moment, the bulk-boundary correlator $G(\alpha,\beta)$ hasn't been solved for the $\mu>0$ case. The full formula was conjectured in physics literature by using the shift equation, see \cite{Hos}. We first define the bulk-boundary correlator by Liouville correlation function within the Seiberg bound: $\alpha,\beta<Q$, $\alpha+\frac{\beta}{2}>Q$. Then we show $G(\alpha,\beta)$ can be analytic continued to a large region which contains the spectrum line $Q+i\R_+$ in Section \eqref{gluingdefinition of G}, we still denote it as $G(Q+iP,\beta)$.

Another ingredient for stating our main theorem is the conformal block. As mentioned above, conformal blocks are formal power series that appear when we represent the $n$-point correlation function by the structure constants. The mathematical definition of conformal blocks \cite{GKRV21} is given by the Fourier-type transform of amplitudes, we won't explain the precise definition at this moment to avoid too much-undeveloped terminology. For an explicit definition, we refer the reader to Section \ref{two point proof}.
For moduli parameter $q\in \D$, We define the torus $\T_q$ by gluing two boundaries of flat annulus $\A_q^1:=\{z\mid|q|\leq|z|\leq 1\}$ which inherits the flat metric $g_\A=\frac{|dz|^2}{|z|^2}$ from $\A_q^1$. The torus $n$-point conformal block is a power series in $n$ plumbing  coordinate $q_i\in \D$, whose coefficients depend on $n$ primary weights $\Delta_{\alpha_i}$, $n$ intermediate weights (spectrum weight) $\Delta_{Q+iP_i}$ and $n$ pairs of Young diagrams.
For later convenience, here we recall some concrete expressions for 1 and 2-point cases.

\noindent The {\bf torus $1$-point block} \footnote{Here $\mathcal{F}$ is slightly different from the conformal block in \cite{GKRV21}. In our paper, it's more convenient to define it as a power series in $q$ with first term (constant term) $1$.} with $\alpha$ weight insertion and moduli $q$ is given by 
\begin{align}
    \mathcal{F}^{\T}(\Delta_{\alpha},\Delta_{Q+iP},q):=\sum^{\infty}_{n=0}\Big(\sum_{|\nu|=n}\mathcal{W}_{\mathbb{A}}(\Delta_{\alpha},\Delta_{Q+iP},\Delta_{Q+iP},
 \nu,\nu )\Big)q^n.
\end{align}
And the {\bf torus $2$-point block} with weights $\alpha_1$, $\alpha_2$ and moduli $q_1$ and $q_2$ is given by
\begin{align}\label{definition of torus 2-point block}
   &\mathcal{F}^{\T}(\Delta_{\alpha_1},\Delta_{\alpha_2},\Delta_{Q+iP_1},\Delta_{Q+iP_2},q_1,q_2)\nonumber\\:=&\sum^{\infty}_{n,m=0}\Big(\sum_{|\nu|=n,|\tilde{\nu}|=m}\mathcal{W}_{\mathbb{A}}(\Delta_{\alpha_1},\Delta_{Q+iP_1},\Delta_{Q+iP_2},
 \nu,\tilde{\nu} )\mathcal{W}_{\mathbb{A}}(\Delta_{\alpha_2},\Delta_{Q+iP_2},\Delta_{Q+iP_1},
 \tilde{\nu}, \nu)\Big)q_1^n q_2^m
\end{align}
where $\mathcal{W}_{\mathbb{A}}$ are some coefficients determined by Virasoro algebra and central charge $c_L$. They are polynomials in primary weight $\Delta_\alpha$ and rational functions in spectrum weight $\Delta_{Q+iP}$ whose degree depend on the size of Young diagrams $\nu$, $\tilde{\nu}$. We refer readers to \cite[Corollary 11.10]{GKRV} for their explicit definition. 

In \cite{GKRV21}, they prove for almost all $(P_1,..,P_n)\in \R^n_+$, the torus $n$-point conformal block is convergent in $n$ moduli parameters $q_i\in \D$. In our paper, the construction of annulus conformal blocks is parallel to theirs. We cut a round concentric circle inside the annulus (which surrounds the inner circle) to derive the conformal block. We will prove the annulus $2$-point conformal block is actually convergent for all $P\in \R_+$ \eqref{continuous2point} and can be matched with torus 2-point conformal block in some sense. 

There is another construction of torus $1$-point conformal block in \cite{GRSS} by Gaussian multiplicative chaos (GMC). They show that for $q\in (\{\text{a small complex neighbourhood of }(0,1)\}\cup\{|z|\leq \frac{1}{2}\})$, the conformal block $(\alpha,P)\mapsto\mathcal{F}^{\T}(\Delta_{\alpha},\Delta_{Q+iP},q)$ has nice analyticity in $\alpha$ and $P$. We need this analyticity result as an input to our paper. Two construction of the torus 1-point block coincides since their coefficients are the same, we will explain this in the appendix \eqref{Identification}.

In this paper, we always suppose the boundary cosmology constant $\mu_\partial>0$ and the bulk cosmology constant $\mu\geq 0$. We define $q\in (0,1)$ as the moduli parameter of the annulus. For the annulus $\mathbb{A}^1_{q}=\{z\mid q\leq|z|\leq 1\}$, the inner boundary $\{|z|=q\}$ is denoted by $\partial^{in}\mathbb{A}^1_{q}$ and the outer boundary $\{|z|=1\}$ is denoted by $\partial^{out}\mathbb{A}^1_{q}$. 
We note boundary cosmology constants $\mu_\partial=\mu_{B_1}>0$ on $\partial^{out}\mathbb{A}^1_{q}$ and $\mu_\partial=\mu_{B_2}>0$ on $\partial^{in}\mathbb{A}^1_{q}$. Our main results are the following. First, we define a universal constant  $C_\D$ by equation \eqref{definition CD}
\begin{theorem}\label{two point bootstrap}
For $\mu\geq 0$, $\mu_\partial>0$, $b_1, b_2\in\partial^{out}\mathbb{A}^1_{q}$ and $\beta_1, \beta_2\in (0,Q)$, we have \footnote{All integral contours appear in following theorems can be modified from $\R_+$ to $\R$, this nontrivial consequence follows reflection property of the Liouville eigenstates, see proposition \eqref{realline}.}
$$\left\langle V_{\frac{\beta_1}{2}}(b_1)V_{\frac{\beta_2}{2}}(\frac{q}{b_2})\right\rangle^{\mathbb{A}^1_{q}}_{g_{\mathbb{A}},\mu, \mu_\partial}=  \frac{C_{\D}^2}{2\sqrt{\pi}}q^{-\frac{1}{12}}\int_{\R^{+}}G^{\mu_{B_1}}(Q+iP,\beta_1)G^{\mu_{B_2}}(Q-iP,\beta_2)q^{\frac{P^2}{2}}\mathcal{F}^{\A}(\Delta_{\beta_1},\Delta_{\beta_2},\Delta_{Q+iP},q,b_1,b_2)dP$$
where the annulus 2-point conformal block $\mathcal{F}^{\A}(\Delta_{\beta_1},\Delta_{\beta_2},\Delta_{Q+iP},q,b_1,b_2)$ is given in \eqref{2pointblock} and 
\begin{align}\label{annulus=torus}
    \mathcal{F}^{\A}(\Delta_{\beta_1},\Delta_{\beta_2},\Delta_{Q+iP},q,b_1,b_2)= \mathcal{F}^{\T}(\Delta_{\beta_1},\Delta_{\beta_2},\Delta_{Q+iP},\Delta_{Q+iP},\frac{q}{b_1b_2},qb_1b_2)
\end{align}
It's convergent for all $P\in\R^+$ in $q\in (0,1)$. \end{theorem} 
There is no priori reason that the RHS of \eqref{annulus=torus} is well defined. In \cite{GKRV21}, they prove torus 2-point block $\mathcal{F}^{\T}(\Delta_{\alpha_1},\Delta_{\alpha_2},\Delta_{Q+iP_1},\Delta_{Q+iP_2},q_1,q_2)$ converges only for almost $(P_1,P_2)\in\R^2_+$. But the diagonal $\{(P_1,P_2)\in\R^2_+\mid P_1=P_2\}$ has measure $0$. The proof will be given in Section \ref{two point proof} and proposition \ref{continuous2point}.

The following theorem answers a conjecture in the physics literature \cite{AB}, which numerically checked that torus 1-point block is a special case of annulus 2-point block. This theorem is independent of other parts of this paper
\begin{theorem}\label{relation of one two}
For $q\in (0,1)$, the annulus 2-point block $ \mathcal{F}^{\A}(\Delta_{\beta_1},\Delta_{\beta_2},\Delta_{Q+iP},q,b_1,b_2)$ is separately continuous in $\beta_1\in[0,Q)$, $\beta_2\in[0,Q)$ and $P\in[0,\infty)$ and $$ \lim_{\beta_2\to 0}\mathcal{F}^{\A}(\Delta_{\beta_1},\Delta_{\beta_2},\Delta_{Q+iP},q,b_1,b_2)=\mathcal{F}^{\T}(\Delta_{\beta_1},\Delta_{Q+iP},q^2).$$
\end{theorem}
The proof will be given in proposition \ref{continuous2point}.
\begin{theorem}\label{gammainsertion}
For $\mu\geq 0$, $\mu_\partial>0$, we have $$\partial_{\mu_{B_1}} \partial_{\mu_{B_2}} \langle 1\rangle^{\mathbb{A}^1_{q}}_{g_{\mathbb{A}},\mu, \mu_\partial}=\frac{C_{\D}^2}{2\sqrt{\pi}}\int_{\R^{+}}\partial_{\mu_{B_1}}U^{\mu_{B_1}}_{\rm FZZ}(Q+iP)\partial_{\mu_{B_2}}U^{\mu_{B_2}}_{\rm FZZ}(Q-iP)\frac{q^{\frac{P^2}{2}}}{\eta(q^2)} dP $$
where $U^{\mu_{B_i}}_{\rm FZZ}$ is the FZZ formula with boundary cosmology constant $\mu_{B_i}$. The LHS is defined by first formally taking the derivatives of boundary cosmology constants, the result correlation function is well defined.
\end{theorem}
The proof will be given in Proposition \ref{proof of gamma insertion}.
\begin{remark}
In the Theorem \ref{two point bootstrap} and Theorem \ref{gammainsertion}, we only give the proof of $\mu>0$ case in this paper. For $\mu=0$ case, one can use the free eigenstate $\Psi^{0}_{Q+iP,\nu,\tilde{\nu}}$ introduced in \cite[Proposition 4.9]{GKRV} and spectrum resolution formula \eqref{fcomplete} as a replacement to repeat the same proof. We also gather the basic properties of $\Psi^{0}_{Q+iP,\nu,\tilde{\nu}}$ in Section \ref{freetheory} for self-consistent. In the next theorem, we require $\mu>0$ to make sure proposition \ref{bbestimate1} holds.
\end{remark}

The following theorem confirms that the conformal bootstrap formula still holds when one of the $\beta_i$'s equal to $0$.
\begin{theorem}\label{one point bootstrap}
For $\mu> 0$, $\mu_\partial>0$ and $\beta\in(0,Q)$, we have
$$\langle V_{\frac{\beta}{2}}(q)\rangle^{\mathbb{A}^1_q}_{g_{\mathbb{A}}}=\frac{C_{\D}^2}{2\sqrt{\pi}}q^{-\frac{1}{12}}\int_{\R^{+}}U^{\mu_{B_1}}_{\rm FZZ}(Q+iP)G^{\mu_{B_2}}(Q-iP,\beta)q^{\frac{P^2}{2}}\mathcal{F}^{\T}(\Delta_{\beta},\Delta_{Q+iP},q^2)dP.$$
\end{theorem}
The proof will be given in Section \ref{proof of 12}.
\begin{remark}
Theorem \ref{gammainsertion} gives an explicit formula for the annulus partition function in LCFT. This conjecture was first proposed in physics literature \cite{Ma}, and later mathematically formulated in \cite[conjecture 1.4]{ARS21}. We also provide an explicit formula for torus 1-point conformal block with weight $\gamma$ \eqref{torus 1-point block}, this was predicted in physics paper \cite{Be} by numerical method. Our derivation of this torus block actually relays on its annulus counterpart formulation.
\end{remark}
\begin{remark}
Theorem \ref{two point bootstrap} shows the bootstrap formula in the {\bf amplitude bound}: $\beta_1,\beta_2\in (0,Q)^2$, see \eqref{annulusamplitudebound}, which is the largest region covered in \cite{GKRV21}. In the case of $\beta_2=0$, $\beta_1\in(0,Q)$, the conformal block can not be directly defined as the Fourier transform of amplitudes, since the amplitude with $0$-insertion is not Hilbert-Schmidt. In theorem \ref{relation of one two}, we see the two-point annulus conformal block still makes sense when taking $\beta_2=0$ and \eqref{one point bootstrap} shows the bootstrap formula still holds (namely the integral contour is still the spectrum $Q+i\R$). Outside the amplitude bound, the bootstrap formula may no longer hold, we need to deform the integral contours and add some discrete terms which correspond to the poles of the structure constant, this phenomenon is discussed in physics \cite{ZaZa} and hasn't been proved.
\end{remark}

\subsection{Outline of this proof}
 We divide the proof of our main theorems into the following steps:
\begin{itemize}
    \item {\bf Step 1.} We first need to cut the two-point correlation function in theorem \ref{two point bootstrap} into two   annulus amplitudes, see Figure \ref{Fig.main1}. These annulus amplitudes have the Dirichlet boundary condition on the red circle, which is called the gluing circle; and the Neumann boundary condition on the black circle, which is decorated by a boundary insertion. Roughly speaking, for $\tilde{\bm{\varphi}}\in \mathcal{H}$, when we parameterize the Dirichlet boundary as $\sqrt{q}e^{i\theta}$,  $\theta\in [0,2\pi)$ 
\begin{align*}
\mathcal{A}^{m}_{g_{\mathbb{A}},{\mathbb{A}}^1_{\sqrt{q}},b_1, \beta_1,\sqrt{q}e^{i\theta}}(\tilde{\bm{\varphi}}):=\int_{\phi|_{\{|z|=\sqrt{q}\}}:=\tilde{\bm{\varphi}}}e^{\frac{\beta_1}{2}\phi(b_1)}e^{-S_{\A^1_{\sqrt{q}}} (g_\A,\phi)}D\phi\\
\mathcal{A}^{m}_{g_{\mathbb{A}},{\mathbb{A}}^{\sqrt{q}}_q,\frac{q}{b_2}, \beta_2,\sqrt{q}e^{i\theta}}(\tilde{\bm{\varphi}}):=\int_{\phi|_{\{|z|=\sqrt{q}\}}:=\tilde{\bm{\varphi}}}e^{\frac{\beta_2}{2}\phi(\frac{q}{b_2})}e^{-S_{\A^{\sqrt{q}}_q} (g_\A,\phi)}D\phi
\end{align*}
For their explicit definition, see \eqref{specialcase}. Then we can prove the following gluing formula
    $$\left\langle V_{\frac{\beta_1}{2}}(b_1)V_{\frac{\beta_2}{2}}(\frac{q}{b_2})\right\rangle^{\mathbb{A}^1_{q}}_{g_{\mathbb{A}},\mu, \mu_\partial}=\sqrt{\pi}\int  \mathcal{A}^{m}_{g_{\mathbb{A}},{\mathbb{A}}^1_{\sqrt{q}},b_1, \beta_1,\sqrt{q}e^{i\theta}}(\tilde{\bm{\varphi}})\times \mathcal{A}^{m}_{g_{\mathbb{A}},{\mathbb{A}}^{\sqrt{q}}_q,\frac{q}{b_2}, \beta_2,\sqrt{q}e^{i\theta} }(\tilde{\bm{\varphi}}) d\mu_0(\tilde{\bm{\varphi}})$$
where $\sqrt{q}e^{i\theta}$ is the parametrization of the cutting circle.
To prove this, we will develop a general framework for amplitudes with Neumann boundary condition in Section \ref{section amplitude} and prove the gluing formula \eqref{gluingtypeI} based on the Markov property of the Gaussian free field.
\begin{figure}[H] 
\centering 
\includegraphics[width=0.5\textwidth]{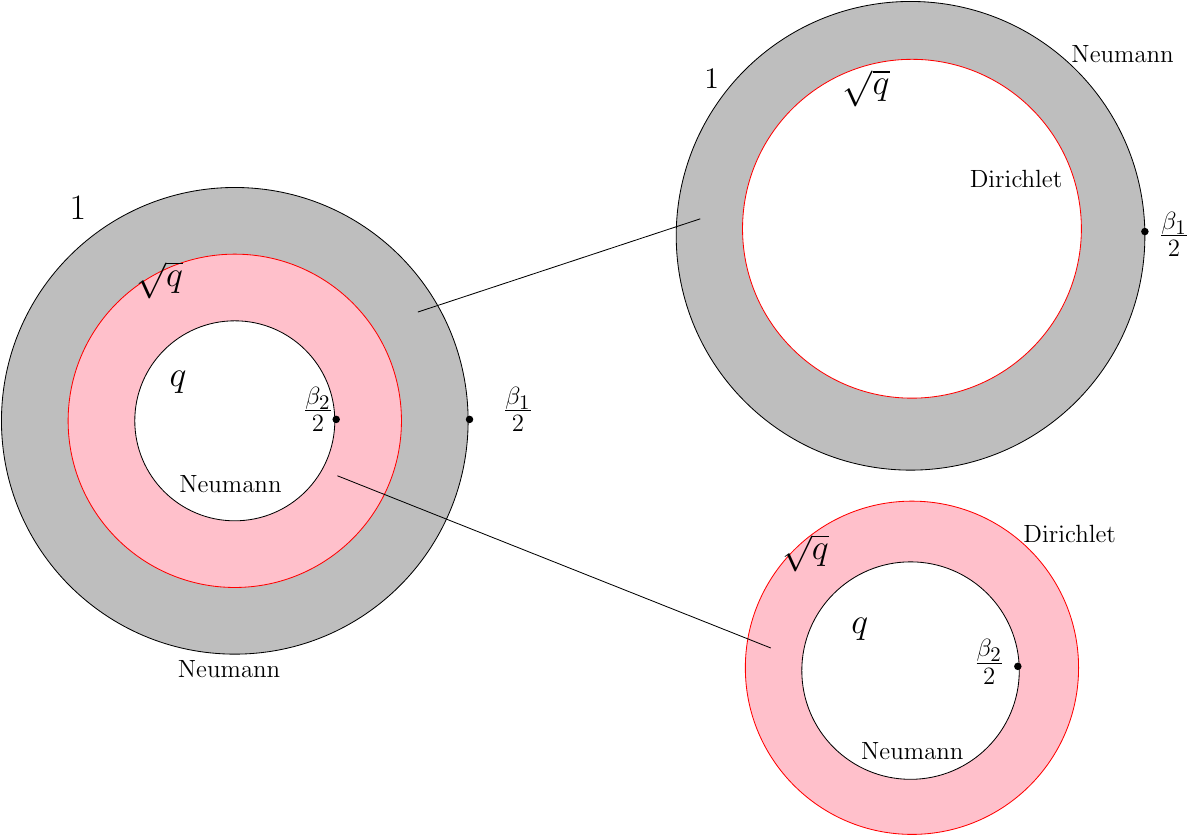} 
\caption{Cut the Neumann annulus into two amplitudes with Neumann boundary.} 
\label{Fig.main1} 
\end{figure}

\item {\bf Step 2.}  In proposition \ref{estimateamplitude}, we prove the above two amplitudes are Hilbert-Schmidt. Then we can apply the Plancherel formula \eqref{holomorphicpsi2} for Hilbert space $\mathcal{H}$ with scalar product $\langle\cdot\mid\cdot\rangle_{L^2}$, 
\begin{align}
    \frac{1}{2\sqrt{\pi}}\sum_{\nu_1, \tilde\nu_1,\nu_2, \tilde\nu_2}\int_{\R_+}& \langle\mathcal{A}^m_{g_{\mathbb{A}},{\mathbb{A}}^1_{\sqrt{q}},b_1, \beta_1,\sqrt{q}e^{i\theta}}\mid \Psi_{Q+ip,\nu_1,\tilde{\nu}_1}\rangle_{L^2}\langle\Psi_{Q+iP,\nu_2,\tilde{\nu}_2}\mid\mathcal{A}^m_{g_{\mathbb{A}},{\mathbb{A}}^{\sqrt{q}}_{q},\frac{q}{b_2}, \beta_2,\sqrt{q}e^{i\theta} }\rangle_{L^2}\nonumber\\
    &\times F^{-1}_{Q+iP}(\nu_1,\tilde{\nu_1}) F^{-1}_{Q+iP}(\nu_2,\tilde{\nu_2})\dd{P} 
\end{align}
here $\Psi_{Q+iP,\nu,\tilde{\nu}}$ is the generalized eigenstate of Liouville Hamiltonian $\mathbf{H}$ labelled by Young diagrams $\nu$, $\tilde{\nu}$, and $F_{Q+iP}(\nu,\tilde{\nu})$ is the Schapovalov matrix (see proposition \ref{prop:mainvir0} for its definition, it appears since $\Psi_{Q+iP,\nu,\tilde{\nu}}$ are not orthogonal). To compute $\langle\mathcal{A}^m_{g_{\mathbb{A}},{\mathbb{A}}^1_{\sqrt{q}},b_1, \beta_1,\sqrt{q}e^{i\theta}}\mid \Psi_{Q+ip,\nu_1,\tilde{\nu}_1}\rangle_{L^2}$, we hope to interpret eigenstates as some amplitudes and use the gluing property. So we need to consider the probabilistic counterpart of the eigenstates $\Psi_{Q+ip,\nu,\tilde\nu}$, which means we analytic $\alpha$ from $Q+iP$ to real $\alpha$ which is very negative. These $\Psi_{\alpha,\nu,\tilde{\nu}}$ can be linked with a disk amplitude with stress energy tensors (SET) around the $\alpha$ insertion, see Figure \ref{Fig.mainA}. The SET is a random field $T(z)$ formally given by $T(z)=Q\partial^2_z\phi-(\partial_z\phi)^2$ and can be rigorously defined through regularized expressions. When take the residue with respect to SET, we can produce Virasoro algebra and Young diagrams, which corresponds to the subscript of the eigenstates.
\begin{figure}[H] 
\centering 
\includegraphics[width=0.5\textwidth]{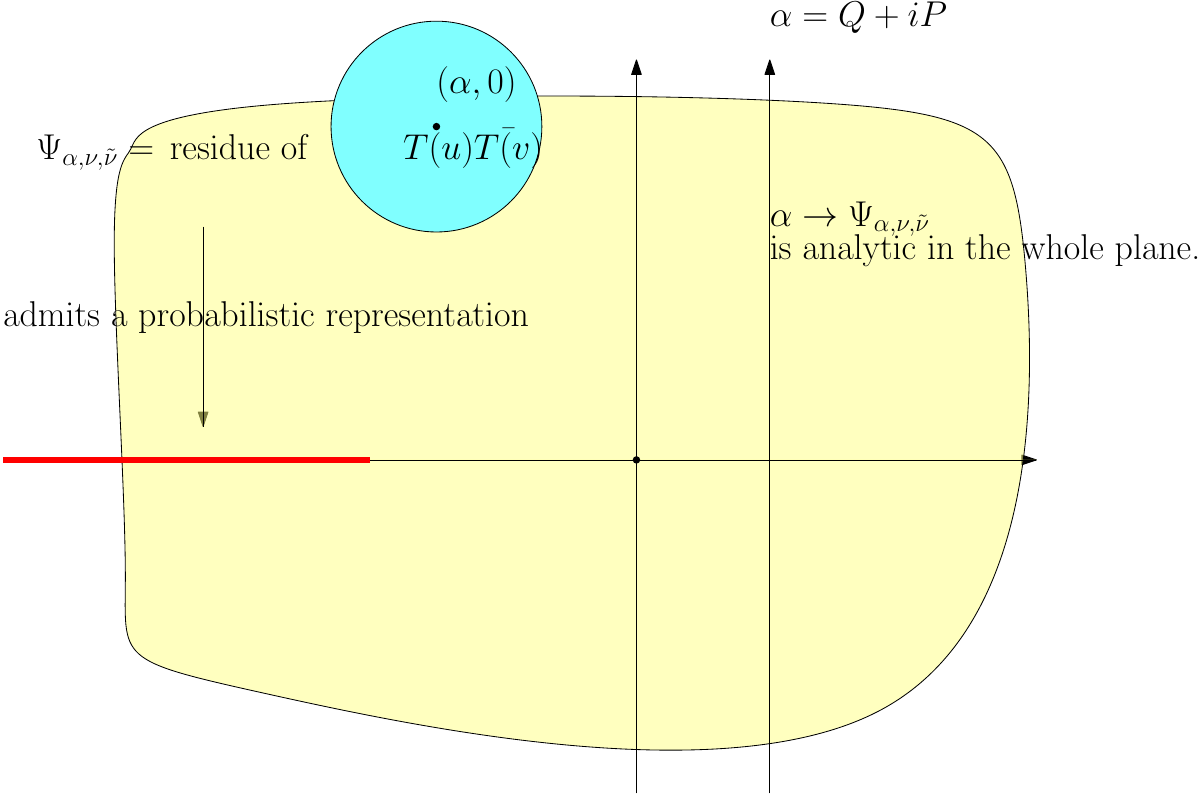} 
\caption{When $\alpha$ negative enough (in the red interval), the eigenstates admit a probabilistic representation.} 
\label{Fig.mainA} 
\end{figure}
For convenience, we first map the annulus amplitude $\mathcal{A}^m_{g_{\mathbb{A}},{\mathbb{A}}^1_{\sqrt{q}},b_1, \beta_1}$ to the upper half-plane by $\psi(z)=i\frac{1+z}{1-z}$. Note $\mathcal{S}=\psi(\mathbb{A}^1_{\sqrt{q}})$ and $s=\psi(b_1)$, the resulting amplitude is $\mathcal{A}_{\psi_*g_{\mathbb{A}},\mathcal{S},\beta,s}$. Since $\Psi_{\alpha,\nu,\tilde{\nu}}$ for $\alpha$ negative enough is not a nice function, we need to regularize the amplitude by adding more insertions. We insert $m$ artificial boundary insertions
$(\beta^{\rm ar}_i,s^{\rm ar}_i)_{1\leq i\leq m}$ so that $\alpha+\frac{\beta}{2}+\sum_i\frac{\beta^{\rm ar}_i}{2}>Q$.
Then we can glue $\Psi_{\alpha,\nu,\tilde{\nu}}$ with the amplitude $\mathcal{A}_{\psi_*g_{\mathbb{A}},\mathcal{S},\beta,s,(\beta^{\rm ar}_i,s^{\rm ar}_i)}$ by $\psi_1(z)=\psi(\sqrt{q}z)$, see Figure \ref{Fig.main2}. 
We call these insertions artificial since their weights $\beta_i^{\rm ar}$ will be finally sent to $0$ when we analytically continue $\langle\mathcal{A}_{\psi_*g_{\mathbb{A}},\mathcal{S},\beta,s,(\beta^{\rm ar}_i,s^{\rm ar}_i)}\mid \Psi_{Q+ip,\nu_1,\tilde{\nu}_1}\rangle_{L^2}$ from $\alpha$ real and negative to $Q+i\R_+$. 

To compute $\langle\mathcal{A}_{\psi_*g_{\mathbb{A}},\mathcal{S},\beta,s,(\beta^{\rm ar}_i,s^{\rm ar}_i)}\mid \Psi_{Q+ip,\nu_1,\tilde{\nu}_1}\rangle_{L^2}$, we first define the generalized correlation function with SET in the half-upper plane in Section \ref{SET}.
Then we are left with computing Ward's identity, which allows us to express the correlation functions with SET
insertions (the generalized correlation function) in terms of derivatives of the ones without SET insertions, which is the ordinary correlation function.
\begin{figure}[H] 
\centering 
\includegraphics[width=0.7\textwidth]{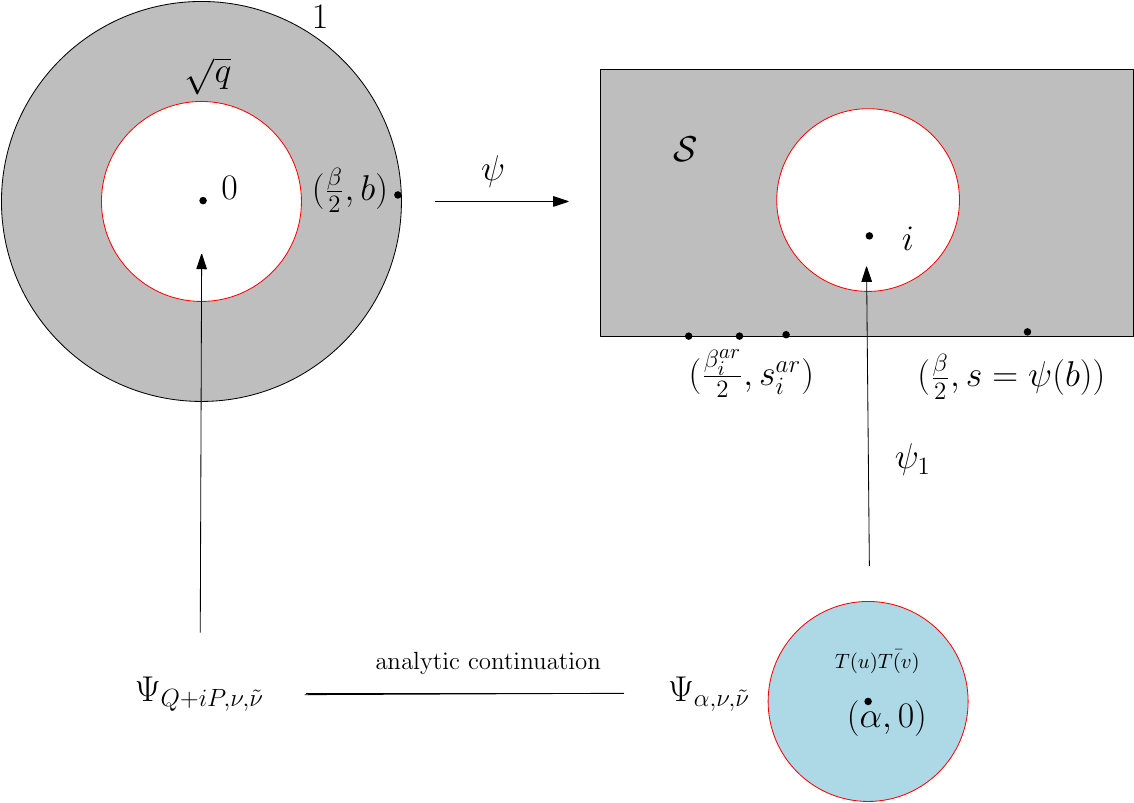} 
\caption{Glue the amplitude of $\mathcal{S}=\psi(\mathbb{A}^1_{\sqrt{q}})$ with the disk amplitude. } 
\label{Fig.main2} 
\end{figure}
\item {\bf Step 3.} The Ward's identity in the upper half-plane can be linked with the full plane case, see Figure \ref{Fig.main3}. This was first pointed out in the physics literature by Cardy \cite{Car}, so-called Cardy's doubling trick, when he studies the Ising model. He defines SET as the metric variation of the correlation functions. For our purpose of gluing amplitudes, we need to define SET by some Gaussian functionals, i.e. $T(z)=Q\partial^2_z\phi-(\partial_z\phi)^2$. Then the gluing property comes from the Markov property of GFF $\phi(z)$. In our setting, Cardy's doubling trick is highly nontrivial. We give the explicit statement in Section \ref{Cardytrick} and postpone its proof to the Appendix \ref{proof cardy}. 

After doubling, we can view the resulting generalized correlation function by gluing two copies of $\Psi_{\alpha,\nu,\tilde{\nu}}$ to $\mathcal{A}_{\psi_*g_{\mathbb{A}},\mathcal{S}\cup\mathcal{S}',\beta,s,(\beta^{\rm ar}_i,s^{\rm ar}_i)}$, where $\mathcal{S}'$ is the complex conjugate of $\mathcal{S}$ and all insertions should be thought as bulk insertions. We know $\mathcal{S}\cup\mathcal{S}'$ is nothing but the image of a thicker annulus $\mathbb{A}^{\frac{1}{\sqrt{q}}}_{\sqrt{q}}$.
We give the precise definition of conformal block by using Ward's identity in section \ref{two point proof} and see the conformal blocks of the annulus can be linked with the conformal blocks of the torus. This leads to the proof of $2$-point bootstrap formula \eqref{two point bootstrap} and indicates the convergence of annulus 2-point conformal blocks as a power series in $q\in(0,1)$ for almost all $P\in\R_+$.
\begin{figure}[H] 
\centering 
\includegraphics[width=0.7\textwidth]{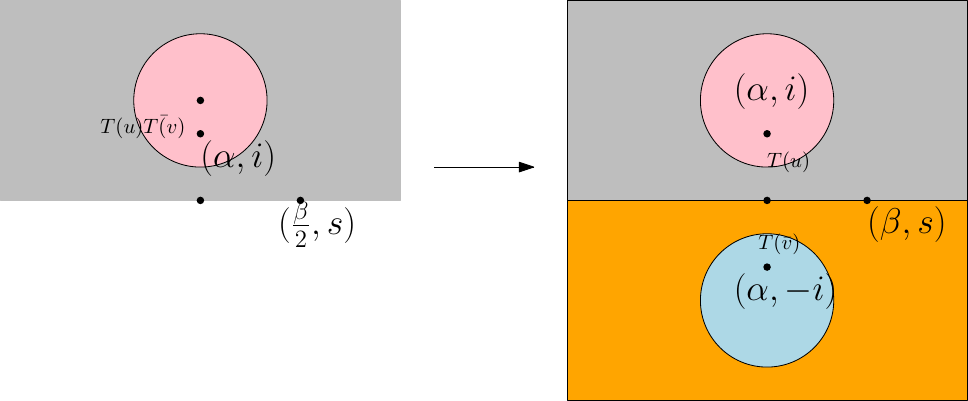} 
\caption{Link the Ward's identity in the upper half-plane to the full plane case by Cardy's trick. } 
\label{Fig.main3} 
\end{figure}
\begin{figure}[H] 
\centering 
\includegraphics[width=0.7\textwidth]{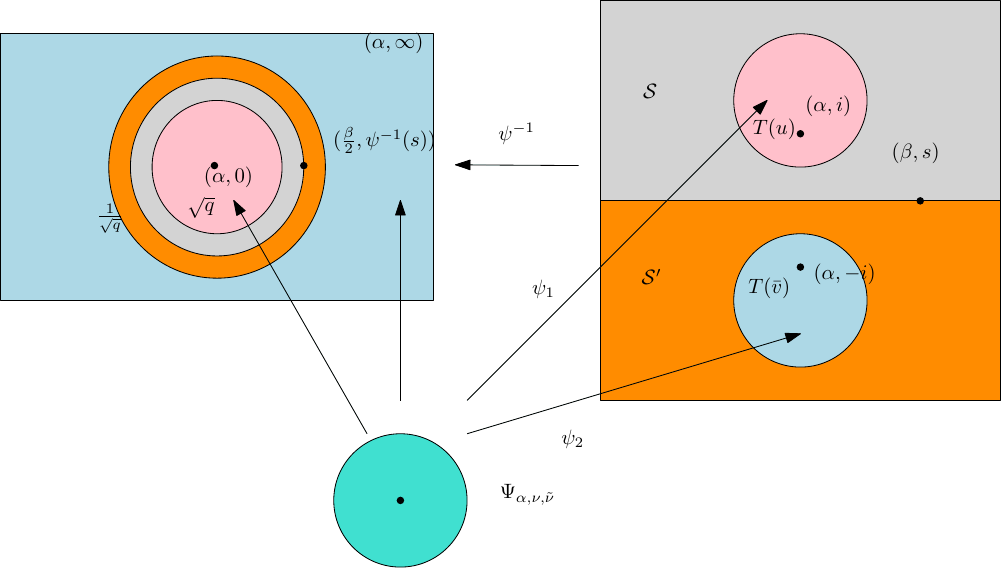} 
\caption{Build a bridge between the annulus conformal block to the torus conformal block. } 
\label{Fig.main4} 
\end{figure}
\item {\bf Step 4.} To show the annulus 2-point conformal block is separately continuous in $\beta_1\in[0,Q)$, $\beta_2\in[0,Q)$ and spectrum parameter $P\in[0,\infty)$. We need some operator formalism for the torus block amplitude \eqref{blockamplitude}. Instead of giving its explicit definition now, we convince the reader that when $q\in (0,1)$, torus block amplitude can be viewed as a positive, self-adjoint operator on the $L^2$ space of Young diagrams. Its trace produces the torus 1-point conformal block and Hilbert-Schmidt norm produces the annulus 2-point block. Then we can use analytic result in \cite{GRSS} and some comparison result between above two norms to get Theorem \ref{relation of one two}. This theorem is independent of other theorems in our paper.
\item {\bf Step 5.} The final step is to prove \eqref{one point bootstrap}. In this case, the amplitude $\mathcal{A}^{m}_{g_{\mathbb{A}},\mathbb{A}^1_{\sqrt{q}},0}$ is no more Hilbert-Schmidt, we can not directly pair it with the spectrum resolution $\Psi_{Q+iP,\nu,\tilde{\nu}}$ and use Plancherel formula. To treat this, we first develop a weaker version of conformal bootstrap \eqref{weaker bootstrap}. These formulae contain the differential with respect to the boundary cosmology constant. We show when boundary insertions have weights $\beta_1=0$ and $\beta_2=\beta\in(0, Q)$, this derivative can actually be removed, which finishes the proof of Theorem \ref{one point bootstrap}. Removing this derivative relies on the explicit FZZ formula and the torus $1$-point block is integrable near $0$ with respect to $P$.
\end{itemize}
In the last section, we give an explicit formula for the bosonic LQG partition function of annulus \eqref{LQG partition} and elucidate some future links.


\subsection{Related work}
Recently, there has been a large effort in probability theory to make sense of Polyakov's path integral
formulation of LCFT within the framework of random conformal geometry. One special case is called pure gravity also known as the scaling limit of random
planar maps in mathematical language. See \cite{LeGall}, \cite{Mier}, \cite{BM17} for construction of these limiting random metric surfaces, called Brownian surface in various topologies. In \cite{MS1}, \cite{MS2}, \cite{MS3}, they establish on the continuum level, the Brownian sphere coincide with $\sqrt{\frac{8}{3}}$-LQG sphere. But when the topology has higher genus, the Brownian surface is less understood, one essential question is what is the law of the moduli of Brownian surface.

In \cite{ARS22} they solve the law of random modulus for Brownian annulus. Their strategies are based on conformal welding and the integrability of random planar map and boundary Liouville CFT, our theorem \eqref{gammainsertion} serve as a fundamental ingredient in their work. Segal's axiom on LCFT should play a crucial role for solving random moduli problem in higher genus case. 

When the Liouville CFT is coupled with conformal matter field, like critical Ising model, a key role was played by the pioneering work \cite{DMS} where was devised a set of tools based on couplings between LQG surface and Schramm-Loewner evolution (SLE) or variants like conformal loop ensembles (CLE,  see \cite{ShWe}). So Segal's axiom for LCFT would
illuminate integrability and bootstrap type formula of $\rm SLE_{\kappa}/\rm CLE_{\kappa}$. See \cite{AS}, \cite{AHS21},  \cite{ARS22} for breakthroughs already achieved by combining different types of integrability (LCFT and random planar maps) in Liouville quantum gravity.

\subsubsection*{Acknowledgements:} The author would like to thank R\'emi Rhodes, Colin Guillarmou, Guillaume Remy and Xin Sun for their useful conversations. 

\section{The LCFT on the Riemann surfaces with boundaries}
In this section, we review some standard facts about LCFT on the Riemann surfaces. The construction of LCFT on all closed Riemann surfaces was established \cite{GRV19}, see also \cite[section 2]{GKRV21} for a review of these standard facts. We will restrict our attention to the boundary case. We summarize without proof the relevant material on boundary LCFT, the detailed proofs can be found in \cite{Wu1}. For notation $(\Sigma,g)$, we mean a compact Riemann surface $\Sigma$ equipped with a metric $g$ with associated volume form ${\rm dv}_g$ and inducing boundary measure ${\rm d\lambda}_g$. We will denote $K_g$ the scalar curvature and $k_g$ the geodesic curvature along the boundary. We consider the standard non-negative Laplacian $\Delta_{g}=d^*d$ where $d^*$ is the $L^2$ adjoint of the exterior derivative $d$. We note $ \partial^g_n$ as the normal inward derivative along the boundary with respect to metric $g$. 
We note $\Delta_{g,D}$, $\Delta_{g,N}$, $\Delta_{g,m}$ as the Laplacians with Dirichlet, Neumann and mixed boundary condition (after self-adjoint extension). When $\Sigma$ is compact and $\partial\Sigma$ is smooth, $\Delta_{g,D}$, $\Delta_{g,N}$, $\Delta_{g,m}$ all have discrete spectrum. As in \cite{RS}, we can define the spectrum zeta functions correspond to different boundary conditions (we note $\la_{N_0}=0$): $\zeta_D(s):=\sum_{j=1}^\infty (\la_{D_j})^{-s}$, $\zeta_N(s):=\sum_{j=1}^\infty (\la_{N_j})^{-s}$ and $\zeta_m(s):=\sum_{j=1}^\infty (\la_{m_j})^{-s}$. Then we can define the determinant of Laplacian correspond to various boundary conditions ${\det}(\Delta_{g,D})=\exp(-\pl_s\zeta_D(s)|_{s=0})$, ${\det} '(\Delta_{g,N})=\exp(-\pl_s\zeta_N(s)|_{s=0})$ and ${\det} '(\Delta_{g,m})=\exp(-\pl_s\zeta_{m}(s)|_{s=0})$. They all admit a meromorphic continuation from ${\rm Re}(s)\gg 1$ to $s\in \C$ and is holomorphic at $s=0$.


The Neumann Green function $G_{g,N}$ on a surface $\Sigma$ is defined 
to be the integral kernel of the resolvent operator $\Delta_{g,N}R_{g,N}=2\pi ({\rm Id}-\Pi_0)$, $\partial^n_{g}R_{g,N}=0$, $R_{g,N}^*=R_{g,N}$ and $R_{g,N}1=0$, where $\Pi_0$ is the orthogonal projection in $L^2(\Sigma,{\rm dv}_g)$ on $\ker \Delta_{g,N}$ (the constants). By integral kernel, we mean that for each $f\in L^2(\Sigma,{\rm dv}_g)$
\[ R_{g,N}f(x)=\int_{\Sigma} G_{g,N}(x,x')f(x'){\rm dv}_g(x').\] 
The Laplacian $\Delta_{g,N}$ has an orthonormal basis of real valued eigenfunctions $(\varphi_{N_j})_{j\in \N_0}$ in $L^2(\Sigma,{\rm dv}_g)$ with associated eigenvalues $\la_{N_j}\geq 0$; we set $\la_{N_0}=0$ and $\varphi_{N_0}=({\rm v}_g(\Sigma))^{-1/2}$.  The Green function then admits the following Mercer's representation in $L^2(\Sigma\times\Sigma, {\rm dv}_g \otimes {\rm dv}_g)$
\begin{equation}
G_{g,N}(x,x')=2\pi \sum_{j\geq 1}\frac{1}{\lambda_{N_j}}\varphi_{N_j}(x)\varphi_{N_j}(x').
\end{equation}
Similarly, we will consider the Green function with Dirichlet boundary conditions $G_{g,D}$ associated to the Laplacian $\Delta_{g,D}$ and mixed boundary condition $G_{g,m}$ associated to the Laplacian $\Delta_{g,m}$. In these case, the associated resolvent operators
$ R_{g,D}f(x)=\int_{\Sigma} G_{g,D}(x,x')f(x'){\rm dv}_g(x')$
solves $\Delta_{g,D}R_{g,D}=2\pi {\rm Id}$ and
$ R_{g,m}f(x)=\int_{\Sigma} G_{g,m}(x,x')f(x'){\rm dv}_g(x')$ 
solves $\Delta_{g,m}R_{g,m}=2\pi {\rm Id}$.

\subsection{Gaussian Free Fields} 
In the case of a surface $\Sigma$ with Neumann boundary, the Gaussian Free Field (GFF in short) is defined as follows. Let $(a_j)_j$ be a sequence of i.i.d. real Gaussians   $\mc{N}(0,1)$ with mean $0$ and variance $1$, defined on some probability space   $(\Omega,\mc{F},\mathbb{P})$,  and define  the Gaussian Free Field with vanishing mean in the metric $g$ by the random functions
\begin{equation}\label{GFFong}
X_{g,N}:= \sqrt{2\pi}\sum_{j\geq 1}a_j\frac{\varphi_{N_j}}{\sqrt{\la_{N_j}}} 
\end{equation}
 where  the sum converges almost surely in the Sobolev space  $H^{s}(\Sigma)$ for $s<0$ defined by
\[ H_N^{s}(\Sigma):=\{f=\sum_{j\geq 0}f_j\varphi_{N_j}\,|\, \|f\|_{s}^2:=|f_0|^2+\sum_{j\geq 1}\lambda_{N_j}^{s}|f_j|^2<+\infty\}.\] 
This Hilbert space is independent of $g$, only its norm depends on a choice of $g$.
The covariance is then the Green function when viewed as a distribution, which we will write with a slight abuse of notation
\[\mathbb{E}[X_{g,N}(x)X_{g,N}(x')]= \,G_{g,N}(x,x').\]
In the case of a surface $\Sigma$ with only Neumann boundary, we will denote the Liouville field by $\phi_g:=c+X_{g,N}$ where $c\in\R$ is a constant that stands for the constant mode of the field.

In the case of a surface with Dirichlet boundary condition or mixed boundary condition on $\Sigma$, we define the Dirichlet Gaussian free field (with covariance $G_{g,D}$) will be denoted $X_{g,D}$ and the mixed boundary Gaussian free field (with covariance $G_{g,m}$) will be denoted $X_{g,m}$ similarly. They are convergent almost surely  in the Sobolev space  $H^{s}(\Sigma)$ (for all $s\in (-1/2,0)$) defined by
\[ H_D^{s}(\Sigma):=\{f=\sum_{j\geq 0}f_j\varphi_{D_j}\,|\, \|f\|_{s}^2:=\sum_{j\geq 0}\lambda_{D_j}^{s}|f_j|^2<+\infty\}.\] and
\[ H_m^{s}(\Sigma):=\{f=\sum_{j\geq 0}f_j\varphi_{m_j}\,|\, \|f\|_{s}^2:=\sum_{j\geq 0}\lambda_{m_j}^{s}|f_j|^2<+\infty\}.\] 
In this context, we  will always consider the harmonic extension $P \boldsymbol{\tilde \varphi}$ of a boundary field $ \boldsymbol{\tilde \varphi}$ on Dirichlet boundary and with Neumann boundary condition: $\partial_n^gP \boldsymbol{\tilde \varphi}=0$ on the Neumann boundary. When the surface exists Dirichlet boundary, the Liouville field will be denoted by  $\phi_g:= X_{g,D}+P \boldsymbol{\tilde \varphi}$ or $\phi_g:= X_{g,m}+P \boldsymbol{\tilde \varphi}$. Hence $\phi_g$ will depend on boundary data in this case. \\
As Gaussian Free Fields rather badly behave (they are distributions), we will need to consider their regularizations. We define {\bf g-regularization} as their geodesic circle average with distance $\epsilon$, we refer reads to \cite[Section 3.2]{GRV19} and \cite{Wu1} for details.
We will denote $X_{g,N,\epsilon}$, $X_{g,D,\epsilon}$  and $X_{g,m,\epsilon}$ their respective $g$-regularizations
and $\phi_{g,\epsilon}$ the $g$-regularization of the Liouville field.

\subsection{Construction of Liouville Conformal Field Theory} 
For $\gamma\in\R$ and $h$ a random distribution admitting a  $g$-regularization $(h_\epsilon)_\epsilon$, we define the bulk GMC measure and the boundary GMC measure as
\begin{equation}\label{GMCg1}
M_{\gamma,g,\eps}(h,\dd x):= \eps^{\frac{\gamma^2}{2}}e^{\gamma h_{ \eps}(x)}{\rm dv}_g(x).
\end{equation}
\begin{equation}\label{GMCg2}
M^{\partial}_{\gamma,g,\eps}(h,\dd x):= \eps^{\frac{\gamma^2}{4}}e^{\frac{\gamma}{2} h_{ \eps}(x)}{\rm d\lambda}_g(x).
\end{equation}
Of particular interest for us is the case when $h=X_{g,N}$ or $h=X_{g,m}$ (and consequently $h=\phi_g$ too). In that case, for $\gamma\in (0,2)$,  the random measures above converge  as $\eps\to 0$ in probability and weakly in the space of Radon measures towards   non trivial random measures respectively denoted by $M_{\gamma,g}(h,\dd x)$, $M^g_\gamma(X_{g,D},\dd x)$ and $M^{\partial}_{\gamma,g}(h,\dd x)$.

Here we also mention a simple fact, when $\Sigma$ has some Dirichlet boundaries, $X_{g,m}$ is invariant in a conformal class. And if $g'=e^{\omega}g$, then due to the renormalization with respect to different metric,
\begin{align*}\label{scalingmeasure}
& M_{\gamma,g'}(X_{g',m},\dd x)= e^{\frac{\gamma}{2}Q \omega(x)} M_{\gamma,g}(X_{g,m},\dd x).\\
& M^{\partial}_{\gamma,g'}(X_{g',m},\dd x)= e^{\frac{\gamma}{4}Q \omega(x)} M^{\partial}_{\gamma,g}(X_{g,m},\dd x)
\end{align*}
Now we define the {\bf vertex operators},
For $z\in\Sigma^{\circ}$, we define the bulk vertex operator $(\alpha,z)$ (or insertion) as
\begin{equation*}
V_{\alpha,g,\eps}(x)=\eps^{\frac{\alpha^2}{2}}  e^{\alpha  \phi_{g,\epsilon}(x) } .
\end{equation*}
For $s\in\partial\Sigma$, we define the boundary vertex operator $(\beta,s)$ (or insertion) as 
\begin{equation*}
V_{\frac{\beta}{2},g,\eps}(s)=\eps^{\frac{\beta^2}{4}}  e^{\frac{\beta}{2}  \phi_{g,\epsilon}(s) } .
\end{equation*}
Notice that in the case of a surface with Dirichlet boundary and if $g'=e^{\omega}g$, then we have
\begin{align}\label{scalingvertex}
& V_{\alpha,g',\eps}(x)=(1+o(1))e^{\frac{\alpha^2}{4}\omega(x)}V_{\alpha,g,\eps}(x)  \\
& V_{\frac{\beta}{2},g',\eps}(s)=(1+o(1))e^{\frac{\beta^2}{8}\omega(x)}V_{\frac{\beta}{2},g,\eps}(s)
\end{align}
when $\eps$ goes to $0$. 

{\bf Liouville path integral for Riemann surface with Neumann boundary}

Suppose $(\Sigma,g)$ is a compact Riemann surface with Neumann boundaries, for each  bounded continuous functional $F:  H^{-s}(\Sigma)\to\R$ (with $s>0$, one can introduce the following $\eps$-regularized functional
\begin{align}\label{actioninsertion}
 \Pi_{\gamma, \mu,\mu_\partial}^{{\bf z},{\boldsymbol  \alpha},{\bf s},{\boldsymbol  \beta}}  (g,F,\eps):=&Z_{GFF,N}(\Sigma,g)
  \int_\R  \E\Big[ F(\phi ) \big(\prod_{i=1}^n V_{\alpha_i,g,\eps}(z_i)\prod_{j=1}^m V_{\frac{\beta_j}{2},g,\epsilon}(s_j)\big) \times  \\  
  &\exp\Big( -\frac{Q}{4\pi}\int_{\Sigma}K_{g}\phi \,{\rm dv}_{g} -\frac{Q}{2\pi}\int_{\partial \Sigma}k_{g}\phi{\rm d\lambda}_g- \mu  M_{\gamma, g}(\phi,\Sigma)-\mu_{\partial} M^{\partial}_{\gamma, g}(\phi,\partial \Sigma) \Big) \Big]\,dc .
 \nonumber
\end{align}
This limit exists and is nontrivial when the following so-called Seiberg bound satisfies
\begin{align}
 &\sum_{i}\alpha_i +\sum_j\frac{\beta_j}{2} >Q\chi(\Sigma) \label{seiberg1}\\
&\forall i,\quad \alpha_i<Q \label{seiberg2}\\ &\forall j,\quad \beta_j<Q.\label{seiberg3}
 \end{align}
Here $\chi(\Sigma)=2-2\mathrm{g}-k$, $k$ is the number of boundary components. The main properties of this functional are the following, which show the Liouville quantum field theory is actually a conformal field theory.
 \begin{proposition}\label{CA1}{\bf (Conformal anomaly)}
Let 
$g_0$ and $g$ be two metrics which are in the same conformal class, i.e. $ g=e^{\omega}g_0$ for some $\omega\in C^\infty(\Sigma)$. Then we have
\begin{align}\label{confan} 
\log\frac{\Pi_{\gamma, \mu,\mu_\partial}^{{\bf z},{\boldsymbol  \alpha},{\bf s},{\boldsymbol  \beta}}  (g,F)}{\Pi_{\gamma, \mu,\mu_\partial}^{{\bf z},{\boldsymbol  \alpha},{\bf s},{\boldsymbol  \beta}}  (g_0,F(\cdot-\tfrac{Q}{2}\omega))}&= 
\frac{1+6Q^2}{96\pi}\Big(\int_{\Sigma}(|d\omega|^2_{g_0}+2R_{g_0}\omega) {\rm dv}_{g_0}+4 \int_{\partial \Sigma} k_{g_0}\omega d\lambda_{g_0}\Big)\nonumber\\
&-\sum_{i=1}^n\Delta_{\alpha_i}\omega(z_i)-\sum_{j=1}^m\frac{1}{2}\Delta_{\beta_j}\omega(s_j)
\end{align}
\end{proposition}
\begin{proposition}{\bf (Diffeomorphism invariance)}\label{DF1}
Let $\psi:\Sigma\to \Sigma$ be an orientation preserving diffeomorphism, $\psi(\partial\Sigma)=\partial\Sigma$. Then  
\[ \Pi_{\gamma, \mu,\mu_\partial}^{{\bf z},{\boldsymbol  \alpha},{\bf s},{\boldsymbol  \beta}}     (\psi^*g ,F)= \Pi_{\gamma, \mu,\mu_\partial}^{{\bf \psi(z)},{\boldsymbol  \alpha},{\bf \psi(s)},{\boldsymbol  \beta}}   (g,F(\cdot \circ \psi)) .\]
\end{proposition}
By taking $F=1$, this defines the  Liouville correlation function $\left\langle\prod_{i=1}^{n} V_{\alpha_{i}}\left(z_{i}\right)\prod_{j=1}^{m} V_{\frac{\beta_j}{2}}\left(s_{j}\right)\right\rangle^{\Sigma}_{\mu,\mu_\partial}$.

\subsection{The LCFT for Sphere, Disk and Annulus}\label{definition of disk}
In this section, we gather the precise LCFT on sphere, disk (half-upper plane) and annulus, which are used in this paper.

{\bf The LCFT on the sphere}\\
We  clarify the relation between 3-point correlation function on sphere and the DOZZ formula \eqref{theDOZZformula}. For $\omega:\hat\C\to \R$ defined by $e^\omega g= |z|_+^{-4}|dz|^2$ (we call $g_{dozz}=|z|_+^{-4}|dz|^2$ the DOZZ metric).\footnote{Here the factor $\frac{1}{2}$ before DOZZ factor is due to the renormalization of path integral, see \cite{KRV20} }
\begin{align}
\label{DOZZ3point}
 \langle  &V_{\alpha_1}(z_1)  V_{\alpha_2}(z_2) V_{\alpha_3}(z_3)  \rangle^{\hat\C}_{g,\mu}\\
  =&|z_1-z_3|^{2(\Delta_{\alpha_2}-\Delta_{\alpha_1}-\Delta_{\alpha_3})}|z_2-z_3|^{2(\Delta_{\alpha_1}-\Delta_{\alpha_2}-\Delta_{\alpha_3})}|z_1-z_2|^{2(\Delta_{\alpha_3}-\Delta_{\alpha_1}-\Delta_{\alpha_2})}
\Big(\prod_{i=1}^3g(z_i)^{-\Delta_{\alpha_i}}\Big)\nonumber \\
&\frac{1}{2}  C_{\gamma,\mu}^{{\rm DOZZ}} (\alpha_1,\alpha_2,\alpha_3 ) \big(\frac{{\rm v}_{g}(\hat\C)}{{\det}'(\Delta_{g})}\big)^{\frac{1}{2}} e^{-6Q^2 S_{\rm L}^0(\hat\C,g,\omega)}  \nonumber
 \end{align} 
 Here 
 \begin{equation}\label{SL0}
S_{\rm L}^0(\hat{\C},g,\omega):=\frac{1}{96\pi}\int_{\hat{\C}}(|d\omega|_g^2+2K_g\omega) {\rm dv}_g
\end{equation}
 We will also consider another metric, called the round metric on the sphere. We identify  the Riemann sphere with the extended complex plane $\hat\C$ by stereographic projection. On the sphere, every metric is (up to diffeomorphism) conformal to the round metric $g_0:=\frac{4}{(1+|z|^2)^2}|dz|^2$. 
 
{\bf The LCFT on the half upper plane}\\
In this paper, we need the bulk-boundary correlation function and bulk 1-point function, for $\omega:H\to \R$ defined by $e^\omega g= |z|_+^{-4}|dz|^2$ and $z\in\H$ and $s\in\R$ we have
\begin{align}\label{DOZZ2point}
 &\langle V_{\alpha}(z)  V_{\frac{\beta}{2}}(s)  \rangle^{\H}_{g,\mu,\mu_\partial}\\
 =&|z-\bar{z}|^{\Delta_{\beta}-2 \Delta_{\alpha}}|z-s|^{-2 \Delta_{\beta}}
g(z)^{-\Delta_{\alpha}}g(s)^{-\frac{\Delta_{\beta}}{2}}\nonumber G(\alpha,\beta) \big(\frac{{\rm v}_{g}(\H)}{{\det}'(\Delta_{g,N})}\big)^{\frac{1}{2}} e^{-6Q^2 S_{\rm L}^0(\H,g,\omega)}  \nonumber
 \end{align} 
\begin{equation}\label{SL0H}
S_{\rm L}^0(\H,g,\omega):=\frac{1}{96\pi}\Big(\int_{\H}(|d\omega|_g^2+2K_g\omega) {\rm dv}_g+4\int_{\R} k_{g}\omega d{\rm \lambda}_{g}\Big)
\end{equation}
 We will also consider another metric on $\H$, which is the restriction of $g_0$ from $\hat{\C}$. On the upper half-plane, every metric is (up to diffeomorphism) conformal to the round metric $g_0:=\frac{4}{(1+|z|^2)^2}|dz|^2$. The above formula $\eqref{DOZZ2point}$ gives the definition of bulk-boundary correlator $G(\alpha,\beta)$ (this function doesn't depend on the position of the insertions but only on their weights) in the Seiberg bound  $\{(\alpha,\beta)|\alpha+\frac{\beta}{2}>Q \text{ and }\alpha,\beta<Q\}.$
 
 It's easy to see the path integral \eqref{actioninsertion} is divergent for bulk 1-point function with weight $\alpha<Q$. So we need some analytic continuation and define it as 
 \begin{align*}
\langle V_{\alpha}(z) \rangle^{\H}_{g,\mu,\mu_\partial}:=Z_{GFF,N}(\H,g)
  \int_\R  \E\Big[  V_{\alpha,g,}(z)e^{ -\frac{Q}{4\pi}\int_{\Sigma}K_{g}\phi \,{\rm dv}_{g} -\frac{Q}{2\pi}\int_{\partial \Sigma}k_{g}\phi{\rm d\lambda}_g}\Big(e^{- \mu  M_{\gamma, g}(\phi,\Sigma)-\mu_{\partial} M^{\partial}_{\gamma, g}(\phi,\partial \Sigma)}-1\Big) \Big]\,dc .
 \end{align*}
 then \begin{align}\label{probfzz}
 \langle V_{\alpha}(z)   \rangle^{\H}_{g,\mu,\mu_\partial}
 =|z-\bar{z}|^{-2 \Delta_{\alpha}}
g(z)^{-\Delta_{\alpha}} U_{\rm FZZ}(\alpha) \big(\frac{{\rm v}_{g}(\H)}{{\det}'(\Delta_{g,N})}\big)^{\frac{1}{2}} e^{-6Q^2 S_{\rm L}^0(\H,g,\omega)}  
 \end{align} 
 \begin{remark}
When we use the notion of the Liouville correlation functions for complex values of $\alpha$ and $\beta$, we mean the RHS of above formulas \eqref{DOZZ3point}, \eqref{DOZZ2point} and \eqref{probfzz}.
 \end{remark}
{\bf The LCFT on the annulus}\\
For annulus $\mathbb{A}_q^p=\{z\in\C\mid q\leq|z|\leq p\}$ with flat metric $g_{\mathbb{A}}=\frac{dz^2}{|z|^2}$. In this metric, both Gaussian curvature and geodesic curvature on boundary are 0. The moduli space of annulus is parameterized by $\frac{q}{p}\in (0,1)$.\\
In the annulus case, we define two bounds of insertion weights:
\begin{itemize}
    \item Amplitude bound: for an annulus with one Dirichlet boundary $\partial^D\mathbb{A}$ and one Neumann boundary $\partial^N\mathbb{A}$. For $(\alpha_i,z_i)\in \mathbb{A}^{\circ}$ and $(\beta_j,s_j)\in \partial^N\mathbb{A}$.  
    \begin{equation}\label{annulusamplitudebound}
        \sum_i\alpha_i+\sum_j\frac{\beta_j}{2}>0 \text{   and    } \alpha_i,\beta_j<Q.
    \end{equation}
    \item Seiberg bound: for an annulus with two Neumann boundaries $\partial^N\mathbb{A}$. For $(\alpha_i,z_i)\in \mathbb{A}^{\circ}$ and $(\beta_j,s_j)\in \partial^N\mathbb{A}$.  
    \begin{equation}\label{annulusseibergbound}
         \sum_i\alpha_i+\sum_j\frac{\beta_j}{2}>0 \text{   and    } \alpha_i,\beta_j<Q.
    \end{equation}
\end{itemize}
\section{Amplitudes with Neumann boundary in the boundary LCFT}\label{section amplitude}
First, we recall the following decomposition of Neumann GFF $X^{\Sigma}_{g,N}$ from \cite[section 7]{Wu1}. We note $m_g(f)=\frac{1}{{\rm v}_g(\Sigma)}\int f d{\rm v}_g.$
\begin{proposition}\label{boundarymarkov}
Suppose $\Sigma$ is a compact Riemann surface with Neumann boundary $\partial^N\Sigma$. For the Neumann GFF $X^{\Sigma}_{g,N}$, the following decomposition holds in law
\begin{equation}\label{Marg}
X^{\Sigma}_{g,N}\stackrel{law}{=}X^{\Sigma_1}_{g,m}+X^{\Sigma_2}_{g,m}+P\varphi^{\Sigma}_{g}-c^{\Sigma}_{g}  
\end{equation}
where $\varphi^{\Sigma}_{g}$ is the restriction of $X^{\Sigma}_{g,N}$ on the circles $\mathcal{C}$ of $\Sigma$, which separate $\Sigma$ into two parts $\Sigma_1$ and $\Sigma_2$, and $X^{\Sigma_i}_{g,m}$ is the mixed boundary GFF on $\Sigma_i$, i.e. $X^{\Sigma_i}_{g,m}=0$ on $\mathcal{C}$ and $\partial^g_{n}X^{\Sigma_i}_{g,m}=0$ on $\partial^N \Sigma_i=\partial^N\Sigma\cap\Sigma_i$. $P\varphi^{\Sigma}_{g}$ is the harmonic extension of $\varphi^{\Sigma}_{g}$ with $\partial^g_{n}P\varphi^{\Sigma}_{g}=0$ on $\partial^N \Sigma_i$ and $c^{\Sigma}_{g}=m^{\Sigma}_{g}(X^{\Sigma_1}_{g,m}+X^{\Sigma_2}_{g,m}+P\varphi^{\Sigma}_{g})$. Here $X^{\Sigma_1}_{g,m}$, $X^{\Sigma_2}_{g,m}$ and $P\varphi^{\Sigma}_{g}$ are all independent.
\end{proposition}


\subsection{Properties of amplitudes with Neumann boundary}
For every compact Riemann surface $\Sigma$ with boundaries $\partial \Sigma$. We separate the boundary components of $\Sigma$ into two kinds, Dirichlet boundaries $\partial_{i}^D \Sigma$, $1\leq i \leq \ell^\Sigma_D$ and Neumann boundaries $\partial_{j}^N \Sigma$, $1\leq j\leq \ell^\Sigma_N$. Note Dirichlet boundary parametrization as $\bm{\zeta}=(\zeta_1,...,\zeta_{\ell^\Sigma_D})$, where $\zeta_i:\T\to \partial_{i}^D \Sigma$ parametrizes the boundary $\partial_{i}^D \Sigma$. We define the mixed boundary condition Gaussian free field $X^\Sigma_{g,m}$ on $\Sigma$, with Dirichlet $0$ boundary condition on $\partial^D_i \Sigma$ and Neumann $0$ boundary condition on $\partial^N_j \Sigma$. 

For $\bm{\tilde{\varphi}}=(\tilde{\varphi}_1,\tilde{\varphi}_2,...,\tilde{\varphi}_{\ell^\Sigma_D})\in H^s(\T)^{\ell^\Sigma_D}$ $(s<0)$, we define $\bm{\tilde{\varphi}}\circ\bm{\zeta}^{-1}$ on $\partial^D \Sigma$ and $\bm{P\tilde{\varphi}}$ as its harmonic extension with Neumann $0$ condition on $\partial^N \Sigma$.
Now we define the {\bf amplitude with Neumann boundary} $\mathcal{A}_{g,\Sigma,\bm{z},\bm{s}, \bm{\alpha},\bm{\beta},\bm{\zeta}}(\bm{\tilde{\varphi}})$ as following:
\begin{definition}[{Amplitudes with Neumann boundary}]\label{typeAAmplitude}

Suppose the $\alpha_i$ and $\beta_j$ satisfy the Seiberg bounds \eqref{seiberg2} and \eqref{seiberg3}.\\
\noindent {\bf (1) }
When $\partial \Sigma^D=\emptyset$, we also suppose \eqref{seiberg1} holds, then for any $F$ continuous  nonnegative function on $H^{p}(\Sigma)$ for some $p<0$, we define $\phi_g=c+X_g^{\Sigma}$ and
\[\mathcal{A}^N_{g,\Sigma,\bm{z},\bm{s}, \bm{\alpha}, \bm{\beta}}(F):=\left\langle F(\phi_g)\prod_{i=1}^n V_{\alpha_i}(z_i)  \prod_{j=1}^mV_{\frac{\beta_j}{2}}(s_j)\right\rangle^\Sigma_{\gamma,\mu, \mu_\partial}\]

\noindent {\bf (2) }
If  $\partial^D\Sigma$ has a positive number ${\ell^\Sigma_D}=b$ boundary  components,
\begin{align}
    &\mathcal{A}^m_{g,\Sigma,\bm{z},\bm{s}, \bm{\alpha}, \bm{\beta},\zeta}(F,\bm{\tilde{\varphi}}) :=Z_{GFF,m}(\Sigma,g)\mathcal{A}^0_{\Sigma,g}(\bm{\tilde{\varphi}})\times\E \Big[F(\bm{P\tilde{\varphi}}+X^\Sigma_{g,m})\prod_{i=1}^n V_{\alpha_i}(z_i)  \prod_{j=1}^mV_{\frac{\beta_j}{2}}(s_j)\times\nonumber\\ &\exp\Big( -\frac{Q}{4\pi}\int_{\Sigma}K_{g}(\bm{P\tilde{\varphi}}+X^\Sigma_{g,m} )\,{\rm dv}_{g} -\frac{Q}{2\pi}\int_{\partial \Sigma}k_{g}(\bm{P\tilde{\varphi}}+X^\Sigma_{g,m}) d\lambda_g -\mu M_{\gamma, g}(\Sigma)-\mu_{\partial} M^{\partial}_{\gamma, g}(\partial^N \Sigma) \Big) \Big]   
\end{align} where $\bm{z}\in \Sigma^{\circ}$, $\bm{s}\in \partial^N \Sigma$ and
\begin{align*}
&Z_{GFF,m}(\Sigma,g)={\det}(\Delta_{g,m})^{-1/2} \exp(-\frac{1}{8\pi}\int_{\partial^N \Sigma} k_g {\rm d\lambda}_g+\frac{1}{8\pi}\int_{\partial ^D\Sigma} k_g {\rm d\lambda}_g)\\
  &M_{\gamma, g}(\Sigma)=\lim _{\varepsilon \rightarrow 0} \int_{\Sigma}\varepsilon^{\frac{\gamma^{2}}{2}} e^{\gamma\left(X^\Sigma_{g,m}+\bm{P\tilde{\varphi}} \right)} d v_{g}\\
  &M^{\partial}_{\gamma, g}(\partial^N \Sigma)=\lim _{\varepsilon \rightarrow 0} \int_{\partial^N \Sigma} \varepsilon^{\frac{\gamma^{2}}{4}} e^{\frac{\gamma}{2}\left(X^\Sigma_{g,m}+\bm{P\tilde{\varphi}}\right)} d \lambda g
\end{align*}
 To define the free amplitude $\mathcal{A}^0_{\Sigma,g}(\bm{\tilde{\varphi}})$, we need to introduce the Dirichlet-Neumann operator $\mathbf{D}^N_{\Sigma}$ for surface $\Sigma$, $\partial^D\Sigma$ is parametrized by $(\zeta_1,..,\zeta_b)$
 \begin{align}
\mathbf{D}_{\Sigma}^{N}: C^{\infty}\left(\T\right)^b & \rightarrow C^{\infty}\left(\T\right)^b \\
\widetilde{\bm{\varphi}} & \longmapsto(-\partial_{n } P\tilde{\boldsymbol{\varphi}}_{|\mc{C}_j}\circ\zeta_j)_{j=1,\dots,b}
\end{align}
where $n$ is the inward unit normal vector fields to $\mc{C}_j$.
Now we define the free amplitude as $\mathcal{A}^0_{\Sigma,g}(\bm{\tilde{\varphi}})=e^{-\frac{1}{2}\left(\widetilde{\varphi},\left(\mathrm{D}^N_{\Sigma}-\mathrm{D}\right) \widetilde{\varphi}\right)_{2}}$, where 
for all $ \widetilde{\varphi} \in C^{\infty}(\T)^b$, $$ (\mathbf{D} \widetilde{\varphi}, \widetilde{\varphi}):=2 \sum_{j=1}^{\ell^\Sigma_D} \sum_{n>0} n\left|\varphi_{n}^{j}\right|^{2}=\frac{1}{2} \sum_{j=1}^{\ell^\Sigma_D} \sum_{n>0}\left(\left(x_{n}^{j}\right)^{2}+\left(y_{n}^{j}\right)^{2}\right).$$ We will also use the notation $\phi_g=X^\Sigma_{g,m}+\bm{P\tilde{\varphi}}$ for simplicity.
\end{definition}
\begin{remark}
If $\partial \Sigma^N=\emptyset$, the above amplitudes reduce to the amplitudes defined in \cite[section 4]{GKRV21}.
\end{remark}
\begin{proposition}[Conformal anomaly for $Z_{GFF,m}$]
Suppose ${g}=e^{\omega}g_0$, then
\[ \log\Big(\frac{Z_{GFF,m}(\Sigma,g)}{Z_{GFF,m}(\Sigma,g_0)}\Big)=\frac{1}{48\pi}\Big(\int_{\Sigma}(|d\omega|^2_{g_0}+2R_{g_0}\omega) {\rm dv}_{g_0}+4 \int_{\partial^N \Sigma} k_{g_0}\omega d\lambda_{g_0}\Big) \]
\end{proposition}
\begin{proof}
One can show this by directly analyzing the heat kernel, see \cite{Wu1} for the Neumann case or the Dirichlet case, the strategies are essentially same.
\end{proof}
Now we introduce the function space $ H^1_0(\Sigma,\partial^D\Sigma,g)=:\left\{\omega \in H^{1}(\Sigma, g) \mid \omega \text { = } 0 \text { on } \partial^{D} \Sigma\right\}$.
\begin{proposition}[{\bf conformal anomaly}]\label{Weyl2}

Let $g=e^{\omega}g_0$ and  $\omega\in H^1_0(\Sigma,\partial^D\Sigma,g)$ if $\partial^D\Sigma\not=\emptyset$, or $\omega\in H^1(\bar\Sigma,g)$ if $\partial^D\Sigma=\emptyset$. Then
\begin{align*}
&\mathcal{A}_{g,\Sigma,\bm{z},\bm{s}, \bm{\alpha}, \bm{\beta},\zeta}(F, \bm{\tilde{\varphi}})=\mathcal{A}_{g_0,\Sigma,\bm{z},\bm{s}, \bm{\alpha}, \bm{\beta},\zeta}(F(\cdot-\frac{Q}{2}\omega), \bm{\tilde{\varphi}})\times\\ &\exp\Big(\frac{1+6Q^2}{96\pi}\big(\int_{\Sigma}(|d\omega|_{g_0}^2+2R_{g_0}\omega) {\rm dv}_{g_0}+4 \int_{\partial^N \Sigma} k_{g_0}\omega d\lambda_{g_0}\big)-\sum_i\Delta_{\alpha_i}\omega(z_i)-\frac{1}{2}\sum_j\Delta_{\beta_j}\omega(s_j) \Big).
\end{align*}
\end{proposition}
\begin{proof}
We only need to treat the case for $\partial^D\Sigma\not=\emptyset$. For $\partial^D\Sigma=\emptyset$ case, see \cite{Wu1} for details. First,  we simply remark that $\mathcal{A}^0_{\Sigma,g}(\bm{\tilde{\varphi}})=e^{-\frac{1}{2}\left(\widetilde{\varphi},\left(\mathrm{D}^N_{\Sigma}-\mathrm{D}\right) \widetilde{\varphi}\right)_{2}}$ doesn't depend on the metrics. Because of the relation $k_{g}=e^{-\omega/2}(k_{g_0}-\partial^{g_0}_n\omega/2)$, the boundary term  in  \eqref{typeAAmplitude} becomes
\begin{equation}\label{boundgeod}
\frac{Q}{2\pi}\int_{\partial\Sigma}k_{g}(P\boldsymbol{\tilde\varphi}+X^\Sigma_{g,m})\dd \ell_{g}=\frac{Q}{2\pi}\int_{\partial^D\Sigma}(k_{g_0}-\partial^{g_0}_n\omega/2)P\boldsymbol{\tilde\varphi}\dd \ell_{g_0}+\frac{Q}{2\pi}\int_{\partial^N\Sigma}(k_{g_0}-\partial^{g_0}_n\omega/2)(P\boldsymbol{\tilde\varphi}+X^\Sigma_{g,m})\dd \ell_{g_0}. 
\end{equation}
Also, the relation for curvatures $K_{g}=e^{-\omega}(K_{g_0} +\Delta_{g_0}\omega)$ allow us deduce that the term involving the curvature in  \eqref{typeAAmplitude} reads (recall that $X^\Sigma_{g,m}=X^\Sigma_{g_0,m}$ when $\partial^D\Sigma\not=\emptyset$, harmonic extension doesn't depend on which metric be chosen in the conformal class neither).
\begin{align*}
\frac{Q}{4\pi}\int_\Sigma K_{g} (X^\Sigma_{g,m} +P\boldsymbol{\tilde\varphi})\dd {\rm v}_{g}=& \frac{Q}{4\pi}\int_\Sigma  K_{g_0} (X^\Sigma_{g_0,m}+P\boldsymbol{\tilde\varphi})\dd {\rm v}_{g_0}+ \frac{Q}{4\pi}\int_\Sigma  \Delta_{g_0}\omega (X^\Sigma_{g_0,m}+P\boldsymbol{\tilde\varphi})\dd {\rm v}_{g_0}.
\end{align*}
In the last integral above, the contribution of the harmonic extension is treated  by the Green identity to produce
$$ \frac{Q}{4\pi}\int_\Sigma  \Delta_{g_0}\omega  P\boldsymbol{\tilde\varphi} \dd {\rm v}_{g_0}=\frac{Q}{4\pi}\int_{\partial^D \Sigma}\partial^{g_0}_n \omega P\boldsymbol{\tilde\varphi}\,\dd \ell_{g_0}+\frac{Q}{4\pi}\int_{\partial^N \Sigma}\partial^{g_0}_n \omega P\boldsymbol{\tilde\varphi}\,\dd \ell_{g_0},$$
which cancels out with the corresponding harmonic extension terms in \eqref{boundgeod}, and the remaining term is $\int_{\partial \Sigma}k_{g_0}(\bm{P\tilde{\varphi}}+X^\Sigma_{g_0,m}) d\lambda_{g_0}$ . Then we apply Girsanov theorem to the exponential term\\ $exp(-\frac{Q}{4\pi}(\int_\Sigma  \Delta_{g_0}\omega X^\Sigma_{g_0,m} \dd {\rm v}_{g_0}-\int_{\partial \Sigma}\partial^{g_0}_n\omega X^\Sigma_{g_0,m}\dd \ell_{g_0}))$ to get the result. The term $1$ in $1+6Q^2$ comes from adding the weyl anomaly of $Z_{GFF,m}$.
\end{proof}
\begin{proposition}[{\bf Diffeomorphism invariance}]\label{Diff2}
Let $\psi:\Sigma\to\Sigma'$ be an  orientation preserving diffeomorphism and $F$  measurable  nonnegative function on $H^{-s}(\Sigma)$. Setting $F_\psi(\phi):=F(\phi\circ\psi)$,  we have
$$\mathcal{A}_{\psi^*g,\Sigma,\bm{z},\bm{s}, \bm{\alpha}, \bm{\beta},\zeta}(F, \bm{\tilde{\varphi}})=\mathcal{A}_{g,\Sigma',\psi(\bm{z}),\psi(\bm{s}), \bm{\alpha}, \bm{\beta},\psi\circ\zeta}(F_\psi, \bm{\tilde{\varphi}})$$
\end{proposition}
To study the gluing property, we specialize to some special metrics $g$ which are called {\bf admissible metric}. 
\begin{definition}
A metric $g$ is called an admissible metric if it behaves like $\frac{|dz|^2}{|z|^2}$ near $\partial^D\Sigma$ and has geodesic curvature $0$ on $\partial^N\Sigma$. In this case, we call $(\Sigma, g,\bm{z},\bm{s}, \bm{\alpha}, \bm{\beta}) $ an admissible marked surface.
\end{definition}
Suppose $\ell^{\Sigma}_D=b$, for boundary fields $\tilde{\boldsymbol{\varphi}}=(\tilde\varphi_1,\dots,\tilde\varphi_b)\in (H^{s}(\T))^b$, we will denote by $(c_j)_{j\leq b}$ their constant modes, set $ \bar c=\tfrac{1}{b}\sum_{j=1}^b c_j$ and $\boldsymbol{\varphi}=(\varphi_1,\dots,\varphi_b)$ the centered fields (recall that $\tilde\varphi_j=c_j+\varphi_j$). Also, recall that $c_+=c\mathbf{1}_{\{c>0\}}$ and $c_-=c\mathbf{1}_{\{c<0\}}$.  The following result will be useful to control the integrable properties of amplitudes:
 
\begin{proposition}\label{estimateamplitude}
Let $(\Sigma, g,\bm{z},\bm{s}, \bm{\alpha}, \bm{\beta}) $ be an admissible surface with $\partial^D \Sigma \neq \varnothing$ witch has $b$ Dirichlet boundary components. Then there exists some constant $a>0$ such that for any $R>0$, there exists $C_{R}>0$ such that $\mu_{0}^{\otimes_{b}}$-almost everywhere in $\widetilde{\varphi} \in H^{s}(\mathbb{T})^{b}$
$$
\mathcal{A}_{g,\Sigma,\bm{z},\bm{s}, \bm{\alpha}, \bm{\beta},\zeta}( \bm{\tilde{\varphi}}) \leqslant C_{R} e^{s \bar{c}_{-}-R \bar{c}_{+}-a \sum_{j=1}^{b}\left(\bar{c}-c_{j}\right)^{2}} A(\varphi)
$$
satisfying $\int A(\varphi)^{2} \mathbb{P}_{\mathbb{T}}^{\otimes_{b}}(\mathrm{d} \varphi)<\infty$, where
$$
s:=\sum_{i} \alpha_{i}+\sum_{j} \frac{\beta_j}{2}-Q \chi(\Sigma).
$$
\end{proposition}
\begin{proof}
We suppose $\partial^N\Sigma\neq\emptyset$, otherwise this proposition was already treated in \cite[Theorem 4.4]{GKRV21}. We separate our proof into three steps:

{\bf 1. The free amplitude part.} Since $\partial^N\Sigma$ are geodesics, we can double this surface along these geodesics to get a new surface $d\Sigma$ with $2b$ Dirichlet boundary components. Then we put $2b$ fields $(\tilde{\boldsymbol{\varphi}}, \tilde{\boldsymbol{\varphi}})=(\tilde\varphi_1,\dots,\tilde\varphi_b, \tilde\varphi_1,\dots,\tilde\varphi_b)$ on these boundaries (we put the same boundary fields on the corresponding symmetric locations). We have the identification of DN operators $D^N_{\Sigma}=D_{d\Sigma}$ on $ C^{\infty}(\partial^D\Sigma)$. By \cite[Lemma 4.5]{GKRV21}, we get 
\begin{align*}
\caA^0_{\Sigma,g}(\tilde{\boldsymbol{\varphi}})=(\caA^0_{d\Sigma,g}(\tilde{\boldsymbol{\varphi}}, \tilde{\boldsymbol{\varphi}}))^{\frac{1}{2}}\leq  \prod_{j=1}^be^{-a( \bar c-c_j)^2}\prod_{n=1}^\infty e^{a_n(\frac{(x_{j,n})^2+(y_{j,n})^2}{2})} 
\end{align*}
where $a_n\leq (\frac{1}{4}-a)$ for $n\leq k_N$ and $a_n\leq \frac{1}{4n^{N}}$ for $n\geq k_N$.\\
{\bf 2. The GMC part.} In the second step, we bound the expectation part in \eqref{typeAAmplitude}. Since $g$ is admissible, we have $k_g=0$ on the  all boundaries $\partial\Sigma$. We first bound the harmonic extension $P\tilde{\boldsymbol{\varphi}}$,  we double $\Sigma$ along geodesic boundaries $\partial^N\Sigma$, then $P\tilde{\boldsymbol{\varphi}}$ can be identified with the harmonic extension of putting $\tilde{\boldsymbol{\varphi}}$ on $d\Sigma$, $P\tilde{\boldsymbol{\varphi}}=\bar c+f+P\boldsymbol{\varphi}$ where $f$ is the harmonic extension on $d\Sigma$ of the boundary fields $c_j-\bar c$ on the boundary curve $\partial^Dd\Sigma$. By the maximum principle, $|f(x)|\leq \sum_{j=1}^b |c_j-\bar c|$ (since the constants $c_j$ are equal on the symmetric positions). The following estimations are routine, we first apply Girsanov transform on the product
\[e^{-\frac{Q}{4\pi}\int_\Sigma K_g Y_g^{\Sigma}\dd {\rm v}_g}\prod_{i=1}^n V_{\alpha_i,g}(z_i)\prod_{j=1}^m V_{\frac{\beta_j}{2},g}(s_j).\]
we get
\begin{equation}\label{eq:boundexp}
e^{v(\mathbf{z},\mathbf{s})}\prod_{i=1}^n e^{\alpha_i P\tilde{\boldsymbol{\varphi}}(x_i)} \prod_{j=1}^m e^{\frac{\beta_j}{2} P\tilde{\boldsymbol{\varphi}}(s_j)}\E \big[ \exp\big(-\mu \int_\Sigma e^{\gamma u(\mathbf{z},\mathbf{s},x)}M_{\gamma,g}(\phi_g,\dd x)\big)-\mu_\partial \int_{\partial\Sigma} e^{\frac{\gamma}{2} u(\mathbf{z},\mathbf{s},x)}M^{\partial}_{\gamma,g}(\phi_g,\dd x)\big]
\end{equation} 
for some explicit function $u(\mathbf{z},\mathbf{s},x)$ and $v(\mathbf{z},\mathbf{s})$.

When $\bar{c}<0$, we can simply bound the expectation term by $1$. 

When $\bar{c}>0$, we can suppose $\mu=0$, and remove a open neighbourhood $O$ contains all $s_j$ on $\partial\Sigma$. By the strategy in \cite[Lemma 4.6]{GKRV21}, we get for all $R>0$, there exists some $C_R$, $a_R>0$ such that 
$$\eqref{eq:boundexp}\leq  C_R e^{(\sum_i\alpha_i+\sum_j\frac{\beta_j}{2}) \bar c_--R \bar c_++a_R\sum_{j=1}^b |c_j-\bar c|}B(\boldsymbol{\varphi})$$
where $\int  B(\boldsymbol{\varphi})^p \,\P_\T^{\otimes_b}(\dd \boldsymbol{\varphi})<\infty$ for all $p<\infty$.\\
{\bf 3. The contribution of the harmonic extension in curvature terms.} This part is essentially the following lemma on Poisson operator, which can be proved by using mirror trick on \cite[Lemma A.1]{GKRV21}
\begin{lemma}
Let $(\Sigma,g)$ be a smooth Riemannian surface with smooth boundary and the geodesic curvature of $\partial^N\Sigma$ is 0. Let $P:C^\infty(\pl^D \Sigma)\to C^\infty(\Sigma)$ be the Poisson operator defined by $\Delta_gP\varphi=0$ and $(P\varphi)|_{\pl^D \Sigma}=\varphi$, $\partial_n^g(P\varphi)|_{\pl^N \Sigma}=0$, and denote by $P(x,y)$ its integral kernel.  Let $P^*:C_c^\infty(\Sigma^\circ)\to C^\infty(\pl^D \Sigma)$ 
be the adjoint operator defined by $P^*f(y):=\int_{\Sigma}P(x,y)f(x)d{\rm v}_g(x)$. Then $P^*$ extends as a continuous map $P^*:C^\infty(\Sigma)\to C^\infty(\pl^D \Sigma)$.
\end{lemma}
Then we can conclude the proof by following the last few lines in \cite[Theorem 4.4]{GKRV}.
\end{proof}
The proposition \eqref{estimateamplitude} shows that, when $s>0$ the amplitude is in the space $e^{\epsilon c_-}L^2(\R\times\Omega_{\T})$ for all $0<\epsilon<s$.
\subsection{Gluing of Amplitudes with Neumann boundary.}
{\bf The Dirichlet-Neumann operator}

Suppose $\Sigma$ is a connected Riemann surface with boundaries $\partial^N \Sigma$. Now we cut along circles $\mathcal{C}$ in the interior which divide the surface into two surfaces $\Sigma_1$ and $\Sigma_2$ and $\mathcal{C}$ is parametrized by $\zeta:\T\to \mathcal{C} $. Suppose $\tilde{\varphi} \in C^{\infty}(\T)$ and we define the harmonic extension of $\tilde{\varphi}\circ\zeta^{-1}$ inside $\Sigma_1$ and $\Sigma_2$ as following
$$
\left\{\begin{array} { l } 
{ \Delta_{g,m} ( P_1 \tilde{\varphi} ) = 0 \quad \text { in } \Sigma _ { 1 } } \\
{ ( P _ { 1 } \tilde{\varphi} ) | _ { \mathcal{C} } = \tilde{\varphi}\circ\zeta^{-1} } \\
{ \partial^g _ { n } ( P_1 \tilde{\varphi} ) = 0 \quad \text{on }  ( \partial \Sigma _ { 1 } \backslash \mathcal{C} ) }
\end{array} \quad \left\{\begin{array}{l}
\Delta_{g,m}\left(P_{2} \tilde{\varphi}\right)=0 \text { in } \Sigma_{2} \\
\left.\left(P_{2} \tilde{\varphi}\right)\right|_{\mathcal{C}}=\tilde{\varphi}\circ \zeta^{-1} \\
\partial^g_{n}\left(P_{2} \tilde{\varphi}\right)=0 \quad \text{on } (\partial \Sigma_{2} \backslash \mathcal{C}).
\end{array}\right.\right.
$$
Now we define the Dirichlet-Neumann operator $\mathbf{D}^N_{\Sigma, \mathcal{C}}$ as
$$
\begin{aligned}
\left.\mathbf{D}^N_{\Sigma, \mathcal{C}}: C^{\infty} (\T\right) &\left.\longrightarrow C^{\infty} ( \T\right) \\
\tilde{\varphi} & \longmapsto-\left(\partial_{n}^{g,+}\left(P_{1} \tilde{\varphi}\right)+\partial_n^{g,-}\left(P_{2} \tilde{\varphi}\right)\right)\circ\zeta
\end{aligned}
$$
We use $+$, $-$ to represent the direction of normal derivative which point inside $\Sigma_1$ and $\Sigma_2$ respectively. 

First, we notice that $\operatorname{Ker}\left(\mathbf{D}^N_{\Sigma, \mathcal{C}}\right)=\mathbb{R} $ (constants), so it's natural to introduce a smaller space $C_{0}^{\infty}(\T)=\left\{\tilde{\varphi} \in C^{\infty}(\T)| \int_{\T}\tilde{\varphi}(e^{i\theta}) d \theta=0\right\}$.
The fact $\int_{\mathcal{C}}(\partial_{n}^{+}(P_1 \tilde{\varphi})+\partial_{n}^{-}(P_{2} \tilde{\varphi}))=0$ follows from Green formula, we know that $\operatorname{Ran}\left(\mathbf{D}^N_{\Sigma, \mathcal{C}}\right)=C_{0}^{\infty}(\T)$. So $\mathbf{D}^N_{\Sigma, \mathcal{C}}$ is invertible on $C_{0}^{\infty}(\mathcal{C})$. 
We can modify the proof of \cite[Theorem 2.1]{Carron} and show
\begin{proposition}
\begin{equation}\label{DNoperator}
    \mathbf{D}^N_{\Sigma, \mathcal{C}}G^{\Sigma}_{g,N}=2\pi {\bf Id    } \text{      in}\quad C_0^{\infty}(\mathbb{T}).
\end{equation}
\end{proposition}
Let us also introduce some notations in this section
\begin{align*}
\label{tildeD}& \widetilde{\mathbf{D}}^N_{\Sigma}  :=\mathbf{D}^N_{\Sigma}-\mathbf{D},  
&  \Pi_0(\tilde\varphi^1,\dots,\tilde\varphi^b):= (2\pi)^{-1/2} (( \tilde\varphi^1,1)_2,\dots,( \tilde\varphi^b,1)_2) \\
& \widetilde{\mathbf{D}}^N_{\Sigma,{\mc{C}'}} :=\mathbf{D}^N_{\Sigma,{\mc{C}'}}-2\mathbf{D}, 
& \Pi'_0(\tilde\varphi^1,\dots,\tilde\varphi^{b'}):=  (2\pi)^{-1/2} (( \tilde\varphi^1,1)_2,\dots,(\tilde\varphi^{b'},1 )_2).\\
&\mathbf{D}_0=\mathbf{D}+\Pi_0
\end{align*}
Given two admissible surfaces $(\Sigma_1,g_1, \bm{z_1}, \bm{s_1}, \bm{\alpha_1}, \bm{\beta_1})$ and $(\Sigma_2,g_2, \bm{z_2}, \bm{s_2}, \bm{\alpha_2}, \bm{\beta_2})$, we glue them together by identify $k$ boundaries in $\partial^D \Sigma_1$ and $k$ boundaries in $\partial^D \Sigma_2$, here $k\leq min\{\ell^{\Sigma_1}_D, \ell^{\Sigma_2}_D\}$. We put the field on the remaining boundary components of $\partial \Sigma_1^D$ and $\partial \Sigma_2^D$ as $\tilde{\bm{\varphi_1}}$ and $\tilde{\bm{\varphi_2}}$ respectively.

\begin{proposition}[{Gluing formula}]\label{gluingtypeI}
\begin{align}
\mathcal{A}&_{g,\Sigma,\bm{z},\bm{s}, \bm{\alpha}, \bm{\beta}}(F_1\otimes F_2,\tilde{\bm{\varphi_1}},\tilde{\bm{\varphi_2}}) =\nonumber\\&
C\int  \mathcal{A}_{g_1,\Sigma_1,\bm{z_1},\bm{s_1}, \bm{\alpha_1}, \bm{\beta_1}}(F_1,\tilde{\bm{\varphi_1}},\tilde{\bm{\varphi}})\times \mathcal{A}_{g_2,\Sigma_2,\bm{z_2},\bm{s_2}, \bm{\alpha_2}, \bm{\beta_2}}(F_2, \tilde{\bm{\varphi}},\tilde{\bm{\varphi_2}})(d\mu_0)^{\otimes k} (\tilde{\bm{\varphi}}).
\end{align}
where $C= \frac{1}{(\sqrt{2} \pi)^{k}}$ if $\partial^D\Sigma \not =\emptyset$ and $C=\frac{\sqrt{\pi}}{(\sqrt{2} \pi)^{k-1}} $ if $\partial^D\Sigma =\emptyset$, and 
 \[F_1\otimes F_2(\phi_{g_1\# g_2}):=F_1(\phi_{g_1\# g_2|\Sigma_1})F_2(\phi_{g_1\# g_2|\Sigma_2}).\]
\end{proposition}
\begin{figure}[H] 
\centering 
\includegraphics[width=0.5\textwidth]{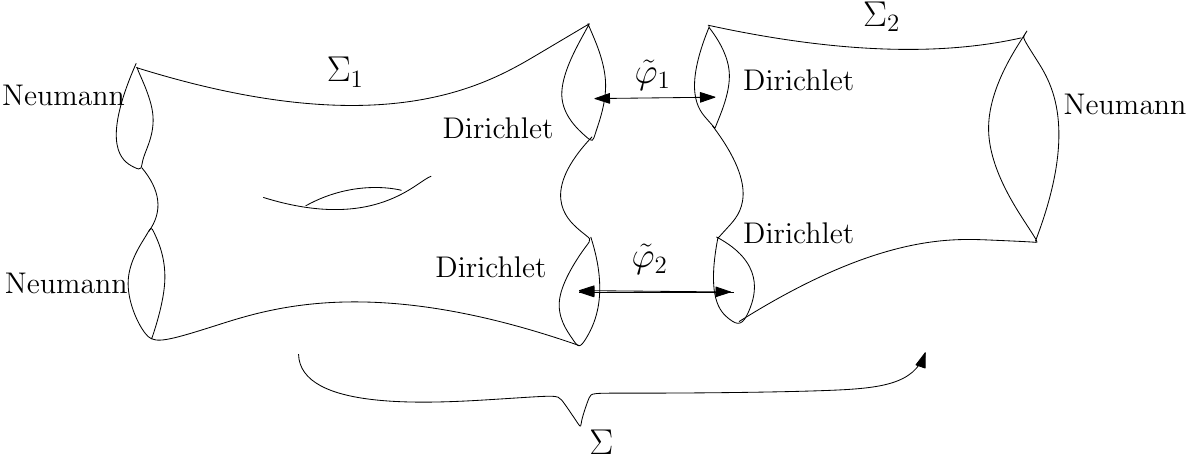} 
\caption{Gluing of two amplitudes with Neumann boundary. The red circles are the Dirichlet boundaries, where we put the fields $\tilde{\varphi}$ for gluing. The black circles are the Neumann boundaries.} 
\label{Fig.main10} 
\end{figure}
\begin{proof}
We only treat the case when $\ell^{\Sigma_1}_D=\ell^{\Sigma_2}_D=1$, $\partial^D\Sigma=\emptyset$ and $F_1=F_2=1$, the general cases are the similar. We glue $\partial^D\Sigma_1$ and $\partial^D\Sigma_2 $ together, since this is the only case we use in this paper. The resulting amplitude is the correlation function
\begin{equation}\label{bcorrelation}
    \mathcal{A}_{g,\Sigma,\bm{z},\bm{s}, \bm{\alpha}, \bm{\beta}}(\bm{\tilde{\varphi}})=\left\langle \prod_{i=1}^n V_{\alpha_i}(z_i)  \prod_{j=1}^mV_{\frac{\beta_j}{2}}(s_j)\right\rangle^\Sigma_{\gamma,\mu, \mu_\partial}
\end{equation}
Let now $Y_1$ and $Y_2$ two  mixed condition GFF respectively on $\Sigma_1$ and $\Sigma_2$, independent of each other. We assume that they are both defined on $\Sigma$ by setting $Y_i=0$ outside of $\Sigma_i$. Then we have the following decomposition in law \eqref{boundarymarkov} 
\begin{align*}
X^{\Sigma}_g\stackrel{\rm law}=Y_1+Y_2+P\boldsymbol{X}-c_g
\end{align*}
where $\boldsymbol{X}$ is the restriction of the GFF $X^{\Sigma}_g$ to the  glued boundary component  $\mathcal{C}=\partial^D\Sigma_1=\partial^D\Sigma_2$ expressed in parametrized coordinates, $P\boldsymbol{X}$ is its harmonic extension to $ \Sigma $ with Neumann 0 boundary condition on $\partial\Sigma$ and $c_g:=\frac{1}{{\rm v}_g(\Sigma)}\int_\Sigma (Y_1+Y_2+P\boldsymbol{X})\,\dd {\rm v}_g$. Therefore, plugging this relation into the amplitude  \eqref{bcorrelation},  and then shifting the $c$-integral by $c_g$, we get      
\begin{align*}
 \mathcal{A}_{\Sigma,g,{\bf z},\boldsymbol{\alpha},\boldsymbol{\beta},\boldsymbol{s}}=&\frac{({\det}'(\Delta_{g,N})/{\rm v}_{g}(\Sigma))^{-1/2}}{Z_{GFF,m}(\Sigma_1)Z_{GFF,m}(\Sigma_2)}\int \frac{\mathcal{A}_{\Sigma_1,g_1,{\bf z}_1,\boldsymbol{\alpha}_1,\boldsymbol{\beta}_1,\boldsymbol{s}_1}(c + \boldsymbol{  \varphi})\mathcal{A}_{\Sigma_2,g_2,{\bf z}_2,\boldsymbol{\alpha}_2,\boldsymbol{\beta}_2,\boldsymbol{s}_2} (c + \boldsymbol{  \varphi})}{\mathcal{A}^0_{\Sigma_1,g_1}(c + \boldsymbol{  \varphi})\mathcal{A}^0_{\Sigma_2,g_2}(c + \boldsymbol{  \varphi} ) }  \dd c \,  \dd\P_{\boldsymbol{X}}    ( \boldsymbol{  \varphi}).
\end{align*}
Now we make a further shift in the $c$-variable in the expression above to subtract the mean $m_{\mc{C}}(\boldsymbol{X}):=\frac{1}{2\pi}\int_0^{2\pi}\boldsymbol{X}(e^{i\theta})\,\dd \theta$ to the field $\boldsymbol{X}$. As a consequence we can replace the law $ \P_{\boldsymbol{X}} $ of $\boldsymbol{X}$ in the above expression by the law $  \P_{\boldsymbol{X}-m_{\mc{C}}(\boldsymbol{X})}$ of the recentered field $\boldsymbol{X}-m_{\mc{C}}(\boldsymbol{X})$.
 
Now we  claim, for measurable bounded functions $F$
\begin{equation}\label{densityDNbis}
\int F( \boldsymbol{ \varphi}) \P_{\boldsymbol{X}-m_{\mc{C}}(\boldsymbol{X})}(\dd \boldsymbol{ \varphi})=\frac{1}{  \det(\mathbf{D}^N_{\Sigma,{\mc{C}},0}(2\mathbf{D}_0)^{-1})^{-1/2} }\int F(    \boldsymbol{ \varphi})\exp(-\frac{1}{2}(\boldsymbol{ \varphi},\widetilde{\mathbf{D}}^N_{\Sigma,{\mc{C}}} \boldsymbol{ \varphi})_2)   \P_\T   (\dd \boldsymbol{ \varphi}).
\end{equation}
We note $\mathbf{D}^N_{\Sigma,{\mc{C}},0}$ as the restriction of $\mathbf{D}^N_{\Sigma,{\mc{C}}}$ in $C_0^{\infty}(\mathcal{C})$. Indeed, we use $  \mathbf{D}^N_{\Sigma, \mathcal{C}}G^{\Sigma}_{g,N}=2\pi \text{Id  in} \quad C_0^{\infty}(\mathcal{C})$ and the following estimate: there is $a>0$ such that for any $\epsilon>0$, there is $N_0>0$ such that (we notice that this kind of estimate can always be treated by the mirror trick)
$$(\boldsymbol{ \varphi},\widetilde{\mathbf{D}}_{\Sigma,{\mc{C}}} \boldsymbol{ \varphi})_2\geq  (a-1)(\pi_{N_0}\boldsymbol{ \varphi},2\mathbf{D}\pi_{N_0}\boldsymbol{ \varphi})_2-\epsilon(\boldsymbol{ \varphi},2\mathbf{D}\boldsymbol{ \varphi})_2.$$

Recall that in this case $\mathbf{D}^N_{\Sigma,{\mc{C}}}=\mathbf{D}^N_{\Sigma_1}+\mathbf{D}^N_{\Sigma_2}$ , we know $\exp(-\frac{1}{2}(\boldsymbol{ \varphi},\widetilde{\mathbf{D}}_{\Sigma,{\mc{C}}} \boldsymbol{ \varphi})_2)=\mathcal{A}^0_{\Sigma_1,g_1}(c + \boldsymbol{   \varphi})\mathcal{A}^0_{\Sigma_2,g_2}(c + \boldsymbol{  \varphi} ) $ (whatever the value of $c$) and the proof is finished by the following lemma.
\end{proof}
\begin{lemma}
\begin{equation}\label{detformula}
\frac{({\det}'(\Delta_{g,N})/{\rm v}_{g}(\Sigma))^{-1/2}}{Z_{GFF,m}(\Sigma_1)Z_{GFF,m}(\Sigma_2)} = \sqrt{\pi}  \det(\mathbf{D}^N_{\Sigma,{\mc{C}},0}(2\mathbf{D}_0)^{-1})^{-1/2},
\end{equation}
\end{lemma}
\begin{proof}
Since we always suppose the boundaries of $\Sigma$ are geodesic, we double the surface and call the resulting surface $d\Sigma$. We gather the theorem $[B]$ and $[B*]$ in \cite{BFK} here\\
$$ {\det}(\Delta^{\Sigma_1}_{g,D})^{-\frac{1}{2}}{\det}(\Delta^{\Sigma_2}_{g,D})^{-\frac{1}{2}}  =  \pi^{1/2}{\det}(\Delta^D_{\Sigma})^{-\frac{1}{2}} {\rm det}_{\rm Fr}(\mathbf{D}^D_{\Sigma,{\mc{C}}}(2\mathbf{D}_0)^{-1})^{1/2} \quad [B]$$
$$ \frac{({\det}'(\Delta^{d\Sigma}_g)/{\rm v}_{g}(d\Sigma))^{-1/2}}{{\det}(\Delta^{d\Sigma_1}_{g,D})^{-\frac{1}{2}}{\det}(\Delta^{d\Sigma_2}_{g,D})^{-\frac{1}{2}}} = \sqrt{2}  \det(\mathbf{D}_{d\Sigma,\mc{C}\cup\sigma(\mc{C}),0}(2\mathbf{D}_0)^{-1})^{-1/2} \quad [B*]$$
where $D_{\Sigma,\mathcal{C}}^D$ is the Dirichlet to Neumann operator corresponds to harmonic extension with Dirichlet $0$ boundary condition on $\partial^N\Sigma$.
By separating odd and even eigenfunctions, we know ${\det}'(\Delta^{d\Sigma}_g)={\det}'(\Delta^{\Sigma}_{g,N}){\det}(\Delta^{\Sigma}_{g,D})$  and ${\det}(\Delta_{g,D}^{d\Sigma_i})={\det}(\Delta^{\Sigma_i}_{g,m}){\det}(\Delta^{\Sigma_i}_{g,D})$ for $i=1, 2$.  Further more, we have
$$\det(\mathbf{D}_{d\Sigma,\mc{C}\cup\sigma(\mc{C}),0}(2\mathbf{D}_0)^{-1})^{-1/2}=\det(\mathbf{D}^D_{\Sigma,\mc{C}}(2\mathbf{D}_0)^{-1})^{-1/2}\det(\mathbf{D}^N_{d\Sigma,\mc{C},0}(2\mathbf{D}_0)^{-1})^{-1/2}$$
Gather all above formulae, we finish the lemma's proof.
\end{proof}
\begin{remark}\label{specialcase}
We develop a general framework for amplitudes with Neumann boundaries in this section, but later we only focus on two special amplitudes:
\begin{itemize}
    \item Annulus amplitude with two Neumann boundaries: $\mathcal{A}^N_{g,\mathbb{A}_q, \bm{s}, \bm{\beta}};$
    \item Annulus amplitude with one Neumann boundary and one Dirichlet boundary: $\mathcal{A}^{m}_{g,\mathbb{A}_{1,\sqrt{q}}, \bm{s}, \bm{\beta},\zeta}(\tilde{\varphi}).$
\end{itemize}
\end{remark}

\section{The Spectrum resolution of the Liouville Hamiltonian }
\subsection{Hilbert space of LCFT}\label{sub:hilbert} 
The circle Gaussian free field is a random Fourier series defined on  the unit circle $\T=\{z\in \C\,|\, |z|=1\}$ 
\begin{equation}\label{GFFcircle0}
\forall \theta\in\R,\quad \varphi(\theta)=\sum_{n\not=0}\varphi_ne^{in\theta} 
\end{equation}
with $\varphi_{n}=\frac{1}{2\sqrt{n}}(x_n+iy_n)$ for $n>0$ where $x_n,y_n$ are i.i.d. standard real Gaussians. Convergence holds in the  Sobolev space  $H^{s}(\T)$ with $s<0$, where $H^s(\T)\subset \C^\Z$ is the set of sequences s.t.
\begin{equation}\label{outline:ws}
\|\varphi\|_{H^s(\T)}^2:=\sum_{n\in\Z}|\varphi_n|^2(|n|+1)^{2s} <\infty.
\end{equation}
One can easily show that $\varphi$ has the same distribution with the restriction of full plane GFF \cite{KRV} on the unit circle. Also, note that the series $ \varphi$ has no constant mode. The constant mode will play an important role in the Liouville theory:  we will view the random series $\tilde{\varphi}:=c+\varphi$ as the coordinate function of the space $\R\times    \Omega_\T$, where the  probability space 
\begin{align}\label{omegat}
  \Omega_\T=(\R^{2})^{\N^*}
\end{align}
  is equipped with the cylinder sigma-algebra $
 \Sigma_\T=\mathcal{B}^{\otimes \N^*}$ ($\mathcal{B}$  stands for the Borel sigma-algebra on $\R^2$) and the   product measure 
 \begin{align}\label{Pdefin}
 \P_\T:=\bigotimes_{n\geq 1}\frac{1}{2\pi}e^{-\frac{1}{2}(x_n^2+y_n^2)}\dd x_n\dd y_n.
\end{align}
Here $ \P_\T$ is supported on $H^s(\T)$ for any $s<0$ in the sense that $ \P_\T(\varphi\in H^s(\T))=1$. Our Liouville Hilbert space is $\mc{H}:=L^2(\R\times \Omega_{\T})$ with underlying measure 
\[\mu_0:=\dd c\otimes  \P_{\T}\] 
and Hermitian product denoted by $\langle\cdot\mid\cdot\rangle_2$.

Besides the Liouville Hilbert space, we need to introduce weighted $L^r$ space. We recall the notations $c_-=c\mathbf{1}_{c<0}$ and  $c_+=c\mathbf{1}_{c>0}$. For $\beta_-,\beta_+\in\R$ and $p\geq 1$ we introduce the weighted $L^r$-spaces $e^{-\beta_- c_- -\beta_+c_+}L^r(\R\times \Omega_{\T})$ as the space of functions with finite $\|\cdot\|_{\beta,r}$-norm, where
\begin{equation}
\|F\|_{\beta,r}^r:=\int e^{r\beta_- c_-+r\beta_+c_+}\E[|F|^r]\,\dd c.
\end{equation}

\subsection{Construction of the Virasoro descendant states}\label{sub:virasoro}
\subsubsection{The free field theory}\label{freetheory} First, we recall how to construct this basis in the case when $\mu=0$, i.e. for the free theory. Here we follow \cite[section 4.4]{GKRV}.  Let us denote by $\mathcal{S}$ the set of smooth functions depending on finitely many coordinates, i.e. of the form
$F(x_1,y_1, \dots,x_n,y_n)$  with $n\geq 1$ and $F\in C^\infty((\R^2)^n)$, with at most polynomial growth at infinity for $F$ and its derivatives. Obviously $\mathcal{S}$ is dense in $L^2(  \Omega_{\T})$. Let  
\begin{equation}\label{smoothexpgrowth}
 \mathcal{C}_\infty:=\mathrm{Span}\{ \psi(c)F\,|\,\psi\in C^\infty(\R)\text{ and }F\in\mathcal{S} \}.
 \end{equation}
We use the  complex coordinates \eqref{GFFcircle0}, i.e. we denote for $n>0$
\begin{align*}
\partial_n:=\frac{\partial}{\partial\varphi_{n}}= \sqrt{n} (\partial_{x_n}-i \partial_{y_n}) \quad \text{ and }\quad \partial_{-n}:=\frac{\partial}{\partial\varphi_{-n}}= \sqrt{n} (\partial_{x_n}+i \partial_{y_n}).
\end{align*}
Then we introduce on $ \mathcal{C}_\infty$  the following operators for $n>0$: 
\[
 \mathbf{A}_n= \tfrac{i}{2}\partial_{n},\ \ \  \mathbf{A}_{-n}=\tfrac{i}{2}(\partial_{-n}-2n\varphi_{n}) \ \ \ 
\widetilde{\mathbf{A}}_n= \tfrac{i}{2}\partial_{-n},\ \ \ \widetilde{\mathbf{A}}_{-n}=\tfrac{i}{2}(\partial_{n}-2n\varphi_{-n})\ \ \ 
\mathbf{A}_0=\widetilde{\mathbf{A}}_0=\tfrac{i}{2}(\partial_c+Q)
\]
and  the {\it normal ordered product} on $ \mathcal{C}_\infty$ by
 $:\!\mathbf{A}_n\mathbf{A}_m\!\!:\,=\mathbf{A}_n\mathbf{A}_m$ if $m>0$ and $\mathbf{A}_m\mathbf{A}_n$ if $n>0$ (i.e. annihilation operators are on the right). The free Virasoro generators are then   defined for all $n \in \Z$ by
\begin{align}
\mathbf{L}_n^0:=-i(n+1)Q\mathbf{A}_n+\sum_{m\in\Z}:\mathbf{A}_{n-m}\mathbf{A}_m: \label{virassoro}\\
\widetilde{\mathbf{L}}_n^0:=-i(n+1)Q\widetilde{\mathbf{A}}_n+\sum_{m\in\Z}:\widetilde{\mathbf{A}}_{n-m}\widetilde{\mathbf{A}}_m:\,\,.\label{virassorotilde}
\end{align}
They map  $ \mathcal{C}_\infty$ into itself. We also define the free Hamiltonian as $\mathbf{H}^0=\mathbf{L}_0^0+\tilde{\mathbf{L}}_0^0$.

For $\alpha\in \C$, we define 
\begin{align}\label{psialphadef}
\Psi^0_\alpha(c,\varphi):=e^{(\alpha-Q)c}\in \mathcal{C}_\infty.
\end{align}
For $\alpha\in \C$, these are generalized eigenstates of ${\bf H}^0$: they never belong to $L^2(\R \times \Omega_\T)$ but rather to some weighted spaces $e^{\beta |c|}L^2(\R\times \Omega_\T)$ for $\beta>|{\rm Re}(\alpha)-Q|$, hence their name ``generalized eigenstates". We have
\begin{equation}\label{L0psialpha}
\mathbf{L}_0^0\Psi^0_\alpha=\widetilde{\mathbf{L}}_0^0\Psi^0_\alpha=\Delta_{\alpha}
\Psi^0_\alpha,  \qquad 
\mathbf{L}_n^0\Psi^0_\alpha=\widetilde{\mathbf{L}}_n^0\Psi^0_\alpha=0,\ \ \ n>0,
\end{equation}
where  $\Delta_\alpha$ is the conformal weight; $\Psi^0_\alpha$ is called {\it highest weight state} with highest weight $\Delta_{\alpha}$.

Next,   let $\mc{T}$ be the set of Young diagrams, i.e. the set of  sequences $\nu$ of integers   with the further requirements that $\nu(k)\geq \nu(k+1)$ and $\nu(k)=0$ for $k$ large enough. For a Young diagram $\nu$ we denote its length as $|\nu|=\sum_k\nu(k)$ and its size as $s(\nu)=\max\{k\,|\,\nu(k)\not=0\}$. Given two Young diagrams $\nu= (\nu(i))_{i \in [1,k]}$ and $\tilde{\nu}= (\tilde{\nu}(i))_{i \in [1,j]}$, we define the operators
\begin{equation*}
\mathbf{L}_{-\nu}^0=\mathbf{L}_{-\nu(k)}^0 \cdots \, \mathbf{L}_{-\nu(1)}^0, \quad \quad \quad   \tilde{\mathbf{L}}_{-\tilde \nu}^0=\tilde{\mathbf{L}}_{-\tilde\nu(j)}^0 \cdots\, \tilde{\mathbf{L}}_{-\tilde\nu(1)}^0
\end{equation*}
and define
\begin{align}\label{psibasis}
\Psi^0_{\alpha,\nu, \tilde\nu}=\mathbf{L}_{-\nu}^0\tilde{\mathbf{L}}_{-\tilde\nu}^0  \: \Psi^0_\alpha,
\end{align}
with the convention that $\Psi^0_{\alpha,\emptyset, \emptyset}=\Psi^0_{\alpha}$.
The vectors $\Psi^0_{\alpha,\nu, \tilde\nu}$ are called the \emph{descendant} states of $\Psi^0_\alpha$.  They satisfy the following properties (see \cite[Prop 4.9]{GKRV})

\begin{proposition}\label{prop:mainvir0}  
{\bf (1) } For each pair of Young diagrams $\nu,\tilde{\nu}\in \mathcal{T}$,  $\Psi^0_{\alpha,\nu, \tilde\nu}$ can be written as   
\begin{align}\label{psibasis1}
\Psi^0_{\alpha,\nu, \tilde\nu}=\mathcal{Q}_{\alpha,\nu,\tilde \nu}\Psi^0_\alpha
\end{align}
where $\mathcal{Q}_{\alpha,\nu,\tilde\nu}  $   is a polynomial in the coefficients $(\varphi_n)_n$ and an eigenfunction of $\mathbf{P}$ with eigenvalue $|\nu|+|\tilde\nu|$.\\
{\bf (2) }  for all $\alpha \in \C$
\begin{equation*}
\mathbf{L}_0^0\Psi^0_{\alpha,\nu, \tilde\nu} = (\Delta_{\alpha}
+|\nu|)\Psi^0_{\alpha,\nu ,\tilde\nu},\ \ \  \tilde{\mathbf{L}}_0^0\Psi^0_{\alpha,\nu, \tilde\nu} = (\Delta_{\alpha}+|\tilde\nu|)\Psi^0_{\alpha,\nu ,\tilde\nu}
\end{equation*}
and
\begin{equation*}
\mathbf{H}^0\Psi^0_{\alpha,\nu,\tilde{\nu}} = (2 \Delta_\alpha+|\nu|+|\tilde{\nu}| )\Psi^0_{\alpha,\nu,\tilde{\nu}}  .
\end{equation*}
{\bf (3) } The inner products of the descendant states obey
\begin{equation}\label{scapo}  
 \langle \mathcal{Q}_{2Q-\bar\alpha,\nu,\tilde\nu} | \mathcal{Q}_{\alpha,\nu',\tilde\nu'} \rangle_{L^2(\Omega_\T)}=\delta_{|\nu| ,|\nu'|}\delta_{|\tilde\nu| ,|\tilde\nu'|}F_{\alpha}(\nu,\nu')F_{\alpha}(\tilde\nu,\tilde\nu')
\end{equation} 
where each coefficient   $F_{\alpha}(\nu,\nu')$  is a  polynomial in  ${\alpha}$, called the {\it Schapovalov form}.  The functions $(\mathcal{Q}_{\alpha,\nu,\tilde \nu})_{\nu,\tilde\nu\in\mathcal{T}}$ are linearly independent for $$\alpha\notin \{{\alpha_{r,s}} \mid \, r,s\in \N^\ast ,rs \leq \max(|\nu|, |\tilde \nu |)\}\quad \quad\text{with }\quad {\alpha_{r,s}}=Q-r\frac{\gamma}{2}-s\frac{2}{\gamma}.$$
{\bf (4) } {\bf Spectral decomposition}: If $u_1,u_2\in L^2(\R\times\Omega_\T)$ then
\begin{align}\label{fcomplete}
\langle u_1\, |\,  u_2\rangle_{2}=\frac{1}{2 \pi}\sum_{\nu,\tilde\nu,\nu',\tilde\nu'\in \mc{T}}\int_\R \langle u_1\,|\,
\Psi^0_{Q+iP,\nu',\tilde{\nu}'} \rangle_{2} \langle \Psi^0_{Q+iP,\nu,\tilde{\nu}}\, |\, u_2\rangle _{2}F^{-1}_{P+iQ}(\nu,\nu')F^{-1}_{P+iQ}(\tilde\nu,\tilde\nu')\, \dd P.
\end{align}
\end{proposition}

\subsection{Scattering construction of eigenstates of \texorpdfstring{${\bf H}$}{H} } 

One of  the main inputs of \cite{GKRV} is the scattering construction, of a complete family of generalized eigenstates for the Liouville Hamiltonian. Given complex number $q=e^{-t+i\theta}$, the Dirichlet annulus amplitude $\mathcal{A}^D_{\A_q^1,g_{\A}}$ forms a semigroup under amplitudes gluing 
\begin{align*}
    \frac{\mathcal{A}^D_{\A_q^1,g_{\A}}}{\sqrt{2}\pi}\times\frac{\mathcal{A}^D_{\A_{q'}^1,g_{\A}}}{\sqrt{2}\pi}=\frac{\mathcal{A}^D_{\A_{qq'}^1,g_{\A}}}{\sqrt{2}\pi}
\end{align*}
This semigroup was studied in \cite[Proposition 5.1]{GKRV} which defines two operator $\mathbf{H}$, $\mathbf{\Pi}$ by $ \mathcal{A}^D_{\A_q^1,g_{\A}}=\sqrt{2}\pi e^{c_{\rm L}t/12}e^{-t\mathbf{H}}e^{i\theta\mathbf{\Pi}}$. The Hamiltonian $\mathbf{H}$ corresponds to dilation and  $\mathbf{\Pi}$ corresponds to rotation.

For Young diagrams $\nu$, $\tilde\nu$, we consider the subset of $\C$ 
\begin{equation}\label{Inunu}
\mathcal{I}_{ \nu,\tilde\nu}:=\{\alpha\in\C \,|\,\exists \beta=\beta(\alpha)\in \R, \,\beta>{\rm Re} (Q-\alpha)\, \text{ and }\,{\rm Re} \big((Q-\alpha)^2\big)-2(|\nu|+|\tilde\nu|)>(\beta-\gamma/2)^2\}
\end{equation}
Now we gather the contents of \cite[Theorem 1.2]{BGKRV} and \cite[Prop 6.9]{GKRV}

\begin{proposition}\label{defprop:desc}
Fix  $\nu$, $\tilde\nu$ Young diagrams and $\ell=|\nu|+|\tilde\nu|$.  There is a holomorphic family in $\alpha\in \C$
$$\alpha\in \C\mapsto \Psi_{\alpha,\nu,\tilde{\nu}}\in e^{-\beta c_-}L^2(\R\times \Omega_{\T})$$ 
with 
  $\beta>|Q-{\rm Re}(\alpha)|$ satisfying ${\bf H}\Psi_{\alpha,\nu,\tilde{\nu}}=(2\Delta_\alpha+\ell) \Psi_{\alpha,\nu,\tilde{\nu}}$. This family is characterized by the intertwining property: for any $\chi\in C^\infty(\R)$   equal to $1$ near $-\infty$ and supported in $\R^-$
  \begin{equation}\label{intert:desc}
\Psi_{\alpha,\nu,\tilde{\nu}}= \lim_{t\to \infty}e^{t(2\Delta_\alpha+\ell)} e^{-t{\bf H}}( \chi(c)\Psi^0_{\alpha,\nu,\tilde{\nu}})\in e^{-\beta c_-}L^2(\R\times \Omega_{\T})
\end{equation}
 for any $\alpha\in \mathcal{I}_{ \nu,\tilde\nu}$ and any $\beta=\beta(\alpha)$ fulfilling the definition of $ \mathcal{I}_{ \nu,\tilde\nu}$. Furthermore, if $\alpha\in\R$ satisfies $\alpha<Q-\gamma$, we can take $\chi=1$ in the above statement.
\end{proposition}
\begin{remark}
The analytic region of $\alpha\to \Psi_{\alpha,\nu,\tilde{\nu}}$ is 
 $\C$. But only in the region $\mathcal{I}_{\nu,\tilde{\nu}}$, it has the intertwining representation (a scattering construction from the free field case).
\end{remark}

\begin{proposition}\label{holomorphicpsi2}
For all $u,v\in e^{\delta c_-}L^2(\R\times \Omega_\T)$ with $\delta>0$
\[\cjg u\mid v\cjd_{L^2}=\frac{1}{2\pi}\sum_{\nu_1,\tilde\nu_1,\nu_2,\tilde\nu_2}\int_0^\infty 
\cjg u\mid\Psi_{Q+ip,\nu_1,\tilde{\nu}_1}\cjd_{L^2} \cjg\Psi_{Q+ip,\nu_2,\tilde\nu_2}\mid v\cjd_{L^2} F^{-1}_{Q+iP}(\nu_1,\nu_2)F^{-1}_{Q+iP}(\tilde\nu_1,\tilde\nu_2) \dd P.\]
\end{proposition}
We can also extend the integration region from of $\R_+$ to the whole real line $\R$, i.e.
\begin{proposition}\label{realline}
For all $u,v\in e^{\delta c_-}L^2(\R\times \Omega_\T)$ with $\delta>0$
\[\cjg u\mid v\cjd_{L^2}=\frac{1}{4\pi}\sum_{\nu_1,\tilde\nu_1,\nu_2,\tilde\nu_2}\int_{-\infty}^\infty 
\cjg u\mid\Psi_{Q+ip,\nu_1,\tilde{\nu}_1}\cjd_{L^2} \cjg\Psi_{Q+ip,\nu_2,\tilde\nu_2}\mid v\cjd_{L^2} F^{-1}_{Q+iP}(\nu_1,\nu_2)F^{-1}_{Q+iP}(\tilde\nu_1,\tilde\nu_2) \dd P.\]
\end{proposition}
\begin{proof}
First, we know the Schapovalov form $F_{Q+iP}(\nu_1,\nu_2)$ and the Kac determinant depend on $P$ only through $\Delta_{Q+iP}=\Delta_{Q-iP}$. So $F^{-1}_{Q+iP}(\nu_1,\nu_2)=F^{-1}_{Q-iP}(\nu_1,\nu_2)$.

By \cite[Theorem 1.1]{BGKRV}, $\Psi_{Q+iP,\nu_1,\tilde{\nu_1}}=R^{\rm DOZZ}(Q+iP)\Psi_{Q-iP,\nu_1,\tilde{\nu_1}}$. So we have 
\begin{align*}
    \cjg u\mid\Psi_{Q+ip,\nu_1,\tilde{\nu}_1}\cjd_{L^2} \cjg\Psi_{Q+ip,\nu_2,\tilde\nu_2}\mid v\cjd_{L^2}&=\cjg u\mid R^{\rm DOZZ}(Q+iP)\Psi_{Q-ip,\nu_1,\tilde{\nu}_1}\cjd_{L^2} \cjg R^{\rm DOZZ}(Q+iP)\Psi_{Q-ip,\nu_2,\tilde\nu_2}\mid v\cjd_{L^2}\\&=\cjg u\mid \Psi_{Q-ip,\nu_1,\tilde{\nu}_1}\cjd_{L^2} \cjg \Psi_{Q-ip,\nu_2,\tilde\nu_2}\mid v\cjd_{L^2}
\end{align*}
Since $|R^{\rm DOZZ}(Q+iP)|=1$.
\end{proof}
\noindent{\bf Reverting orientation operators}

Let $o:\T\to\T$ be the map defined by $o(e^{i\theta})=e^{-i\theta}$. It will serve to analyze the effect of reverting orientation on the boundary of the amplitudes. We introduce the following operators
\begin{align}\label{defCandO}
\forall F\in e^{-\beta c_-}L^2(\R\times\Omega_\T), & & \mathbf{O}F(\varphi)=F(\varphi\circ o), & & \mathbf{C}F(\varphi)=\bbar F(\varphi ) 
\end{align}
defined on the weighted $L^2$-spaces with $\beta\in\R$. The operator $\mathbf{C}$ is nothing but complex conjugation. Recall that the semigroup $(e^{-t\mathbf{H}})_{t\geq 0}$ extends continuously to $e^{-\beta c_-}L^2(\R\times\Omega_\T)$ for $\beta\geq 0$ (see \cite[Lemma 6.5]{GKRV}). 
\begin{proposition}{\cite[Proposition 6.11.]{GKRV21}}\label{revert}
For $p\in\R^*$ and $\nu,\tilde\nu\in\mathcal{T}$ Young diagrams, we have the following relations
\begin{align*}
  &e^{-t\mathbf{H}} \mathbf{O}= \mathbf{O}e^{-t\mathbf{H}}, \quad \quad e^{-t\mathbf{H}} \mathbf{C}= \mathbf{C}e^{-t\mathbf{H}} .\\
&\mathbf{O} \mathbf{L}_n^0=\widetilde{\mathbf{L}}_n^0\mathbf{O},\quad \quad \mathbf{O}\widetilde{\mathbf{L}}_n^0= \mathbf{L}_n^0\mathbf{O},\quad\quad
\mathbf{C} \mathbf{L}_n^0=\widetilde{\mathbf{L}}_n^0\mathbf{C},\quad\quad \mathbf{C}\widetilde{\mathbf{L}}_n^0= \mathbf{L}_n^0\mathbf{C} .\\
&\mathbf{O} \mathbf{C}\Psi_{Q+ip,\nu,\tilde\nu}= \mathbf{C}\mathbf{O}\Psi_{Q+ip,\nu,\tilde\nu}=\Psi_{Q-ip,\nu,\tilde\nu}.
\end{align*}
\end{proposition}
\noindent{\bf Orientation on Riemann surface}.

 A compact Riemann surface $\Sigma$ with real analytic boundary $\partial\Sigma=\sqcup_{j=1}^{b}\caC_j$ is a 
compact oriented surface with smooth boundary with a family of charts 
 $\omega_j:V_j\to \omega_j(V_j)\subset \C$ for $j=1,\dots,j_0$ where $\cup_j V_j$ is an open covering of $\Sigma$ and $\omega_k \circ \omega_j^{-1}$ are holomorphic maps (where they are defined), and $\omega_j(V_j\cap\pl \Sigma)$ is a real analytic curve if $V_j\cap\pl\Sigma\not=\emptyset$. Let  $\T:=\{z\in \C\,|\, |z|=1\}$ be the standard unit circle, then for each $j\in [1,b]$, fix a point $p_j\in \mc{C}_j$  and an orientation $\sigma_j\in\{-1,+1\}$, and define the parametrization $\zeta_j:\T\to \mc{C}_j$ by 
$\zeta_j(e^{i\theta})=\omega_j^{-1}(e^{-i\sigma_j\theta}\omega_j(p_j))$ (in particular $\zeta_j(1)=p_j$). Observe that the parametrization of $\mc{C}_j$ is entirely described by $(\sigma_j,p_j)$ and the  complex coordinate chart $\omega_j$ near $\mc{C}_j$.  
We say that the boundary $\mc{C}_j$ is \emph{outgoing} if the orientation $(\zeta_j)_*(d\theta)$ is the orientation of $\mc{C}_j$ induced by that of $\Sigma$ as described above, i.e. if $\sigma_j=-1$, otherwise the parametrized boundary $\mc{C}_j$ is called \emph{incoming} if $\sigma_j=+1$. 

For example, in the {\bf annulus} topology, we call the outer circle is outgoing if the parametrization is counter-clockwise; and the inner circle is outgoing if it is clockwise.

\noindent{\bf Convention for pairing}

Suppose $\Sigma$ has $b$ Dirichlet boundary components, the amplitude  $\caA_{\Sigma,\boldsymbol{\zeta}}$ is a function $\caA_{\Sigma,\boldsymbol{\zeta}}:H^s(\T)^b\to \C$. Let us use the notation
 \begin{align}\label{ampf}
 \big\cjg \caA_{\Sigma,\boldsymbol{\zeta}},\otimes_{j=1}^{b}f_j\big\cjd_2:=\caA_{\Sigma,\boldsymbol{\zeta}}\big(\prod_{j=1}^b \bbar{f_j(\tilde\varphi_j)}\big):=\int \caA_{\Sigma,\boldsymbol{\zeta}}(\tilde{\boldsymbol{\varphi}})\big(\prod_{j=1}^b \bbar{f_j(\tilde\varphi_j)}\big) \dd \mu_0^{\otimes_b}(  \tilde{\boldsymbol{\varphi}}) .
\end{align}
Hence if we denote by $\tilde{\boldsymbol{\zeta}}$ the parametrization where all the boundary components are "in" then 
\begin{align*}
 \big\cjg \caA_{\Sigma,\boldsymbol{\zeta}}\mid\otimes_{j=1}^{b}f_j\big\cjd_2= \big\cjg \caA_{\Sigma,{\boldsymbol{\tilde\zeta}}}\mid\otimes_{j=1}^{b}{\bf O}^{(1-\sigma_j)/2}f_j\big\cjd_2.
\end{align*}
Recalling that ${\bf CO}\Psi_{Q+ip,\nu,\tilde\nu}=\Psi_{Q-ip,\nu,\tilde\nu}$ (Prop. \eqref{revert}) we obtain
\begin{align*}
 \big\cjg \caA_{\Sigma,\boldsymbol{\zeta}}\mid \otimes_{j=1}^b{\bf C}^{(1+\sigma_j)/2}\Psi_{Q+ip_j,\nu_j,\tilde\nu_j}\big\cjd_2=\caA_{\Sigma,\tilde{\boldsymbol{\zeta}}}(\otimes_{j=1}^b{\bf C}\Psi_{Q+i\sigma_jp_j,\nu_j,\tilde\nu_j})
\end{align*}

\section{The stress energy tensors in the upper half-plane}\label{SET}
Before introducing heavy notations, let us give a brief summary of the following two sections.
\begin{itemize}
    \item Step 1. We introduce the amplitude with one Dirichlet boundary and one Neumann boundary. On the Neumann boundary, there are insertions $(\beta,s)$ and $(\beta^{ar},s^{ar})$, we also cut a small hole with size $e^{-t}$ around $s$; 
    \item Step 2. We introduce the $\alpha$ insertion Dirichlet disk amplitude with a hole of size $e^{-t}$ and SET insertions. See Section \ref{diskamplitude};
    \item Step 3. We glue the above two objects along the Dirichlet boundary. Then we get an upper half-plane with one hole of size $e^{-t}$ in the interior and one hole of size $e^{-t}$ on the Neumann boundary. We denote the union of two holes by $U_t$. The SET insertions locate in the interior hole. See Proposition \ref{glluu}. ;
    \item Step 4. We compute Ward's identity in the object given by step 3. We get some differential operators to act on the ordinary correlation function with $U_t$ removed in the potential and some error terms. see Proposition \ref{beforeres};
    \item Step 5. We take the residue of the object in step 4 to get some new differential operators labeled by Young diagrams and new error terms. See Proposition \ref{withtbis};
    \item Step 6. Combining the regularity argument Proposition \ref{difflemma} and Proposition \ref{withtbis}, we can take $t\to\infty$ to make error terms vanish after choosing proper insertion weight. This gives Ward identity without holes. See Proposition \ref{abr}.
\end{itemize} 
We first introduce some basic materials on how to define the stress energy tensor on a subset of $\H$.  We equip $\H$  with conformal metric $g=e^{\omega(z)}|\dd z|^2$. The stress energy tensor (SET) does not make sense as a random field but can be given sense at the level of correlation functions as the limit of a regularized SET. Let us introduce the Liouville field
\begin{align*}
\Phi_g:=\phi_g+\tfrac{Q}{2}\omega
\end{align*}
and its $|dz|^2$-regularization $\Phi_{g,\epsilon}$ with the following caveat: we  make a special choice of regularization by choosing a  nonnegative smooth compactly supported function $\rho:\C\to \R$ such that $\int \rho(z)\,\dd z=1$ and thus $\int z\partial_z \rho(z)\,\dd z=-1$. Then, as usual, we set $\rho_\epsilon(z)=\epsilon^{-2}\rho(z/\epsilon)$ and denote by $h_\epsilon=h\star \rho_\epsilon$ the regularisation of a distribution $h$.  In particular $\Phi_{g,\epsilon}(z)$ is a.s. smooth.
Then the regularized SET is defined by for $z\in \H$
\begin{equation}\label{defSET}
\quad T^{\H}_{g,\epsilon}(z):=Q\partial^2_{z}\Phi_{g,\epsilon}(z)-(\partial_z \Phi_{g,\epsilon}(z))^2 +a_{\Sigma,g,\epsilon}(z)
\end{equation}
with the renormalization constant  given by $a_{\Sigma,g,\epsilon}:=\E[(\partial_z X^{\H}_{g,\epsilon}(z))^2]$.
We denote also by $\bar T_{g,\epsilon}(z)$  the complex conjugate of $T_{g,\epsilon}(z)$. 
Then we have the following proposition
\begin{proposition}\label{constanteSET} For a fixed $z\in \H$, we have $a_{\H,g,\epsilon}(z)=o(1)$.

\end{proposition}
\begin{proof}
\begin{lemma}
 For $g=e^{\omega(z)}|\dd z|^2$, the Green function is of the form $G^{\H}_g(z,z')=\ln\frac{1}{|z-z'||z-\bar{z'}|}+h(z,\bar z)+h(z',\bar z')$. 
\end{lemma}
\begin{proof}
First, we choose a metric $g_0$ s.t. the real line $\R$ is geodesic. Then we double $\H$ to get the full plane, using $G^{\H}_{g_0}(z,z')=G^{\hat\C}_{g_0}(z,z')+G^{\hat\C}_{g_0}(z,\bar{z'})$, we know the lemma holds for $(\C,g_0)$. For general case $( \H,g)$, the Green function can be written as $$G^{\H}_{g}(z,z')=G^{\H}_{g_0}(z,z')-m_g(G^{\H}_{g_0}(z,\cdot))-m_g(G^{\H}_{g_0}(\cdot,z'))+m_g(G^{\H}_{g_0}(\cdot,\cdot)).$$ since we can always find $\omega$ such that $g=e^{\omega}g_0$.
\end{proof}
The rest of proof is parallel to \cite[Lemma 10.1]{GKRV21}, the new thing here is that we need to treat the term $\ln\frac{1}{|(z+u\epsilon)-\overline{( z+v\epsilon)}|}$. We may suppose $\Im(z)>\delta>0$, we consider the smooth function $h_1(z,\bar z,z',\bar z')=\ln\frac{1}{z-\bar{z}'}+\ln\frac{1}{\bar{z}-z'}$, then we have
$$\epsilon^{-2}\int\int\partial_z\rho(u)\partial_z\rho(v)h(z+u\epsilon,\bar z+\epsilon\bar u,z+v\epsilon,\bar z+\epsilon\bar v)\dd u\dd v=\partial^2_{zz'} h(z,\bar z,z,\bar z)+o(1)\label{cst1}.$$
We extend $\psi$ to the double surface $d\Sigma$. Similarly, due to $|\psi(z)-\psi(\bar{z})|>\delta>0$, we know that $h_2(z,\bar z,z',\bar z')=\ln\frac{1}{|\psi(z)-\psi(\bar{z}')|}$ is a smooth function, and $\partial^2_{zz'} h_2(z,\bar z,z,\bar z)=0.$
\end{proof}
\begin{remark}
In this paper, we only need the case $\psi$ is a M\"obius transform which implies $S_{\psi}=0$.
\end{remark}

\subsection{Generalized disk amplitude with SET insertions.}\label{diskamplitude}
In this section, we restrict our attention to either $\Sigma=\C$ or  $\Sigma\subset\C$  a bounded region given by removing finite disks in the interior of $\C$ (notice $\Sigma$ has no Neumann boundaries).

We recall some results from \cite{GKRV} and then explain generalized disk amplitude can be understood as the probabilistic counterpart of generalized eigenfunction of Liouville Hamiltonian. First  we will use the following corollary of Proposition \ref{defprop:desc} 

\begin{proposition}\label{eigenana}
Let  $\nu,\tilde\nu$  be Young diagrams and $\ell=|\nu|+|\tilde\nu|$.  There exists a real number $\bf{a}$ $< Q-\gamma$ such that $ (-\infty ,\bf{a}]\subset \mathcal{I}_{\nu,\tilde{\nu}}$. Since $\alpha\to\Psi_{\alpha,\nu,\tilde{\nu}}(c,\varphi)$ is analytic in $\alpha\in\C$, so its analytic region covers both $Q+i\R_+$ and $ (-\infty ,\bf{a}]$.

\end{proposition}
We introduce some notation from \cite{GKRV}. 

Let  $f:\C^k\times\C^{\tilde k}\to\C$ and  ${\bf a} \in \C^k$,  ${\bf b} \in \C^{\tilde k}$.       We will denote multiple   nested contour integrals of $f$ as follows:
\begin{align*}
\oint_{|\mathbf{u}-\boldsymbol{a}|=\boldsymbol{\delta} }\oint_{|\mathbf{v} -\boldsymbol{\tilde a}|=\tilde{\boldsymbol{\delta}}}f ({\bf u},{\bf v})\dd {\bf \bar v}\dd {\bf u}:=
\oint_{|u_k-a_k|=\delta_k}\dots \oint_{|u_1-a_1|=\delta_1} \oint_{|v_{\tilde k}-\tilde a_{\tilde k}|=\tilde{\delta}_{\tilde k}}\dots  \oint_{|v_1-\tilde a_1|=\tilde{\delta}_1}  f({\bf u},{\bf v}) \dd \bar v_1\dots \dd \bar v_{j}  \dd u_1\dots \dd u_k\nonumber
\end{align*}
where $\boldsymbol{\delta} :=(\delta_1,\dots, \delta_k)$ with $0<\delta_1<\dots<\delta_k<1$ and similarly for $\tilde{\boldsymbol{\delta}}$.  We always suppose $\delta_i\neq \tilde\delta_j$ for all $i,j$. Given Young diagrams $\nu$, $\tilde\nu$  we denote
\begin{equation}\label{notation}
\mathbf{u}^{1-\nu}:=\prod u_i^{1-\nu(i)},\ \ \ \bar{\mathbf{v}}^{1-\tilde{\nu}}:=\prod \bar {v}_i^{1-\tilde{\nu}(i)}.
\end{equation}
More generally, we will often make use of the shorthand $f({\bf u}):=\prod_i f(u_i)$ if $f:\C\to \C$ is a function and ${\bf u}\in \C^k$.
Recall also   the definitions for Young diagrams in subsection \ref{sub:virasoro}, in particular $s(\nu)$ is the size of $\nu$.

We write $g_\D=|dz|^2$ for the flat metric on $\D$. Let $g=e^\omega g_\D$ where $\omega$ is smooth on $\D$. Recall that, on $\D$, $\Phi_g=X_{g,D}+P\tilde\varphi+ \tfrac{Q}{2}\omega$ and we denote $\Phi_{g,\epsilon}$ its $|dz|^2$-regularization as above. The SET (distinguish this SET with the one in the last subsection) is then given by
$$
T_{g,\epsilon}=Q\partial^2_{z}\Phi_{g,\epsilon}-\big(\partial_z \Phi_{g,\epsilon} \big)^2 +\E[(\partial_z X_{g,D,\epsilon}(z))^2]
$$
and $\bbar T_{g,\epsilon}$ is its complex conjugate. 

Let us consider now $\alpha\in\R$ with $\alpha<\bf{a}$. Then from Proposition \ref{defprop:desc} (taking $\chi=1$) we have
\begin{equation}\label{limtpsi}
 \lim_{t\to +\infty}e^{t (2\Delta_\alpha+|\nu|+|{\tilde{\nu}}|)}e^{-t\mathbf{H}}\Psi^0_{\alpha,\nu,{\tilde{\nu}}}=\Psi_{\alpha,\nu,{\tilde{\nu}}}
 \end{equation}
  in $e^{-\beta c_- }L^2(\R\times \Omega_\T)$ for any $\beta>(Q-\alpha)$.  
Then from \cite[Lemma 7.7]{GKRV} we have:

\begin{lemma}\label{TTLemma} Let $\alpha \in\R$, $\boldsymbol{ \delta},\boldsymbol{\tilde\delta}\in \R^k$. Then
\begin{align}\label{tlim}
 e^{t (2\Delta_\alpha+|\nu|+|{\tilde{\nu}}|)}e^{-t\mathbf{H}}\Psi^0_{\alpha,\nu,{\tilde{\nu}}} 
 =&
  \frac{1}{(2\pi i)^{s(\nu)+s(\tilde{\nu})}}  
 \oint_{|\mathbf{u}|=\boldsymbol{\delta}_t}   \oint_{|\mathbf{v}|=\boldsymbol{\tilde\delta}_t}  
\mathbf{u}^{1-\nu}\bar{\mathbf{v}}^{1-\tilde\nu}  
\Psi_{t, \alpha}(\mathbf{u},\mathbf{v})\dd   \bar{\mathbf{v}}\dd   \mathbf{u}
 \end{align}
where $\boldsymbol{\delta}_t:=e^{-t}\boldsymbol{\delta},\boldsymbol{\tilde \delta}_t:=e^{-t}\boldsymbol{\tilde\delta}$, and 
\begin{equation}\label{limphieps}
 \Psi_{t, \alpha}(\mathbf{u},\mathbf{v}) :=\lim_{{\boldsymbol{\epsilon}\to 0}}\Psi_{t,\epsilon,\alpha}(\mathbf{u},\mathbf{v})  
\end{equation} with
\begin{align}\label{psiepsi}
\Psi_{t,\epsilon,\alpha}(\mathbf{u},\mathbf{v}) =e^{-Q c}
\E\Big( T_{g_\D,\epsilon}(\mathbf{u})\bbar T_{g_\D,\epsilon}(\mathbf{v}) V_{\alpha,g_\D}(0) e^{-\mu e^{\gamma c}M_\gamma(\phi_{g_\D},\D\setminus\D_{e^{-t}})}\Big) \quad \in e^{-\beta c_- }L^2(\R\times \Omega_\T),
 \end{align}
the expectation is over the Dirichlet field $X_{g_\D,D}$  and we denoted $\D_{e^{-t}}=e^{-t}\D$. The limit \eqref{limphieps} holds in $ e^{-\beta c_- }L^2(\R\times \Omega_\T)$, for $\beta>(Q-\alpha)$, uniformly over the compact subsets of $\{(\mathbf{u},\mathbf{v})|u_i,v_j\in \D_{e^{-t}}\setminus\{0\}, \text{all distinct}\}$.
  \end{lemma}
  
Note that the expression  \eqref{psiepsi} is the expectation involved in the definition of amplitudes \eqref{blockamplitude} with the function $F$ given by  
$$F=T_{g_\D,\epsilon}(\mathbf{u})\bbar T_{g_\D,\epsilon}(\mathbf{v}) e^{\mu e^{\gamma c}M_\gamma(\phi_{g_\D},\D_{e^{-t}})}.$$
It has  a hole in the potential, i.e. $\D_{e^{-t}}$, and further SET insertions.
Hence we need to extend the definition of amplitudes for the limiting object as $\epsilon\to 0$. 

The framework is the following. Consider an admissible surface $(\Sigma,g, {\bf z},\boldsymbol{\zeta})$ embedded in the Riemann sphere $\hat{\C}$ (viewed as the complex plane), with marked points ${\bf z}=(z_1,\dots,z_m)$ and associated weights  $\boldsymbol{\alpha}=(\alpha_1,\dots,\alpha_m)\in\R^m$ satisfying $\alpha_i<Q$ for all $i$. Furthermore we consider an open set $U\subset\Sigma$, which will stand for   holes in $\Sigma$ and which is not necessarily connected: this will typically be the case as we will put holes around many different vertex insertions. We also consider two vectors $\mathbf{u}=(u_1,\dots,u_k)\in U^k ,\mathbf{v}=(v_1,\dots,v_{\tilde k})\in U^{\tilde k}$. 
Let us introduce the sets
\begin{align}\label{caOt}
 \caO^{\bf z}_{\Sigma,U}=\{({\bf u},{\bf v})\in U^{k+\tilde k}\,|\, u_j,v_{j'} \text{ all distinct and } \forall j ,\quad \forall i, \:  u_j\neq z_i\; ,v_j\neq z_i  \},\\
  \caO^{\rm ext}_{\Sigma,U} =\{({\bf z},{\bf u},{\bf v})\in \Sigma^m\times U^{k+\tilde k}\,|\, z_i,u_j,v_{j'} \text{ all distinct} \}.
\end{align}

\begin{definition}[{\textbf{Generalized amplitudes}}]\label{def:ampholes}
Let $({\bf u},{\bf v})\in  \caO^{\bf z}_{\Sigma,U}$.
 
\noindent {\bf (1) }
If $\partial\Sigma=\emptyset$, i.e. $\Sigma=\hat\C$, then we assume that the Seiberg bound \eqref{seiberg1} (thus with ${\bf g}=0$)
is satisfied and   the generalized amplitude with holes and SET insertions  is   defined by 
\[
 \caA_{\Sigma,U, g, {\bf z},\boldsymbol{\alpha}, \boldsymbol{\zeta}} (T_g( \mathbf{u})\bbar T_g( \mathbf{v}))  :=   \langle T_g( \mathbf{u})\bbar T_g( \mathbf{v})\prod_{i=1}^mV_{\alpha,g}(z_i)\rangle_{\hat\C,U,g}.
\]

 \vskip 3mm
 
\noindent {\bf (2)}  If  $\partial\Sigma$ has $b>0$ boundary connected components then 
for boundary fields $\tilde{\boldsymbol{\varphi}}:=(\tilde\varphi_1,\dots,\tilde\varphi_b)\in (H^{s}(\T))^b$ with $s<0$ we set  
\begin{multline}
 \label{amplitudeSET}
 \caA_{\Sigma,U, g, {\bf z},\boldsymbol{\alpha}, \boldsymbol{\zeta}}( T_g( \mathbf{u})\bbar T_g( \mathbf{v}),\tilde{\boldsymbol{\varphi}}) \\
 :=\lim_{\epsilon\to 0}  Z^{\Sigma}_{GFF,g}\caA^0_{\Sigma,g}(\tilde{\boldsymbol{\varphi}})
 \E \big[T_{g,\epsilon}({\bf u})\bbar T_{g,\epsilon}({\bf v})\prod_{i=1}^m V_{\alpha_i,g }(z_i)e^{-\frac{Q}{4\pi}\int_\Sigma K_g\phi_g\dd {\rm v}_g-\frac{Q}{2\pi}\int_{\partial\Sigma}k_g\phi_g\dd \lambda_g -\mu M_\gamma^g (\phi_g,\Sigma\setminus U)}\big].
 \end{multline}
\end{definition}

The following statement is a straightforward adaptation from \cite[Prop 9.1]{GKRV}.

\begin{proposition}\label{ampSETholo}
The limit $\epsilon\to 0$ in Definition \ref{def:ampholes} is well defined and  defines a  continuous function of the variables $({\bf z},{\bf u},{\bf v})\in   \caO^{\rm ext}_{\Sigma,U} $, $\mu_0^{\otimes b}$ almost surely in $\tilde{\boldsymbol{\varphi}}\in (H^{s}(\T))^b$, with $s<0$.  

\end{proposition} 

Next, we reformulate the Weyl covariance in terms of generalized amplitudes:
\begin{proposition}\label{WeylSET2}{\bf  (Conformal Anomaly)}  Let   $\omega\in H^1_0(\Sigma)$ if $\partial\Sigma\not=\emptyset$ or $\omega\in H^1(\Sigma)$ if $\partial\Sigma=\emptyset$. Then 
\begin{align*}
\caA_{\Sigma,U,e^\omega g,{\bf z},\boldsymbol{\alpha}, \boldsymbol{\zeta}}(T_{e^\omega g}( \mathbf{u})\bbar T_{e^\omega g}( \mathbf{v}),\tilde{\boldsymbol{\varphi}})= e^{c_LS_{\rm L}^0(\Sigma,g,\omega)-\sum_i\Delta_{\alpha_i}\omega(z_i) }\caA_{\Sigma, U,g,{\bf z},\boldsymbol{\alpha},\boldsymbol{\zeta} }\big( T_g( \mathbf{u})\bbar T_g( \mathbf{v}), \tilde{\boldsymbol{\varphi}}\big).
\end{align*}
\end{proposition}

A direct consequence of the definition is that we can rewrite \eqref{tlim} as a generalized amplitude on $\D$ with boundary  $\partial\D$ equipped with the  parametrization $\zeta_\D(e^{i\theta})=e^{i\theta}$: 
\begin{proposition}
    
\label{ampdiskSET}
 Let $\alpha\in\R$ with $\alpha<Q$ and $\nu,\tilde{\nu}\in \mc{T}$. Then 
\begin{align*}
  e^{(2\Delta_{\alpha}+|\nu|+|\tilde\nu|)t}   e^{-t\mathbf{H}}\Psi^0_{\alpha,\nu,\tilde\nu}(\tilde\varphi)= &\frac{Z_{\D,g_\D} ^{-1}}{(2\pi i)^{s(\nu)+s(\tilde \nu)}} 
 \oint_{|\mathbf{u}|=\boldsymbol{\delta}_t}   \oint_{|\mathbf{v}|=\boldsymbol{\tilde\delta}_t}  
\mathbf{u}^{1-\nu}\bbar{\mathbf{v}}^{1-\tilde\nu} \mathcal{A}_{\D,\D_{e^{-t}},g_\D,0,\alpha,\zeta_\D}(T_g({\bf u})   \bbar T_g({\bf v}),\tilde\varphi)    \dd   \bbar{\mathbf{v}}\dd   \mathbf{u}.
 \end{align*}
\end{proposition} 
For convenience, we define $Z(g)=e^{-c_LS_L^0(\D, g_{\D}, \omega)}$.
\begin{proposition}{\bf  (Conformal change of generalized disk amplitude with SET insertions)} 
 Let $(\Sigma,g)$ and $(\Sigma',g')$ be two admissible surfaces embedded in the complex plane with $\partial \Sigma\not=\emptyset$ and such that there exists a biholomorphism  $\psi :\Sigma\to\Sigma'$ with $g'=\psi_\ast g$.  We suppose $g$ is of  the form  $g=e^{\omega(z)}|dz|^2$ with $\omega\in C^\infty(\Sigma )$.  
Define
\begin{align}\label{Schwarzian1}
T^\psi_{g'}:=\psi'^2T_{g'}\circ\psi+\frac{Q^2}{2}
 {\rm S}_\psi
\end{align} Then
    \begin{align} \label{psitpsiP}
\Psi_{t,\alpha_j}(\mathbf{u},\mathbf{v}) =
 Z(g)Z_{\D,g_\D}^{-1}\caA_{\psi_j(\D),\psi_j(e^{-t}\D),\hat g_j,z_j,\alpha_j,\psi_j\circ \zeta_\D} (  T^{\psi_j}_{\hat g_j}(\mathbf{u})\bar T^{\psi_j}_{\hat g_j}(\mathbf{v}))
\end{align}
\end{proposition}

\subsection{Generalized correlation function with SET insertions}

 We will construct the SET insertions with the caveat that the Liouville potential is removed from a neighborhood of the insertions. In Section \ref{diskamplitude}, we explained the eigenstates are given in terms of limits of such amplitudes as these neighborhoods are shrunk to points so that the building block amplitudes evaluated at eigenstates will lead to such SET correlation functions. 
 We consider a bounded open set $U \subset \H$  and we introduce the sets
\begin{align*}
 \caO^{\bf z, \bf s}_{\H,U}=\{({\bf u},{\bf v})\in U^{k+\tilde k}\,|\, u_j,v_{j'} \text{ all distinct and } \forall j ,\forall i, \:  u_j\neq z_i\; ,v_j\neq z_i\; u_j\neq s_i\; ,v_j\neq s_i \},\\
  \caO^{\rm ext}_{\H,U} =\{({\bf z},{\bf s},{\bf u},{\bf v})\in \H^n\times\R^m\times U^{k+\tilde k}\,|\, z_i,u_j,v_{j'} \text{ all distinct} \}.
\end{align*}

For $\mathbf{u}= (u_1, \dots, u_k)$ we set
\begin{equation*}
  T^{\H}_{g,\epsilon}( \mathbf{u}):= \prod_{i=1}^k  T^{\H}_{g,\epsilon}(u_i)
\end{equation*}
and similarly for $ \bar{T}^{\H}_{g,\epsilon}( \mathbf{v})$. For $\phi_g=c+X^{\H}_g$, we define
\begin{align}\label{amplitudeSET1}
  & \langle T^{\H}_g( \mathbf{u})\bar T^{\H}_g( \mathbf{v})\prod_{i=1}^n V_{\alpha_i,g}(z_i)\prod_{j=1}^m V_{\frac{\beta_j}{2},g}( s_j) \rangle_{\H,U,g}
   :=  ({\det}'(\Delta_{g,N})/{\rm v}_{g}(\H))^{-1/2}  \lim_{\epsilon\to 0} \lim_{\epsilon'\to 0}  \int_\R  \E\Big[ \prod_{i=1}^n V_{\alpha_i,g,\epsilon'}( z_i) \prod_{j=1}^m V_{\frac{\beta_j}{2},g,\epsilon'}( s_j) \nonumber\\
   & \times T^{\H}_{g,\epsilon}({\bf u})\bar T^{\H}_{g,\epsilon}({\bf v}) 
  \exp\Big( -\frac{Q}{4\pi}\int_{\H}K_{g}\phi_g d{\rm v}_{g}  -\frac{Q}{2\pi}\int_{\R}k_{g}\phi_g {\rm d\lambda}_{g}- \mu  M_{g, \gamma}( \phi_g, \H\setminus U) -\mu_\partial  M^{\partial}_{g, \gamma}( \phi_g,\R\setminus U) \Big) \Big] \dd c . 
\end{align}
 The following statement is a  straightforward adaptation from \cite[Prop 9.1]{GKRV}. We stress here that the role of the hole $U$ is crucial as it removes any problematic singularity in the treatment of correlation functions. 
\begin{proposition}\label{ampSETholomorphic}
Let $(\alpha_i,\beta_j)$ satisfy the Seiberg bounds. Then
 the limit
  in \eqref{amplitudeSET1}  exists  and  defines  
 a function which is smooth in  $\bf u$ and $\bf v$ and continuous in $\bf z$  for  $({\bf z},{\bf s}, {\bf u},{\bf v}) \in \caO^{\rm ext}_{\H,U}$. It is also smooth in the variables $z_i$'s that belong to  $  U$.
\end{proposition} 
\begin{proof}
The argument is parallel to \cite[propsition 10.2]{GKRV21}, after using Gaussian integration by parts formula then we use the scaling formula \eqref{scaling} to get the bound for the integration correlation functions. The smoothness inherited from the smoothness of Green function at non-coinciding point.
\end{proof}
  \begin{proposition}[\bf Conformal Anomaly]\label{WeylSET}
 Let   $\omega\in H^1_0(\partial^D\Sigma,g)$. Then
\begin{align*}
&\langle T^{\H}_{e^\omega g}( \mathbf{u})\bar T^{\H}_{e^\omega g}( \mathbf{v})\prod_{i=1}^nV_{\alpha_i,{e^\omega g}}(z_i)\prod_{j=1}^m V_{\frac{\beta_j}{2},{e^\omega g}}( s_j)  \rangle_{\H,U,e^{\omega}g}=\langle T_g( \mathbf{u})\bar T_g( \mathbf{v})\prod_{i=1}^nV_{\alpha_i,g}(z_i) \prod_{j=1}^m V_{\frac{\beta_j}{2},g}( s_j) \rangle_{\H,U,g}\\&\times\exp\Big(\frac{1+6Q^2}{96\pi}\Big(\int_{\Sigma}(|d\omega|_g^2+2K_g\omega) {\rm dv}_g+4 \int_{\partial^N \Sigma} k_{g}\omega d\lambda_{g}\big)-\sum_i\Delta_{\alpha_i}\omega(z_i)-\frac{1}{2}\sum_j\Delta_{\beta_j}\omega(s_j) \Big).
\end{align*}
\end{proposition}
\begin{proof}
The proof is essentially same as \eqref{Weyl2}. We omit the proof to make this paper shorter.
\end{proof}

\subsection{Glue generalized disk amplitudes with SET insertions to an amplitude with Neumann boundary.}
Now we want to glue the disk amplitudes with SET insertions to the upper half-plane with one interior hole $\Omega$ ($\partial\Omega$ is the gluing boundary). We note $\mathcal{S}:=\H\backslash\Omega$ and it equiped with metric $\hat g$. We define conformal bijective map $\psi_1:(\D,g)\to(\Omega,\psi_{1 *}g)$ and
 \begin{align}\label{TpsiSET}
 \psi_1^\ast T_{\hat g}:={\psi'_1}^2  T_{\hat g}\circ\psi_1+\tfrac{c_{\rm L}}{12}S_{\psi_1} .
\end{align}
We get the metric (still call it $\hat g$) $\hat g=\psi_{1 *}g\#\hat g$ on $\H$. We also introduce $\tilde{U}_t= \mathcal{Q}_t\cap{\H} $ where $\mathcal{Q}_t$ is a disk of center $s$ and radius $e^{-t}$.
 We claim

\begin{proposition}\label{glluu}
Let $\alpha,\beta\in\R $ and artificial weights $\beta^{\rm ar}_i\in\R$ for $i=1,\dots,m$ satisfy the Seiberg bounds and the artificial insertions ${\bf   s}^{\rm ar}_i\in\R$ for $i=1,\dots,m$ and 
$\mathbf{u}\in (\D_{e^{-t}})^{k}$, $\mathbf{v}\in (\D_{e^{-t}})^{\tilde{k}}$. Then
\begin{align}\label{pantclaim1}
&\sqrt{\pi}\caA^m_{\mathcal{S}, \tilde{U}_t ,\hat{g},s,\beta, {\bf   s}^{\rm ar},\boldsymbol{\beta}^{\rm ar}}(\Psi_{t, \alpha}(\mathbf{u},\mathbf{v}))\\
 = & ( Z(g)Z_{\D,g_\D}^{-1})\langle  \psi_1^\ast T^{\H}_{\hat g}(\mathbf{u}){\psi_1^\ast \bar T^{\H}_{\hat g}}(\mathbf{v})V_{\alpha,\hat g}(z)V_{\frac{\beta}{2},\hat g}(s)\prod_{i=1}^mV_{\frac{\beta_i^{\rm ar}}{2},\hat g}(s_i^{\rm ar})\rangle_{\H,\hat g,U_t}\nonumber
\end{align}
where $z=\psi_1(0)$ and $U_t=\psi_1(\D_{e^{-t}}) \cup\tilde{U}_t $.
\end{proposition}

\begin{proof} First we observe that the LHS is well defined: consider the  generalized amplitudes   $\Psi_{t, \alpha}(\mathbf{u},\mathbf{v})$, which belong respectively to  $e^{-\epsilon c_- }L^2(\R\times \Omega_\T)$ with $\epsilon>Q-\alpha$. The LHS makes sense provided that  $(\alpha-Q)+s>0$ (with $s=\frac{\beta}{2}+\sum_i\frac{\beta^{\rm ar}_i}{2}$), which is satisfied due to the Seiberg bound. Then we use gluing formula \eqref{gluingtypeI}  
\begin{align}\label{glueS}
    \caA_{\H,U_t,z,\alpha,s,\beta,{\bf   s}^{\rm ar},\boldsymbol{\beta}^{\rm ar}}(\psi_1^*T^{\H}_{\hat{g},\epsilon}({\bf u})\psi_1^\ast \bar T^{\H}_{\hat g,\epsilon}(\mathbf{v}))=&\sqrt{\pi}\int \caA_{\H\backslash\Omega,\mathcal{Q}_t,z,\alpha,s,\beta,{\bf   s}^{\rm ar},\boldsymbol{\beta}^{\rm ar}}(\tilde{\varphi})\nonumber\\&\times\caA_{\psi_1(\D).\psi_1(\D_{e^{-t}}),z,\alpha}(T^{(1)}_{\epsilon}(\mathbf{u})\bar T^{(1)}_{\epsilon}(\mathbf{v}),\tilde{\varphi}) d\mu_0(\tilde{\varphi})
\end{align}
where
\begin{align}\label{apu}
T^{(1)}_{\epsilon}(u):=T^{\psi_1}_{\hat g_1,\epsilon}(u)+\psi'_1(u)^2(a_{\H,\hat g,\epsilon}-a_{\psi_1(\D),\hat g_j})(\psi_1(u))+\frac{1}{12}S_{\psi_1}(u).
\end{align}
By \eqref{constanteSET}, we have $\lim_{\epsilon\to 0}a_{\H,\hat g,\epsilon}=0$ and
\begin{align}\label{apu1}
\lim_{\epsilon\to 0}\psi'_1(u)^2a_{\psi_1(\D),\hat g}(\psi_1(u))=\frac{1}{12}S_{\psi_1}(u)+a_{\D,g}(u)=\frac{1}{12}S_{\psi_1}(u)
\end{align}
Since $a_{\D,g}(u)=0$.
The equality \eqref{glueS} holds  provided that   the integrand in the RHS is integrable but this is again a consequence of the Seiberg bound. Also, these estimates ensure that we can pass to the limit as $\epsilon\to 0$.
\end{proof} 
It will be convenient to consider any metric Euclidean in the unit disk: we choose the DOZZ metric $g_{\rm dozz}=|z|_+^{-4} |dz|^2$.  Proposition \ref{WeylSET} gives
\begin{align}\label{Zzhat1}
&\langle  \psi_1^\ast T^{\H}_{\hat g}(\mathbf{u}){\psi_1^\ast \bar T^{\H}_{\hat g}}(\mathbf{v})V_{\alpha,\hat g}(z)V_{\frac{\beta}{2},\hat g}(s)\prod_{i=1}^mV_{\frac{\beta^{\rm ar}_i}{2},\hat g}(s^{\rm ar}_i)\rangle_{\H,\hat g,U_t}\\&=Z(\H,\hat g)\langle  \psi_1^\ast T^{\H}_{g_{dozz}}(\mathbf{u}){\psi_1^\ast \bar T^{\H}_{g_{dozz}}}(\mathbf{v})V_{\alpha,g_{dozz}}(z)V_{\frac{\beta}{2},g_{dozz}}(s)\prod_{i=1}^mV_{\frac{\beta^{\rm ar}_i}{2},g_{dozz}}(s^{\rm ar}_i)\rangle_{\H,g_{dozz},U_t}\nonumber
\end{align}
with
\begin{equation}\label{Zzhat}
Z(\H,\hat g):=e^{-c_L S^0_{\rm L}(\H,g_{\rm dozz},\hat \omega)-\Delta_{\alpha}\hat\omega(z)-\Delta_{\beta}\hat\omega(s)-\sum_{i=1}^m\Delta_{\beta_i}\hat\omega(s_i)}
\end{equation}
where we wrote  $\hat g =e^{ \hat\omega} g_{\rm dozz}$ for some function $ \hat\omega\in H^1(\H)$.

\noindent Next we show $\caA^m_{\mathcal{S}, \tilde{U}_t ,\hat{g},s,\beta, {\bf   s}^{\rm ar},\boldsymbol{\beta}^{\rm ar}}$ converges as linear functional when $t\to \infty$
\begin{proposition}\label{limitofbeta}
For $(\alpha,\beta,\beta_i^{\rm ar})$ in the Seiberg bound, the following limit exists
\begin{align}
    \lim_{t\to\infty}\caA^m_{\mathcal{S}, \tilde{U}_t ,\hat{g},s,\beta, {\bf   s}^{\rm ar},\boldsymbol{\beta}^{\rm ar}}=\caA^m_{\mathcal{S},\hat{g},s,\beta, {\bf   s}^{\rm ar},\boldsymbol{\beta}^{\rm ar}}
\end{align}
where the converge is as continuous linear functional on $e^{-\epsilon c_-}L^2(\R\times\Omega_{\T})$ for $\epsilon<\frac{\beta}{2}+\sum_i\frac{\beta_i^{\rm ar}}{2}$
\end{proposition}
\begin{proof}
It's easy to see above convergence holds $\mu_0$-almost surely, then we can use the bound \eqref{estimateamplitude} to deduce for $\delta<\frac{\beta}{2}+\sum_i\frac{\beta_i^{\rm ar}}{2}$
\begin{align}
    |\caA^m_{\mathcal{S}, \tilde{U}_t ,\hat{g},s,\beta, {\bf   s}^{\rm ar},\boldsymbol{\beta}^{\rm ar}}-\caA^m_{\mathcal{S} ,\hat{g},s,\beta, {\bf   s}^{\rm ar},\boldsymbol{\beta}^{\rm ar}}|\in e^{\delta c_-}L^2
\end{align}
Then we can use dominate convergence to see
\begin{align*}
    \lim_{t\to\infty}\int|\caA^m_{\mathcal{S}, \tilde{U}_t ,\hat{g},s,\beta, {\bf   s}^{\rm ar},\boldsymbol{\beta}^{\rm ar}}-\caA^m_{\mathcal{S}, \hat{g},s,\beta, {\bf   s}^{\rm ar},\boldsymbol{\beta}^{\rm ar}}|^2e^{-2\epsilon c_-}d\mu_0(\tilde{\varphi})=0.
\end{align*}
\end{proof}
We then have our final result about gluing
(recall the definition of $\bf a$ in Proposition \ref{eigenana}):
\begin{proposition}\label{witht}
Let $\alpha,\beta\in\R $ and $\beta^{\rm ar}_i\in\R$ $i=1,..,m$ satisfy the Seiberg bound, together with the condition $\alpha <\bf a$. Consider Young diagrams $\nu,\tilde\nu$, Then 
\begin{align}\label{pantclaimmm}
& \sqrt{\pi} \caA_{\H\backslash\Omega, \hat g,\alpha,z,\beta,s,{\bf s},\boldsymbol{\beta}^{\rm ar}}\big(\Psi_{\alpha,\nu,\tilde\nu}\big)=  ( Z(g)Z_{\D,g_\D}^{-1})Z(\H,\hat g)\lim_{t\to\infty}
   \frac{1}{(2\pi i)^{(s(\nu)+s(\tilde{\nu}))}} \\ &\times\oint_{|\mathbf{u}|=\boldsymbol{\delta}_t}   \oint_{|\mathbf{v}|=\tilde{\boldsymbol{\delta}}_t}  
\mathbf{u}^{1-\nu}\bar{\mathbf{v}}^{1-\tilde{\nu}} \langle  \psi_1^\ast T^{\H}_{g_{dozz}}(\mathbf{u}){\psi_1^\ast \bar T^{\H}_{g_{dozz}}}(\mathbf{v})V_{\alpha,g_{dozz}}(z)V_{\frac{\beta}{2},\hat g}(s)\prod_{i=1}^mV_{\frac{\beta^{\rm ar}_i}{2},g_{dozz}}(s^{\rm ar}_i)\rangle_{\H,g_{dozz},U_t}\dd  \mathbf{u}\dd  \mathbf{v}   \nonumber
\end{align}
 we used the notations 
${\bf u}=( u_1,\dots, u_k)$ and ${\bf v}=(v_1,\dots, v_{\tilde{k}})$, $\boldsymbol{\nu}=(\nu_1,..,\nu_k)$, $\boldsymbol{\tilde{\nu}}=(\tilde{\nu}_1,..,\tilde{\nu}_k)$ the powers $\mathbf{u}^{1-\boldsymbol{\nu}}$ and $\bar{\mathbf{v}}^{1-\tilde{\boldsymbol{\nu}}} $ are shorthands respectively for  $\prod_{i=1}^k\mathbf{u}_i^{1-\nu_i}$ and $\prod_{i=1}^{\tilde{k}}\bar{\mathbf{v}}_i^{1-\tilde \nu_i}$, and $\boldsymbol{\delta}_t=e^{-t}\boldsymbol{\delta}$. 
\end{proposition} 

\section{Ward's identity in the upper half-plane}

In this section, we prove the Ward's identity in the upper half-plane. 
For our interest, we only need to insert one pair of SET insertions near imaginary number $i$, but this method can be generalized to any number SET insertions in the interior of upper half-plane. We refer the readers to \cite[Appendix E]{GKRV21} for the derivation of Ward's identity in the full plane $\hat \C$ case. The main difficult here is that in the upper half-plane, the action of SET $T^{\H}(z)$ and $\overline{T}^{\H}(z)$ are coupled with each other. So clarify the form of Ward's identity needs some novel ideas. John Cardy first studied the conformal field theory of the Ising model in the upper half-plane by playing a doubling trick. Cardy defines SET by variation method and he describes a strategy to link the conformal block of Ising model 2-spin correlation in the upper half-plane to the full plane 4-spin correlation. In our case, we define the SET by Gaussian functional to inherit the Markov property of GFF, which makes Cardy's doubling trick highly nontrivial.

Now we introduce some basic notations needed in Ward's identity, we only need to compute the generalized correlation function of the following form
\begin{align}\label{amplitudeSET2}
     \langle T^{\H}_g( \mathbf{u})\bar T^{\H}_g( \mathbf{v})V_{\alpha,g}(z)V_{\frac{\beta}{2},g}( s)\prod_{i=1}^m V_{\frac{\beta^{\rm ar}_i}{2},g}( s^{\rm ar}_i) \rangle_{\H,U,g}
\end{align}
Let  $\psi_1:\D\to \H$ be a conformal map and $\caD=\psi_1(\D)$. Let  $A_j$,  $j=1,\dots, k$ and $ \tilde A_j $, $j=1,\dots, \tilde k$ be disjoint annuli in $\D$ surrounding a disjoint ball $B$ centred at origin. For $t\geq 0$ we set  $\caD_{t}=\psi_1(e^{-t}\D)$, let $\caA_{j,t}=\psi_1(e^{-t}A_j)$
and similarly for $\tilde\caA_{j,t}$ and $\caB_{t}=\psi_1(e^{-t}B)$. Thus we have a family of conformal annuli surrounding a conformal ball in $\caD_{t}$ and all these sets are separated from each other and from the boundary $\partial \caD_{t}$ by a distance $\mathcal{O}(e^{-t})$. We also introduce $\tilde{U}_t= \mathcal{Q}_t\cap{\H} $ where $\mathcal{Q}_t$ is a ball of center $s$ and radius $e^{-t}$.  

We then consider the correlation function \eqref{amplitudeSET2} with $U=U_t= \caD_{t}\cup \tilde{U}_t$, ${\bf u}=(u_1,u_2,...,u_k)$ and  ${\bf v}=(v_1,v_2,..,v_{\tilde{k}})$ with $u_{j}\in \caA_{j,t}$, similarly for   $v_j$. $\partial U_t=\partial \caD_t\cup \{s+e^{-t+i\theta}\mid \theta\in (0,\pi)\}$.
 The Ward identity will involve derivatives in all the variables $\bf u,v,z,s$ and we will consider these in the distributional sense in the variables $s^{\rm ar}_i$ for $i=1,\dots,m$. 
  \subsection{Cardy's doubling trick}\label{Cardytrick}
We choose the map $\psi_1(0)=i$ and define $\caD$ as above.
We suppose $v_i\in \tilde\caA_{i,t}$ for $1\leq i\leq \tilde k$ and $u_i\in \caA_{i,t}$ for $1\leq i\leq k$. We can order $\frac{1}{2}e^{-t}>|u_k-i|>|u_{k-1}-i|>...>|u_1-i|>|v_{\tilde{k}}-i|>|v_{\tilde{k}-1}-i|>...>|v_1-i|>ce^{-t}$ and define a new sequence $x_i$, 
\begin{equation}
    x_{i}= \begin{cases}\overline{v_{i}}, & \text { if } 1 \leqslant i \leqslant \widetilde{k} \\ u_{i}, & \text { if } \tilde{k}+1 \leqslant i \leqslant \tilde{k}+k\end{cases}
\end{equation}
To simplify our notation, we define 
$$\mathcal{T}(\mathbf{x})=\prod_{i=\tilde{k}+1}^{k+\tilde{k}}T^{\H}_g(u_i)\prod_{i=1}^{\tilde{k}}\bar{T}^{\H}_g(v_i)=T^{\H}_g(\mathbf{u})\bar{T}^{\H}_g(\mathbf{v})$$ and $\mathcal{T}(\mathbf{x}^{(\ell)})$ by removing the $\ell_{th}$ coordinate in $u_i$ or $v_i$, depending on the number  $\ell$. And similarly for $\mathcal{T}(\mathbf{x}^{(\ell,\ell')})$. In the following, the $x_\ell$ in the notation $X(x_\ell)$ or $\partial_{x_\ell}X(x_\ell)$ should be understood as $u_\ell$ or $v_\ell$ depending on the number $\ell$.\\
To make following formulae shorter, we also define $\mathcal{V}_{\mathbf{\alpha}}(\mathbf{z}):=\prod_{i=1}^nV_{\alpha_i }(z_i)$ and $\mathcal{V}_{\frac{\mathbf{\beta}}{2}}(\mathbf{s})=\prod_{j=1}^mV_{\frac{\beta_j}{2} }(s_j)$.
We define $U_t$ as the union of a disk $\caD_t$ centered at $z_1$ and a half disk $\mathcal{Q}_t\cap \H$ centered at $s_1$. Denote $\H_t=\H\backslash U_t$ and $\R_t=\R\backslash U_t$.
\begin{proposition}\label{WardTH}
 Recall that $c_L=1+6Q^2$, when the metric is $g=e^{\omega}|dz|^2$ we have
 \begin{align} \label{WardT}
 \langle  \mathcal{T}(\mathbf{x})\prod_{i=1}^n V_{\alpha_i }(z_i)\prod_{j=1}^m V_{\frac{\beta_j}{2} }(s_j) \rangle_{\H,U_t,g} =& \frac{1}{2}\sum_{\ell=1}^{\tilde{k}+k-1}\frac{c_L}{(x_{\tilde{k}+k}-x_\ell)^4} \langle \mathcal{T}(\mathbf{x}^{(k+\tilde{k},\ell)}) \mathcal{V}_{\mathbf{\alpha}}(\mathbf{z})\mathcal{V}_{\frac{\mathbf{\beta}}{2} }(\mathbf{s}) \rangle_{\H,U_t,g} \nonumber
\\
 &+\sum_{\ell=1}^{\tilde{k}+k-1}\Big(\frac{2}{(x_{\tilde{k}+k}-x_\ell)^2} +\frac{1}{(x_{\tilde{k}+k}-x_\ell) }\partial_{x_\ell} \Big)\langle \mathcal{T}(\mathbf{x}^{(k+\tilde{k})})\mathcal{V}_{\mathbf{\alpha}}(\mathbf{z})\mathcal{V}_{\frac{\mathbf{\beta}}{2} }(\mathbf{s})\rangle_{\H,U_t,g}\nonumber
  \\
  &+\sum_{p=1}^{n}\Big(\frac{\Delta_{\alpha_p}}{(x_{\tilde{k}+k}-z_p)^2} +\frac{\partial_{z_p}+\Delta_{\alpha_p}\partial_{z_p}\omega}{(x_{\tilde{k}+k}-z_p)  } \Big)\langle  \mathcal{T}(\mathbf{x}^{(k+\tilde{k})})\mathcal{V}_{\mathbf{\alpha}}(\mathbf{z})\mathcal{V}_{\frac{\mathbf{\beta}}{2} }(\mathbf{s}) \rangle_{\H,U_t,g} \nonumber
  \\
   &+\sum_{p=1}^{n}\Big(\frac{\Delta_{\alpha_p}}{(x_{\tilde{k}+k}-\bar{z_p})^2} +\frac{\partial_{\bar z_p}+\Delta_{\alpha_p}\partial_{\bar{z}_p}\omega}{(x_{\tilde{k}+k}-\bar{z_p})  } \Big)\langle  \mathcal{T}(\mathbf{x}^{(k+\tilde{k})})\mathcal{V}_{\mathbf{\alpha}}(\mathbf{z})\mathcal{V}_{\frac{\mathbf{\beta}}{2} }(\mathbf{s}) \rangle_{\H,U_t,g} \nonumber
   \\
    &+\sum_{q=1}^{m}\Big(\frac{\Delta_{\beta_q}}{(x_{\tilde{k}+k}-s_q)^2} +\frac{\partial_{s_q}+\frac{1}{2}\Delta_{\beta_q}\partial_{s_q}\omega}{(x_{\tilde{k}+k}-s_q)  } \Big)\langle  \mathcal{T}(\mathbf{x}^{(k+\tilde{k})})\mathcal{V}_{\mathbf{\alpha}}(\mathbf{z})\mathcal{V}_{\frac{\mathbf{\beta}}{2} }(\mathbf{s}) \rangle_{\H,U_t,g} \nonumber
    \\
  &-\frac{i}{2}\oint_{\partial U_t}\frac{\mu}{  x_{\tilde{k}+k}-  x}\langle   \mathcal{T}(\mathbf{x}^{(k+\tilde{k})}) V_{\gamma }(x)\mathcal{V}_{\mathbf{\alpha}}(\mathbf{z})\mathcal{V}_{\frac{\mathbf{\beta}}{2} }(\mathbf{s}) \rangle_{\H,U_t,g} \dd\bar x\\
   &+\frac{i}{2}\oint_{\partial U_t}\frac{\mu}{  x_{\tilde{k}+k}-  \bar{x}}\langle   \mathcal{T}(\mathbf{x}^{(k+\tilde{k})}) V_{\gamma }(x)\mathcal{V}_{\mathbf{\alpha}}(\mathbf{z})\mathcal{V}_{\frac{\mathbf{\beta}}{2} }(\mathbf{s}) \rangle_{\H,U_t,g} \dd x\\
   &+\frac{\mu_\partial}{x_{\tilde{k}+k}-(s_1+e^{-t})}\langle   \mathcal{T}(\mathbf{x}^{(k+\tilde{k})})  V_{\frac{\gamma}{2}}(s_1+e^{-t})\mathcal{V}_{\mathbf{\alpha}}(\mathbf{z})\mathcal{V}_{\frac{\mathbf{\beta}}{2} }(\mathbf{s})\rangle_{\H,U_t,g}\nonumber\\
   &-\frac{\mu_\partial}{x_{\tilde{k}+k}-(s_1-e^{-t})}\langle   \mathcal{T}(\mathbf{x}^{(k+\tilde{k})})  V_{\frac{\gamma}{2}}(s_1-e^{-t})\mathcal{V}_{\mathbf{\alpha}}(\mathbf{z})\mathcal{V}_{\frac{\mathbf{\beta}}{2} }(\mathbf{s})\rangle_{\H,U_t,g}\nonumber
 \end{align}
 \end{proposition}
\begin{proof}
The proof is tedious and similar to Ward's identity in \cite[Appendix E]{GKRV21}. The difficulties come from the presence of the boundary cosmology constant and the coupling relation of $T^{\H}(\mathbf{u})$ and  $T^{\H}(\mathbf{v})$. We postpone its proof in the Appendix \ref{proof cardy}.
\end{proof}
\begin{remark}
    Here we explain why the contour integral \eqref{WardT} surround $\beta$ insertion makes sense. Notice here the $\beta$ is very negative, so when the $V_{\gamma}(x)$ approaches real line, the singularity behaves like $|x-\bar{x}|^{-\frac{\gamma^2}{2}}$ when $\gamma^2<\frac{4}{3}$ and $ |x-\bar{x}|^{\frac{Q^2}{4}-2}$ when $\gamma^2\geq\frac{4}{3}$. Using the polar coordinates, we get the integral $\int_{0}^{\pi}\frac{1}{|sin(\theta)|^{a}}d\theta$ with $a<1$, which is integrable. For this fusion type estimate, see \cite{Wu1}.
\end{remark}
The key observation from Cardy's doubling trick is the following.
\begin{proposition}\label{cardydoublingproperty}
Iterate the above recursive relation, we can get the following differential operator $D^{\H}(z_i,s_j,\mathbf{u},\mathbf{v})$ and error term $\epsilon_t(z_i,s_j,\mathbf{u},\mathbf{v})$
\begin{align*}
&\Big(\prod^n_{i=1} g(z_i)^{\Delta_{\alpha_i}} \prod_{j=1}^m g(s_j)^{\frac{\Delta_{\beta_j}}{2}}\langle  T^{\H}(\mathbf{u})\bar T^{\H}({ \mathbf{v}})\prod_{i=1}^nV_{\alpha_i }(z_i)\prod_{j=1}^mV_{\frac{\beta_j}{2} }(s_j) \rangle_{\H,U_t,g}\Big)\\ =&D^{\H}(z_i,s_j,\mathbf{u},\mathbf{v})\Big(\prod^n_{i=1} g(z_i)^{\Delta_{\alpha_i}} \prod_{j=1}^m g(s_j)^{\frac{\Delta_{\beta_j}}{2}}\langle  \prod_{i=1}^nV_{\alpha_i }(z_i)\prod_{j=1}^mV_{\frac{\beta_j}{2} }(s_j) \rangle_{\H,U_t,g}\Big)+\epsilon_t(z_i,s_j,\mathbf{u},\mathbf{v})
\end{align*}
and recall the full plane Ward's identity in \cite[Proposition 10.4]{GKRV21}
\begin{align*}
    &\Big(\prod^n_{i=1} g(z_i)^{\Delta_{\alpha_i}} \prod^n_{i=1}g(z'_i)^{\Delta_{\alpha_i}} \prod_{j=1}^m g(s_j)^{\Delta_{\beta_j}}\langle T^{\hat\C}(\mathbf{u})T^{\hat\C}({ \bar{\mathbf{v}}})\prod_{i=1}^nV_{\alpha_i }(z_i)\prod_{i=1}^nV_{\alpha_i}(z'_i)\prod_{j=1}^mV_{\beta_j }(s_j) \rangle_{\hat\C,U_t\cup\bar U_t,g}\Big)\\
    =&D^{\hat\C}(z_i,z'_i,s_j,\mathbf{u},\mathbf{\bar{v}})\Big(\prod^n_{i=1} g(z_i)^{\Delta_{\alpha_i}} \prod^n_{i=1}g(z'_i)^{\Delta_{\alpha_i}} \prod_{j=1}^m g(s_j)^{\Delta_{\beta_j}}\langle \prod_{i=1}^nV_{\alpha_i }(z_i)\prod_{i=1}^nV_{\alpha_i}(z'_i)\prod_{j=1}^mV_{\beta_j }(s_j) \rangle_{\hat\C,U_t\cup\bar U_t,g}\Big)\\
    &+\epsilon_t(z_i,z'_i,s_j,\mathbf{u},\mathbf{\bar{v}})
\end{align*}
Here $\bar{U}_t$ is the reflection of $U_t$ along the real line, then we have 
 $$ {D}^{\hat\C}(z_i,\bar{z}_i,s_j,\mathbf{u},\mathbf{\bar{v}})={D}^{\H}(z_i,s_j,\mathbf{u},\mathbf{v})$$
\end{proposition} 
\begin{proof}
We directly compare the first five terms in RHS of \eqref{WardTH}(the half plane case) with Ward's identity in \cite[Theorem 10.3]{GKRV21} (the full plane case) except for the last term. Then we get to the conclusion.
\end{proof}
  
\subsection{Estimation of error terms}
\begin{proposition}[\bf{Ward's identity in $\H$}]\label{beforeres}
For $\alpha,\beta,{\bf \beta}^{\rm ar}$ in the Seiberg bound and $g=e^{\omega}|dz|^2$. The following identity holds
 \begin{align}
 &e^{\Delta_{\alpha}\omega(z)} e^{\frac{\Delta_{\beta}}{2}\omega(s)}\prod_{i=1}^m e^{\frac{\Delta_{\beta^{\rm ar}_i}}{2} \omega (s^{\rm ar}_i)} \langle T^{\H}_g( \mathbf{u})\bar  T^{\H}_g( \mathbf{v})V_{\alpha,g}(z)V_{\frac{\beta}{2},g}( s)\prod_{i=1}^m V_{\frac{\beta^{\rm ar}_j}{2},g}( s^{\rm ar}_i) \rangle_{\H,U_t,g}\\
=  &D^{\H}({\bf u},{\bf  v},z,s,{\bf s}^{\rm ar}) \left (   e^{\Delta_{\alpha}\omega(z)} e^{\frac{\Delta_{\beta}}{2}\omega(s)}\prod_{i=1}^m e^{\frac{\Delta_{\beta^{\rm ar}_i}}{2} \omega (s^{\rm ar}_i)}\langle T^{\H}_g( \mathbf{u})\bar  T^{\H}_g( \mathbf{v})V_{\alpha,g}(z)V_{\frac{\beta}{2},g}( s)\prod_{i=1}^m V_{\frac{\beta^{\rm ar}_i}{2},g}( s^{\rm ar}_i) \rangle_{\H,U_t,g}\right )\nonumber\\
+&\epsilon_t{({\bf u},{\bf  v},z,s,{\bf s}^{\rm ar})}
\end{align}
 where  
\begin{align}\label{reminder}
|\epsilon_t{({\bf u},{\bf v},z,s,{\bf s}^{\rm ar})}|\leq Ce^{(k+\tilde{k}-1+\gamma \max\{\alpha,\beta,\frac{\beta}{2}\})t}
\end{align}
and the differential operator $D^{\H}({\bf u},{\bf v},z,s,{\bf s}^{\rm ar})$ doesn't involve the metric terms. 
\begin{proof}
We iterate \eqref{WardTH}, and gather all terms contain the contour integral on $\partial U_t$ and $\overline{\partial U_t}\cap \R=\{s_1+e^{-t}\}\cup\{s_1-e^{-t}\}$ (the $Q_t$ terms defined in \eqref{Q_tterm}). To estimate these contour integrals, we need the following lemma to control $x$ close to $z$ or $x$ close to $s$.
\begin{lemma}[\bf Fusion estimate]\label{fusion}
For $x\in \partial U_t$, we have
\begin{align}
   &\langle V_{\gamma}(x) V_{\alpha,g}(z)V_{\frac{\beta}{2},g}( s)\prod_{i=1}^m V_{\frac{\beta^{\rm ar}_i}{2},g}( s^{\rm ar}_i) \rangle_{\H,U_t,g}\leq |x-z|^{-\gamma\alpha}|x-s|^{-\gamma\beta}|x-\bar{x}|^{-\frac{\gamma^2}{2}}\quad \text{when }\gamma^2<\frac{4}{3}\\
   &\langle V_{\gamma}(x) V_{\alpha,g}(z)V_{\frac{\beta}{2},g}( s)\prod_{i=1}^m V_{\frac{\beta^{\rm ar}_i}{2},g}( s^{\rm ar}_i) \rangle_{\H,U_t,g}\leq |x-z|^{-\gamma\alpha}|x-s|^{-\gamma\beta}|x-\bar{x}|^{\frac{Q^2}{4}-2}\quad \text{when }\gamma^2\geq\frac{4}{3}\\
   &\langle V_{\frac{\gamma}{2}}(s+e^{-t}) V_{\alpha,g}(z)V_{\frac{\beta}{2},g}( s)\prod_{i=1}^m V_{\frac{\beta^{\rm ar}_i}{2},g}( s^{\rm ar}_i) \rangle_{\H,U_t,g}\leq e^{-\frac{\gamma}{2}\beta t}\nonumber
\end{align}
\end{lemma}
For proof of this lemma, see [Proposition 6.4]\cite{Wu1}.
\end{proof}
As in \cite[Lemma 10.6]{GKRV21}, we can show
\begin{proposition}\label{difflemma}
 Let  $L<2-\gamma\max\{\alpha,\beta,\frac{\beta}{2}\}$ and fix $t_0 \geq 0$. For all $t \geq 0$, the correlation functions $$\langle V_{\alpha,g}(z)V_{\frac{\beta}{2},g}( s)\prod_{i=1}^m V_{\frac{\beta{\rm ar}_i}{2},g}( s^{\rm ar}_i) \rangle_{\H,U_t,g}$$  are  $C^L$ in $(z,s)\in U_{t_0}=\caD_{t_0}\times(\mathcal{Q}_{t_0}\cap\R)$ and $C^0$  in $(s^{\rm ar}_{1},\dots, s^{\rm ar}_n)$ in the region $s^{\rm ar}_i\in \R\backslash \overline{U}_{t_0}$ with $s^{\rm ar}_i\neq s^{\rm ar}_j$ and they converge uniformly on compact subsets of the aforementioned region together with all these derivatives as $t\to \infty$ to $\langle V_{\alpha,g}(z)V_{\frac{\beta}{2},g}( s)\prod_{i=1}^m V_{\frac{\beta^{\rm ar}_i}{2},g}( s^{\rm ar}_i) \rangle_{\H,g}$.
\end{proposition}
\begin{proof}
    Here we only point out the main difference. Define $(\mathcal{Q}_t\cap \H)\subset(\mathcal{Q}_{t_0}\cap \H)$, we come to estimate the following integral
    \begin{align*}
        \int_{(\mathcal{Q}_{t_0}\cap \H)\setminus(\mathcal{Q}_t\cap \H)}  (|s-\omega_{\pi(1)}|+\epsilon)^{-L} \langle V_{\gamma}(\omega_{\pi(1)}) V_{\alpha,g}(z)V_{\frac{\beta}{2},g}( s)\prod_{i=1}^m V_{\frac{\beta^{\rm ar}_i}{2},g}( s^{\rm ar}_i) \rangle_{\H,U_t,g} d\omega_{\pi(1)}
    \end{align*}
    by Lemma \ref{fusion}, we know when $L<2-\gamma\max\{\alpha,\beta,\frac{\beta}{2}\} $, this proposition holds.
\end{proof}
Computing RHS of \eqref{witht} by Ward's identity, we get
\begin{proposition}\label{withtbis}
 Given $\boldsymbol{\nu}=(\nu_1,\dots,\nu_k)\in \mc{T}$, $\tilde{\boldsymbol{\nu}}=(\tilde{\nu}_1,\dots,\tilde{\nu}_{\tilde{k}})\in \mc{T}$, $\max\{\alpha,\beta,\frac{\beta}{2}\}<\bf{a}$ and  $\beta^{\rm ar}_i<Q$ for $i=1,..,m$ satisfying $\alpha+\frac{\beta}{2}+\sum_{i=1}^m\frac{\beta^{\rm ar}_i}{2}>Q$
  \begin{align*}
 & \oint_{|\mathbf{u}|=\boldsymbol{\delta}_t}   \oint_{|\mathbf{v}|=\boldsymbol{\tilde\delta}_t}  
\mathbf{u}^{1-\boldsymbol{\nu}}\bbar{\mathbf{v}}^{1-\tilde{\boldsymbol{\nu}}}
\langle\psi_1^\ast T^{\H}_{g_{dozz}}(\mathbf{u}){\psi_1^\ast \bar T^{\H}_{g_{dozz}}}(\mathbf{v})V_{\alpha,g_{dozz}}(z)V_{\frac{\beta}{2},g_{dozz}}(s)\prod_{i=1}^mV_{\frac{\beta^{\rm ar}_i}{2},g_{dozz}}(s^{\rm ar}_i)\rangle_{\H,g_{dozz},U_t}
\dd{\bf u}\dd{\bf v}\\
  &= (2\pi i)^{(s(\nu)+s(\tilde{\nu}))}\caD^{\H}_{\nu,\tilde{\nu},\alpha,z,\beta,s,\boldsymbol{\beta}^{\rm ar},{\bf s}^{\rm ar}}
  \langle V_{\alpha,g}(z)V_{\frac{\beta}{2},g}( s)\prod_{i=1}^m V_{\frac{\beta^{\rm ar}_i}{2},g}( s^{\rm ar}_i) \rangle_{\H,U_t,g_{dozz}}+\epsilon'_t{(\nu,\tilde{\nu},\alpha,z,\beta,s,\boldsymbol{\beta}^{\rm ar},{\bf s}^{\rm ar})}
 \end{align*}
where $\epsilon'_t{(\nu,\tilde{\nu},\alpha,z,\beta,s,\boldsymbol{\beta}^{\rm ar},{\bf s}^{\rm ar})}\to 0$ as $t\to\infty$.
\end{proposition} 
\begin{proof}
Here we only convince readers why we add the condition $\beta, \frac{\beta}{2}<\bf a$.
\begin{align*}
|\epsilon'_t{(\nu,\tilde{\nu},\alpha,z,\beta,s,\boldsymbol{\beta}^{\rm ar},{\bf s}^{\rm ar})}|&\leq\oint_{|\mathbf{u}|=\boldsymbol{\delta}_t}   \oint_{|\mathbf{v}|=\boldsymbol{\tilde\delta}_t}  
|\mathbf{u}^{1-\boldsymbol{\nu}}\bar{\mathbf{v}}^{1-\tilde{\boldsymbol{\nu}}} \epsilon_t{({\bf u},{\bf \bar v},z,s,{\bf s}^{\rm ar})}|\dd{\bf u}\dd{\bf v} \\&\leq C e^{(|\boldsymbol{\nu}|+|\tilde{\boldsymbol{\nu}}|-2k-2\tilde k)t}e^{(2(k+\tilde k-1)+\gamma \max\{\alpha,\beta,\frac{\beta}{2}\})t}
\end{align*}

and when $\max\{\alpha,\beta,\frac{\beta}{2}\}<\bf a$, $|\epsilon'_t{(\nu,\tilde{\nu},\alpha,z,\beta,s,\boldsymbol{\beta}^{\rm ar},{\bf s}^{\rm ar})}|\to 0$ as $t\to\infty$.
\end{proof}
This estimate shows when we shrink the radius of circle to 0, we can ignore the $\partial U_t$ terms and $Q_t$ terms in Ward's identity.
By choosing number $\bf a$ sufficiently small, one can apply Lemma \ref{difflemma} to get the following convergence
$$
\caD^{\H}_{\nu,\tilde{\nu},\alpha,z,\beta,s,\boldsymbol{\beta}^{\rm ar},{\bf s}^{\rm ar}}
  \langle V_{\alpha,g}(z)V_{\frac{\beta}{2},g}( s)\prod_{i=1}^m V_{\frac{\beta_j}{2},g}( s_i) \rangle_{\H,U_t,g_{dozz}}\to \caD^{\H}_{\nu,\tilde{\nu},\alpha,z,\beta,s,\boldsymbol{\beta}^{\rm ar},{\bf s}^{\rm ar}}
  \langle V_{\alpha,g}(z)V_{\frac{\beta}{2},g}( s)\prod_{i=1}^m V_{\frac{\beta_j}{2},g}( s_i) \rangle_{\H,g_{dozz}}
$$
as $t\to \infty$ in the strong sense in the variables $z,s$ and in the distributional sense for the variables ${\bf s}^{\rm ar}$. 

Now we prove the following theorem to illustrate the amplitude can be analytic continued to some analytic neighbourhood of  $(\beta,\boldsymbol{\beta}^{\rm ar})\in (\infty,Q)^{m+1}$.

\begin{proposition}[Weights analyticity of amplitudes]\label{holomorphic}
\label{anala} There exists a complex neighborhood $ \mc{U}$ of $\{  (\beta,\boldsymbol{\beta}^{\rm ar})  \in \R\times\R^{m} \, |\, \, (\beta,\boldsymbol{\beta}^{\rm ar})\in (\infty,Q)^{m+1}\}$  s.t. the function   
\begin{align*}
(\beta,\boldsymbol{\beta}^{\rm ar}, s,{\bf  s}^{\rm ar})\to \caA^m_{\H\backslash\Omega,\hat g,s,\beta,{\bf s}^{\rm ar},\boldsymbol{\beta}^{\rm ar}}
\end{align*}
extends holomorphically in the variables $(\beta,\boldsymbol{\beta}^{\rm ar})\in \mc{U}$. This extension is continuous in 
  $(\beta,\boldsymbol{\beta}^{\rm ar}, s,{\bf  s}^{\rm ar})\in  \mc{U}\times \{(s,{\bf s}^{ar})\in (\H\backslash\Omega)^{m+1}| \text{all s and $s^{ar}$ are not same}\}$   with values in 
$$e^{\epsilon\bar c_--R\bar c_+-a\sum_{j=1}^b(c_j-\bar c)^2} L^2(H^{-s}(\T)^{b},\mu_0^{\otimes_b})$$ for some $a>0$, arbitrary $\epsilon<s$ with $s:= {\rm Re}(\frac{\beta}{2}) +\sum_{i=1}^m{\rm Re}(\frac{\beta_i}{2})$ and arbitrary $R>0$. 
\end{proposition}
\begin{proof}
we choose a metric $g$ s.t. the Ricci curvature in $\H\backslash\Omega$ is 0, and the geodesic curvature on boundary $\R$ and $\partial\Omega$ are $0$. For simplicity, we consider the case $\boldsymbol{\beta}^{\rm ar}=0$. We define $\mathbb{R}_{k}:=\mathbb{R} \backslash \left(s-e^{-k}, s+e^{-k}\right)$ and
\begin{align}\label{estimateextension}
\caA^{m,k}_{\H\backslash\Omega,g,{\bf s},\boldsymbol{\beta}}(\tilde{\varphi})=&\E_{\tilde{\varphi}} \Big[  \epsilon_k^{\frac{\beta^2}{4}} e^{\frac{\beta}{2}(c+\phi^{\circ}_{g,k}(s))}\exp\Big(-\mu  e^{\gamma c}M_{\gamma, g}(\phi^{\circ}_g,\H)-\mu_{\partial} e^{\frac{\gamma}{2}c}M^{\partial}_{\gamma, g}(\phi^{\circ}_g,\R_k )\Big) \Big] 
\end{align}
where $\phi^{\circ}_{g,k}=P\varphi+X_{g,m,k}^{\H\backslash\Omega}$ and $X^{\H\backslash\Omega}_{g,m,k}(s)$ is regularized by taking average of $X^{\H\backslash\Omega}_{g,m}(s)$ with respect to the upper half-plane circle $B(s,e^{-k})^+$.

The Markov property of the mixed condition GFF $X^{\H\backslash\Omega}_{g,m}(s)$ implies that $X^{\H\backslash\Omega}_{g,m,k+t}(s)-X^{\H\backslash\Omega}_{g,m}(s)$ is a standard Brownian motion independent of $F(X_{g,m}(z))$, $z\in \R_k$. Note $\beta=a+ib$, then 
by Markov property
\begin{align*}
    |A^{m,k+1}_{\H\backslash\Omega,g,{\bf s},\boldsymbol{\beta}}(\tilde{\varphi})-& A^{m,k}_{\H\backslash\Omega,g,{\bf s},\boldsymbol{\beta}}(\tilde{\varphi})|\leq e^{\frac{a}{2}c}|\E_{\tilde{\varphi}} \Big[\epsilon_{k+1}^{\frac{(a+ib)^2}{4}}e^{\frac{a+ib}{2}X^{\H\backslash\Omega}_{g,m,k+1}(s)}\Big(e^{-\mu_{\partial} e^{\frac{\gamma}{2}c}M^{\partial}_{\gamma, g}(\phi^{\circ}_g,\R_{k+1})} -e^{-\mu_{\partial} e^{\frac{\gamma}{2}c}M^{\partial}_{\gamma, g}(\phi^{\circ}_g,\R_k)}\Big)\Big] |\\
    \leq&e^{\frac{a}{2}c}2^\frac{(k+1)b^2}{4}|\E_{\tilde{\varphi}}[\Big(e^{-\mu_\partial e^{\frac{\gamma}{2}c}(Y_k+\delta Y_k)}-e^{-\mu_\partial e^{\frac{\gamma}{2}c}Y_k}\Big)e^{-\mu  e^{\gamma c}M_{\gamma, g}(\phi_g+G_{m,k+1}(x,s),\H)}]|
\end{align*}
where the regularized Green function $G_{m,k}(x,s)=\E[X^{\H\backslash\Omega}_{g,m,k}(s)X^{\H\backslash\Omega}_{g,m}(u)]$ and 
\begin{align}
     &Y_k=M^{\partial}_{\gamma, g}(\phi^{\circ}_g+G_{m,k+1}(x,s),\R_k)\\
     &\delta Y_k=M^{\partial}_{\gamma, g}(\phi^{\circ}_g+G_{m,k+1}(x,s),\R_k\backslash\R_{k+1})
\end{align}
 Now we use that for any fixed $R\in (0,1)$, there exists a constant $C_R$ such that $1-e^{-x}<C_R x^R$ holds for all $x>0$. Then we take $x=\mu_\partial e^{\frac{\gamma}{2}c}\delta Y_k$. 
 Then we use multi-fractal property for GMC and choose suitable $R$ and $b$ to conclude $A^{m,k}_{\H\backslash\Omega,g,{\bf s},\boldsymbol{\beta}}(\tilde{\varphi})$ converges locally uniform. This part follows the same lines in \cite[Proposition C.1.]{GKRV}.
\end{proof}
We may now combine this result with Proposition \ref{witht}, Proposition \ref{holomorphic} and 
Lemma  \ref{TTLemma} to get
\begin{proposition}[{Ward's identity for $\mathbb{H}$}]\label{abr}
Given $\boldsymbol{\nu},\tilde{\boldsymbol{\nu}}$ Young diagrams, $\max\{\alpha,\beta,\frac{\beta}{2}\}< \bf a$ for  and  $\beta^{\rm ar}_i<Q$ for $i=1,..,m$ satisfying $\alpha+\frac{\beta}{2}+\sum_i\frac{\beta^{\rm ar}_i}{2}>2Q$ 
 \begin{align}\label{final1}
 & \sqrt{\pi} \caA_{\H\backslash\Omega,\hat g,z,\alpha,s,\beta,{\bf s}^{\rm ar},\boldsymbol{\beta}^{\rm ar}}\big(\Psi_{\alpha,\nu,\tilde{\nu}}\big)\\
  =& ( Z(g)Z_{\D,g_\D}^{-1})Z(\H,\hat g)
\caD^{\H}_{\nu,\tilde{\nu},\alpha,z,\beta,s,\boldsymbol{\beta}^{\rm ar},{\bf s}^{\rm ar}}
  \langle V_{\alpha,g}(z)V_{\frac{\beta}{2},g}( s)\prod_{i=1}^m V_{\frac{\beta_i^{\rm ar}}{2},g}( s^{\rm ar}_i) \rangle_{\H,g_{dozz}}\nonumber
 \end{align}
 The LHS in  \eqref{final1} is continuous in $(z,s,{\bf s}^{\rm ar})$ and holomorphic in $(\alpha,\beta,\boldsymbol{\beta}^{\rm ar})$ in a complex neighborhood of Seiberg bounds and the 
 equality \eqref{final1} holds in the distributional sense. 
 \end{proposition}

\end{proposition}

\section{Properties of bulk-boundary correlator \texorpdfstring{$G(Q+iP,\beta)$}{g0}}\label{gluingdefinition of G}
In this paper, we define the bulk boundary correlator $G(\alpha,\beta)$ as following:
\begin{definition}
We define the Liouville bulk-boundary correlator as
\begin{equation}\label{eigenstatedefineof G}
G(\alpha,\beta):=\frac{\mathcal{A}^{m}_{g_{\mathbb{A}},\mathbb{A}^1_{\sqrt{q}},b,\beta,\sqrt{q}e^{i\theta}}(\Psi_{\alpha})}{C_{\D}|q |^{-{c}_{\rm L}/24+\Delta_{\alpha}}}
\end{equation}
\end{definition}
We will show $(\alpha,\beta)\to G(\alpha,\beta)$ is analytic in the region \eqref{analyticregion}. When we restrict it to the Seiberg bound, due to the gluing formula 
\eqref{gluingtypeI}, it fits into to the Liouville correlation function:
\begin{align*}
 &\langle V_{\alpha}(z)  V_{\frac{\beta}{2}}(s)  \rangle^{\H}_{g,\mu,\mu_\partial}\\
 =&|z-\bar{z}|^{\Delta_{\beta}-2 \Delta_{\alpha}}|z-s|^{-2 \Delta_{\beta}}
g(z)^{-\Delta_{\alpha}}g(s)^{-\frac{\Delta_{\beta}}{2}}\nonumber \frac{1}{2}G(\alpha,\beta) \big(\frac{{\rm v}_{g}(\H)}{{\det}'(\Delta_{g,N})}\big)^{\frac{1}{2}} e^{-6Q^2 S_{\rm L}^0(\H,g,\omega)}  \nonumber
 \end{align*} 

\begin{proposition}[Analytic region of $G(\alpha,\beta)$]\label{analyticregion}
Define the map $(\alpha,\beta)\to G(\alpha,\beta)$, then it is analytic in the following region, recall \eqref{holomorphic} $\mathcal{U}$ is a complex neighbourhood of $(0,Q)$
\begin{align}
    \{(\alpha,\beta)\in \C\times \mathcal{U}\mid |\Re(\alpha)- Q|+ \Re(\frac{\beta}{2})>0\}
\end{align}
\end{proposition}
\begin{proof}
We know $\beta\to\mathcal{A}^{m}_{g_{\mathbb{A}},\mathbb{A}^1_{\sqrt{q}},b,\beta}(\tilde{\varphi})$ is analytic in $\beta\in \mathcal{U}$ in  $e^{\delta c_{-}}L^2(\R\times\Omega_{\T})$ space for $0<\delta<\frac{\Re(\beta)}{2}$. And we also have $\alpha\to\Psi_{\alpha}$ is analytic in $\alpha\in C$ in $e^{-\epsilon c_{-}}L^2(\R\times\Omega_{\T})$ space for $\epsilon>|\Re(\alpha)-Q|$. Pairing these two, we know  $(\alpha,\beta)\to \langle \mathcal{A}^{m}_{g_{\mathbb{A}},\mathbb{A}^1_{\sqrt{q}},b,\beta}\mid\mathbf{C} \Psi_{\alpha}\rangle_{L^2(\R\times\Omega_{\T})}$ is analytic in above region.
\end{proof}
\begin{proposition}\label{bbestimate1}
    For $\mu>0$ and fixed $\beta>0$, the $\lim_{P\to 0}\frac{1}{P}G(Q+iP,\beta)$ exists.
\end{proposition}
\begin{proof}
For $\mu>0$, by \cite[Theorem 4.4]{BGKRV}, we have $\Psi_{Q}=0$. Then we deduce $G(Q,\beta)=0$. To prove the order of $0$ is at least one, we show $G(\alpha,\beta)$ is analytic near $\alpha=Q$. By previous proposition, we see $\alpha\to G(\alpha,\beta)$ is analytic in a region which contains a complex neighbourhood of $\alpha=Q$.
\end{proof}
We remark that the previous proposition is not true in the $\mu=0$ case.

\begin{proposition}
For $\beta\in (0,Q)$, we have $\overline{G(Q+iP,\beta)}=G(Q-iP,\beta)$.
\end{proposition}
\begin{proof}
W only need to prove $\overline{\mathcal{A}_{g_{\A},\A^1_{\sqrt{q}},\beta,b}(\Psi_{Q+iP})}=\mathcal{A}_{g_{\A},\A^1_{\sqrt{q}},\beta,b}(\Psi_{Q-iP})$. we know $\overline{\mathcal{A}_{g_{\A},\A^1_{\sqrt{q}},\beta,b}(\Psi_{Q+iP})}=\mathcal{A}_{g_{\A},\A^1_{\sqrt{q}},\beta,b}(\mathbf{C}\Psi_{Q+iP})$.

By the notation in \eqref{defprop:desc}, we have for $\alpha\in \mathcal{I}_{0,0}$, 
\begin{equation}
\mathbf{C}\Psi_{\alpha}= \lim_{t\to \infty}e^{t(2\Delta_{\bar{\alpha}})} e^{-t{\bf H}}( \chi(c)\mathbf{C}\Psi^0_{\alpha})=\lim_{t\to \infty}e^{t(2\Delta_{\bar{\alpha}})} e^{-t{\bf H}}( \chi(c)\mathbf{C}\Psi^0_{\bar{\alpha}})=\Psi_{\bar{\alpha}}
\end{equation}
Then for $\alpha=Q+iP$, we use analytic continuation.

\end{proof}


\section{The Conformal bootstrap formula on annulus with two boundary insertions}\label{two point proof}

In this section, we prove $\mathcal{A}^{m}_{g_{\mathbb{A}},\mathbb{A}^1_{\sqrt{q}},b,\beta}(\Psi_{Q+iP,\nu,\tilde{\nu}})$ can be factorized to structure constant $G(Q+iP,\beta)$ and some polynomial. Our strategy is the following:

\vskip 1mm

\noindent (1) As we saw $\alpha\to\Psi_{\alpha,\nu,\tilde\nu}$ is analytic in $\C$. If  $ \alpha\in\R$ is negative enough (depending on $\nu,\tilde\nu$), 
  it can be expressed as a limit of a generalized disk amplitude with SET insertion (Lemma \ref{ampdiskSET}).

\vskip 1mm

\noindent (2) The states $\Psi_{\alpha,\nu,\tilde\nu}$ for $\alpha\in\R$ small  are in general not in the domain of the amplitude $\caA_{\H\backslash\Omega, \hat g,\beta,s}$ due to violation of the Seiberg bounds. To remedy this we regularize this  amplitude by inserting several additional vertex operators  $V_{\frac{\beta^{\rm ar}_i}{2},\hat g}(s_i^{\rm ar})$ which locate at $s_i^{\rm ar}\in \R$ for $i=1,\dots,m$ with $\beta^{\rm ar}_i<Q$ and $\sum_i\beta^{\rm ar}_i$ are large enough to meet the Seiberg bound. Then we can use Ward's identity to get proposition \eqref{abr}.

\vskip 1mm

\noindent (3) 
Then we will use analyticity of eigenstate and amplitudes to take the $\alpha=Q+iP$ and all $\beta^{\rm ar}_i\to 0$ for $i=1,..,m$, it results to the LCFT bulk-boundary correlator and a real valued polynomial in $\Delta_{Q+iP}$ and $\Delta_\beta$. 

First we prove two lemmas about conformal mapping,
\begin{lemma}
We define the map \begin{align}
\psi: \hat{\mathbb{C}} & \longrightarrow \hat{\mathbb{C}} \\
z & \mapsto i\frac{1+z}{1-z}
\end{align}
Then $\psi(0)=i$, $\psi(\D)=\H$. We define $\Gamma=\psi(\{|z|=\sqrt{q}\})$, $\Gamma'=\psi(\{|z|=\frac{1}{\sqrt{q}}\})$, $\Omega=\psi(\{|z|\leq\sqrt{q}\})$, $\Omega'=\psi(\{|z|\geq\frac{1}{\sqrt{q}}\})$ then $\Gamma$ and $\Gamma'$ are Euclidean round circles and the complex conjugation image of $\Gamma$ is $\Gamma'$.
\end{lemma} 
\begin{proof}
Since the M\"obius map $\psi$ maps origin-avoid circles to circles. The inverse map of $\psi$ is $\psi^{-1}(u)=\frac{u-i}{u+i}$, if $|\frac{u_0-i}{u_0+i}|=\sqrt{q}$, then $|\frac{\bar{u}_0-i}{\bar{u}_0+i}|=\frac{1}{\sqrt{q}}$.
\end{proof}
\begin{lemma}
We also two maps on $\D$, such that $\psi_1(u)=\psi(\sqrt{q}u)$ and $\psi_2(u)=\psi(\frac{1}{\sqrt{q}u})$, then $S_{\psi_1}=S_{\psi_2}=0$, $\psi_2(u)=\overline{\psi_1(\bar{u})}$ and $\psi_2'(u)^2=\overline{\psi_1'(\bar{u})^2}$.
\end{lemma}
\begin{proof}
By definition, $\psi_2(u)=i\frac{\sqrt{q}u+1}{\sqrt{q}u-1}$, since $\psi_1$ and $\psi_2$ are M\"obius map, we have $S_{\psi_1}=S_{\psi_2}=0$. The next claim can be shown by direct computation.
\end{proof}
We also note $\H\setminus\Omega=\mathcal{S}$, the reflection of $\mathcal{S}$ along the real line is  denoted by $\mathcal{S}'$, see figure \eqref{Fig.main4}.


Now we deal with mixed boundary condition annulus amplitude $\mathcal{A}^{m}_{g_{\mathbb{A}},\mathbb{A}_{\sqrt{q}},\beta, \psi^{-1}(s),\sqrt{q}e^{i\theta}}(\tilde{\varphi})$. We use the notation ${\mathbb{A}}^p_q$ to represent annulus $\{q<|z|<p\}$ and equip it with flat metric $g_{\mathbb{A}}:=\frac{|dz|^2}{|z|^2}$. For $z=\psi^{-1}(s)\in\partial^{out}\mathbb{A}_q$,  we hope to evaluate $\mathcal{A}^{m}_{g_{\mathbb{A}},\mathbb{A}_{\sqrt{q}},\beta, \psi^{-1}(s),\sqrt{q}e^{i\theta}}(\Psi_{Q+iP,\nu,\tilde{\nu}})$. To work within the probabilistic representation, we need to add several extra boundary insertions $(\psi^{-1}({\bf   s}^{\rm ar}), \boldsymbol{\beta}^{\rm ar})$ to meet the Seiberg bound $\alpha+\frac{\beta}{2}+\sum_i\frac{\beta_i^{\rm ar}}{2}>Q$. We note ${\bf   s}=(s, {\bf   s}^{\rm ar})$ and $\boldsymbol{\beta}=(\beta,\boldsymbol{\beta}^{\rm ar})$ and choose a admissible metric $g=e^{\omega}|dz|^2$ on $\D$ and $\omega$ is $0$ near the origin. Note $\hat{g}=\psi_*g_{\A}\#\psi_{1*}g$ and $\zeta_\D=e^{i\theta}$
\begin{align*}
   &\quad\quad\sqrt{\pi}\mathcal{A}^{m}_{g_{\mathbb{A}},\mathbb{A}^1_{\sqrt{q}},\psi^{-1}(s),\beta, \psi^{-1}({\bf   s}^{\rm ar}), \boldsymbol{\beta}^{\rm ar},\sqrt{q}\zeta_\D}(\Psi_{\alpha,\nu,\tilde{\nu}})
   =\sqrt{\pi}\mathcal{A}^{m}_{\psi_*g_{\mathbb{A}},\mathcal{S},s,\beta, {\bf   s}^{\rm ar}, \boldsymbol{\beta}^{\rm ar},\psi_1(\zeta_\D)}(\Psi_{\alpha,\nu,\tilde{\nu}})\\
   &=( Z(g)Z_{\D,g_\D}^{-1})Z(\H,\hat g)\lim_{t\to\infty} 
  \frac{1}{(2\pi i)^{s(\nu)+s(\tilde{\nu})}}  \\
  &\quad\times\oint_{|\mathbf{u}|=\boldsymbol{\delta}_t}   \oint_{|\mathbf{v}|=\tilde{\boldsymbol{\delta}}_t}  
\mathbf{u}^{1-\boldsymbol{\nu}}\bar{\mathbf{v}}^{1-\tilde{\boldsymbol{\nu}}} \langle  \psi_1^\ast T^{\H}_{g_{\rm dozz}}(\mathbf{u})\overline{\psi_1^\ast T^{\H}_{g_{\rm dozz}}}(\mathbf{v})V_{\alpha,g_{\rm dozz}}(z)V_{\frac{\beta}{2},g_{\rm dozz}}(s)\prod_{j=1}^mV_{\frac{\beta^{\rm ar}_j}{2},g_{\rm dozz}}(s^{\rm ar}_j)  
\rangle_{\H,g_{\rm dozz},U_t}\dd  \mathbf{u}\dd  \bar{\mathbf{v} }  \\
  &= ( Z(g)Z_{\D,g_\D}^{-1})Z(\H,\hat g)g_{dozz}(z)^{-\Delta_{\alpha}}g_{dozz}(s)^{-\frac{\Delta_{\beta}}{2}}\prod_{j=1}^m g_{dozz}(s_j^{\rm ar})^{-\frac{\Delta_{\beta_j^{\rm ar}}}{2}}\\
  &\quad\times \Big(\caD^{\H}_{\nu,\tilde{\nu},\alpha,z,\beta,s,\boldsymbol{\beta}^{\rm ar},{\bf s}^{\rm ar}} \langle V_{\alpha,g_{\rm dozz}}(z)V_{\frac{\beta}{2},g_{\rm dozz}}(s) \prod_{j=1}^mV_{\frac{\beta^{\rm ar}_j}{2},g_{\rm dozz}}(s^{\rm ar}_j)   \rangle_{ \H,{g_{\rm dozz}}}\Big)
\end{align*}
here
\begin{equation}
    Z(\H,\hat g):=e^{-c_LS_L^0(\H,g_{dozz},\hat{\omega})-\Delta_{\alpha}\hat{\omega}(z)-\frac{1}{2}\Delta_{\beta}\hat{\omega}(s)-\frac{1}{2}\sum_j\Delta_{\beta^{\rm ar}_j}\hat{\omega}(s^{\rm ar}_j)}
\end{equation}

On the other hand side, we deal with Dirichlet boundary amplitude $\mathcal{A}^{D}_{g_{\mathbb{A}},\mathbb{A}^{\frac{1}{\sqrt{q}}}_{\sqrt{q}},\beta, \psi^{-1}(s),{\boldsymbol{\tilde\zeta}}}(\Psi_{Q+iP,\nu,\emptyset}, \Psi_{Q+iP,\tilde{\nu},\emptyset})$, here ${\boldsymbol{\tilde\zeta}}=(\sqrt{q}e^{i\theta},\frac{1}{\sqrt{q}}e^{-i\theta})$ means both boundaries of the annulus are incoming. When $\psi^{-1}(s)$ on the boundary, we should understand the notation "weight $\beta$" as an boundary insertion $V_{\frac{\beta}{2}}(\psi^{-1}(s))$, and  when $\psi^{-1}(s)$ in the bulk, we should understand the notation "weight $\beta$" as an bulk insertion $V_{\beta}(\psi^{-1}(s))$.
\begin{align*}
    &\quad\quad\frac{1}{\pi}\mathcal{A}^{D}_{g_{\mathbb{A}},\mathbb{A}^{\frac{1}{\sqrt{q}}}_{\sqrt{q}},\psi^{-1}(s),\beta, \psi^{-1}({\bf   s}^{\rm ar}), \boldsymbol{\beta}^{\rm ar},{\boldsymbol{\tilde\zeta}}}(\Psi_{\alpha,\nu,\emptyset}, \Psi_{\alpha,\tilde{\nu},\emptyset})=\frac{1}{\pi}\mathcal{A}^{D}_{\psi_*g_{\mathbb{A}},D\mathcal{S}, s,\beta, {\bf   s}^{\rm ar}, \boldsymbol{\beta}^{\rm ar},\psi_1(\zeta_\D),\psi_2(\zeta_\D)}(\Psi_{\alpha,\nu,\emptyset}, \Psi_{\alpha,\tilde{\nu},\emptyset})\\
     &=( Z(g)Z_{\D,g_\D}^{-1})^2 Z(\hat{\C},\hat g)\lim_{t\to\infty} 
  \frac{1}{(2\pi i)^{s(\nu)+s(\tilde{\nu})}}  \\
  &\quad\times\oint_{|\mathbf{u}|=\boldsymbol{\delta}_t}\oint_{|\mathbf{v}|=\tilde{\boldsymbol{\delta}}_t}  \mathbf{u}^{1-\boldsymbol{\nu}}\mathbf{v}^{1-\tilde{\boldsymbol{\nu}}} \langle  \psi_1^\ast T^{\hat{\C}}_{g_{\rm dozz}}(\mathbf{u}){\psi_2^\ast  T^{\hat{\C}}_{g_{\rm dozz}}}(\mathbf{v})V_{\alpha,g_{\rm dozz}}(z)V_{\alpha,g_{\rm dozz}}(\bar{z})V_{\beta,g_{\rm dozz}}(s)\prod_{j=1}^mV_{\beta^{\rm ar}_j,g_{\rm dozz}}(s^{\rm ar}_j)  
\rangle_{\hat \C,g_{\rm dozz},U_t}\dd  \mathbf{u}\dd  \mathbf{v}   \\
  &=( Z(g)Z_{\D,g_\D}^{-1})^2Z(\hat{\C},\hat g)g_{dozz}(z)^{-2\Delta_{\alpha}}g_{dozz}(s)^{-\Delta_{\beta}}\prod_{j=1}^m g_{dozz}(s_j^{\rm ar})^{-\Delta_{\beta_j^{\rm ar}}} \\
  &\quad\times \Big(\caD^{\hat{\C}}_{\nu,\tilde{\nu},\alpha,z,\alpha,\bar{z},\beta,s,\boldsymbol{\beta}^{\rm ar},{\bf s}^{\rm ar}} \langle V_{\alpha,g_{\rm dozz}}(z)V_{\alpha,g_{\rm dozz}}(\bar{z})V_{\beta,g_{\rm dozz}}(s)\prod_{j=1}^mV_{\beta^{\rm ar}_j,g_{\rm dozz}}(s^{\rm ar}_j)   \rangle_{ \hat \C,{g_{\rm dozz}}}\Big)
\end{align*}
here $\hat{g}=\psi_*g_{\A}\#\psi_{1*}g\#\psi_{2*}g$ and
\begin{equation}
    Z(\hat{\C},\hat g):=e^{-c_LS_L^0(\hat{\C},g_{dozz},\hat{\omega})-2\Delta_{\alpha}\hat{\omega}(z)-\Delta_{\beta}\hat{\omega}(s)-\sum_j\Delta_{\beta^{\rm ar}_j}\hat{\omega}(s^{\rm ar}_j)}
\end{equation}
we will take $z=i$ in the end of the proof. 
\begin{proposition}
The above identities holds for $\beta^{\rm ar}=0$, $\alpha=Q+iP$ and $\beta\in (0,Q)$.
\end{proposition} 
\begin{proof}
First we can decrease the value of $\bf a$ to replace the condition $\max\{\alpha,\beta,\frac{\beta}{2}\}<\bf a$ by  $\max\{\alpha,\beta\}<\bf a$. Note $\mathcal{E}=\{(z,s,s^{\rm ar})\in\C\times\R^{m+1})|\text{ all points are different}\}$ and test function $f(z,s,s^{\rm ar})\in \mathcal{C}^{\infty}_0(\mathcal{E})$
\begin{align}\label{final2}
 & \int_{\mathcal{E} }\sqrt{\pi} \caA_{\H\backslash\Omega,\hat g,z,\alpha,s,\beta,{\bf s}^{\rm ar},\boldsymbol{\beta}^{\rm ar}}\big(\Psi_{\alpha,\nu,\tilde{\nu}}\big) f(z,s,s^{\rm ar}) \\
 & = ( Z(g)Z_{\D,g_\D}^{-1})\int_{\mathcal{E}}\langle V_{\alpha,g}(z)V_{\frac{\beta}{2},g}( s)\prod_{i=1}^m V_{\frac{\beta^{\rm ar}_i}{2},g}( s^{\rm ar}_i) \rangle_{\H,g_{dozz}} \caD^{\H}_{\nu,\tilde{\nu},\alpha,z,\beta,s,\boldsymbol{\beta}^{\rm ar},{\bf s}^{\rm ar}} \Big( Z(\H,\hat g)
f(z,s,s^{\rm ar}) \Big)\nonumber
 \end{align}
 Since for $(\alpha,\beta)$ in the Seiberg bound, we have
  \begin{align}
 & \sqrt{\pi} \caA_{\H\backslash\Omega,\hat g,z,\alpha,s,\beta,{\bf s}^{\rm ar},\boldsymbol{\beta}^{\rm ar}}\big(\Psi_{\alpha}\big)\\
 =&  ( Z(g)Z_{\D,g_\D}^{-1})Z(\H,\hat g)
  \langle V_{\alpha,g}(z)V_{\frac{\beta}{2},g}( s)\prod_{i=1}^m V_{\frac{\beta_i^{\rm ar}}{2},g}( s^{\rm ar}_i) \rangle_{\H,g_{dozz}}\nonumber
 \end{align}
 
Then we use both $\alpha\to\Psi_{\alpha,\nu,\tilde{\nu}}$ and $\alpha\to\Psi_{\alpha,\emptyset,\emptyset}$ are holomorphic in $\alpha$ in a connected region ($\C$) containing
 $\alpha<\bf a$ and the spectrum line $Q+i\R$. By \eqref{holomorphic}, the mappings $(\alpha,\beta,\beta^{\rm ar})\to \textbf{RHS}/\textbf{LHS}$ are holomorphic in
 \begin{align}
     \mathcal{G}:=\{(\alpha,\beta,\beta^{\rm ar})\in \C^{m+2}|\alpha\in \C,\beta\in \mathcal{U},|\Re(\alpha)-Q|+\Re(\frac{\beta}{2})+\sum_{i=1}^m \Re(\frac{\beta_i^{\rm ar}}{2})>0\}
 \end{align}
 we can choose $m$ large enough so that this set contains a nonempty connected components which has a nonempty intersection with $(-\infty,a)^2\times\R^m$ and contains a neighbourhood of
 \begin{align}
    \mathcal{S}_b:=\{(\alpha,\beta,0,..,0)\in \C^{m+2}|\alpha=Q+iP,0<\beta<Q \} 
 \end{align}
So we finish the proof.
\end{proof}
Recall \eqref{DOZZ3point} and \eqref{DOZZ2point}, we have (here we abuse the notation in the LHS, the rigorous notation should be given by amplitude gluing)
\begin{align}
    &g_{dozz}(w_1)^{\Delta_{Q+iP}}g_{dozz}(w_2)^{\Delta_{\beta}}g_{dozz}(w_3)^{\Delta_{Q+iP}}\left\langle V_{Q+iP}(w_1)V_{\beta}(w_3) V_{Q-iP}(w_2)\right\rangle_{\hat{\C},g_{dozz}}\nonumber\\=&|w_1-w_2|^{2\Delta_{\beta}-4 \Delta_{Q+iP}}|w_1-w_3|^{-2 \Delta_{\beta}}|w_2-w_3|^{-2 \Delta_{\beta}}\Big(\frac{{\rm v}_{g_{dozz}}(\hat\C)}{{\det}'(\Delta^{\hat{C}}_{g_{dozz}})}\big)^{\frac{1}{2}}\frac{1}{2}C^{\rm DOZZ}_{\gamma,\mu}(Q+iP,\beta,Q-iP)\\
    &g_{dozz}(z)^{\Delta_{Q+iP}}g_{dozz}(s)^{\frac{\Delta_{\beta}}{2}}\left\langle V_{Q+iP}(z)V_{\frac{\beta}{2}}(s) \right\rangle_{\H,g_{dozz}}\nonumber\\=&(-i)^{\Delta_{\beta}-2 \Delta_{Q+iP}}(z-\bar{z})^{\Delta_{\beta}-2 \Delta_{Q+iP}}(z-s)^{- \Delta_{\beta}}(\bar z-s)^{- \Delta_{\beta}}\Big(\frac{{\rm v}_{g_{dozz}}(\H)}{{\det}'(\Delta^{\H}_{g_{dozz},N})}\big)^{\frac{1}{2}}G(Q+iP,\beta)
    \end{align}
Now we use \eqref{cardydoublingproperty}, the holomorphic part in the sphere case is  $(w_1-w_2)^{\Delta_{\beta}-2 \Delta_{Q+iP}}(w_1-w_3)^{- \Delta_{\beta}}(w_2-w_3)^{- \Delta_{\beta}}$, when we set $w_1=w_2=z=i$ and $w_3=s$, these two differential operators produce the same polynomial.
By symmetry, we have $Z(\hat{\C},\hat g)=(Z(\H,\hat g))^2$ and ${\rm v}_{g_{dozz}}(\hat{\C})=2{\rm v}_{g_{dozz}}(\H)=2\pi$.
From the diffeomorphism invariance of Dirichlet boundary amplitude \cite[Proposition 4.7]{GKRV21}, we know $$\mathcal{A}^{D}_{ g_{\mathbb{A}},\mathbb{A}^{\frac{1}{\sqrt{q}}}_{\sqrt{q}},\psi^{-1}(s),\beta, \psi^{-1}({\bf   s}^{\rm ar}), \boldsymbol{\beta}^{\rm ar},(\sqrt{q}e^{i\theta},\frac{1}{\sqrt{q}}e^{-i\theta})}(\tilde{\varphi})=\mathcal{A}^{D}_{g_{\mathbb{A}},\mathbb{A}^1_{q},\sqrt{q}\psi^{-1}(s),\beta, \sqrt{q}\psi^{-1}({\bf   s}^{\rm ar}), \boldsymbol{\beta}^{\rm ar},(qe^{i\theta},e^{-i\theta})}(\tilde{\varphi}),$$ then by \cite[Corollary 11.10]{GKRV21}, for plumbed annulus there exists a polynomial $w_\A$
\begin{align*}
    \mathcal{A}^{D}_{g_{\mathbb{A}},\mathbb{A}^{\frac{1}{\sqrt{q}}}_{\sqrt{q}},\psi^{-1}(s),\beta,{\boldsymbol{\tilde\zeta}}}(\Psi_{Q+iP,\nu,\emptyset}, \Psi_{Q+iP,\tilde{\nu},\emptyset})=&C_{\A}C_{\gamma,\mu}^{{\rm DOZZ}} (Q+iP,\beta,Q+iP) w_{\mathbb{A}}(\Delta_{\beta},\Delta_{Q+iP},\Delta_{Q+iP},
 \nu,\tilde{\nu} )\\
   \times&|q |^{-{c}_{\rm L}/12+2\Delta_{Q+ip}}  (\frac{\sqrt{q}}{\psi^{-1}(s)})^{|\nu|} (\sqrt{q}\psi^{-1}(s))^{|\tilde{\nu}|}\nonumber 
\end{align*}
for some universal constant $C_\A$ and $w_\A(\Delta_\beta,\emptyset,\emptyset)=1$,
and then we conclude there exists a constant $C_{\D}$ such that
$$\mathcal{A}^{m}_{g_{\mathbb{A}},\mathbb{A}^1_{\sqrt{q}},b,\beta,\sqrt{q}e^{i\theta}}(\Psi_{Q+iP,\nu,\tilde{\nu}})=C_\D G(Q+iP,\beta)w_{\mathbb{A}}(\Delta_{\beta},\Delta_{Q+iP},
 \nu,\tilde{\nu} )|q |^{-{c}_{\rm L}/24+\Delta_{Q+ip}}   (\frac{\sqrt{q}}{b})^{|\nu|} (\sqrt{q}b)^{|\tilde{\nu}|}$$
 more precisely, we have
 \begin{equation}\label{definition CD}
     C_\D=\sqrt{C_\A}\times\sqrt{\frac{{\det}'(\Delta^{\hat\C}_{g_{dozz}})^{\frac{1}{2}}}{{\det}'(\Delta^{\H}_{g_{dozz},N})}}\times\frac{2^{\frac{1}{4}}}{\pi^{\frac{3}{4}}}
 \end{equation}

\begin{proposition}{\bf  (Disk descendants formula)} \\
For $b\in\partial^{out}\mathbb{A}_q$ and above constant $C_{\D}$, we have the following disk descendants formula,
\begin{align}\label{diskdescendant}
    \quad\quad\mathcal{A}^{m}_{g_{\mathbb{A}},\mathbb{A}^1_{\sqrt{q}},b,\beta,\sqrt{q}e^{i\theta}}(\Psi_{Q+iP,\nu,\tilde{\nu}})=C_{\D}G(Q+iP,\beta)w_{\mathbb{A}}(\Delta_{\beta},\Delta_{Q+iP},
 \nu,\tilde{\nu} )|q |^{-{c}_{\rm L}/24+\Delta_{Q+ip}}  (\frac{\sqrt{q}}{b})^{|\nu|} (\sqrt{q}b)^{|\tilde{\nu}|}
\end{align}
 For convenient usage, we use the orthogonal eigenfunctions 
    \begin{align*}
        H_{Q+i P, \nu, \tilde{\nu}}=\sum_{|\nu^{\prime}|=|\nu|,|\tilde{\nu}'|=|\tilde{\nu}|} F_{Q+i P}^{-1 / 2}\left(\nu, \nu^{\prime}\right) F_{Q+i P}^{-1 / 2}\left(\tilde{\nu}, \tilde{\nu}'\right)\Psi_{Q+i p, \nu^{\prime}, \tilde{\nu}'}  
    \end{align*}
and define the coefficient $\mathcal{W}_{\mathbb{A}}(\Delta_{\beta},\Delta_{Q+iP},
 \nu,\tilde{\nu} )$ by absorbing the Schapovalov forms within $w_{\mathbb{A}}$, then
\begin{align*}
    \quad\quad\mathcal{A}^{m}_{g_{\mathbb{A}},\mathbb{A}^1_{\sqrt{q}},b,\beta}(H_{Q+iP,\nu,\tilde{\nu}})=C_{\D}G(Q+iP,\beta)\mathcal{W}_{\mathbb{A}}(\Delta_{\beta},\Delta_{Q+iP},
 \nu,\tilde{\nu} )|q |^{-{c}_{\rm L}/24+\Delta_{Q+ip}}  (\frac{\sqrt{q}}{b})^{|\nu|} (\sqrt{q}b)^{|\tilde{\nu}|}
\end{align*}
\end{proposition}
Now we show symmetric property of $w_\A$
\begin{proposition}\label{symmetric property}
For $\beta\in (0,Q)$, we have
$$w_{\mathbb{A}}(\Delta_{\beta},\Delta_{Q+iP},
 \nu,\tilde{\nu} )=\overline{w_{\mathbb{A}}(\Delta_{\beta},\Delta_{Q+iP},
 \tilde{\nu},\nu )}=w_{\mathbb{A}}(\Delta_{\beta},\Delta_{Q+iP},
 \tilde{\nu},\nu )$$
\end{proposition}
\begin{proof}
The first equality is easy, we take $b=1$ and take conjugation in both sides of \eqref{diskdescendant}. Then we use $C\Psi_{Q+iP,\nu,\tilde{\nu}}=\Psi_{Q-iP,\tilde{\nu},\nu}$. The second inequality follows from $w_\A$ is real-valued when $\beta\in (0,Q)$, see proposition \eqref{realvalue}.
\end{proof}
\subsection{Proof of theorem \texorpdfstring{\eqref{two point bootstrap}}{2point}}
Now we define the {\bf block amplitudes} and rewrite above proposition, here we only define two kinds of block amplitudes, for the general case, see \cite[Lemma 8.2]{GKRV21}
\begin{proposition}\label{blockamplitude}
    For $|b|=1$ and ${\boldsymbol{\zeta}}=(\sqrt{q}e^{i\theta},\frac{1}{\sqrt{q}}e^{i\theta})$, we define $\mathcal{B}^D_{g_{\mathbb{A}},\mathbb{A}^{\frac{1}{\sqrt{q}}}_{\sqrt{q}},b,\beta}(P_1,P_2,\nu,\tilde{\nu})$ by 
\begin{align}
    \frac{{A}^{D}_{g_{\mathbb{A}},\mathbb{A}^{\frac{1}{\sqrt{q}}}_{\sqrt{q}},b,\beta,{\boldsymbol{\zeta}}}(H_{Q+iP_1,\nu,\nu'},\mathbf{C} H_{Q+iP_2,\tilde{\nu},\tilde{\nu}'})}{C_{\A}\times C^{\rm DOZZ}(Q+iP_1,\beta,Q-iP_2)}=\mathcal{B}^D_{g_{\mathbb{A}},\mathbb{A}^{\frac{1}{\sqrt{q}}}_{\sqrt{q}},b,\beta}(P_1,P_2,\nu,\tilde{\nu})\overline{\mathcal{B}^D_{g_{\mathbb{A}},\mathbb{A}^{\frac{1}{\sqrt{q}}}_{\sqrt{q}},b,\beta}(P_1,P_2,\nu',\tilde{\nu}')}
\end{align}
and fix the rotation by

$\mathcal{B}^D_{g_{\mathbb{A}},\mathbb{A}^{\frac{1}{\sqrt{q}}}_{\sqrt{q}},b,\beta}(P_1,P_2,\nu,\tilde{\nu}):=|q|^{-c_{\mathrm{L}} / 24+ \Delta_{Q+i p_{1}}}|\sqrt{q}|^{ \Delta_{Q+i p_{2}}- \Delta_{Q+i p_{1}}}(\sqrt{q} / b)^{\left|\nu\right|}(\sqrt{q}b)^{\left|\tilde{\nu}\right|} \mathcal{W}_\A(\Delta_\beta,\Delta_{Q+iP_1},\Delta_{Q+iP_2},\nu,\tilde{\nu})$.

\noindent We also define $\mathcal{B}^{m}_{g_{\mathbb{A}},\mathbb{A}^1_{\sqrt{q}},b,\beta}(P,\nu,\tilde{\nu})$ by
\begin{align}
    \frac{\mathcal{A}^{m}_{g_{\mathbb{A}},\mathbb{A}_{\sqrt{q}},b,\beta}(H_{Q+iP,\nu,\tilde{\nu}})}{C_{\D}\times G(Q+iP,\beta)}=\mathcal{B}^{m}_{g_{\mathbb{A}},\mathbb{A}^1_{\sqrt{q}},b,\beta}(P,\nu,\tilde{\nu})
\end{align}
Then we have the following identity
\begin{align}
    \mathcal{B}^{m}_{g_{\mathbb{A}},\mathbb{A}^1_{\sqrt{q}},b,\beta}(P,\nu,\tilde{\nu})=\mathcal{B}^D_{g_{\mathbb{A}},\mathbb{A}^{\frac{1}{\sqrt{q}}}_{\sqrt{q}},b,\beta}(P,P,\nu,\tilde{\nu})
\end{align}
\end{proposition}


We cut along the middle circle $\{|z|=\sqrt{q}\}$ to separate this annulus into two annuli ${\mathbb{A}}^{\sqrt{q}}_q$ and ${\mathbb{A}}^1_{\sqrt{q}}$. When we insert $b_1$ with weight $\beta_1$ on $\partial^{out}\mathbb{A}_q=\{|z|=1\}$ and  $\frac{q}{b_2}$ with weight $\beta_2$ on $\partial^{in}\mathbb{A}_q=\{|z|=q\}$, both two boundaries are geodesic and we put the Neumann boundary condition on them. Then we have
\[\mathcal{A}^N_{g_{\mathbb{A}},\mathbb{A}^1_{q},b_1,\frac{q}{b_2},\beta_1,\beta_2}=\left\langle V_{\frac{\beta_1}{2}}(b_1)V_{\frac{\beta_2}{2}}(\frac{q}{b_2})\right\rangle^{\mathbb{A}^1_{q}}_{\gamma,\mu, \mu_\partial}\]
By the gluing formula for the amplitudes with Neumann boundary \eqref{gluingtypeI}
$$\mathcal{A}^N_{g_{\mathbb{A}},\mathbb{A}^1_{q},b_1,\frac{q}{b_2},\beta_1,\beta_2}=\sqrt{\pi}\int  \mathcal{A}^m_{g_{\mathbb{A}},{\mathbb{A}}^1_{\sqrt{q}},b_1, \beta_1,\sqrt{q}e^{i\theta}}(\tilde{\bm{\varphi}})\times \mathcal{A}^m_{g_{\mathbb{A}},{\mathbb{A}}^{\sqrt{q}}_q,\frac{q}{b_2}, \beta_2,\sqrt{q}e^{i\theta} }(\tilde{\bm{\varphi}})d\mu_0(\tilde{\bm{\varphi}})$$
Now we apply the spectrum resolution \eqref{holomorphicpsi2}, the right side equals
$$\frac{1}{2\sqrt{\pi}}\sum_{\nu, \tilde\nu}\int_{\R_+} \mathcal{A}^m_{g_{\mathbb{A}},{\mathbb{A}}^1_{\sqrt{q}},b_2, \beta_2,\sqrt{q}e^{i\theta}}(\mathbf{O}\mathbf{C}H_{Q+ip,\nu,\tilde\nu})\times \mathcal{A}^m_{g_{\mathbb{A}},{\mathbb{A}}^1_{\sqrt{q}},b_1, \beta_1,\sqrt{q}e^{i\theta} }(H_{Q+iP,\nu,\tilde\nu}) \dd{P} $$\\
where $\mathbf{O}$ is the orientation reserving operator and $\mathbf{C}$ is the conjugation operator, $\mathbf{O} \mathbf{C}H_{Q+iP,\nu,\tilde\nu}=H_{Q-iP,\nu,\tilde\nu}$. 
By proposition \ref{diskdescendant}, we get the following formula
\begin{align}
    \frac{C_{\D}^2}{2\sqrt{\pi}} q^{-\frac{1}{12}}\int_{\R^{+}}G(Q+iP,\beta_1)G(Q-iP,\beta_2)q^{\frac{P^2}{2}}\mathcal{F}^{\A}(\Delta_{\beta_1},\Delta_{\beta_2},\Delta_{Q+iP},q,b_1,b_2)dP
\end{align}
where 
\begin{align}\label{2pointblock}
    \mathcal{F}^{\A}(\Delta_{\beta_1},\Delta_{\beta_2},\Delta_{Q+iP},q,b_1,b_2)=\sum^{\infty}_{n,m=0}\Big(\sum_{|\nu|=n,|\tilde{\nu}|=m}\mathcal{W}_{\mathbb{A}}(\Delta_{\beta_1},\Delta_{Q+iP},
 \nu,\tilde{\nu} )\mathcal{W}_{\mathbb{A}}(\Delta_{\beta_2},\Delta_{Q+iP},
\nu, \tilde{\nu} )\Big)(\frac{q}{b_1b_2})^n(qb_1b_2)^m
\end{align}
To finish the proof, we need the following statement about convergence of conformal block. Actually, we prove it is absolutely convergent
\begin{proposition}\label{as convergent}
For $\beta_1\in(0,Q)$ and $\beta_2\in(0,Q)$, we have
    for almost all $P\in\R^{+}$, the conformal block $\mathcal{F}^{\A}(\Delta_{\beta_1},\Delta_{\beta_2},\Delta_{Q+iP},q,b_1,b_2)$ is convergent in $q\in (0,1)$.
\end{proposition}
\begin{proof}
By proposition \ref{estimateamplitude}, we set $u_1=\mathcal{A}^{m}_{g_{\mathbb{A}},\mathbb{A}^1_{\sqrt{q}},b_1,\beta_1}\in e^{\eps c_-}L^2(\R\times\Omega_{T}) $ and $u_2=\mathcal{A}^{m}_{g_{\mathbb{A}},\mathbb{A}^1_{\sqrt{q}},b_2,\beta_2}\in e^{\eps c_-}L^2(\R\times\Omega_{T}) $ for  $\eps\in(0,\min\{\beta_1,\beta_2\})$. Then we pair it with $H_{Q+iP,\nu,\tilde{\nu}}$, by \cite[(6.28)]{GKRV21}, we know
\begin{align}
    \langle u_1\mid u_2\rangle_{L^2}=\frac{1}{2\pi}\int_{\R_+} \sum_{\nu, \tilde\nu} \langle u_1\mid H_{Q+iP,\nu,\tilde{\nu}}\rangle_{L^2} \langle H_{Q+iP,\nu,\tilde{\nu}}\mid u_2\rangle_{L^2} \dd P
\end{align}
Now we claim $ \sum_{\nu, \tilde\nu} |\langle u_1\mid H_{Q+iP,\nu,\tilde{\nu}}\rangle_{L^2} \langle H_{Q+iP,\nu,\tilde{\nu}}\mid u_2\rangle_{L^2}|\in L^1(\R_+,dP)$ by Cauchy-Schwartz inequality,
\begin{align}
    |\langle u_1\mid {H_{Q+iP,\nu,\tilde{\nu}}}\rangle_{L^2} \langle H_{Q+iP,\nu,\tilde{\nu}}\mid u_2\rangle_{L^2}|\leq \frac{1}{2}\Big(|\langle u_1\mid {H_{Q+iP,\nu,\tilde{\nu}}}\rangle_{L^2}|^2+| \langle H_{Q+iP,\nu,\tilde{\nu}}\mid u_2\rangle_{L^2}|^2\Big)
\end{align}
when take the sum over Young diagrams $\nu,\tilde{\nu}$ and take the integral over $P$, the RHS would be dominated by $||u_1||_{L^2}+||u_2||_{L^2}\leq C$.  we have for almost all $P\in \R_+$, $ \sum_{\nu, \tilde\nu} \langle u_1\mid H_{Q+iP,\nu,\tilde{\nu}}\rangle_{L^2} \langle H_{Q+iP,\nu,\tilde{\nu}}\mid u_2\rangle_{L^2}$ is absolutely convergent.

Now we already proved $G(Q+iP,\beta_1)G(Q-iP,\beta_2)q^{\frac{P^2}{2}}\mathcal{F}^{\A}(\Delta_{\beta_1},\Delta_{\beta_2},\Delta_{Q+iP},q,b_1,b_2)$ is almost surely absolutely convergent. To derive the absolute convergence of $\mathcal{F}^{\A}(\Delta_{\beta_1},\Delta_{\beta_2},\Delta_{Q+iP},q,b_1,b_2)$, we use the zero set of $P\to G(Q+iP,\beta)$ has measure $0$ for fixed $\beta$. This is a general fact for real analytic  function and we know $G(Q+iP,\beta)$ is a real analytic function of $P$, see \cite[Theorem 6.3.3]{KP}.
We will prove this conformal block is actually convergent for all $P\in\R_+$ by proving it's continuous in $P$, see proposition \ref{continuous2point}.
\end{proof}
Now we can identify the annulus 2-point block with the special case of torus 2-point block in the regime $\beta_1, \beta_2\in (0,Q)$, i.e. \eqref{annulus=torus}, by using formula \eqref{symmetric property}.
\section{Continuity of two-point conformal block  }
\subsection{Traces, integral kernels and Hilbert-Schmidt norms} 
We recall in this section basic facts on Hilbert-Schmidt operators and traces, we refer to \cite[Chapter 4]{Gohberg-Goldberg-Krupnik}.
Let $(M_1,\mu_1)$ and $(M_2,\mu_2)$ be two separable measured spaces and let $H_i:=L^2(M_i,\mu_i)$.
Let $ \mathbf{K}: H_1\to H_2$ be a bounded operator. We say that $ \mathbf{K}$ is Hilbert-Schmidt if $\mathbf{K}^*\mathbf{K}:H_1\to H_1$ is compact and the eigenvalues $(\lambda_j^2)_{j\in\nn}$ of $\mathbf{K}^*\mathbf{K}$ (with $\la_j\geq 0$) are in $\ell^1(\nn)$. Denote by $|{\bf K}|:=\sqrt{\mathbf{K}^*\mathbf{K}}$.
 If $H_1=H_2$, we say that ${\bf K}:H_1\to H_1$ is trace class if $\sum_j\la_j<\infty$, and its trace is then defined by the converging series 
\[{\rm Tr}({\bf K})=\sum_{j}\cjg {\bf K}e_j,e_j\cjd_{L^2}\]
where $(e_j)_j$ is an orthonormal basis of $H_1$.
The space $\mc{L}_2(H_1,H_2)$ of Hilbert-Schmidt operators (which will be denoted $\mc{L}_2(H_1)$ when $M_1=M_2$) is a Hilbert space and the space $\mc{L}_1(H_1)$ of trace class operators is a Banach space, when equipped with the norms
\[\|\mathbf{K}\|_{{\rm HS}}^2:=\sum_{j} \lambda_j^2={\rm Tr}({\bf K}^*{\bf K})=\sum_{j}\|\mathbf{K}e_j\|^2_{L^2}, \quad \|\mathbf{K}\|_{{\rm Tr}}:=\sum_{j=1}^\infty \la_j={\rm Tr}(|{\bf K}|).\] 
Notice that, if $(e'_i)_i$ is an orthonormal basis of $H_2$, then  $\|\mathbf{K}\|_{{\rm HS}}^2=\sum_{i,j}|\cjg {\bf K}e_j,e_i'\cjd|^2$, thus ${\bf K}^*\in \mc{L}_2(H_2,H_1)$ 
and $\|\mathbf{K}\|_{{\rm HS}}=\|\mathbf{K}^*\|_{{\rm HS}}$. We also have $\mc{L}_1(H_1)\subset \mc{L}_2(H_1)\subset \mc{L}(H_1)$ (with  $\mc{L}(H_1)$ the space of continuous operators from $H_1$ into itself) and 
\begin{equation}\label{orderschatten}
\|\,  |{\bf K}|\, \|_{\mc{L}(H_1)}\leq \|{\bf K}\|_{\rm HS}\leq \|{\bf K}\|_{\rm Tr}.
\end{equation}
The trace of $K$ is controlled by its trace class norm: 
\begin{equation}\label{boundtrace1}
|{\rm Tr}({\bf K})| \leq \|{\bf K}\|_{{\rm Tr}}.
\end{equation}
Here we also gether some basic notation about Young diagrams which will be used in the next subsection. The set $\mc{T}$ of Young diagrams  is   countable set  and can be partitioned as 
\begin{equation}\label{decofT} 
\mc{T}= \bigcup_{n=0}^\infty \mc{T}_n,\quad \mc{T}_n:=\{\nu \in \mc{T}\,| \, |\nu|=n\} \textrm{ if }n>0, \quad \mc{T}_0=\{0\}.
\end{equation}
Set $d_{n}:=|\mc{T}_n|$ its cardinality, we let $\cjg\cdot,\cdot\cjd_{d_n}$ be the canonical Hermitian product on $\C^{d_n}$ and we define 
\begin{equation}\label{defofHT} 
H_\mc{T}:=\bigoplus_{n=0}^\infty \C^{d_n} \textrm{ with Hermitian product }\cjg v,v'\cjd_{\mc{T}}:= \sum_{n=0}^\infty \cjg v_n,v'_n\cjd_{d_n}
\end{equation}
where $v=\sum_n v_n$ with $v_n\in \C^{d_n}$ and $v'=\sum_n v'_n$. 
The Hilbert space 
$(H_\mc{T},\cjg \cdot,\cdot\cjd_{\mc{T}})$ is isomorphic to space $L^2(\mc{T},\mu_{\mc{T}})$ where the measure $\mu_{\mc{T}}$ is the counting measure on the set $\mc{T}$, namely
\[  \int_{\mc{T}} f(\nu)d\mu_{\mc{T}}= \sum_{\nu\in \mc{T}}f(\nu), \quad    \forall f:\mc{T}\to \R^+.\]
\subsection{Proof of the theorem \texorpdfstring{\eqref{relation of one two} }{continuity}}
Before proof, we introduce some necessary terminology.
\noindent We define $ \mathcal{B}^D_{g_{\mathbb{A}},\mathbb{A}^{\frac{1}{\sqrt{q}}}_{\sqrt{q}},b,\beta}(P,P,\cdot,\cdot)$ or similarly $ \mathcal{B}^m_{g_{\mathbb{A}},\mathbb{A}^{1}_{\sqrt{q}},b,\beta}(P,\cdot,\cdot)$ as an operator on $L^2(\caT,\mu_{\caT})$ given by (first we define this for almost $P\in\R_+$ and later it becomes clear that this is true for all $\R_+$.) $$ (\mathcal{B}^D_{g_{\mathbb{A}},\mathbb{A}^{\frac{1}{\sqrt{q}}}_{\sqrt{q}},b,\beta}(P,P,\cdot,\cdot)f)(\nu)=\sum_{\nu'}\mathcal{B}^D_{g_{\mathbb{A}},\mathbb{A}^{\frac{1}{\sqrt{q}}}_{\sqrt{q}},b,\beta}(P,P,\nu,\nu')f(\nu').$$ An orthogonal basis of $L^2(\caT,\mu_{\caT})$ is $\{{\rm 1}_{\nu}(\cdot),\nu\in\caT\}.$
First we show 
\begin{proposition}
For $q\in (0,1)$ and $\beta\in [0,Q)$, 
$\mathcal{B}^D_{g_{\mathbb{A}},\mathbb{A}^{\frac{1}{\sqrt{q}}}_{\sqrt{q}},b,\beta}(P,P,\cdot,\cdot)$ is a positive, self-adjoint operator on $L^2(\caT,\mu_{\caT})$.
\end{proposition} 
\begin{proof}
First we decompose the amplitude $\mathcal{A}^D_{g_{\mathbb{A}},\mathbb{A}^{\frac{1}{\sqrt{q}}}_{\sqrt{q}},b,\beta,\boldsymbol{\zeta}}(\tilde{\varphi}_1,\tilde{\varphi}_2)$ along the circle $\{|z|=1\}$,
Then by perturbation argument, we can extend the gluing formula \eqref{gluingtypeI} to $b\in\{|z|=1\}$ case (there is an insertion on the gluing circle) 
\begin{align}
    \mathcal{A}^D_{g_{\mathbb{A}},\mathbb{A}^{\frac{1}{\sqrt{q}}}_{\sqrt{q}},b,\beta,\boldsymbol{\zeta}}(\tilde{\varphi}_1,\tilde{\varphi}_2)_{\epsilon}=\frac{\epsilon^{\frac{\beta^2}{2}}}{\sqrt{2}\pi}\int \mathcal{A}^D_{g_{\mathbb{A}},\mathbb{A}^{\frac{1}{\sqrt{q}}}_{1}}(\tilde{\varphi}_1,\tilde{\varphi})e^{\beta\tilde{\varphi}_\epsilon(b)}\mathcal{A}^D_{g_{\mathbb{A}},\mathbb{A}^{1}_{\sqrt{q}}}(\tilde{\varphi},\tilde{\varphi}_2)d\mu_0(\tilde{\varphi})
\end{align}
Here $\tilde{\varphi}_\epsilon$ is the first $[\frac{1}{\epsilon}]$ terms of its Fourier expansion. Then it's easy to see $\mathcal{A}^D_{g_{\mathbb{A}},\mathbb{A}^{\frac{1}{\sqrt{q}}}_{\sqrt{q}},b,\beta}(\tilde{\varphi}_1,\tilde{\varphi}_2)$ converges to $\mathcal{A}^D_{g_{\mathbb{A}},\mathbb{A}^{\frac{1}{\sqrt{q}}}_{\sqrt{q}},b,\beta}(\tilde{\varphi}_1,\tilde{\varphi}_2)$ as a linear functional on $\big(e^{-\delta c_-}L^2(\R\times\Omega)\Big)^{{\otimes}^2}$ for $\delta\in (0,\beta)$.
by \cite[Proposition 6.1]{GKRV21}, we have for any $F\in L^2(\R\times\Omega_{\T})$
\begin{align}\label{symmetry}
    \mathcal{A}^D_{g_{\mathbb{A}},\mathbb{A}^{\frac{1}{\sqrt{q}}}_{\sqrt{q}},b,\beta}(F,\mathbf{C}F)=q^{\frac{c_L}{12}}\langle e^{\beta\tilde{\varphi}(b)}e^{\frac{1}{2}\ln(q) \mathbf{H}}F\mid e^{\frac{1}{2}\ln(q) \mathbf{H}}F\rangle_{\rm L^2(\R\times\Omega_{\T})}\geq 0
\end{align}
Recall that the semigroup $(e^{-t\mathbf{H}})_{t\geq 0}$ extends continuously to $e^{-\delta c_-}L^2(\R\times\Omega_\T)$ for $\delta\geq 0$ (see \cite[Lemma 6.5]{GKRV}). By density in weighted $L^2$-spaces this relation also holds for $F\in e^{-\epsilon c_-}L^2(\R\times\Omega_{\T})$ when $\epsilon\in [0,\beta)$. Then we take $F=H_{Q+iP,\nu,\emptyset}$ and deduce $ \mathcal{A}^D_{g_{\mathbb{A}},\mathbb{A}^{\frac{1}{\sqrt{q}}}_{\sqrt{q}},b,\beta}(H_{Q+iP,\nu,\emptyset},\mathbf{C}H_{Q+iP,\nu,\emptyset})\geq 0$. By the same reason, we know $\mathcal{A}^D_{g_{\mathbb{A}},\mathbb{A}^{\frac{1}{\sqrt{q}}}_{\sqrt{q}},b,\beta}(H_{Q+iP,\emptyset,\emptyset},\mathbf{C}H_{Q+iP,\emptyset,\emptyset})=\frac{\pi}{\sqrt{2}\pi}q^{-\frac{c_L}{12}+2\Delta_{Q+iP}}C^{\rm DOZZ}(Q+iP,\beta,Q-iP)\geq 0$. So we can deduce $C^{\rm DOZZ}(Q+iP,\beta,Q-iP)>0$ and $\mathcal{B}^D_{g_{\mathbb{A}},\mathbb{A}^{\frac{1}{\sqrt{q}}}_{\sqrt{q}},b,\beta}(P,P,\nu,\nu)\geq 0$.

By \eqref{symmetry}, we see $\mathcal{A}^D_{g_{\mathbb{A}},\mathbb{A}^{\frac{1}{\sqrt{q}}}_{\sqrt{q}},b,\beta}(F,\mathbf{C}G)=\overline{\mathcal{A}^D_{g_{\mathbb{A}},\mathbb{A}^{\frac{1}{\sqrt{q}}}_{\sqrt{q}},b,\beta}(G,\mathbf{C}F)}$, so we have $\mathcal{B}^D_{g_{\mathbb{A}},\mathbb{A}^{\frac{1}{\sqrt{q}}}_{\sqrt{q}},b,\beta}(P,P,\nu,\tilde{\nu})=\overline{\mathcal{B}^D_{g_{\mathbb{A}},\mathbb{A}^{\frac{1}{\sqrt{q}}}_{\sqrt{q}},b,\beta}(P,P,\tilde{\nu},\nu)}.$
\end{proof}
Now we calculate the operator $\caB^D_{g_{\mathbb{A}},\mathbb{A}^{\frac{1}{\sqrt{q}}}_{\sqrt{q}},b,\beta}(P,P,\cdot,\cdot)$ in the following norms:
\begin{align}
   &||\mathcal{B}^D_{g_{\mathbb{A}},\mathbb{A}^{\frac{1}{\sqrt{q}}}_{\sqrt{q}},b,\beta}(P,P,\cdot,\cdot)||_{\rm Tr}=\sum_{\nu}|\mathcal{B}^D_{g_{\mathbb{A}},\mathbb{A}^{\frac{1}{\sqrt{q}}}_{\sqrt{q}},b,\beta}(P,P,\nu,\nu)|=\sum_{\nu}\mathcal{B}^D_{g_{\mathbb{A}},\mathbb{A}^{\frac{1}{\sqrt{q}}}_{\sqrt{q}},b,\beta}(P,P,\nu,\nu)\label{hsnorm}\\
   &||\mathcal{B}^m_{g_{\mathbb{A}},\mathbb{A}^{\frac{1}{\sqrt{q}}}_{\sqrt{q}},b,\beta}(P,\cdot,\cdot)||^2_{\rm HS}=\sum_{\nu,\tilde{\nu}}|\mathcal{B}^m_{g_{\mathbb{A}},\mathbb{A}^{1}_{\sqrt{q}},b,\beta}(P,\nu,\tilde{\nu})|^2\\
   &\langle \mathcal{B}^m_{g_{\mathbb{A}},\mathbb{A}^{1}_{\sqrt{q}},b_1,\beta_1}(P,\cdot,\cdot)\mid  \overline{\mathcal{B}^m_{g_{\mathbb{A}},\mathbb{A}^{1}_{\sqrt{q}},b_2,\beta_2}}(P,\cdot,\cdot)\rangle_{\rm HS}=\sum_{\nu,\tilde{\nu}}\mathcal{B}^m_{g_{\mathbb{A}},\mathbb{A}^{1}_{\sqrt{q}},b_1,\beta_1}(P,\nu,\tilde{\nu})\mathcal{B}^m_{g_{\mathbb{A}},\mathbb{A}^{1}_{\sqrt{q}},b_2,\beta_2}(P,\nu,\tilde{\nu})\label{hsinner}
\end{align}
\begin{remark}
Since the trace norm can be identified with torus 1-point conformal block for $\beta\in [0,Q)$. By \cite{GRSS}, we see when $q\in(0,1)$ the trace norm (hence the Hilbert-Schmidt norm) is bounded. So we can extend above construction to all $P\in\R$.
\end{remark}
Now we can deduce the conformal block $\mathcal{F}^{\A}(\Delta_{\beta_1},\Delta_{\beta_2},\Delta_{Q+iP},q,b_1,b_2)$ is continuous on $P$.
\begin{proposition}\label{continuous2point}
 For $P$ and $P_0\in [0,\infty)$,
 we have $$\lim_{P\to P_0}\mathcal{F}^{\A}(\Delta_{\beta_1},\Delta_{\beta_2},\Delta_{Q+iP},q,b_1,b_2)=\mathcal{F}^{\A}(\Delta_{\beta_1},\Delta_{\beta_2},\Delta_{Q+iP_0},q,b_1,b_2).$$
 Similarly, it's also separately continuous in $\beta_i\in [0,Q)$. 
\end{proposition}
\begin{proof}
This claim is equivalent to
\begin{align}\label{HSconvergence}
    \lim_{P\to P_0}\langle \mathcal{B}^m_{g_{\mathbb{A}},\mathbb{A}^{1}_{\sqrt{q}},b_1,\beta_1}(P,\cdot,\cdot)\mid  \overline{\mathcal{B}^m_{g_{\mathbb{A}},\mathbb{A}^{1}_{\sqrt{q}},b_2,\beta_2}}(P,\cdot,\cdot)\rangle_{\rm HS}=\langle \mathcal{B}^m_{g_{\mathbb{A}},\mathbb{A}^{1}_{\sqrt{q}},b_1,\beta_1}(P_0,\cdot,\cdot)\mid  \overline{\mathcal{B}^m_{g_{\mathbb{A}},\mathbb{A}^{1}_{\sqrt{q}},b_2,\beta_2}}(P_0,\cdot,\cdot)\rangle_{\rm HS}
\end{align}
From \cite{GRSS}, we know for $\beta_1\in [0,Q)$, torus 1-point conformal block is continuous in $P$, i.e. $\lim_{P\to P_0}||\mathcal{B}^m_{g_{\mathbb{A}},\mathbb{A}^{1}_{\sqrt{q}},b_1,\beta_1}(P,\cdot,\cdot)||_{\rm Tr}=||\mathcal{B}^m_{g_{\mathbb{A}},\mathbb{A}^{1}_{\sqrt{q}},b_1,\beta_1}(P_0,\cdot,\cdot)||_{\rm Tr}$. To show $\mathcal{B}^m_{g_{\mathbb{A}},\mathbb{A}^{1}_{\sqrt{q}},b_1,\beta_1}(P,\cdot,\cdot)$ converges to $\mathcal{B}^m_{g_{\mathbb{A}},\mathbb{A}^{1}_{\sqrt{q}},b_1,\beta_1}(P_0,\cdot,\cdot)$ in the trace norm, we cite the following lemma from \cite[Theorem 1]{Si}.
\begin{lemma}
For $\{x_n\}$ in the trace class, and $x$ also in the trace class. If we have $||x_n||_{tr}\to||x||_{tr}$ as $n\to \infty$ and $x_n$ converge to $x$ in the weak operator topology. Then $\lim_{n\to \infty}||x_n-x||_{tr}=0$.
\end{lemma}

Since $||\mathcal{B}^m_{g_{\mathbb{A}},\mathbb{A}^{1}_{\sqrt{q}},b_1,\beta_1}(P,\cdot,\cdot)||_{op}\leq||\mathcal{B}^m_{g_{\mathbb{A}},\mathbb{A}^{1}_{\sqrt{q}},b_1,\beta_1}(P,\cdot,\cdot)||_{tr}\leq C$, it's enough to check $\mathcal{B}^m_{g_{\mathbb{A}},\mathbb{A}^{1}_{\sqrt{q}},b_1,\beta_1}(P,\cdot,\cdot)$ converges to $\mathcal{B}^m_{g_{\mathbb{A}},\mathbb{A}^{1}_{\sqrt{q}},b_1,\beta_1}(P,\cdot,\cdot)$ in the weak operator topology on a basis, i.e. for all pair of young diagrams $\nu$, $\tilde{\nu}$
\begin{align}
    \lim_{P\to P_0}\mathcal{B}^m_{g_{\mathbb{A}},\mathbb{A}^{1}_{\sqrt{q}},b_1,\beta_1}(P,\nu,\tilde{\nu})=\mathcal{B}^m_{g_{\mathbb{A}},\mathbb{A}^{1}_{\sqrt{q}},b_1,\beta_1}(P_0,\nu,\tilde{\nu})
\end{align}
which is immediately from $\mathcal{B}^m_{g_{\mathbb{A}},\mathbb{A}^{1}_{\sqrt{q}},b_1,\beta_1}(P,\nu,\tilde{\nu})$ is a rational function on $\Delta_{Q+iP}$ with no poles on $P\in\R$.
Actually, the poles of rational function $\mathcal{B}^m_{g_{\mathbb{A}},\mathbb{A}^{1}_{\sqrt{q}},b_1,\beta_1}(P,\nu,\tilde{\nu})$ only comes from Schapovalov form since $w_{\mathbb{A}}(\Delta_{\beta_1},\Delta_{Q+iP},\Delta_{Q+iP},
 \nu,\tilde{\nu} )$ is a polynomial of $\Delta_{Q+iP}$ by \cite[Theorem 11.17]{GKRV21}. 
 Then \eqref{HSconvergence} follows from the strong convergence in the trace norm. For continuity on $\beta_i$, we can treat similarly.
\end{proof}
\begin{corollary}
For $q\in (0,1)$ and $\beta_1\in[0,Q)$, we have $$ \lim_{\beta_2\to 0}\mathcal{F}^{\A}(\Delta_{\beta_1},\Delta_{\beta_2},\Delta_{Q+iP},q,b_1,b_2)=\mathcal{F}^{\T}(\Delta_{\beta_1},\Delta_{Q+iP},q^2).$$
\end{corollary}
\begin{proof}
Since the two-point annulus block is continuous on $\beta_2$, we can use the proposition \eqref{0coefficient} in the next section to simplify the expression, i.e.
$$\mathcal{F}^{\A}(\Delta_{\beta_1},0,\Delta_{Q+iP},q,b_1,b_2)=\mathcal{F}^{\T}(\Delta_{\beta_1},\Delta_{Q+iP},q^2)$$
\end{proof}
\section{The \texorpdfstring{$\gamma$}{gamma} weight torus 1-point conformal block}\label{gamma insertion proof}

In this section, we give an explicit formula for torus  1-point block when the insertion weight is $\gamma$, which solves a conjecture in \cite{Be}. We will first treat the $0$-weight case, then we relate the $\gamma$-weight with $0$-weight by its annulus counterpart.
\begin{proposition}
\label{0coefficient}
For fixed Young diagrams $\nu$ and $\tilde{\nu}$, we have
\begin{align}\label{fixyoung}
    \lim_{\beta \to 0}\mathcal{W}_{\mathbb{A}}(\Delta_{\beta},\Delta_{Q+iP},\Delta_{Q+iP},
 \nu,\tilde{\nu} )= 1_{\{\nu=\tilde{\nu}\}}.
\end{align}
\end{proposition}
\begin{proof}

By the Ward identity, we know the left side of \eqref{fixyoung} is a polynomial of $\Delta_{\beta}$, so the limit must exist.
We define
$S(L_{-\nu},L_{-\tilde{\nu}}):=\lim_{\beta\to 0}w_{\mathbb{A}}(\Delta_{\beta},\Delta_{Q+iP},\Delta_{Q+iP},
 \nu,\tilde{\nu} )$, and we hope to prove $S(L_{-\nu},L_{-\tilde{\nu}})=F_{Q+iP}(\nu,\tilde{\nu})$.
 
Now we claim
\begin{align}\label{transform}
    S(L_{-n}L_{-\nu},L_{-\tilde{\nu}})=S(L_{-\nu},L_{n}L_{-\tilde{\nu}})
\end{align}

Before proving this claim, we first give some remarks on $w_{\mathbb{A}}(\Delta_{\beta},\Delta_{Q+iP},\Delta_{Q+iP},
 \nu,\tilde{\nu} )$. The original definition of $w_{\mathbb{A}}(\Delta_{\beta},\Delta_{Q+iP},\Delta_{Q+iP},
 \nu,\tilde{\nu} )$ in \cite[Corollary 11.10]{GKRV21} comes from uniformization of a plumbed pant \cite[Corollary]{GKRV21} to a complex plane with three holes which center at $(z_1,z_2,z_3)\in\D^3$. But this construction is not easy to get some recursive relation. Here we give another construction, as before, we insert enough artificial points $(\beta_i^{ar},z_i^{ar})$ on $\mathcal{A}^{D}_{g_{\mathbb{A}},\mathbb{A}^{1}_{q},z,\beta,{\boldsymbol{\zeta}}}$ where $q<|z|<1$, and pairing it with $\Psi_{\alpha_1,\nu,\emptyset}$ and $\Psi_{\alpha_2,\tilde{\nu},\emptyset}$ by  m\"obius maps $\psi_1:z\to qz$ and $\psi_2:z\to 1/z$ respectively. Now we hope to compute the Ward's identity and the corresponding coefficient $w'_{\mathbb{A}}(\Delta_{\beta},\Delta_{Q+iP},\Delta_{Q+iP},
 \nu,\tilde{\nu} )$ in this setting.
 
we define for $\boldsymbol{\delta}>\max\{\boldsymbol{\delta}_{1},\boldsymbol{\delta}_{2}\}$,
\begin{align}\label{manypoints}
     &\left(\mathbf{L}_{-{\nu}} V_{\alpha_{1}}(0) \mid \mathbf{L}_{-n} \mathbf{L}_{-\tilde{\nu}} V_{\alpha_{2}}(\infty)\mid V_{\beta}(z)\right)_{t}= \lim_{z'\to \infty}\nonumber \oint_{\left|\mathbf{u}_{1}\right|= \boldsymbol{\delta}_{1}} \oint_{|v|=1 /\boldsymbol {\delta}} \oint_{\left|\mathbf{u}_{2}\right|=1 / \boldsymbol {\delta}_2} \\&\mathbf{u}_{1}^{1-\nu} \mathbf{u}_{2}^{1+\tilde{\nu}} v^{1+n}\left\langle T^{\hat\C}\left(\mathbf{u}_{1}\right) T^{\hat\C}\left(\mathbf{u}_{2}\right) T^{\hat\C}(v) V_{\alpha_{1}}(0) V_{\alpha_{2}}(z') V_{\beta}(z) \prod_{i=1}^m V_{\beta^{ar}_i}(z^{ar}_i)\right\rangle_{\hat\C,g_{dozz},U_{t}} \dd \mathbf{u}_{1} \dd \mathbf{u}_{2} \dd v 
\end{align}
Now we claim
\begin{lemma}
\begin{align}\label{torusward}
&\left(\mathbf{L}_{-\nu}V_{\alpha_{1}}(0) \mid \mathbf{L}_{-n} \mathbf{L}_{-\tilde{\nu}} V_{\alpha_{2}}(\infty)\mid V_{\beta}(z)\right)_{t}\nonumber\\: =&\left(\mathbf{L}_{n} \mathbf{L}_{-\nu} V_{\alpha_{1}}(0) \mid \mathbf{L}_{-\tilde{\nu}} V_{\alpha_{2}}(\infty)\mid V_{\beta}(z)\right)_{t}-\mathbf{D}_{-n}^{(0)}\left(\mathbf{L}_{-\nu} V_{\alpha_{1}}(0) \mid \mathbf{L}_{-\tilde{\nu}} V_{\alpha_{2}}(\infty)\mid V_{\beta}(z)\right)_{t}+\epsilon_{t}(\mathbf{z}) \end{align}
where $$\mathbf{D}_{-n}^{(0)}=-z^{n+1} \partial_{z}-(n+1)z^{n} \Delta_{\beta}+\sum_{i=1}^m \Big(-{(z_{i}^{ar})}^{n+1} \partial_{z_i^{ar}}-(n+1){(z^{ar}_{i})}^{n} \Delta_{\beta_i^{ar}}\Big).$$
If $\max\{\alpha_1,\alpha_2,\beta\}<\bf a$, we have $\epsilon_{t}(\mathbf{z})$ goes to $0$ when $t\to\infty $
\end{lemma}
\begin{proof}
By using Ward's identity on sphere to  $\eqref{manypoints}$, we expend the $T^{\hat\C}(v)$ insertion. Since $|z|<1$ and $|z_i^{ar}|<1$, the metric terms $\frac{\Delta_{\beta}\partial_z\omega(z)}{v-z}=0$, $\frac{\Delta_{\beta_i^{\rm ar}}\partial_{z_i^{\rm ar}}\omega({z_i^{\rm ar}})}{v-{z_i^{\rm ar}}}=0$ . When ignore the $\epsilon_t(z)$ term, for $|z'|>\frac{1}{\boldsymbol{\delta}_{2}}$, we can see the map $$v\to \langle T\left(\mathbf{u}_{1}\right) T\left(\mathbf{u}_{2}\right) T(v) V_{\alpha_{1}}(0) V_{\alpha_{2}}(z') V_{\beta}(z) \prod_{i=1}^m V_{\beta^{ar}_i}(z^{ar}_i)\rangle_{\hat\C,g_{dozz},U_{t}} $$ is meromorphic in the annular region $\{v\in\C\mid\boldsymbol{\delta}<|v|<\frac{1}{\boldsymbol{\delta}}\}$ whose poles are $z$ and $z_i^{\rm ar}$, $1\leq i\leq m$. Computing these residues and take $z'\to\infty$ give the conclusion.
\end{proof}
\begin{lemma}\label{virasorostructure}

$$
\Psi_{\alpha, \nu, \widetilde{\nu}}=\mathbf{L}_{-\nu_{k}} \ldots \mathbf{L}_{-\nu_{1}} \widetilde{\mathbf{L}}_{-\widetilde{\nu}_{k^{\prime}}} \ldots \widetilde{\mathbf{L}}_{-\widetilde{\nu}_{1}} \Psi_{\alpha }
$$
and for each $n \in \mathbb{Z}$, and $\alpha \notin Q \pm\left(\frac{2}{\gamma} \mathbb{N}+\frac{\gamma}{2} \mathbb{N}\right)$,
 we have $\mathbf{L}_n\Psi_{Q+iP,\nu,\emptyset}\in Span\{\Psi_{Q+iP,\nu',\emptyset}\mid |\nu'|=|\nu|-n\}$.
\end{lemma}
\begin{proof}
This follows from \cite[Theorem 1.2]{BGKRV}
\end{proof}
Now we perform the standard analytic continuation argument by choosing $\alpha_1=Q+iP$, $\alpha_2=Q-iP$, $\beta\in(0,Q)$ and all $\beta_i^{\rm ar}=0$. By \cite[proposition 11.16]{GKRV21}, we can take the restriction of $\mathbf{D}_{-n}^{(0)}\mid_{(z,z_1^{ar},..,z_m^{ar})=(z,0,..,0)}=-z^{n+1} \partial_{z}-(n+1){z^{n}} \Delta_{\beta}$.
We recall that
\begin{align}
 \langle  &V_{Q+iP}(0)  V_{\beta}(z) V_{Q-iP}(\infty)  \rangle^{\hat\C}_{g_{\rm DOZZ}}=|z|^{-2\Delta_{\beta}}\frac{1}{2}C_{\gamma,\mu}^{{\rm DOZZ}} (Q+iP,\beta,Q-iP ) \big(\frac{{\rm v}_{g_{\rm DOZZ}}(\hat\C)}{{\det}'(\Delta^{\hat C}_{g_{\rm DOZZ}})}\big)^{\frac{1}{2}}  \nonumber
 \end{align} 
 Then by above two lemmas, we can construct a polynomial $w'_{\mathbb{A}}(\Delta_{\beta},\Delta_{Q+iP},\Delta_{Q+iP},
 \nu,\tilde{\nu} )$ recursively with $w'_{\mathbb{A}}(\Delta_{\beta},\Delta_{Q+iP},\Delta_{Q+iP},
 \emptyset,\emptyset )=1$ which satisfies: for some constant $C'_\A$.
 \begin{align*}
    \mathcal{A}^{D}_{g_{\mathbb{A}},\mathbb{A}^{1}_{q},z,\beta,{\boldsymbol{\zeta}}}(\Psi_{Q+iP,\nu,\emptyset}, \mathbf{C}\Psi_{Q+iP,\tilde{\nu},\emptyset})&=C'_{\A}C_{\gamma,\mu}^{{\rm DOZZ}} (Q+iP,\beta,Q+iP) w'_{\mathbb{A}}(\Delta_{\beta},\Delta_{Q+iP},\Delta_{Q+iP},
 \nu,\tilde{\nu} )\\
   &\times|q |^{-{c}_{\rm L}/12+2\Delta_{Q+ip}}  (\frac{q}{z})^{|\nu|} z^{|\tilde{\nu}|}\nonumber 
\end{align*}
 By taking $\nu=\tilde{\nu}=\emptyset$, we see $C'_\A=C_\A$. Then we get $$w'_{\mathbb{A}}(\Delta_{\beta},\Delta_{Q+iP},\Delta_{Q+iP},
 \nu,\tilde{\nu} )=w_{\mathbb{A}}(\Delta_{\beta},\Delta_{Q+iP},\Delta_{Q+iP},
 \nu,\tilde{\nu} ). $$
 We also observe that when $|\nu|=|\tilde\nu|$, $\mathcal{A}^{D}_{g_{\mathbb{A}},\mathbb{A}^{1}_{q},z,\beta,{\boldsymbol{\zeta}}}(\Psi_{Q+iP,\nu,\emptyset}, \mathbf{C}\Psi_{Q+iP,\tilde{\nu},\emptyset})$ doesn't depend on $z$.
 
 Now we start to prove \eqref{transform} by induction,
\begin{lemma}\label{inductionbasic}
For the case $\nu\neq\emptyset$ and $\tilde{\nu}=\emptyset$,
we have $$\lim_{\beta\to 0} w_{\mathbb{A}}(\Delta_{\beta},\Delta_{Q+iP},\Delta_{Q+iP},
 \nu,\emptyset )=0.$$
\end{lemma}
\begin{proof}
By \cite[Proposition 7.12]{GKRV}, for $\nu=(\nu_1,...,\nu_k)$, we have $$w_{\mathbb{\A}}(\Delta_{\beta},\Delta_{Q+iP},\Delta_{Q+iP},
 \nu,\emptyset )=\prod_{i=1}^k(\nu_i\Delta_{\beta}+\Delta_{Q+iP}-\Delta_{Q-iP}+\sum_{j<i}\nu_j)$$
 The $i=1$ term in the RHS is $\Delta_{\beta}\nu_1$, take $\beta\to 0$ then we prove the claim.
\end{proof}

We can and divide \eqref{torusward} by $C_{\gamma,\mu}^{\rm DOZZ}(Q+iP,\beta,Q-iP)$
and send $\beta\to 0$. By induction on $k=|\nu|+|\tilde{\nu}|$, and use lemma \eqref{inductionbasic} (as induction basis) one can see $$\lim_{\beta\to 0}\frac{\left(\mathbf{L}_{-{\nu}} V_{Q+iP}(0) \mid  \mathbf{L}_{-\tilde{\nu}} V_{Q-iP}(\infty)\mid V_{\beta}(z)\right)}{C_{\gamma,\mu}^{\rm DOZZ}(Q+iP,\beta,Q-iP)}$$ doesn't depend on position $z$. Indeed, the first term in \eqref{torusward} has index $|\nu|+|\tilde{\nu}|+n$, the second term is a linear combination of terms with index $|\nu|+|\tilde{\nu}|-n$ (by lemma \eqref{virasorostructure}), the third term is $\mathbf{D}_{-n}^{(0)}$ act on a term with index $|\nu|+|\tilde{\nu}|$. When taking $\beta\to 0$, it's easy to see the third term vanish (since we suppose it doesn't depend on $z$). Then the first term equals the second term and doesn't depend on $z$.

By Ward's identity, we have $$\frac{\left(\mathbf{L}_{-{\nu}} V_{Q+iP}(0) \mid  \mathbf{L}_{-\tilde{\nu}} V_{Q-iP}(\infty)\mid V_{\beta}(z)\right)}{C_{\gamma,\mu}^{\rm DOZZ}(Q+iP,\beta,Q-iP)}=\frac{1}{2} \big(\frac{{\rm v}_{g_{\rm dozz}}(\hat\C)}{{\det}'(\Delta^{\hat C}_{g_{\rm dozz}})}\big)^{\frac{1}{2}} w_{\mathbb{A}}(\Delta_{\beta},\Delta_{Q+iP},\Delta_{Q+iP},
 \nu,\tilde{\nu} )z^{-\Delta_{\beta}+|\nu|-|\tilde{\nu}|}$$
Then we take the limit $\beta \to 0$ in both sides, we know the LHS doesn't depend on $z$, so we can set $z=1$ on the RHS.
This shows $S(L_{-n}L_{-\nu},L_{-\tilde{\nu}})=S(L_{-\nu},L_{n}L_{-\tilde{\nu}})$ for any $\nu$ and $\tilde{\nu}$.

By the commutative relation of Liouville Virasoro algebra and lemma \eqref{inductionbasic}, we see $S(L_{-\nu},L_{-\tilde{\nu}})=F_{Q+ip}(\nu.\tilde{\nu})$.
\noindent so we get
\begin{align*}
    \lim_{\beta\to 0}\mathcal{W}_{\mathbb{A}}(\Delta_{\beta},\Delta_{Q+iP},
 \nu,\tilde{\nu} )=1_{\{\nu=\tilde{\nu}\}}
\end{align*}
\end{proof}
\begin{remark}
In the previous proposition, we prove the correlation in the free field case can be recovered from the Liouville case by sending the middle-position weight to 0, i.e.
\begin{align}
    \lim_{\beta\to 0}\frac{\mathcal{A}^D_{\A^{1}_{q},g_{\A},b,\beta,\zeta}(\Psi_{Q+iP,\nu_1,\tilde{\nu}_1},\mathbf{C}\Psi_{Q-iP,\nu_2,\tilde{\nu}_2})}{C\times C^{DOZZ}_{\gamma,\mu}(Q+iP,\beta,Q-iP)}=\langle\frac{\Psi^0_{Q+iP,\nu_1,\tilde{\nu}_1}}{\Psi^0_{Q+iP}}\mid\frac{\Psi^0_{Q+iP,\nu_2,\tilde{\nu}_2}}{\Psi^0_{Q+iP}}\rangle_{L^2(\Omega_{\T})}\end{align}
    Where $C=\frac{\pi}{2^{1/2}e} |q |^{-{c}_{\rm L}/12+2\Delta_{Q+iP}}    (q/b)^{|\nu_{1}|}(\bar{q}/\bar{b})^{ |\tilde{\nu}_{1}|} b^{|\nu_{2}|}\bar{b}^{|\tilde\nu_{2}|}$.
\end{remark}
\begin{corollary}\label{realvalue}
For $\beta\in [0,Q)$, $w_{\mathbb{A}}(\Delta_{\beta},\Delta_{Q+iP},\Delta_{Q+iP},
 \nu,\tilde{\nu} )$ is real valued.
\end{corollary}
\begin{proof}
In Lemma \ref{inductionbasic}, we see $w_{\mathbb{A}}(\Delta_{\beta},\Delta_{Q+iP},\Delta_{Q+iP},
 \nu,\emptyset )$ is real valued. Then we choose $q$ real take $z=\sqrt{q}$, and repeat above induction. We see $w_{\mathbb{A}}(\Delta_{\beta},\Delta_{Q+iP},\Delta_{Q+iP},
 \nu,\tilde{\nu} )$ is real valued. We can also see $F_{Q+iP}$ is a real valued matrix from this.
\end{proof}
Now we start to compute the block amplitude for $\gamma$-weight insertion,
\begin{proposition}For $\beta\in (0,Q)$, we have
\begin{align}\label{derivative}
    \partial_{\mu_{\partial}}\Big(\mathcal{A}^{m}_{g_{\mathbb{A}},\mathbb{A}^1_{\sqrt{q}},b,\beta}(\Psi_{Q+iP,\nu,\tilde{\nu}})\Big)=\Big(\partial_{\mu_{\partial}}\mathcal{A}^{m}_{g_{\mathbb{A}},\mathbb{A}^1_{\sqrt{q}},b,\beta}\Big)(\Psi_{Q+iP,\nu,\tilde{\nu}})
\end{align}
 
\end{proposition}
 \begin{proof} we start from a simple lemma. Although we only need the $\beta>0$ case, but we prove a stronger version
 \begin{lemma}\label{nicefunction}
  For $\beta>-\gamma$, we have $\partial_{\mu_{\partial}}\mathcal{A}^{m}_{g_{\mathbb{A}},\mathbb{A}^1_{\sqrt{q}},b,\beta}\in e^{\epsilon c_{-}}L^2$ for $0<\epsilon<\frac{\beta+\gamma}{2}$.
 \end{lemma}
 \begin{proof}
 First, we show that for Note $\phi=P\tilde{\varphi}+X_{m}^{\mathbb{A}^1_{\sqrt{q}}}=c+\phi^{\circ}$
\begin{align*}
  \text{the GMC part of }\partial_{\mu_{\partial}}\mathcal{A}^{m}_{g_{\mathbb{A}},\mathbb{A}^1_{\sqrt{q}},b,\beta}(\tilde{\varphi})&=\E[M^{\partial}_\gamma(\phi,\partial^{out} \mathbb{A}_q) e^{\frac{\beta}{2}\phi(b)}e^{-\mu M_{\gamma}(\phi,\mathbb{A}^1_{\sqrt{q}})-\mu_\partial M^{\partial}_\gamma(\phi,\partial^{out} \mathbb{A}^1_q)}]\\
  &\leq e^{(\frac{\gamma}{2}(1-\delta)+\frac{\beta}{2})c}\E[M^{\partial}_\gamma(\phi^{\circ}+\delta c,\partial^{out} \mathbb{A}_q) e^{\frac{\beta}{2}\phi^{\circ}(b)}e^{-\frac{\mu_\partial}{2} M^{\partial}_\gamma(\phi,\partial^{out} \mathbb{A}^1_q)-\frac{\mu_\partial}{2} M^{\partial}_\gamma(\phi,\partial^{out} \mathbb{A}^1_q)}]\\
  &\leq e^{(\frac{\gamma}{2}(1-\delta)+\frac{\beta}{2})c}e^{\frac{\beta}{2}P\varphi(b)}\E[ e^{\frac{\beta}{2}X_{m}^{\mathbb{A}^1_{\sqrt{q}}}}e^{-\frac{\mu_\partial}{2} M^{\partial}_\gamma(\phi,\partial^{out} \mathbb{A}^1_q\backslash O)}]
\end{align*}
where $O$ is a neighbourhood of $b$. The second line to the third line is derived by using H\"older inequality
\begin{align}\label{trick}
    \E[M^{\partial}_\gamma(\phi^{\circ}+\delta c,\partial^{out} \mathbb{A}^1_q) e^{-\frac{\mu_\partial}{2} M^{\partial}_\gamma(\phi,\partial^{out} \mathbb{A}^1_q)}]\leq& C\E[M^{\partial}_\gamma(\phi^{\circ}+\delta c,\partial^{out} \mathbb{A}^1_q) \big(\frac{\mu_\partial}{2} M^{\partial}_\gamma(\phi,\partial^{out} \mathbb{A}^1_q)\Big)^{-\delta}]\\
    &\leq C \E[M_{\gamma}^{\partial}(\phi^{\circ},\partial^{out} \mathbb{A}_q)^{1-\delta}]<\infty
    \end{align}
In the second line, we use the following lemma
\begin{lemma}
$\E[M_{\gamma}^{\partial}(X_m^{\A^1_{\sqrt{q}}},\partial^{out}\A^1_q)]<\infty$
\end{lemma}
\begin{proof}
First we double the surface $\A^1_{\sqrt{q}}$ along the Neumann boundary $\partial^{out}\A^1_q$. Then we know for $s\in\partial^{out}\A^1_q$,  $X_m^{\A^1_{\sqrt{q}}}(s)\stackrel{law}=\sqrt{2}X_D^{\A^{\frac{1}{\sqrt{q}}}_{\sqrt{q}}}(s)$. By \cite[Lemma 3.2]{GRV19}, there exists a smooth function $W^D(z)$ such that $\lim_{\epsilon\to 0}\E[X_D^{\A^{\frac{1}{\sqrt{q}}}_{\sqrt{q}}}(z)^2]+\ln(\epsilon)=W^D(z)$ uniformly in compact set of $\A^{\frac{1}{\sqrt{q}}}_{\sqrt{q}}$. So $\E[M_{\gamma}^{\partial}(X_m^{\A^1_{\sqrt{q}}},\partial^{out}\A^1_q)]=\int_\T \E[e^{\frac{\gamma}{2}X_m^{\A^1_{\sqrt{q}}}(s)-\frac{\gamma^2}{8}\E[X_m^{\A^1_{\sqrt{q}}}(s)^2]}]e^{\frac{\gamma^2}{4}W^D(s)}\frac{\dd s}{|s|}<\infty$.
\end{proof}
By Girsanov theorem, we get term $\E[ e^{-\frac{\mu_\partial}{2} M^{\partial}_\gamma(\phi+u(b),\partial^{out} \mathbb{A}^1_q\backslash O)}]$. When $c<0$, we simply take $\mu_\partial=0$. When $c>0$, we can bound $\E[ e^{-\frac{\mu_\partial}{2} M^{\partial}_\gamma(\phi+u(b),\partial^{out} \mathbb{A}^1_q\backslash O)}]\leq C_R \E[\Big( \frac{\mu_\partial}{2} M^{\partial}_\gamma(\phi+u(b),\partial^{out} \mathbb{A}^1_q\backslash O)\Big)^{-R}]$ as in \cite[Lemma 4.6]{GKRV21} to control the $c>0$ case.
\end{proof}
\noindent For fixed Young diagrams $\nu$ and $\tilde{\nu}$, we have $\Psi_{Q+iP,\nu,\tilde{\nu}}\in \cap_{\epsilon>0}e^{-\epsilon c_-}L^2(\R\times\Omega_{\T}).$
Then $$\sup_{\mu_\partial>\delta>0}|\partial_{\mu_{\partial}}\mathcal{A}^{m}_{g_{\mathbb{A}},\mathbb{A}^1_{\sqrt{q}},b,\beta}(\tilde{\varphi})(\Psi_{Q+iP,\nu,\tilde{\nu}})(\tilde{\varphi})|\in L^1(\R\times\Omega_{\T}),$$
by dominate convergence, we have for $\beta\in (0,Q)$
\begin{align}
    \partial_{\mu_{\partial}}\Big(\mathcal{A}^{m}_{g_{\mathbb{A}},\mathbb{A}^1_{\sqrt{q}},b,\beta}(\Psi_{Q+iP,\nu,\tilde{\nu}})\Big)=\Big(\partial_{\mu_{\partial}}\mathcal{A}^{m}_{g_{\mathbb{A}},\mathbb{A}^1_{\sqrt{q}},b,\beta}\Big)(\Psi_{Q+iP,\nu,\tilde{\nu}})
\end{align}
As a special case, we also have $\partial_{\mu_\partial}\Big(\caA^m_{g_{\A},\A^1_{\sqrt{q}},b,\beta}(\Psi_{Q+iP})\Big)=(\partial_{\mu_\partial}\caA^m_{g_{\A},\A^1_{\sqrt{q}},b,\beta})(\Psi_{Q+iP})$.

\end{proof}

\begin{lemma}
$\beta\to\Big(\partial_{\mu_{\partial}}\mathcal{A}^{m}_{g_{\mathbb{A}},\mathbb{A}^1_{\sqrt{q}},b,\beta}\Big)(c,\varphi)$ is analytic in a complex neighbourhood of $\{\beta\in \R\mid \beta<\frac{2}{\gamma}\}$ as a function in $e^{\delta c_-}L^2(\Omega\times\R)$ for any $\delta<\frac{\Re(\beta)+\gamma}{2}$. 
\end{lemma}
\begin{proof}
Use the same regularization strategy in \eqref{holomorphic}, note $\beta=a+ib$, we define 
\begin{align*}
    &\Big(\partial_{\mu_\partial}\mathcal{A}^{m,k}_{g_{\mathbb{A}},\mathbb{A}^1_{\sqrt{q}},b,\beta}\Big)(c,\varphi)=-\E_{\tilde{\varphi}} \Big[  \epsilon_k^{\frac{\beta^2}{4}} e^{\frac{\beta}{2}(c+\phi_{g,k}(b))}e^{\frac{\gamma}{2}c}M^{\partial}_{\gamma, g}(\phi_g,\R_k )\exp\Big(-\mu  e^{\gamma c}M_{\gamma, g}(\phi_g,\H)-\mu_{\partial} e^{\frac{\gamma}{2}c}M^{\partial}_{\gamma, g}(\phi_g,\R_k )\Big) \Big]\\
    &\leq e^{\frac{\Re(\beta)}{2}c}e^{\frac{\gamma}{2}(1-\epsilon)c}2^\frac{(k+1)b^2}{4}e^{\frac{\epsilon}{2}c}|\E_{\tilde{\varphi}}\Big[\Big((Y_k+\delta Y_k)e^{-\mu_\partial e^{\frac{\gamma}{2}c}(Y_k+\delta Y_k)}-Y_k e^{-\mu_\partial e^{\frac{\gamma}{2}c}Y_k}\Big)e^{-\mu  e^{\gamma c}M_{\gamma, g}(\phi_g+G_{m,k+1}(x,s),\H)}\Big]|
\end{align*}
 As in \cite[Proposition C.1.]{GKRV}, after using Girsanov theorem, we are left to control
 \begin{align}
     e^{\frac{\gamma}{2}c\epsilon}|\E_{\tilde{\varphi}}\Big[(Y_k+\delta Y_k)e^{-\mu_\partial e^{\frac{\gamma}{2}c} (Y_k+\delta Y_k)}-Y_k e^{-\mu_\partial e^{\frac{\gamma}{2}c} Y_k}\Big]|
 \end{align}
in the case $\delta Y_k> 1$,
\begin{align*}
    &e^{\frac{\epsilon}{2}c}\E_{\tilde{\varphi}}\Big[\indic_{\delta Y_k> 1}|(Y_k+\delta Y_k)e^{-\mu_\partial e^{\frac{\gamma}{2}c }(Y_k+\delta Y_k)}-Y_k e^{-\mu_\partial e^{\frac{\gamma}{2}c} Y_k}|\Big]\\
    \leq &C_1 \E_{\tilde{\varphi}}\Big[\indic_{\delta Y_k>1}(Y_k+\delta Y_k)^{1-\epsilon}e^{-\frac{\mu_\partial}{2} e^{\frac{\gamma}{2}c }Y_k}\Big]+C_2 E_{\tilde{\varphi}}\Big[\indic_{\delta Y_k>1}Y_k^{1-\epsilon}e^{-\frac{\mu_\partial}{2} e^{\frac{\gamma}{2}c }Y_k}\Big]\\
    \leq& C_1\E_{\tilde{\varphi}}[(Y_k+\delta Y_k)]^{1-\epsilon}\E_{\tilde{\varphi}}[\indic_{\delta Y_k>1}e^{-\frac{\mu_\partial}{2\epsilon} e^{\frac{\gamma}{2}c }Y_k}]^{\epsilon}+ C_2\E_{\tilde{\varphi}}[(Y_k)]^{1-\epsilon}\E_{\tilde{\varphi}}[\indic_{\delta Y_k>1}e^{-\frac{\mu_\partial}{2\epsilon} e^{\frac{\gamma}{2}c }Y_k}]^{\epsilon}\\
    \leq & C  \E_{\tilde{\varphi}}[(\delta Y_k)^{\eta}]^\epsilon \E_{\tilde{\varphi}}[e^{-\frac{\mu_\partial}{2\epsilon} e^{\frac{\gamma}{2}c }Y_k}]^{\epsilon}
\end{align*}
From the third line to the forth line, we use when $\Re(\beta)<\frac{2}{\gamma}$, the expectation of $(Y_k+\delta Y_k)$ is finite.

For the case $\delta Y_k \leq 1$,
\begin{align*}
   &e^{\frac{\epsilon}{2}c}\E_{\tilde{\varphi}}\Big[\indic_{\delta Y_k\leq 1}|(Y_k+\delta Y_k)e^{-\mu_\partial e^{\frac{\gamma}{2}c }(Y_k+\delta Y_k)}-Y_k e^{-\mu_\partial e^{\frac{\gamma}{2}c} Y_k}|\Big]\\
  \leq  & e^{\frac{\epsilon}{2}c}\E_{\tilde{\varphi}}\Big[\indic_{\delta Y_k\leq 1}\delta Y_k e^{-\mu_\partial e^{\frac{\gamma}{2}c }Y_k}\Big]+ e^{\frac{\epsilon}{2}c}\mu_\partial e^{\frac{\gamma}{2}c}\E_{\tilde{\varphi}}\Big[\indic_{\delta Y_k\leq 1}Y_k \delta Y_k e^{-\mu_\partial e^{\frac{\gamma}{2}c }Y_k}\Big]\\
    \leq & C e^{\frac{\epsilon}{2}c}\E_{\tilde{\varphi}}\Big[\indic_{\delta Y_k\leq 1}\delta Y_k e^{-\frac{\mu_\partial}{2} e^{\frac{\gamma}{2}c }Y_k}\Big] \text{ ( since } e^{\frac{\gamma}{2}c }Y_ke^{-\frac{\mu_\partial}{2} e^{\frac{\gamma}{2}c }Y_k}\leq C)\\
   \leq & C e^{\frac{\epsilon}{2}c}\E_{\tilde{\varphi}}\Big[(\delta Y_k)^{\eta}\Big]\E_{\tilde{\varphi}}\Big[e^{-\frac{\mu_\partial}{2} e^{\frac{\gamma}{2}c }Y_k}\Big] \text{(here $\eta\in (0,1)$ and use FKG inequality)}
\end{align*}
Now we choose $\eta$ as in \cite[Proposition C.1]{GKRV}, we prove the locally uniform convergence.
\end{proof}
From the above lemma, we know $\beta\to\Big(\partial_{\mu_{\partial}}\mathcal{A}^{m}_{g_{\mathbb{A}},\mathbb{A}^1_{\sqrt{q}},b,\beta}\Big)(\Psi_{Q+iP,\nu,\tilde{\nu}})$ is analytic in a complex neighbourhood of  $0$. 
 So the LHS of \eqref{derivative} has the same $\beta\to 0$ limit with RHS, which is 
$-\int_{\T}\mathcal{A}^{m}_{g_{\mathbb{A}},\mathbb{A}^1_{\sqrt{q}},b,\gamma}(\Psi_{Q+iP,\nu,\tilde{\nu}})\dd b$.

Now we analysis the LHS, by Ward's identity \eqref{diskdescendant} and take the derivative of $\mu_\partial$
$$LHS=C_{\D}\partial_{\mu_\partial}G(Q+iP,\beta)w_{\mathbb{A}}(\Delta_{\beta},\Delta_{Q+iP},
 \nu,\tilde{\nu} )|q |^{-{c}_{\rm L}/24+\Delta_{Q+ip}}  (\frac{\sqrt{q}}{b})^{|\nu|} (\sqrt{q}b)^{|\tilde{\nu}|}$$
In the case of $\nu=\tilde{\nu}=\emptyset$, we have $\lim_{\beta\to 0}\partial_{\mu_\partial}\Big(\caA^m_{g_{\A},\A^1_{\sqrt{q}},b,\beta}(\Psi_{Q+iP})\Big)=\lim_{\beta\to 0}(\partial_{\mu_\partial}\caA^m_{g_{\A},\A^1_{\sqrt{q}},b,\beta})(\Psi_{Q+iP})=-2\pi\caA^m_{g_{\A},\A^1_{\sqrt{q}},b,\gamma}(\Psi_{Q+iP})$, so $\lim_{\beta\to0}\partial_{\mu_\partial}G(Q+iP,\beta)=-2\pi G(Q+iP,\gamma)$. 
 
 We know $\lim_{\beta\to 0}w_{\mathbb{A}}(\Delta_{\beta},\Delta_{Q+iP},
 \nu,\tilde{\nu} )$ exists since it's a polynomial of $\Delta_{\beta}$,
 compare it with 
 \begin{align}\label{anotherproof}
\int_{\T}\mathcal{A}^{m}_{g_{\mathbb{A}},\mathbb{A}^1_{\sqrt{q}},b,\gamma}(\Psi_{Q+iP,\nu,\tilde{\nu}})\dd b=&C_{\D}G(Q+iP,\gamma)w_{\mathbb{A}}(\Delta_{\gamma},\Delta_{Q+iP},
 \nu,\tilde{\nu} )|q |^{-{c}_{\rm L}/24+\Delta_{Q+ip}} \int_{\T} (\frac{\sqrt{q}}{b})^{|\nu|} (\sqrt{q}b)^{|\tilde{\nu}|} \dd b\nonumber\\
 =&C_{\D}G(Q+iP,\gamma)w_{\mathbb{A}}(\Delta_{\gamma},\Delta_{Q+iP},
 \nu,\tilde{\nu} )|q |^{-{c}_{\rm L}/24+\Delta_{Q+ip}} q^{|\nu|}\mathbf{1}_{|\nu|=|\tilde{\nu}|}
 \end{align}
 change to the orthogonal basis, we know when $|\nu|=|\tilde{\nu}|$ $$\mathcal{W}_{\mathbb{A}}(\Delta_{\gamma},\Delta_{Q+iP},
 \nu,\tilde{\nu} )=\lim_{\beta\to 0}\mathcal{W}_{\mathbb{A}}(\Delta_{\beta},\Delta_{Q+iP},
 \nu,\tilde{\nu} )=\mathbf{1}.$$
 So we get an explicit formula for torus 1-point block with weight $\gamma$.
 \begin{proposition}\label{torus 1-point block}
  \begin{align}
     \mathcal{F}(\Delta_{\gamma},\Delta_{Q+iP},q)=\sum^{\infty}_{n=0}\Big(\sum_{|\nu|=n}\mathcal{W}_{\mathbb{A}}(\Delta_{\gamma},\Delta_{Q+iP},
 \nu,\nu )\Big)(q^2)^n=\sum^{\infty}_{n=0}P(n)(q^2)^n=\frac{q^{\frac{1}{12}}}{\eta(q^2)}.
 \end{align}

 \end{proposition}

 In the last equality, we use $$\frac{q^{\frac{1}{24}}}{\eta(q)}=
\sum_{n=0}^{\infty} P(n) q^{n} =\left(1+q+q^{2}+\ldots\right)\left(1+q^{2}+q^{4}+\ldots\right)\left(1+q^{3}+\ldots\right) \cdots 
=\prod_{n=1}^{\infty} \frac{1}{\left(1-q^{n}\right)}$$
with this observation, we can also get a formula for annulus 2-point block with both boundary components weighted $\gamma$.
\begin{remark}
For the case $|\nu|\neq |\tilde{\nu}|$, we get another proof of $\lim_{\beta\to 0}\mathcal{W}_{\mathbb{A}}(\Delta_{\beta},\Delta_{Q+iP},
 \nu,\tilde{\nu} )=0 $ from the formula \eqref{anotherproof}. But we don't know the value of $\mathcal{W}_{\mathbb{A}}(\Delta_{\gamma},\Delta_{Q+iP},
 \nu,\tilde{\nu} )$ when $|\nu|\neq |\tilde{\nu}|$.
\end{remark}
\subsection{Proof of theorem \texorpdfstring{\eqref{gammainsertion}}{gamma}}
We start from a lemma, which confirms the $\mu_\partial$ derivative of FZZ formula coincides with our gluing construction of bulk-boundary correlator with boundary insertion $\gamma$.
\begin{lemma}\label{relation of UG}
For $\mu\geq 0$ and $\mu_\partial>0$, we have $\partial_{\mu_\partial} U_{\rm FZZ}(Q+iP)=-2\pi G(Q+iP,\gamma)$.
\end{lemma}
\begin{proof}
We define two maps $F_1:\alpha\to\partial_{\mu_\partial}U_{FZZ}(\alpha)$ and $F_2:\alpha\to -2\pi G(\alpha,\gamma)$. We know $F_1=F_2$ on $\alpha\in(Q-\frac{\gamma}{2},Q) $ and both $F_1$ and $F_2$ are analytic in a connected component contains both $\{\alpha=Q+i\R_+\}$ and $\{\alpha\in(Q-\frac{\gamma}{2},Q)\}$ (the claim on analyticity of $F_1$ follows from the explicit FZZ formula).
\end{proof} 
\noindent when we put different boundary cosmology constants $\mu_{B_1}$, $\mu_{B_2}$ on $\partial\mathbb{A}^{out}$ and $\partial\mathbb{A}^{in}$ respectively, we have
\begin{proposition}
\label{proof of gamma insertion}
For $\mu\geq 0$ and $\mu_{\partial}>0$
$$\partial_{\mu_{B_1}} \partial_{\mu_{B_2}}\langle 1\rangle^{\mathbb{A}^1_{q}}_{g_{\mathbb{A}},\mu, \mu_{\partial}, }=\frac{C_{\D}^2}{2\sqrt{\pi}}\int_{\R_{+}}\partial_{\mu_{B_1}}U_{\rm FZZ}^{\mu_{B_1}}(Q+iP)\partial_{\mu_{B_2}}U_{\rm FZZ}^{\mu_{B_2}}(Q-iP)\frac{q^{\frac{P^2}{2}}}{\eta(q^2)} dP $$
\end{proposition}
\begin{proof}
We write the LHS as \begin{align*}
   \int  \Big(\int_{\T}\mathcal{A}^{m,\mu_{B_1}}_{g_{\mathbb{A}},{\mathbb{A}}^1_{\sqrt{q}},s,\gamma}(\tilde{\bm{\varphi}})\dd s\Big)\times \Big(\int_{\T}\mathcal{A}^{m,\mu_{B_2}}_{g_{\mathbb{A}},{\mathbb{A}}^1_{\sqrt{q}},s,\gamma}(\tilde{\bm{\varphi}})\dd s\Big)\dd\mu_0(\tilde{\bm{\varphi}})
\end{align*}
By \eqref{nicefunction}, we know $\int_{\T}\mathcal{A}^{m,\mu_{B_2}}_{g_{\mathbb{A}},{\mathbb{A}}^1_{\sqrt{q}},s,\gamma}(\tilde{\bm{\varphi}})\dd s\in e^{\epsilon c_-}L^2$ for $\epsilon>0$.
Then we apply the spectrum resolution and use the previous lemma, the conformal block is given by
  \begin{align}
     \sum^{\infty}_{n=0}\Big(\sum_{|\nu|=|\tilde{\nu}|=n}\mathcal{W}_{\mathbb{A}}(\Delta_{\gamma},\Delta_{Q+iP},
 \nu,\tilde{\nu })\mathcal{W}_{\mathbb{A}}(\Delta_{\gamma},\Delta_{Q+iP},\nu,
 \tilde{\nu} )\Big)(q^2)^n=\sum^{\infty}_{n=0}P(n)(q^2)^n=\frac{q^{\frac{1}{12}}}{\eta(q^2)}.
 \end{align}
\end{proof}


In \cite{GRSS}, the authors expresses the torus 1-point conformal block as a GMC integral to analyse the analytic property of conformal block. More precisely, for $\alpha \in\left(-\frac{4}{\gamma}, Q\right), q \in(0,1)$, and $P \in \mathbb{R}$, they define the probabilistic 1-point toric conformal block with $\alpha$-weight insertion by
$$
\mathcal{G}_{\gamma, P}^{\alpha}(q):=\frac{1}{Z} \mathbb{E}\left[\left(\int_{0}^{1} \Theta_{\tau}(x)^{-\frac{\alpha \gamma}{2}} e^{\pi \gamma P x} e^{\frac{\gamma}{2} Y_{\tau}(x)} d x\right)^{-\frac{\alpha}{\gamma}}\right]
$$
Where $Z$ is a renormalisation constant s.t. $\lim _{q \rightarrow 0} \mathcal{G}_{\gamma, P}^{\alpha}(q)=1$ and $\lim _{P \rightarrow+\infty} \mathcal{G}_{\gamma, P}^{\alpha}(q)=1$. We won't dive into the definition of these terms and refer interested readers to the original paper. The main result \cite{GRSS} is that the above satisfies the Zamolodchikov's recursion, this also algebraically implies the GMC construction of torus 1-point conformal block coincides with the conformal block defined in \cite{GKRV21}. For our interest, we take $\alpha=\gamma$, then we get $\frac{1}{Z} \mathbb{E}\left[\left(\int_{0}^{1} \Theta_{\tau}(x)^{-\frac{ \gamma^2}{2}} e^{\pi \gamma P x} e^{\frac{\gamma}{2} Y_{\tau}(x)} d x\right)^{-1}\right]$.
In this section we actually provided a nice formula about the GMC integral.
\begin{proposition}
\begin{align}
    \frac{1}{Z} \mathbb{E}\left[\left(\int_{0}^{1} \Theta_{\tau}(x)^{-\frac{ \gamma^2}{2}} e^{\pi \gamma P x} e^{\frac{\gamma}{2} Y_{\tau}(x)} d x\right)^{-1}\right]=\frac{q^{\frac{1}{12}}}{\eta(q^2)}
\end{align}
\end{proposition}
The left hand side doesn't depend on $P$ and $\gamma$, one interesting question is to understand this phenomenon by some path decomposition like in \cite{CHY}.

\section{Proof of the one-point conformal bootstrap formula}
\subsection{A weaker version of conformal bootstrap formula}

We notice that marking one boundary insertion is enough to make sure the correlation function converges. For $q\in\partial^{in}\mathbb{A}_q^1$ with $\beta$-weight insertion, if we still start from $$\mathcal{A}^N_{g_{\mathbb{A}},\mathbb{A}^1_{q},q,\beta}=\frac{1}{2\sqrt{\pi}}\int  \mathcal{A}^m_{g_{\mathbb{A}},{\mathbb{A}}^1_{\sqrt{q}}}(\tilde{\bm{\varphi}})\times \mathcal{A}^m_{g_{\mathbb{A}},{\mathbb{A}}^{\sqrt{q}}_q ,q,\beta}(\tilde{\bm{\varphi}})\dd\mu_0(\tilde{\bm{\varphi}})$$
Then due to $\mathcal{A}^m_{g_{\mathbb{A}},{\mathbb{A}}_{\sqrt{q}}^1 }(\tilde{\bm{\varphi}})$ is not in $L^2(\R\times\Omega_{\T})$, the standard Plancherel type formula no longer holds, and we can not define the conformal block as the Fourier transform of the amplitudes any more. In this section, we first prove a weaker version of conformal bootstrap formula 
\begin{proposition}
\label{weaker bootstrap}
For $\mu\geq 0$, $\mu_{B_1},\mu_{B_2}>0$,  $q\in\partial^{in}\mathbb{A}^1_{q}$ and $\beta\in(0,Q)$, we have
$$\partial_{\mu_{B_1}}\langle V_{\frac{\beta}{2}}(q)\rangle^{\mathbb{A}^1_q}_{g_{\mathbb{A}}}=\frac{C_{\D}^2}{2\sqrt{\pi}}q^{-\frac{1}{12}}\int_{\R^{+}}\partial_{\mu_{B_1}}U^{\mu_{B_1}}_{\rm FZZ}(Q+iP)G^{\mu_{B_2}}(Q-iP,\beta)q^{\frac{P^2}{2}}\mathcal{F}^{\T}(\Delta_{\beta},\Delta_{Q+iP},q^2)dP$$
where 
\end{proposition}

\begin{proof}
First we differentiate $\mu_{B_1}$ on both side of 
$$\mathcal{A}^N_{g_{\mathbb{A}},\mathbb{A}^1_{q},q,\beta}=\frac{1}{2\sqrt{\pi}}\int  \mathcal{A}^m_{g_{\mathbb{A}},{\mathbb{A}}^1_{\sqrt{q}}}(\tilde{\bm{\varphi}})\times \mathcal{A}^m_{g_{\mathbb{A}},{\mathbb{A}}^{\sqrt{q}}_q ,q,\beta}(\tilde{\bm{\varphi}})\dd\mu_0(\tilde{\bm{\varphi}})$$
then by dominate convergence
$$\partial_{\mu_{B_1}}\mathcal{A}^N_{g_{\mathbb{A}},\mathbb{A}^1_{q},q,\beta}=-\frac{1}{2\sqrt{\pi}}\int  \Big(\int_{\T}\mathcal{A}^m_{g_{\mathbb{A}},{\mathbb{A}}^1_{\sqrt{q}},s,\gamma}(\tilde{\bm{\varphi}})\dd s\Big)\times \mathcal{A}^m_{g_{\mathbb{A}},{\mathbb{A}}^{\sqrt{q}}_q ,q,\beta}(\tilde{\bm{\varphi}})\dd\mu_0(\tilde{\bm{\varphi}})$$
Since $\int_{\T}\mathcal{A}^m_{g_{\mathbb{A}},{\mathbb{A}}^1_{\sqrt{q}},s,\gamma}(\tilde{\bm{\varphi}})\dd s\in e^{\epsilon c_-}L^2$ for some $\epsilon >0$, we can apply the spectrum resolution and use the disk descendant \eqref{diskdescendant}
$$\sqrt{\pi}\partial_{\mu_{B_1}}\mathcal{A}^N_{g_{\mathbb{A}},\mathbb{A}^1_{q},q,\beta}=-\int_{\R_+}2\pi  G^{\mu_{B_1}}(Q+iP,\gamma)G^{\mu_{B_2}}(Q-iP,\beta)q^{\frac{P^2}{2}}\mathcal{F}^w(\Delta_{\gamma},\Delta_{\beta},\Delta_{Q+iP},q)\dd P$$
where 
\begin{align}\label{annulus2pointblock}
   \mathcal{F}^w(\Delta_{\gamma},\Delta_{\beta},\Delta_{Q+iP},q):=\sum^{\infty}_{n=0}\Big(\sum_{|\nu|=|\tilde{\nu}|=n}\mathcal{W}_{\mathbb{A}}(\Delta_{\beta},\Delta_{Q+iP},
 \nu,\tilde{\nu} )\mathcal{W}_{\mathbb{A}}(\Delta_{\gamma},\Delta_{Q+iP},
 \nu,\tilde{\nu} )\Big)(q)^{2n}
\end{align}
We use $\mathcal{W}_{\mathbb{T}}(\Delta_{\gamma},\Delta_{Q+iP},\nu,\tilde{\nu} )= 1$ when $|\nu|=|\tilde{\nu}|$, so 
 $$\mathcal{F}^w(\Delta_{\gamma},\Delta_{\beta},\Delta_{Q+iP},q)=\mathcal{F}^{\T}(\Delta_{\beta},\Delta_{Q+iP},q^2)$$
Then we identify $ \partial_{\mu_\partial} U_{\rm FZZ}(Q+iP)=-2\pi G(Q+iP,\gamma)$ by lemma \eqref{relation of UG}, which finishes the proof.
\end{proof}

\subsection{Proof of the theorem \texorpdfstring{\eqref{one point bootstrap}}{one point}}\label{proof of 12}

In this subsection, we will remove the $\partial_{\mu_{B_1}}$ of both sides in the weaker bootstrap formula \eqref{weaker bootstrap}, and finish the proof of \eqref{one point bootstrap}.

First we recall the explicit FZZ formula \eqref{fzz}, We have for $\mu^*_{B_1}>\delta>0$
$$\langle V_{\frac{\beta}{2}}(q)\rangle^{\mathbb{A}^1_q}_{g_{\mathbb{A}},\mu^*_{B_1},\mu_{B_2}}=-\int^{\infty}_{\mu^*_{B_1}}\partial_{\mu_{B_1}}\langle V_{\frac{\beta}{2}}(q)\rangle^{\mathbb{A}^1_q}_{g_{\mathbb{A}}} \dd\mu_{B_1}=-\lim_{L\to\infty}\int^{L}_{\mu^*_{B_1}}\partial_{\mu_{B_1}}\langle V_{\frac{\beta}{2}}(q)\rangle^{\mathbb{A}^1_q}_{g_{\mathbb{A}}} \dd\mu_{B_1}$$
To exchange the order of $\mu_{B_1}$ integral and $P$ integral, we need the following proposition
\begin{proposition}
We have $\int_{\R_{+}}\int^{L}_{\mu^*_{B_1}}|\partial_{\mu_{B_1}}U^{\mu_{B_1}}_{\rm FZZ}(Q+iP)G^{\mu_{B_2}}(Q-iP,\beta)q^{\frac{P^2}{2}}\mathcal{F}^{w}(\Delta_{\gamma},\Delta_{\beta},\Delta_{Q+iP},q)|\dd\mu_{B_1}\dd P\leq C$
\end{proposition}
\begin{proof}
Direct computation shows (here we define $A(P)=:\frac{4}{\gamma} \left({\pi \mu} \frac{\Gamma\left(\frac{\gamma^{2}}{4}\right)}{\Gamma\left(1-\frac{\gamma^{2}}{4}\right)}\right)^{-\frac{iP}{\gamma}} \Gamma\left(1+i\frac{\gamma P}{2}\right) \Gamma\left(i\frac{2P}{\gamma}\right) $)
$$U^{\mu_{B_1}}_{\rm FZZ}(Q+iP)=A(P) \cos (iP \pi s)$$ and $$\partial_{\mu_{B_1}}U^{\mu_{B_1}}_{\rm FZZ}(Q+iP)=A(P) iP\pi\sin (iP \pi s)\frac{\frac{1}{\sqrt{\mu}}\sqrt{\sin \frac{\pi \gamma^{2}}{4}}}{\frac{\gamma\pi}{2}\sqrt{1-\frac{\mu_{B_1}^2}{\mu}\sin \frac{\pi \gamma^{2}}{4}}}=A(P) iP\pi\sin (iP \pi s)\frac{\frac{1}{\sqrt{\mu}}\sqrt{\sin \frac{\pi \gamma^{2}}{4}}}{\frac{\gamma\pi}{2}\sin(\frac{\pi\gamma s}{2})}$$
Then it's easy to see when $s\to 0$, $s\to\partial_{\mu_{B_1}}U^{\mu_{B_1}}_{\rm FZZ}(Q+iP)$ has no singularity. Without loss of generality, we can suppose $\frac{{\mu^*_{B_1}}^{2}}{\mu} \sin \frac{\pi \gamma^{2}}{4} < 1 <\frac{L^{2}}{\mu} \sin \frac{\pi \gamma^{2}}{4} $. So 
\begin{align*}
    \int^{L}_{\mu^*_{B_1}}|\partial_{\mu_{B_1}}U^{\mu_{B_1}}_{\rm FZZ}(Q+iP)|\dd\mu_{B_1}=&\int^{\sqrt{\frac{\mu}{\sin \frac{\pi \gamma^{2}}{4}}}}_{\mu^*_{B_1}}|\partial_{\mu_{B_1}}U^{\mu_{B_1}}_{\rm FZZ}(Q+iP)|\dd\mu_{B_1}+\int_{\sqrt{\frac{\mu}{\sin \frac{\pi \gamma^{2}}{4}}}}^L |\partial_{\mu_{B_1}}U^{\mu_{B_1}}_{\rm FZZ}(Q+iP)|\dd\mu_{B_1}\\&\leq C_1 |A(P)|e^{C_2 P}+C_3|A(P)|\ln(L)
\end{align*} 
For universal constants $C_1, C_2, C_3>0$.
Now we use the estimation of Gamma function \eqref{estimate gamma}, we see that for fixed $L$,
this term grows like $e^{C P}$ when $P$ large enough. Later we will see we can remove an exponential decay $q_0^{\frac{P^2}{2}}$ from the conformal block to control this. We also notice that $A(P)$ has a simple pole at $P=0$. We need to define truncation functions $f(P)$ and $g(P)$ as following
$$f(P) =\left\{\begin{array}{lll}
P & \text { when } & P\in (0,1] \\
1 & \text { when} & P\in (1,+\infty)
\end{array}\right.$$ and
$$g(P) =\left\{\begin{array}{lll}
\frac{1}{P}& \text { when } & P\in (0,1] \\
1 & \text { when} & P\in (1,+\infty)
\end{array}\right.$$
By using $f(P)g(P)=1$ and Cauchy-Schwartz inequality we have
\begin{align*}
    &\int_{\R_{+}}\int^{L}_{\mu^*_{B_1}}|\partial_{\mu_{B_1}}U^{\mu_{B_1}}_{\rm FZZ}(Q+iP)|\dd\mu_{B_1}\times |G^{\mu_{B_2}}(Q-iP,\beta)q^{\frac{P^2}{2}}\mathcal{F}^{w}(\Delta_{\gamma},\Delta_{\beta},\Delta_{Q+iP},q)|\dd P\\
    =&\int_{\R_{+}}\int^{L}_{\mu^*_{B_1}}|\partial_{\mu_{B_1}}U^{\mu_{B_1}}_{\rm FZZ}(Q+iP)|\dd\mu_{B_1} |G^{\mu_{B_2}}(Q-iP,\beta)\langle \int_\T\mathcal{B}^m_{g_{\mathbb{A}},\mathbb{A}^{1}_{\sqrt{q}},b_1,\gamma}(P,\cdot,\cdot)\dd b_1\mid  \overline{\mathcal{B}^m_{g_{\mathbb{A}},\mathbb{A}^{1}_{\sqrt{q}},b_2,\beta}}(P,\cdot,\cdot)\rangle_{\rm HS}|\dd P\\\leq&\int_{\R_{+}}\int^{L}_{\mu^*_{B_1}}|\partial_{\mu_{B_1}}U^{\mu_{B_1}}_{\rm FZZ}(Q+iP)|\dd\mu_{B_1} |G^{\mu_{B_2}}(Q-iP,\beta)|| \int_\T\mathcal{B}^m_{g_{\mathbb{A}},\mathbb{A}^{1}_{\sqrt{q}},b_1,\gamma}(P,\cdot,\cdot)\dd b_1||_{\rm HS} ||\mathcal{B}^m_{g_{\mathbb{A}},\mathbb{A}^{1}_{\sqrt{q}},b_2,\beta}(P,\cdot,\cdot)||_{\rm HS}\dd P\\ \leq&\Big(\int_{\R_{+}}|\int^{L}_{\mu^*_{B_1}}|f(P)\partial_{\mu_{B_1}}U^{\mu_{B_1}}_{\rm FZZ}(Q+iP)|\dd\mu_{B_1}|^2 || \int_\T\mathcal{B}^m_{g_{\mathbb{A}},\mathbb{A}^{1}_{\sqrt{q}},b_1,\gamma}(P,\cdot,\cdot)\dd b_1||^2_{\rm HS} \dd P\Big)^{\frac{1}{2}}\times\\&\Big(\int_{\R_{+}} |g(P)G^{\mu_{B_2}}(Q-iP,\beta)|^2||\mathcal{B}^m_{g_{\mathbb{A}},\mathbb{A}^{1}_{\sqrt{q}},b_2,\beta}(P,\cdot,\cdot)||^2_{\rm HS}\dd P\Big)^{\frac{1}{2}}:=(I_{\gamma})^{\frac{1}{2}}\times (I_{\beta})^{\frac{1}{2}}
\end{align*}
The fact $I_{\gamma}<\infty$ follows from the above estimation and the explicit formula of $ ||\int_\T \mathcal{B}^m_{g_{\mathbb{A}},\mathbb{A}^{1}_{\sqrt{q}},b_1,\gamma}(P,\cdot,\cdot)\dd b_1||^2_{\rm HS}=q^{\frac{P^2}{2}}\frac{q^{\frac{1}{12}}}{\eta(q)} $. Now we prove $I_{\beta}$ is also finite
\begin{lemma}
 For fixed $\beta\in(0,Q)$, $I_{\beta}<\infty$.
\end{lemma}
\begin{proof}
We separate $I_{\beta}$ into two parts, the first part contains $P$ integral from $0$ to $1$ and the second part contains $P$ integral from $1$ to $\infty$, note as $I_{\beta}^{(1)}$ and $I_{\beta}^{(2)}$ respectively.

We know $I_{\beta}^{(2)}<\int_{\R_{+}} |G^{\mu_{B_2}}(Q-iP,\beta)|^2||\mathcal{B}^m_{g_{\mathbb{A}},\mathbb{A}^{1}_{\sqrt{q}},b_2,\beta}(P,\cdot,\cdot)||^2_{\rm HS}\dd P=||\mathcal{A}_{g_{\mathbb{A}},\mathbb{A}^{1}_{\sqrt{q}},b_2,\beta}||_{L^2}\leq C$.

For $I_{\beta}^{(1)}$, first we know $|g(P)G^{\mu_{B_2}}(Q-iP,\beta)|\leq C$ by proposition \eqref{bbestimate1}. We know the torus one-point block is bounded near $P=0$ by \cite{GRSS}, then we can deduce
$$ ||\mathcal{B}^m_{g_{\mathbb{A}},\mathbb{A}^{1}_{\sqrt{q}},b_2,\beta}(P,\cdot,\cdot)||^2_{\rm HS}\leq ||\mathcal{B}^D_{g_{\mathbb{A}},\mathbb{A}^{\frac{1}{\sqrt{q}}}_{\sqrt{q}},b_2,\beta}(P,P,\cdot,\cdot)||^2_{\rm tr}\leq C$$
so we conclude $I_{\beta}^{(1)}$ is finite.
\end{proof}

Now we can use the Fubini-Tonelli theorem, recall that
$$\mathcal{F}^w(\Delta_{\gamma},\Delta_{\beta},\Delta_{Q+iP},q)=\mathcal{F}^{\T}(\Delta_{\beta},\Delta_{Q+iP},q^2)$$
\begin{align*}
    &\int^{L}_{\mu^*_{B_1}}\int_{\R_{+}}\partial_{\mu_{B_1}}U^{\mu_{B_1}}_{\rm FZZ}(Q+iP)G^{\mu_{B_2}}(Q-iP,\beta)q^{\frac{P^2}{2}}\mathcal{F}^{\T}(\Delta_{\beta},\Delta_{Q+iP},q^2)\dd P\dd\mu_{B_1}\\=&\int_{\R_{+}}\int^{L}_{\mu^*_{B_1}}\partial_{\mu_{B_1}}U^{\mu_{B_1}}_{\rm FZZ}(Q+iP)G^{\mu_{B_2}}(Q-iP,\beta)q^{\frac{P^2}{2}}\mathcal{F}^{\T}(\Delta_{\beta},\Delta_{Q+iP},q^2)\dd\mu_{B_1}\dd P\\=&\int_{\R_{+}}(U^{L}_{\rm FZZ}(Q+iP)-U^{\mu_{B_1*}}_{\rm FZZ}(Q+iP))G^{\mu_{B_2}}(Q-iP,\beta)q^{\frac{P^2}{2}}\mathcal{F}^{\T}(\Delta_{\beta},\Delta_{Q+iP},q^2)\dd P
\end{align*}
Note $S(P)=A(P)G^{\mu_{B_2}}(Q-iP,\beta)q^{\frac{P^2}{2}}\mathcal{F}^{\T}(\Delta_{\beta},\Delta_{Q+iP},q^2)$, as above we can prove $S(P)\in L^1(\R_{+})$, so by Riemann–Lebesgue lemma, we have 
$$\lim_{L\to \infty} \int_{\R_{+}}cos(iPs)S(P)\dd P=0$$
since when $L\to\infty$, $s$ goes to imaginary infinite.
\end{proof}
\section{The bosonic LQG partition function of the annulus}\label{annulus lqg}
In this section, we give an exact formula of LQG partition function in annulus topology as an application of previous result. Our goal is to integrate over the moduli parameter $q=e^{-2\pi\tau}\in(0,1)$ the partition function of LCFT coupling with bosonic matter field and ghost field, the  partition function of LQG is defined by (for a review of this part, we refer readers to \cite{GRV19})
$$
\mathcal{Z}^{bos}_{L Q G,\gamma}(\mathbb{A})=\int_{0}^{\infty}  \mathcal{Z}_{\text {Ghost }}(q,g_{\mathbb{A}}) \mathcal{Z}^{bos}_{\text {Matter }}(q,g_{\mathbb{A}}) \mathcal{Z}_{L C F T,\gamma}(q,g_{\mathbb{A}})d \tau
$$
where we have following expressions for the different partition functions,
$$
\begin{aligned}
&\mathcal{Z}_{G F F}(q,g_{\mathbb{A}})=\frac{1}{\sqrt{2}\eta(q^2)} \\
&\mathcal{Z}^{bos}_{\text {Matter }}(q,g_{\mathbb{A}})=\left(\mathcal{Z}_{G F F}\right)^{c_{m}}=(\frac{1}{\sqrt{2}\eta(q^2)})^{c_{m}} \\
&\mathcal{Z}_{\text {Ghost }}(q,g_{\mathbb{A}})=\eta(q^2)^2\\
&\mathcal{Z}_{L C F T,\gamma}(q,g_{\mathbb{A}})=\frac{C_{\D}^2}{2\sqrt{\pi}} \int_{\R^{+}}\partial_{\mu_{B_1}}U_{\rm FZZ}(Q+iP)\partial_{\mu_{B_2}}U_{\rm FZZ}(Q-iP)\frac{q^{\frac{P^2}{2}}}{\eta(q^2)} dP,
\end{aligned}
$$

The real constant $c_{m}$ is the central charge of the matter fields. It is linked to $Q$ by the following relation:
$$
c_{m}-25+6 Q^{2}=0
$$
\begin{proposition}\label{LQG partition}
For $\gamma\in(0,2)$, the partition function of the bosonic Liouville quantum gravity of annulus is given by 
$$
\mathcal{Z}^{bos}_{LQ G,\gamma}(\mathbb{A})=\frac{C_{\D}^2C_{\gamma}}{2^{\frac{c_m+3}{2}}\sqrt{\pi}}\int_{0}^{\infty}  \frac{P^2 sinh(\pi Ps_1)sinh(\pi Ps_2)}{sinh(\pi\frac{\gamma P}{2})sinh(\pi\frac{2P}{\gamma})}(\int_{0}^{\infty}q^{\frac{P^2}{2}}{\eta(q^2)^{6Q^2-24}}d\tau )dP 
$$
and $C_{\gamma}=\frac{16\pi^2 \sin(\frac{\pi\gamma^2}{4})}{\mu\gamma^2sin(\frac{\pi\gamma s_1}{2})\sin(\frac{\pi\gamma s_2}{2})}$
\end{proposition}

\begin{remark}
The $\gamma=2$ case corresponds to the $c_m=1$ bosonic string theory. When $\gamma=2$ the correlation functions are $0$, so we need an extra $\sqrt{\ln(\frac{1}{\epsilon})}$ multiplier to renormalize the GMC, see \cite{GRV19} for closed Riemann surfaces and \cite{Wu1} for Riemann surfaces with boundaries. One may wonder do we still have the bootstrap equation.
It is  possible to prove Segal's axiom for $\gamma=2$. The behavior of bosonic LQG partition function deeply relies on the analyticity of the conformal block on the spectrum parameter $P$, but this kind of analyticity result is not known except the torus 1-point block case in \cite{GRSS} and our result about annulus 2-point block \eqref{continuous2point}.
\end{remark}

\appendix
\section*{Appendix}

\section{Special functions}
{\bf The double gamma function}\\
We will now provide some explanations on the functions $\Gamma_{\frac{\gamma}{2}}(x)$ and $S_{\frac{\gamma}{2}}(x)$ that we have introduced. For all $\gamma \in (0,2) $ and for $\mathrm{Re}(x) >0$, $\Gamma_{\frac{\gamma}{2}}(x)$ is defined by the integral formula,
\begin{equation}
\ln \Gamma_{\frac{\gamma}{2}}(x) = \int_0^{\infty} \frac{dt}{t} \left[ \frac{ e^{-xt} -e^{- \frac{Qt}{2}}   }{(1 - e^{- \frac{\gamma t}{2}})(1 - e^{- \frac{2t}{\gamma}})} - \frac{( \frac{Q}{2} -x)^2 }{2}e^{-t} + \frac{ x -\frac{Q}{2}  }{t} \right],
\end{equation} 
where we have $Q = \frac{\gamma}{2} +\frac{2}{\gamma}$. Since the function $\Gamma_{\frac{\gamma}{2}}(x)$ is continuous it is completely determined by the following two shift equations
\begin{align}
\frac{\Gamma_{\frac{\gamma}{2}}(x)}{\Gamma_{\frac{\gamma}{2}}(x + \frac{\gamma}{2}) }&= \frac{1}{\sqrt{2 \pi}}
\Gamma(\frac{\gamma x}{2}) ( \frac{\gamma}{2} )^{ -\frac{\gamma x}{2} + \frac{1}{2}
}, \\
\frac{\Gamma_{\frac{\gamma}{2}}(x)}{\Gamma_{\frac{\gamma}{2}}(x + \frac{2}{\gamma}) }&= \frac{1}{\sqrt{2 \pi}} \Gamma(\frac{2
x}{\gamma}) ( \frac{\gamma}{2} )^{ \frac{2 x}{\gamma} - \frac{1}{2} },
\end{align}
and by its value in $\frac{Q}{2}$, $\Gamma_{\frac{\gamma}{2}}(\frac{Q}{2} ) =1$. Furthermore $x \mapsto \Gamma_{\frac{\gamma}{2}}(x)$ admits a meromorphic extension to all of $\mathbb{C}$ with single poles at $x = -n\frac{\gamma}{2}-m\frac{2}{\gamma}$  for any $n,m \in \mathbb{N}$ and $\Gamma_{\frac{\gamma}{2}}(x)$ is never equal to $0$.  \\

{\bf The DOZZ formula}\\
We set $l(z)=\frac{\Gamma (z)}{\Gamma (1-z)}$ where $\Gamma$ denotes the standard Gamma function. We introduce Zamolodchikov's special holomorphic function $\Upsilon_{\frac{\gamma}{2}}(z)$ by the following expression for $0<\Re (z)< Q$
\begin{equation}\label{def:upsilon}
\ln \Upsilon_{\frac{\gamma}{2}} (z)  = \int_{0}^\infty  \left ( \Big (\frac{Q}{2}-z \Big )^2  e^{-t}-  \frac{( \sinh( (\frac{Q}{2}-z )\frac{t}{2}  )   )^2}{\sinh (\frac{t \gamma}{4}) \sinh( \frac{t}{\gamma} )}    \right ) \frac{dt}{t}.
\end{equation}
The function $\Upsilon_{\frac{\gamma}{2}}$  is then defined on all $\C$ by analytic continuation of the expression \eqref{def:upsilon} as expression \eqref{def:upsilon}  satisfies the following remarkable functional relations: 
\begin{equation}\label{shiftUpsilon}
\Upsilon_{\frac{\gamma}{2}} (z+\frac{\gamma}{2}) = l( \frac{\gamma}{2}z) (\frac{\gamma}{2})^{1-\gamma z}\Upsilon_{\frac{\gamma}{2}} (z), \quad
\Upsilon_{\frac{\gamma}{2}} (z+\frac{2}{\gamma}) = l(\frac{2}{\gamma}z) (\frac{\gamma}{2})^{\frac{4}{\gamma} z-1} \Upsilon_{\frac{\gamma}{2}} (z).
\end{equation}
The function $\Upsilon_{\frac{\gamma}{2}}$ has no poles in $\C$ and the zeros of $\Upsilon_{\frac{\gamma}{2}}$ are simple (if $\gamma^2 \not \in \Q$) and given by the discrete set $(-\frac{\gamma}{2} \N-\frac{2}{\gamma} \N) \cup (Q+\frac{\gamma}{2} \N+\frac{2}{\gamma} \N )$. 
With these notations, the DOZZ formula is defined for $\alpha_1,\alpha_2,\alpha_3 \in \C$ by the following formula where we set $\bar{\alpha}=\alpha_1+\alpha_2+\alpha_3$ 
\begin{equation}\label{theDOZZformula}
C_{\gamma,\mu}^{{\rm DOZZ}} (\alpha_1,\alpha_2,\alpha_3 )  = (\pi \: \mu \:  l(\frac{\gamma^2}{4})  \: (\frac{\gamma}{2})^{2 -\gamma^2/2} )^{\frac{2 Q -\bar{\alpha}}{\gamma}}   \frac{\Upsilon_{\frac{\gamma}{2}}'(0)\Upsilon_{\frac{\gamma}{2}}(\alpha_1) \Upsilon_{\frac{\gamma}{2}}(\alpha_2) \Upsilon_{\frac{\gamma}{2}}(\alpha_3)}{\Upsilon_{\frac{\gamma}{2}}(\frac{\bar{\alpha}}{2}-Q) 
\Upsilon_{\frac{\gamma}{2}}(\frac{\bar{\alpha}}{2}-\alpha_1) \Upsilon_{\frac{\gamma}{2}}(\frac{\bar{\alpha}}{2}-\alpha_2) \Upsilon_{\frac{\gamma}{2}}(\frac{\bar{\alpha}}{2} -\alpha_3)   }
\end{equation}      
 The DOZZ formula is meromorphic with poles corresponding to the zeroes of the denominator of expression \eqref{theDOZZformula}.\\
 The has following asymptotic when $\Im(x)\to \infty$

 \noindent{\bf The Reflection Coefficient}
 
 We define the reflection coefficient as
 \begin{equation}\label{Rcoefficient}
 R^{\mathrm{DOZZ}}(\alpha)=-\left(\pi \mu l\left(\frac{\gamma^{2}}{4}\right)\right)^{\frac{2(Q-\alpha)}{\gamma}} \frac{\Gamma\left(-\frac{\gamma(Q-\alpha)}{2}\right)}{\Gamma\left(\frac{\gamma(Q-\alpha)}{2}\right)} \frac{\Gamma\left(-\frac{2(Q-\alpha)}{\gamma}\right)}{\Gamma\left(\frac{2(Q-\alpha)}{\gamma}\right)}
 \end{equation}
 and we have the following relation
 \begin{equation}
     C_{\gamma}^{\mathrm{DOZZ}}\left(\alpha_{1}, \alpha_{2}, \alpha_{3}\right)=R^{\mathrm{DOZZ}}\left(\alpha_{1}\right) C_{\gamma}^{\mathrm{DOZZ}}\left(2 Q-\alpha_{1}, \alpha_{2}, \alpha_{3}\right)
 \end{equation}

{\bf Asymptotic of Gamma function}\label{estimate gamma}
 For fixed $x \in \mathbb{R}$ there holds the following asymptotic estimate for $|\Gamma(x+i y)|$ at infinity:
$$
|\Gamma(x+i y)|=\sqrt{2 \pi}|y|^{x-1 / 2} e^{-x-|y| \pi / 2}\left[1+O\left(\frac{1}{|y|}\right)\right] \quad(|y| \rightarrow \infty)
$$

{\bf The FZZ formula}
 It is well known that the (bulk) one-point correlation function $\left\langle e^{\alpha \phi(z)}\right\rangle^{\mathbb{H}}_{\mu,\mu_\partial}$ of LCFT must have the following form
\begin{align}\label{onepointfunction}
    \left\langle e^{\alpha \phi(z)}\right\rangle^{\mathbb{H}}_{\mu,\mu_\partial}=\frac{U(\alpha)}{|z-\bar{z}|^{2 \Delta_{\alpha}}}g(z)^{-\Delta_{\alpha}}\big(\frac{{\rm v}_{g_{\rm dozz}}(\H)}{{\det}'(\Delta_{g_{\rm dozz},N})}\big)^{\frac{1}{2}}
\end{align}
 
where $U(\alpha)$ is called the bulk 1-point structure constant. the following exact formula for $U(\alpha)$ was proved in \cite{ARS21}
\begin{equation}\label{fzz}
    U_{\mathrm{FZZ},s}(\alpha):=\frac{4}{\gamma} \left(\pi \mu \frac{\Gamma\left(\frac{\gamma^{2}}{4}\right)}{\Gamma\left(1-\frac{\gamma^{2}}{4}\right)}\right)^{\frac{Q-\alpha}{\gamma}} \Gamma\left(\frac{\gamma \alpha}{2}-\frac{\gamma^{2}}{4}\right) \Gamma\left(\frac{2 \alpha}{\gamma}-\frac{4}{\gamma^{2}}-1\right) \cos ((\alpha-Q) \pi s)
\end{equation}

where the parameter $\theta$ is related to the ratio of cosmological constants $\frac{\mu_{B}}{\sqrt{\mu}}$ through the relation:
$$
\cos \frac{\pi \gamma s}{2}=\frac{\mu_{B}}{\sqrt{\mu}} \sqrt{\sin \frac{\pi \gamma^{2}}{4}}, \quad \text { with } \quad\left\{\begin{array}{l}
s \in\left[0, \frac{1}{\gamma}\right), \quad \text { when } \quad \frac{\mu_{B}^{2}}{\mu} \sin \frac{\pi \gamma^{2}}{4} \leq 1 \\
s \in i[0,+\infty), \quad \text { when } \quad \frac{\mu_{B}^{2}}{\mu} \sin \frac{\pi \gamma^{2}}{4} \geq 1
\end{array}\right.
$$
 \section{The scaling formula in the boundary case}\label{scaling}
For $g(z)=e^{\omega(z)}|d^2z|$, we define $\tilde{V}_{\alpha,g,\epsilon}$,  $\tilde{V}_{\frac{\beta}{2},g,\epsilon}$by
\begin{equation*}
\tilde{V}_{\alpha,g,\epsilon}(z)= \epsilon^{\alpha^2/2}  e^{\alpha  \Phi_{g,\epsilon}(z)} \quad \tilde{V}_{\frac{\beta}{2},g,\epsilon}(s)= \epsilon^{\beta^2/4}  e^{\frac{\beta}{2}  \Phi_{g,\epsilon}(s)} 
\end{equation*}
For $\mathbf{x}=x_{1}, \cdots, x_{n} \in \H$ and $\mathbf{y}=y_{1}, \cdots, y_{m} \in \R $, we define
$$ 
\tilde{G}(\mathbf{x} ; \mathbf{z};\mathbf{s})_\epsilon:=\left\langle\prod_{i=1}^{n} \tilde{V}_{\gamma, \epsilon}\left(x_{i}\right) \prod_{k=1}^{N} \tilde{V}_{\alpha_{k}, \epsilon}\left(z_{k}\right)\prod_{j=1}^{M} \tilde{V}_{\frac{\beta_j}{2}, \epsilon}\left(s_{k}\right)\right\rangle_{\H,U}
$$ and
$$\tilde{G}^{\partial}(\mathbf{y} ; \mathbf{z};\mathbf{s})_\epsilon:=\left\langle\prod_{i=1}^{m} \tilde{V}_{\frac{\gamma}{2}, \epsilon}\left(y_{i}\right) \prod_{k=1}^{N} \tilde{V}_{\alpha_{k}, \epsilon}\left(z_{k}\right)\prod_{j=1}^{M} \tilde{V}_{\frac{\beta_j}{2}, \epsilon}\left(s_{k}\right)\right\rangle_{\H,U}$$
When $\lim\to 0$ this limit exist and note as $\tilde{G}(\mathbf{x} ; \mathbf{z};\mathbf{s})$ and $\tilde{G}^{\partial}(\mathbf{y} ; \mathbf{z};\mathbf{s})$.
The scaling formula helps us get rid of metric dependence when we perform Gaussian integration by parts and get some union bound when we remove the SET insertions in the generalized correlation function. The uniform integrability mainly comes from the fusion rules studied in \cite{BW}, for the proof, we refer the reader to \cite{Wu1}
\begin{proposition}[The scaling identity]\label{KPZboundary}
For the round metric $g(z)=e^{\omega(z)}|dz^2|$ and we renormalize the Gaussian free field by supposing it has 0 average along the Gaussian curvature ($K_g=1$)
$$
 \int_{\H}\mu \gamma \tilde{G}(x ; \mathbf{z};\mathbf{s}) {\rm dv}_g(x)+\int_{\R}\mu_{\partial} \frac{\gamma}{2}  \tilde{G}(y ; \mathbf{z};\mathbf{s}) {\rm d\lambda}_g(y)=\left(\sum_{k=1}^{N} \alpha_{k}+\sum_{j=1}^{M} \frac{\beta_{j}}{2}-Q\right) \tilde{G}(\mathbf{z};\mathbf{s})
$$
and $\sup_{\epsilon}\tilde{G}(\mathbf{x} ; \mathbf{z};\mathbf{s})_\epsilon\in L^1(\H^n) $ and $\sup_{\epsilon} \tilde{G}^{\partial}(\mathbf{y} ; \mathbf{z};\mathbf{s})_\epsilon\in L^1(\R^m)$.
\end{proposition}

\begin{remark}
This relation also holds if we replace the positive constant $\mu_{\partial}$ by a piecewise constant function $\mu_{\partial}(s).$
\end{remark}
\section{Gaussian integration by parts}
\begin{lemma}
For $g=e^{\omega}|dz|^2$ is the round metric $\frac{4}{(1+|z|^2)^2}|dz|^2$, we define the Gaussian free field $\tilde{X}^{\H}_{g,N}$ on $\H$ which has a zero mean with respect to Ricci curvature of the metric $g=e^{\omega}|dz|^2$, define $\tilde{G}_{g,N}(z,z')=\E[\tilde{X}_{g,N}(z)\tilde{X}_{g,N}(z')]$. Then
\begin{align}
    \tilde{G}_{g,N}(z,z')=\ln\frac{1}{|z-z'||z-\bar{z}'|}-\frac{1}{2}\omega(z)-\frac{1}{2}\omega(z')+C.
\end{align}
\end{lemma}
\begin{proof}
Since $k_g=0$ on the real line $\R$, we can double $\H$ to get $\hat\C$. This property is clear true for $\hat\C$, then we use $ \tilde{G}^{\H}_{g,N}(z,z')= \tilde{G}^{\hat\C}_{g,N}(z,z')+ \tilde{G}^{\hat\C}_{g,N}(z,\bar{z'})$.
\end{proof}

For a centered Gaussian vector $(X,Y_1,\dots, Y_N)$ and a smooth function $f$ on $\R^N$, the Gaussian integration by parts formula is $$\E[X \,f(Y_1,\dots,Y_N)]=\sum_{k=1}^N\E[XY_k]\E[\partial_{k}f(Y_1,\dots,Y_N)]. $$
Applied to the LCFT this leads to the following formula. For the round metric $g=e^{\omega}|dz|^2$, let $\Phi=c+\tilde{X}^{\H}_{g,N}+\frac{Q}{2}\omega$ be the Liouville field  and $F$ a smooth function on $\R^N$.  Define for $u,v\in\bar{\H}$
\begin{align}
    C^{\H}(u,v)=-\frac{1}{2}\big(\frac{1}{u-v}+\frac{1}{u-\bar{v}}\Big),\ \  C^{\H}_{\epsilon,\epsilon'}(u,v)=\int \rho_{\epsilon}(u-u') \rho_{\epsilon'}(v-v')C^{\H}(u',v')\dd u\dd v
\end{align}
 with $( \rho_{\epsilon})_\epsilon$ a mollifying family of the type $\rho_\epsilon(\cdot)=\epsilon^{-2}\rho(|\cdot|/\epsilon)$ as in Section \eqref{SET}. Note $\tilde{V}_{\alpha,g,\epsilon}(z):= \epsilon^{\alpha^2/2}  e^{\alpha  \Phi_{g,\epsilon}(z)}$ and $\tilde{V}_{\frac{\beta}{2},g,\epsilon}(s)= \epsilon^{\beta^2/4}  e^{\frac{\beta}{2} \Phi_{g,\epsilon}(s)}$.
 
 Notice that $ \partial_u\tilde{G}_{g,N}^{\H}(u,v)=C^{\H}(u,v)-\frac{1}{2} \partial_u\omega(u)$. The virtue of this Liouville field is that the annoying metric-dependent terms $-\frac{1}{2} \partial_u\omega(u)$ drop out from the formulae. This fact is nontrivial and it was proven in \eqref{KPZboundary},  for the case $t=\infty$ with $F$ corresponding to the product of vertex operators. The proof goes the same way to produce \eqref{basicipp} with a finite $t$.
 \begin{proposition}
 For $z$, $x_1,\dots,x_N\in\bar{\H}$, we have
\begin{align} \label{basicipp}
\langle \partial_z\Phi_\epsilon(z) F(\Phi_{\epsilon'}(x_1),\dots,\Phi_{\epsilon'}(x_N))\rangle_t
&=
\sum_{k=1}^NC_{\epsilon,\epsilon'}(z,x_k)\langle \partial_{k}F(\Phi_{\epsilon'}(x_1),\dots,\Phi_{\epsilon'}(x_N))\rangle_t
\\
&-\mu\gamma\int_{\H_t} C_{\epsilon,0}(z,x)\langle \tilde{V}_\gamma(x) F(\Phi_{\epsilon'}(x_1),\dots,\Phi_{\epsilon'}(x_N))\rangle_t \dd x  \nonumber\\
&-\mu_{\partial}\frac{\gamma}{2}\int_{\R} C_{\epsilon,0}(z,s)\langle \tilde{V}_{\frac{\gamma}{2}}(x) F(\Phi_{\epsilon'}(x_1),\dots,\Phi_{\epsilon'}(x_N))\rangle_t \dd s  \nonumber
\end{align}
For $F$ such that all terms are well defined.
 \end{proposition}
\begin{proof}
Follow the same lines in proof of \eqref{KPZboundary}, it's easy to show
\begin{align}
   &\int_{\H} \mu\gamma\langle \tilde{V}_{\gamma}(x) F(\Phi_{\epsilon'}(x_1),\dots,\Phi_{\epsilon'}(x_N))\rangle_t\dd x+ \int_{\R} \mu_{\partial}\frac{\gamma}{2}\langle \tilde{V}_{\frac{\gamma}{2}}(s) F(\Phi_{\epsilon'}(x_1),\dots,\Phi_{\epsilon'}(x_N))\rangle_t\dd s\nonumber\\
   =&\sum_{i=1}^N\langle  \partial_{k}F(\Phi_{\epsilon'}(x_1),\dots,\Phi_{\epsilon'}(x_N))\rangle_t-Q\langle  F(\Phi_{\epsilon'}(x_1),\dots,\Phi_{\epsilon'}(x_N))\rangle_t
\end{align}
Then we can see the collection of metric dependent terms vanish.
\end{proof}

Now let us introduce a systematic way to calculate Ward's identity, first we introduce the notation
 $$(\caO_1,\caO_2,\caO_3)  :=( \partial_z\Phi_\epsilon, \partial_z^2\Phi_\epsilon, (\partial_z\Phi_\epsilon)^2-\E(\partial_z\Phi_\epsilon)^2).$$
 Applying the integration by parts formula to $\partial_z\Phi_\epsilon(u_k)$ (or to $\partial_z^2\Phi_\epsilon(u_k)$ if $i_k=2$ below) we obtain (here $V_{\alpha_i}$ can be either bulk or boundary insertion).
 
 \begin{align}\label{computationsystem}
 \langle \prod_{j=1}^k \caO_{i_j}(u_j)\prod_{i=1}^n\tilde{V}_{\alpha_i }(z_i)\prod_{j=1}^m\tilde{V}_{\frac{\beta_j}{2} }(s_j) \rangle_t=&\sum_{l=1}^{k-1}  \langle \{\caO_{i_k}(u_k)\caO_{i_l}(u_l)\} \prod_{j\neq k,l} \caO_{i_j}(u_j)\prod_{i=1}^n\tilde{V}_{\alpha_i }(z_i)\prod_{j=1}^m\tilde{V}_{\frac{\beta_j}{2} }(s_j) \rangle_t \nonumber\\
 +&\sum_{l=1}^n \alpha_l \langle \{\caO_{i_k}(u_k)\Phi(z_l)\}\prod_{j\neq k} \caO_j(u_j)\prod_{i\neq l}\tilde{V}_{\alpha_i }(z_i)\prod_{j=1}^m\tilde{V}_{\frac{\beta_j}{2} }(s_j) \rangle_t\nonumber\\
 +&\sum_{l=1}^n \frac{\beta_l}{2} \langle \{\caO_{i_k}(u_k)\Phi(s_l)\}\prod_{j\neq k} \caO_j(u_j)\prod_{i\neq l}\tilde{V}_{\alpha_i }(z_i)\prod_{j=1}^m\tilde{V}_{\frac{\beta_j}{2} }(s_j) \rangle_t\nonumber\\
 -&\mu\gamma \int_{\H_t} \langle \{\caO_{i_k}(u_k)\Phi(x)\} \prod_{j\neq k} \caO_j(u_j)\tilde{V}_\gamma(x)\prod_{i\neq l}\tilde{V}_{\alpha_i }(z_i) \prod_{j=1}^m\tilde{V}_{\frac{\beta_j}{2} }(s_j)\rangle_t\dd x\nonumber\\-&\mu_{\partial}\frac{\gamma}{2} \int_{\R} \langle \{\caO_{i_k}(u_k)\Phi(s)\} \prod_{j\neq k} \caO_j(u_j)\tilde{V}_{\frac{\gamma}{2}}(s)\prod_{i\neq l}\tilde{V}_{\alpha_i }(z_i)\prod_{j=1}^m\tilde{V}_{\frac{\beta_j}{2} }(s_j) \rangle_t\dd s
 \end{align}
where $\alpha_0=\alpha$ and $z_0=0$ and the contractions are defined by
\begin{align*}
 \{\caO_1(u)\caO_1(v)\} &=\partial_vC_{\epsilon,\epsilon}(u,v) & & \{\caO_2(u)\caO_2(v)\} =\partial_u\partial^2_{v}C_{\epsilon,\epsilon}(u,v)\\
 \{\caO_1(u)\caO_3(v)\} &=2\caO_1(v)\partial_vC_{\epsilon,\epsilon}(u,v) & &  \{\caO_1(u)\caO_2(v)\} =\partial^2_{v}C_{\epsilon,\epsilon}(u,v)\\
 \{\caO_3(u)\caO_1(v)\} &=\caO_1(u)\partial_vC_{\epsilon,\epsilon}(u,v) & & \{\caO_2(u)\caO_1(v)\} =\partial_{u}\partial_{v}C_{\epsilon,\epsilon}(u,v) \\
\{\caO_3(u)\caO_3(v)\} &=2\caO_1(u)\caO_1(v)\partial_vC_{\epsilon,\epsilon}(u,v)& &  \{\caO_2(u)\caO_3(v)\} =2\caO_1(v)\partial_{u}\partial_{v}C_{\epsilon,\epsilon}(u,v)\\
\{\caO_3(u)\caO_2(v)\} &=\caO_1(u)\partial^2_{v}C_{\epsilon,\epsilon}(u,v)  & &  
\end{align*}
etc. Similarly 
$$\{\caO_1(u)\Phi(x)\} = C_{\epsilon,0}(u,x), \quad\{\caO_2(u)\Phi(x)\} = \partial_u C_{\epsilon,0}(u,x), \quad\{\caO_3(u)\Phi(x)\} =\caO_1(u)C_{\epsilon,0}(u,x). $$
$$\{\caO_1(u)\Phi(s)\} = C_{\epsilon,0}(u,s), \quad\{\caO_2(u)\Phi(x)\} = \partial_u C_{\epsilon,0}(u,s), \quad\{\caO_3(u)\Phi(s)\} =\caO_1(u)C_{\epsilon,0}(u,s). $$


\section{Proof of Cardy's doubling trick}\label{proof cardy}
{\it Proof.} The remaining part of this subsection is devoted to Cardy's doubling trick.  We will suppose that the metric $g$ is the restriction of round metric to the upper half plane, i.e. $g=e^{\omega}|dz|^2=\frac{4}{(1+|z|^2)^2}|dz|^2$. The result for other metrics can be deduced via the Weyl anomaly formula. Our computation follows systematically from \eqref{computationsystem}. In the proof, we always use the notation  $\tilde{V}_{\alpha,g,\epsilon}(z):= \epsilon^{\alpha^2/2}  e^{\alpha  \Phi_{g,\epsilon}(z)}$ and $\tilde{V}_{\frac{\beta}{2},g,\epsilon}(s)= \epsilon^{\beta^2/4}  e^{\frac{\beta}{2} \Phi_{g,\epsilon}(s)}$, then by \eqref{KPZboundary}, we can remove the metric dependency. So in the following proof, we can pretend $\omega=0$. We also simply use suffix $\_t$ to represent $\_{\H,U_t,g}$.

We start with a simple lemma to see how the Gaussain free field on half upper plane linked with the full plane, which proof is given by direct computation.
\begin{lemma}
If $a$, $b$ $\in \mathbb{N}^+$, we have
\begin{equation}
    \E[\,\partial_{x}^aX^{\mathbb{H}}(x)\overline{\partial_{y}^bX^{\mathbb{H}}(y)}\,]=\E[\,\partial_x^aX^{\hat\C}(x)\partial_{\bar{y}}^bX^{\hat \C}(\bar{y})\,]
\end{equation}
\begin{equation}
    \E[\,\partial_{x}^aX^{\mathbb{H}}(x)\partial_{y}^bX^{\mathbb{H}}(y)\,]=\E[\,\partial_x^aX^{\hat\C}(x)\partial_{y}^bX^{\hat \C}(y)\,]
\end{equation}
\begin{equation}
    \E[\,\overline{\partial_{x}^aX^{\mathbb{H}}(x)}  \overline{\partial_{y}^bX^{\mathbb{H}}(y)}\,]=\E[\,\partial_{\bar{x}}^aX^{\hat\C}(\bar{x})\partial_{\bar{y}}^bX^{\hat \C}(\bar{y})\,]
\end{equation}
\end{lemma}

Our computation follows systematically from \eqref{computationsystem}. In the proof, we always use the notation  $\tilde{V}_{\alpha,g,\epsilon}(z):= \epsilon^{\alpha^2/2}  e^{\alpha  \Phi_{g,\epsilon}(z)}$ and $\tilde{V}_{\frac{\beta}{2},g,\epsilon}(s)= \epsilon^{\beta^2/4}  e^{\frac{\beta}{2} \Phi_{g,\epsilon}(s)}$, then by \eqref{KPZboundary}, we can remove the metric dependency. So in the following proof, we can pretend $\omega=0$. We also simply use suffix $\_t$ to represent $\_{\H,U_t,g}$.

So, applying the Gaussian integration by parts to the SET-insertion $T^{\H}(u_k)$  produces plenty of terms which we group in four contributions:
\begin{equation}\label{IPPstress}
\langle  \mathcal{T}(\mathbf{x})   \prod_{i=1}^n\tilde{V}_{\alpha_i  }(z_i)\prod_{j=1}^m\tilde{V}_{\frac{\beta_j}{2} }(s_j) \rangle_t =M(\mathbf{x})+T(\mathbf{x}) + N(\mathbf{x})+ C(\mathbf{x})  .
\end{equation}
The first contribution in \eqref{IPPstress} collects the contractions hitting only one $V_{\alpha_p}$:
\begin{align}\label{IPPstressM}
M(\mathbf{x},\alpha)= & \sum_{p=1}^{n} (\frac{Q\alpha_p}{2}- \frac{\alpha_p^2}{4}) \Big(\frac{1}{(x_{k+\tilde{k}}-z_p)^2} +\frac{1}{(x_{k+\tilde{k}}-\bar{z_p})^2}\Big) \langle    \mathcal{T}(\mathbf{x}^{(k+\tilde{k})})  \tilde{\mathcal{V}}_{\mathbf{\alpha}}(\mathbf{z})\tilde{\mathcal{V}}_{\frac{\mathbf{\beta}}{2} }(\mathbf{s})\rangle_t \\&-\sum_{p=1}^{n}  \frac{\alpha_p^2}{2} \Big(\frac{1}{x_{k+\tilde{k}}-z_p} \frac{1}{x_{k+\tilde{k}}-\bar{z_p}}\Big) \langle    \mathcal{T}(\mathbf{x}^{(k+\tilde{k})}) \tilde{\mathcal{V}}_{\mathbf{\alpha}}(\mathbf{z})\tilde{\mathcal{V}}_{\frac{\mathbf{\beta}}{2} }(\mathbf{s})\rangle_t. 
 \\&:= M_1(\mathbf{x},\alpha)+ M_2(\mathbf{x},\alpha) 
\end{align}
\begin{align}
    M(\mathbf{x},\beta)= & \sum_{q=1}^{m} (\frac{Q\beta_q}{2}- \frac{\beta_q^2}{4}) \Big(\frac{1}{(x_{k+\tilde{k}}-s_q)^2} \Big) \langle    \mathcal{T}(\mathbf{x}^{(k+\tilde{k})})  \tilde{\mathcal{V}}_{\mathbf{\alpha}}(\mathbf{z})\tilde{\mathcal{V}}_{\frac{\mathbf{\beta}}{2} }(\mathbf{s})\rangle_t 
\end{align}
The second contribution collects the terms coming from contractions of SET insertions and producing lower-order SET insertions:
 \begin{align}\label{OtherstressT}
 T(\mathbf{x})= \frac{1}{2}   \sum_{\ell=1}^{k+\tilde{k}-1}\frac{1+6Q^2}{(x_{k+\tilde{k}}-x_\ell)^4}  \langle   \mathcal{T}(\mathbf{x}^{(\ell,k+\tilde{k})})   \tilde{\mathcal{V}}_{\mathbf{\alpha}}(\mathbf{z})\tilde{\mathcal{V}}_{\frac{\mathbf{\beta}}{2} }(\mathbf{s})\rangle_{t}
 \end{align}
 The third contribution is given by terms where all contractions hit only $\tilde{V}_\gamma$ or only $\tilde{V}_{\frac{\gamma}{2}}$:
\begin{align}\nonumber
N^{\mu}(\mathbf{x})=&(\frac{\mu\gamma^2}{4}-\frac{\mu\gamma Q}{2}) \int_{\H_t} \Big(\frac{1}{(x_{k+\tilde{k}}-x)^2}+\frac{1}{(x_{k+\tilde{k}}-\bar{x})^2}\Big)\langle   \mathcal{T}(\mathbf{x}^{(k+\tilde{k})})  \tilde{V}_\gamma(x)\tilde{\mathcal{V}}_{\mathbf{\alpha}}(\mathbf{z})\tilde{\mathcal{V}}_{\frac{\mathbf{\beta}}{2} }(\mathbf{s})\rangle_t \,\dd x
\\
  & \textcolor{red}{+ \mu\frac{\gamma^2}{2}   \int_{\H_t} \frac{1}{(x_{k+\tilde{k}}-x)(x_{k+\tilde{k}}-\bar{x})}\langle   \mathcal{T}(\mathbf{x}^{(k+\tilde{k})})  \tilde{V}_\gamma(x)\tilde{\mathcal{V}}_{\mathbf{\alpha}}(\mathbf{z})\tilde{\mathcal{V}}_{\frac{\mathbf{\beta}}{2} }(\mathbf{s})\rangle_t \,\dd x.} \label{Nterm} 
  \\&=N^{\mu}(\mathbf{x},\partial)+N^{\mu}(\mathbf{x},\bar{\partial})+\textcolor{red}{T_{1}(\mathbf{x})}
\end{align}
and
\begin{align}
N^{\mu_{\partial}}(\mathbf{x})=&(\frac{\mu_{\partial}\gamma^2}{4}-\frac{\mu_{\partial}\gamma Q}{2}) \int_{\R_t} \frac{1}{(x_{k+\tilde{k}}-s)^2}\langle   \mathcal{T}(\mathbf{x}^{(k+\tilde{k})})  \tilde{V}_{\frac{\gamma}{2}}(s)\tilde{\mathcal{V}}_{\mathbf{\alpha}}(\mathbf{z})\tilde{\mathcal{V}}_{\frac{\mathbf{\beta}}{2} }(\mathbf{s})\rangle_t \,\dd s\\&=-\mu_{\partial} \int_{\R_t} \frac{1}{(x_{k+\tilde{k}}-s)^2}\langle   \mathcal{T}(\mathbf{x}^{(k+\tilde{k})})  \tilde{V}_{\frac{\gamma}{2}}(s)\tilde{\mathcal{V}}_{\mathbf{\alpha}}(\mathbf{z})\tilde{\mathcal{V}}_{\frac{\mathbf{\beta}}{2} }(\mathbf{s}) \rangle_t \,\dd s
\end{align}
Finally, $C$ gathers all the other terms 
\begin{align}
\textcolor{black}{C_1(\mathbf{x} )}:=&  - \sum_{\ell\not =\ell'=1}^{k+\tilde{k}-1}\frac{ Q^2}{(x_{k+\tilde{k}}-x_\ell)^3(x_{k+\tilde{k}}-x_{\ell'})^3}  \langle      \mathcal{T}(\mathbf{x}^{({k+\tilde{k}},\ell,\ell')})   \tilde{\mathcal{V}}_{\mathbf{\alpha}}(\mathbf{z})\tilde{\mathcal{V}}_{\frac{\mathbf{\beta}}{2} }(\mathbf{s}) \rangle_{t}
\nonumber\\
\textcolor{black}{C_2(\mathbf{x} )}:=&  \sum_{\ell\not =\ell'=1}^{k+\tilde{k}-1}\frac{2Q^2}{(x_{k+\tilde{k}}-x_\ell)^3(x_\ell-x_{\ell'})^3}  \langle     \mathcal{T}(\mathbf{x}^{(k+\tilde{k},\ell,\ell')})   \tilde{\mathcal{V}}_{\mathbf{\alpha}}(\mathbf{z})\tilde{\mathcal{V}}_{\frac{\mathbf{\beta}}{2} }(\mathbf{s})\rangle_{t} 
\nonumber\\
\textcolor{black}{C_3(\mathbf{x} )}:=&    \sum_{\ell\not =\ell'=1}^{k+\tilde{k}-1}\frac{ Q}{(x_{k+\tilde{k}}-x_\ell)^3(x_{k+\tilde{k}}-x_{\ell'})^2}  \langle    \partial_{x_{\ell'}}\Phi(x_{\ell'})  \mathcal{T}(\mathbf{x}^{(k+\tilde{k},\ell,\ell')})   \tilde{\mathcal{V}}_{\mathbf{\alpha}}(\mathbf{z})\tilde{\mathcal{V}}_{\frac{\mathbf{\beta}}{2} }(\mathbf{s}) \rangle_{t}
\nonumber\\
\textcolor{black}{C_4(\mathbf{x} )}:=& - \sum_{\ell\not =\ell'=1}^{k+\tilde{k}-1}\frac{2Q}{(x_{k+\tilde{k}}-x_\ell)^3(x_\ell-x_{\ell'})^2}  \langle    \partial_{x_{\ell'}}\Phi(x_{\ell'})   T(\mathbf{x}^{(\ell,k+\tilde{k},\ell')})   \tilde{\mathcal{V}}_{\mathbf{\alpha}}(\mathbf{z})\tilde{\mathcal{V}}_{\frac{\mathbf{\beta}}{2} }(\mathbf{s}) \rangle_{t} 
\nonumber\\
\textcolor{black}{C_{5}(\mathbf{x} }):=&  \sum_{\ell=1}^{k+\tilde{k} -1}\sum_{p=1}^{n}\frac{ Q\alpha_p}{ (x_{k+\tilde{k}}-x_{\ell})^3} (\frac{1}{x_\ell-z_p}+\frac{1}{x_\ell-\bar{z_p}}) \langle     \mathcal{T}(\mathbf{x}^{(k+\tilde{k},\ell )})   \tilde{\mathcal{V}}_{\mathbf{\alpha}}(\mathbf{z})\tilde{\mathcal{V}}_{\frac{\mathbf{\beta}}{2} }(\mathbf{s})\rangle_{t} 
\nonumber\\
&+ \sum_{\ell=1}^{k+\tilde{k} -1}\sum_{q=1}^{m}\frac{ Q\beta_q}{ (x_{k+\tilde{k}}-x_{\ell})^3} (\frac{1}{x_\ell-s_q}) \langle     \mathcal{T}(\mathbf{x}^{(k+\tilde{k},\ell )})   \tilde{\mathcal{V}}_{\mathbf{\alpha}}(\mathbf{z})\tilde{\mathcal{V}}_{\frac{\mathbf{\beta}}{2} }(\mathbf{s}) \rangle_{t} 
\nonumber\\
\textcolor{black}{C_{6}(\mathbf{x},\mu )}:=&-  \mu Q\gamma  \sum_{\ell=1}^{k+\tilde{k}-1}\int_{\H_t}\frac{ 1}{(x_{k+\tilde{k}}-x_\ell)^3 } (\frac{1}{x_\ell-x} +\frac{1}{x_\ell-\bar{x}})\langle   \mathcal{T}(\mathbf{x}^{(k+\tilde{k},\ell )})   \tilde{V}_\gamma(x)\tilde{\mathcal{V}}_{\mathbf{\alpha}}(\mathbf{z})\tilde{\mathcal{V}}_{\frac{\mathbf{\beta}}{2} }(\mathbf{s}) \rangle_{t} \,\dd x
\nonumber\\
\textcolor{black}{C_{6}(\mathbf{x} ,\mu_{\partial})}:=&-  \mu_\partial Q\gamma  \sum_{\ell=1}^{k+\tilde{k}-1}\int_{\H_t}\frac{ 1}{(x_{k+\tilde{k}}-x_\ell)^3 } (\frac{1}{x_\ell-s} )\langle   \mathcal{T}(\mathbf{x}^{(k+\tilde{k},\ell )})   \tilde{V}_{\frac{\gamma}{2}}(s)\tilde{\mathcal{V}}_{\mathbf{\alpha}}(\mathbf{z})\tilde{\mathcal{V}}_{\frac{\mathbf{\beta}}{2} }(\mathbf{s}) \rangle_{t} \,\dd s
\nonumber\\
\textcolor{black}{C_7(\mathbf{x})}:=&   \sum_{\ell\not=\ell'=1}^{k+\tilde{k}-1}\frac{ Q}{(x_{k+\tilde{k}}-x_\ell)^2(x_{k+\tilde{k}}-x_{\ell'})^3}  \langle \partial_{x_{\ell}}\Phi(x_{\ell})    \mathcal{T}(\mathbf{x}^{(k+\tilde{k},\ell,\ell')})   \tilde{\mathcal{V}}_{\mathbf{\alpha}}(\mathbf{z})\tilde{\mathcal{V}}_{\frac{\mathbf{\beta}}{2} }(\mathbf{s})\rangle_{t} 
\nonumber\\
\textcolor{black}{C_{8}(\mu,\mathbf{x} )}:=&-   \mu  \gamma  \sum_{\ell=1}^{k+\tilde{k}-1}\int_{\H_t} \frac{ 1}{(x_{k+\tilde{k}}-x_\ell)^2 } (\frac{1}{x_{k+\tilde{k}}-x}+\frac{1}{x_{k+\tilde{k}}-\bar{x}}) \langle  \partial_{x_{\ell}}\Phi(x_{\ell})   \mathcal{T}(\mathbf{x}^{(k+\tilde{k},\ell )})   V_\gamma(x)\tilde{\mathcal{V}}_{\mathbf{\alpha}}(\mathbf{z})\tilde{\mathcal{V}}_{\frac{\mathbf{\beta}}{2} }(\mathbf{s}) \rangle_{t} \,\dd x
\nonumber\\
\textcolor{black}{C_{8}(\mu_{\partial},\mathbf{x} )}:=&-   \mu_{\partial}  \gamma  \sum_{\ell=1}^{k+\tilde{k}-1}\int_{\R_t} \frac{ 1}{(x_{k+\tilde{k}}-x_\ell)^2 } (\frac{1}{x_{k+\tilde{k}}-s}) \langle  \partial_{x_{\ell}}\Phi(x_{\ell})   \mathcal{T}(\mathbf{x}^{(k+\tilde{k},\ell )})   V_{\frac{\gamma}{2}}(s)\tilde{\mathcal{V}}_{\mathbf{\alpha}}(\mathbf{z})\tilde{\mathcal{V}}_{\frac{\mathbf{\beta}}{2} }(\mathbf{s}) \rangle_{t} \,\dd s
\nonumber\\
\textcolor{black}{C_9(\mathbf{x} )}:=&     \sum_{\ell=1}^{k+\tilde{k}-1} \sum_{p=1}^{n}\frac{ \alpha_p}{(x_{k+\tilde{k}}-x_\ell)^2} (\frac{1}{x_{k+\tilde{k}}-z_p}+\frac{1}{x_{k+\tilde{k}}-\bar{z_p}})\langle \partial_{x_{\ell}}\Phi(x_{\ell})       \mathcal{T}(\mathbf{x}^{(k+\tilde{k},\ell)})   \tilde{\mathcal{V}}_{\mathbf{\alpha}}(\mathbf{z})\tilde{\mathcal{V}}_{\frac{\mathbf{\beta}}{2} }(\mathbf{s}) \rangle_{t} 
\nonumber\\
&+\sum_{\ell=1}^{k+\tilde{k}-1} \sum_{q=1}^{m}\frac{ \beta_q}{(x_{k+\tilde{k}}-x_\ell)^2} (\frac{1}{x_{k+\tilde{k}}-s_q})  \langle  \partial_{x_{\ell}}\Phi(x_{\ell})    \mathcal{T}(\mathbf{x}^{(k+\tilde{k},\ell)})   \tilde{\mathcal{V}}_{\mathbf{\alpha}}(\mathbf{z})\tilde{\mathcal{V}}_{\frac{\mathbf{\beta}}{2} }(\mathbf{s})\rangle_{t} 
\nonumber\\
\textcolor{black}{C_{10}(\mathbf{x} )}:=&  \sum_{p\not =p'=1}^{n}- \frac{\alpha_p\alpha_{p'}}{4}(\frac{ 1}{x_{k+\tilde{k}}-z_p}+\frac{ 1}{x_{k+\tilde{k}}-\bar{z_p}}) (\frac{ 1}{x_{k+\tilde{k}}-z_{p'}}+\frac{ 1}{x_{k+\tilde{k}}-\bar{z_{p'}}})  \langle    \mathcal{T}(\mathbf{x}^{(k+\tilde{k} )})   \tilde{\mathcal{V}}_{\mathbf{\alpha}}(\mathbf{z})\tilde{\mathcal{V}}_{\frac{\mathbf{\beta}}{2} }(\mathbf{s}) \rangle_{t} 
\nonumber\\
&+\sum_{p=1,q=1}^{n,m}- \frac{\alpha_p\beta_q}{4}(\frac{ 1}{x_{k+\tilde{k}}-z_p}+\frac{ 1}{x_{k+\tilde{k}}-\bar{z_p}}) (\frac{ 1}{x_{k+\tilde{k}}-s_q})  \langle    \mathcal{T}(\mathbf{x}^{(k+\tilde{k} )})   \tilde{\mathcal{V}}_{\mathbf{\alpha}}(\mathbf{z})\tilde{\mathcal{V}}_{\frac{\mathbf{\beta}}{2} }(\mathbf{s}) \rangle_{t} 
\nonumber\\
&+\sum_{q\not= q'=1}^{m}- \frac{\beta_q\beta_{q'}}{4}(\frac{ 1}{x_{k+\tilde{k}}-s_q})(\frac{ 1}{x_{k+\tilde{k}}-s_{q'}})  \langle    \mathcal{T}(\mathbf{x}^{(k+\tilde{k} )})   \tilde{\mathcal{V}}_{\mathbf{\alpha}}(\mathbf{z})\tilde{\mathcal{V}}_{\frac{\mathbf{\beta}}{2} }(\mathbf{s}) \rangle_{t} 
\nonumber\\
\end{align}
\begin{align}
\textcolor{black}{C_{11}(\mathbf{x} ,\mu,\mu}):= -  &\frac{\mu^2  \gamma^2}{4}   \int_{\H_t}\int_{\H_t}(\frac{ 1}{x_{k+\tilde{k}}-x}+\frac{ 1}{x_{k+\tilde{k}}-\bar{x}}) (\frac{ 1}{x_{k+\tilde{k}}-x'}+\frac{ 1}{x_{k+\tilde{k}}-\bar{x}'})   \langle    \mathcal{T}(\mathbf{x}^{(k+\tilde{k}  )})   \tilde{V}_\gamma(x)\tilde{V}_\gamma(x')\tilde{\mathcal{V}}_{\mathbf{\alpha}}(\mathbf{z})\tilde{\mathcal{V}}_{\frac{\mathbf{\beta}}{2} }(\mathbf{s}) \rangle_{t} \,\dd x\,\dd x'\
\nonumber\\
\textcolor{black}{C_{11}(\mathbf{x} ,\mu,\mu_{\partial}}):= &-  \frac{\mu\mu_{\partial}  \gamma^2}{4}   \int_{\H_t}\int_{\R_t}(\frac{ 1}{x_{k+\tilde{k}}-x}+\frac{ 1}{x_{k+\tilde{k}}-\bar{x}}) (\frac{ 1}{x_{k+\tilde{k}}-s})   \langle    \mathcal{T}(\mathbf{x}^{(k+\tilde{k}  )})   \tilde{V}_\gamma(x)\tilde{V}_\frac{\gamma}{2}(s)\tilde{\mathcal{V}}_{\mathbf{\alpha}}(\mathbf{z})\tilde{\mathcal{V}}_{\frac{\mathbf{\beta}}{2} }(\mathbf{s}) \rangle_{t} \,\dd x\,\dd s\
\nonumber\\
\textcolor{black}{C_{11}(\mathbf{x} ,\mu_{\partial},\mu_{\partial}}):=& -  \frac{\mu_{\partial}^2 \gamma^2}{4}   \int_{\R_t}\int_{\R_t}(\frac{ 1}{x_{k+\tilde{k}}-s}) (\frac{ 1}{x_{k+\tilde{k}}-s'})   \langle    \mathcal{T}(\mathbf{x}^{(k+\tilde{k}  )})   \tilde{V}_\frac{\gamma}{2}(s')\tilde{V}_\frac{\gamma}{2}(s)\tilde{\mathcal{V}}_{\mathbf{\alpha}}(\mathbf{z})\tilde{\mathcal{V}}_{\frac{\mathbf{\beta}}{2} }(\mathbf{s})\rangle_{t} \,\dd s\,\dd s'\
\nonumber\\
C_{12}(\mathbf{x} ):=&    \sum_{\ell\not=\ell'=1}^{k+\tilde{k}-1}\frac{ 1}{(x_{k+\tilde{k}}-x_\ell)^2(x_{k+\tilde{k}}-x_{\ell'})^2}  \langle \partial_{x_{\ell}}\Phi(x_{\ell})    \partial_{x_{\ell'}}\Phi(x_{\ell'})     \mathcal{T}(\mathbf{x}^{(k+\tilde{k},\ell,\ell')})   \tilde{\mathcal{V}}_{\mathbf{\alpha}}(\mathbf{z})\tilde{\mathcal{V}}_{\frac{\mathbf{\beta}}{2} }(\mathbf{s}) \rangle_{t} 
\nonumber\\
\textcolor{black}{C_{13}(\mathbf{x},\mu )}:=&  \mu Q\gamma  \sum_{\ell=1}^{k+\tilde{k}-1}\int_{\H_t}\frac{ 1}{(x_{k+\tilde{k}}-x_\ell)^3 }(\frac{1}{x_{k+\tilde{k}}-x}+\frac{1}{x_{k+\tilde{k}}-\bar{x}} ) \langle   \mathcal{T}(\mathbf{x}^{(k+\tilde{k},\ell )})   \tilde{V}_\gamma(x)\tilde{\mathcal{V}}_{\mathbf{\alpha}}(\mathbf{z})\tilde{\mathcal{V}}_{\frac{\mathbf{\beta}}{2} }(\mathbf{s}) \rangle_{t} \,\dd x
\nonumber\\
\textcolor{black}{C_{13}(\mathbf{x},\mu_{\partial} )}:=&  \mu_{\partial} Q\gamma  \sum_{\ell=1}^{k+\tilde{k}-1}\int_{\R_t}\frac{ 1}{(x_{k+\tilde{k}}-x_\ell)^3 }(\frac{1}{x_{k+\tilde{k}}-s} ) \langle   \mathcal{T}(\mathbf{x}^{(k+\tilde{k},\ell )})   \tilde{V}_{\frac{\gamma}{2}}(s)\tilde{\mathcal{V}}_{\mathbf{\alpha}}(\mathbf{z})\tilde{\mathcal{V}}_{\frac{\mathbf{\beta}}{2} }(\mathbf{s}) \rangle_{t} \,\dd s
\nonumber\\
\textcolor{green}{C_{14}(\mathbf{x} ,\mu):=}& \textcolor{green}{\sum_{p=1}^{n}  \frac{\mu  \gamma\alpha_p}{2} \int_{\H_t}(\frac{1}{x_{k+\tilde{k}}-x}+\frac{1}{x_{k+\tilde{k}}-\bar{x}} ) (\frac{1}{x_{k+\tilde{k}}-z_p}+\frac{1}{x_{k+\tilde{k}}-\bar{z_p}} ) \langle    \mathcal{T}(\mathbf{x}^{(k+\tilde{k} )})   \tilde{V}_\gamma(x)\tilde{\mathcal{V}}_{\mathbf{\alpha}}(\mathbf{z})\tilde{\mathcal{V}}_{\frac{\mathbf{\beta}}{2} }(\mathbf{s}) \rangle_{t} \,\dd x
\nonumber}\\
&\textcolor{green}{+\sum_{q=1}^{m}\frac{\mu  \gamma\beta_q}{2} \int_{\H_t}(\frac{1}{x_{k+\tilde{k}}-x}+\frac{1}{x_{k+\tilde{k}}-\bar{x}} ) (\frac{1}{x_{k+\tilde{k}}-s_q} ) \langle    \mathcal{T}(\mathbf{x}^{(k+\tilde{k} )})   \tilde{V}_\gamma(x)\tilde{\mathcal{V}}_{\mathbf{\alpha}}(\mathbf{z})\tilde{\mathcal{V}}_{\frac{\mathbf{\beta}}{2} }(\mathbf{s}) \rangle_{\mathcal{D}} \,\dd x
\nonumber}\\
\textcolor{cyan}{C_{14}(\mathbf{x} ,\mu_{\partial}):=}&   \textcolor{cyan}{\sum_{p=1}^{n}\frac{\mu_{\partial}  \gamma\alpha_p}{2} \int_{\R_t}(\frac{1}{x_{k+\tilde{k}}-s} ) (\frac{1}{x_{k+\tilde{k}}-z_p}+\frac{1}{x_{k+\tilde{k}}-\bar{z_p}} ) \langle    \mathcal{T}(\mathbf{x}^{(k+\tilde{k} )})   \tilde{V}_{\frac{\gamma}{2}}(s)\tilde{\mathcal{V}}_{\mathbf{\alpha}}(\mathbf{z})\tilde{\mathcal{V}}_{\frac{\mathbf{\beta}}{2} }(\mathbf{s}) \rangle_{t} \,\dd s
\nonumber}\\
&\textcolor{cyan}{+\sum_{q=1}^{m}\frac{\mu_{\partial}  \gamma\beta_q}{2} \int_{\R_t}(\frac{1}{x_{k+\tilde{k}}-s} ) (\frac{1}{x_{k+\tilde{k}}-s_q} ) \langle    \mathcal{T}(\mathbf{x}^{(k+\tilde{k} )})   \tilde{V}_{\frac{\gamma}{2}}(s)\tilde{\mathcal{V}}_{\mathbf{\alpha}}(\mathbf{z})\tilde{\mathcal{V}}_{\frac{\mathbf{\beta}}{2} }(\mathbf{s}) \rangle_{t} \,\dd s
\nonumber}\\
\textcolor{black}{C_{15}(\mathbf{x} )}:=& -\sum_{\ell=1}^{k+\tilde{k}-1} \sum_{p=1}^{n}\frac{ Q\alpha_p}{(x_{k+\tilde{k}}-x_{\ell})^3} (\frac{1}{x_{k+\tilde{k}}-z_p}+\frac{1}{x_{k+\tilde{k}}-\bar{z_p}} ) \langle     \mathcal{T}(\mathbf{x}^{(k+\tilde{k},\ell )})   \tilde{\mathcal{V}}_{\mathbf{\alpha}}(\mathbf{z})\tilde{\mathcal{V}}_{\frac{\mathbf{\beta}}{2} }(\mathbf{s}) \rangle_{t} .\\
&-\sum_{\ell=1}^{k+\tilde{k}-1} \sum_{q=1}^{m}\frac{ Q\beta_q}{(x_{k+\tilde{k}}-x_{\ell})^3} (\frac{1}{x_{k+\tilde{k}}-s_q} ) \langle     \mathcal{T}(\mathbf{x}^{(k+\tilde{k},\ell )})   \tilde{\mathcal{V}}_{\mathbf{\alpha}}(\mathbf{z})\tilde{\mathcal{V}}_{\frac{\mathbf{\beta}}{2} }(\mathbf{s}) \rangle_{t} .\label{Cterms}
\end{align}
  
\paragraph{The $N$-term.}
\begin{lemma}
Consider a domain $D \subset \mathbb{H}$ and two function $f, g$ defined on $D .$ Then one has:
$$
\begin{aligned}
&\int_{D} \partial_{z} f(z) g(z) d z=-\int_{D} f(z) \partial_{z} g(z) d z+\frac{i}{2} \int_{\partial D} f(z) g(z) d \bar{z} \\
&\int_{D} \partial_{\bar{z}} f(z) g(z) d z=-\int_{D} f(z) \partial_{\bar{z}} g(z) d z-\frac{i}{2} \int_{\partial D} f(z) g(z) d z
\end{aligned}
$$
The integrations along $\partial D$ are in the counterclockwise direction.
\end{lemma} 
 The first step is to rewrite the $N$-contribution to make it cancel with some $C$-terms. For this, we regularize the vertex insertions  (besides the $\tilde{V}_\gamma$ insertion, we also regularize the $\tilde{V}_{\alpha_i}$'s for later need) 
in $N(\mathbf{u},\mathbf{z})$ given by \eqref{Nterm}, and performing an integration by parts (Green formula) in the $x$ integral  we get  
\begin{align*} 
 N^{\mu}(\mathbf{x},\partial) =& - \mu  \lim_{\epsilon\to 0} \int_{\H_t} \partial_x\frac{1}{x_{k+\tilde{k}}-x }\langle   \mathcal{T}(\mathbf{x}^{(k+\tilde{k})})  \tilde{V}_{\gamma,\epsilon}(x)\tilde{\mathcal{V}}_{\mathbf{\alpha},\epsilon}(\mathbf{z})\tilde{\mathcal{V}}_{\frac{\mathbf{\beta}}{2},\epsilon }(\mathbf{s}) \rangle_t \,\dd x\\
=& \lim_{\epsilon\to 0}B_{t,\epsilon}(\mathbf{x},\partial) + \mu \gamma  \lim_{\epsilon\to 0} \int_{\H_t}  \frac{1}{x_{k+\tilde{k}}-x } \langle   T(\mathbf{x}^{(k+\tilde{k})})  \partial_x\Phi_{\epsilon}(x)\tilde{V}_{\gamma,\epsilon}(x)\tilde{\mathcal{V}}_{\mathbf{\alpha},\epsilon}(\mathbf{z})\tilde{\mathcal{V}}_{\frac{\mathbf{\beta}}{2},\epsilon }(\mathbf{s})\rangle_t \,\dd x\\
:= &\lim_{\epsilon\to 0}B_{t,\epsilon}(\mathbf{x},\partial) + \tilde N^{\mu}(\mathbf{x},\partial)
  \end{align*} 
 where 
 \begin{align}\label{bryterm1}
B_{t,\epsilon}(\mathbf{x},\partial) :=&
 -\frac{i}{2}\mu\oint_{\partial U_t}
\frac{1}{x_{k+\tilde{k}}-x} \langle   \mathcal{T}(\mathbf{x}^{(k+\tilde{k})})  \tilde{V}_{\gamma,\epsilon}(x)\tilde{\mathcal{V}}_{\mathbf{\alpha}}(\mathbf{z})\tilde{\mathcal{V}}_{\frac{\mathbf{\beta}}{2} }(\mathbf{s}) \rangle_t \,\dd\bar x\\
&\textcolor{blue}{-\frac{i}{2}\mu\int_{\R_t}\frac{1}{x_{k+\tilde{k}}-s}\langle   T(\mathbf{u}^{(k)})  \tilde{V}_{\gamma,\epsilon}(s)\tilde{\mathcal{V}}_{\mathbf{\alpha}}(\mathbf{z})\tilde{\mathcal{V}}_{\frac{\mathbf{\beta}}{2} }(\mathbf{s}) \rangle_t \,\dd s}\\
:=&B_{t,\epsilon}^1(\mathbf{x},\partial)+\textcolor{blue}{B_{t,\epsilon}^2(\mathbf{x},\partial)}
\end{align}
and we write
\begin{align*}
\partial_x\tilde{V}_{\gamma,\epsilon}(x)=\gamma\partial_x\Phi_\epsilon(x)\tilde{V}_{\gamma,\epsilon}(x).
\end{align*}
\begin{remark}
    By fusion type estimate \eqref{fusion}, it's easy to see $B_{t,\epsilon}^1(\mathbf{x},\partial)$ has $\epsilon \to 0$ limit. The $B_{t,\epsilon}^2(\mathbf{x},\partial)$ limit should be problematic, but it cancels with $B_{t,\epsilon}^2(\mathbf{x},\Bar{\partial})$ at the fixed $\epsilon$ level.
\end{remark}

In $\tilde N^{\mu}(\mathbf{x},\partial)$ we integrate by parts the $ \partial_{x}\Phi_\epsilon(x)$ and end up with
  \begin{align}\nonumber
\tilde N^{\mu}(\mathbf{x},\partial)
=&  -\mu Q \gamma\sum_{\ell=1}^{k+\tilde{k}-1}  \int_{\H_t}  \frac{1}{(x_{k+\tilde{k}}-x)(x-x_\ell)^3 } \langle   \mathcal{T}(\mathbf{x}^{(\ell,k+\tilde{k})}) \tilde{V}_\gamma(x)\tilde{\mathcal{V}}_{\mathbf{\alpha}}(\mathbf{z})\tilde{\mathcal{V}}_{\frac{\mathbf{\beta}}{2} }(\mathbf{s}) \rangle_t \,\dd x
\\
& +\mu  \gamma\sum_{\ell=1}^{k+\tilde{k}-1}  \int_{\H_t}  \frac{1}{(x_{k+\tilde{k}}-x)(x-x_\ell)^2 } \langle \partial_{x_\ell}X(x_\ell)  \mathcal{T}(\mathbf{x}^{(\ell,k+\tilde{k})})  \tilde{V}_\gamma(x)\tilde{\mathcal{V}}_{\mathbf{\alpha}}(\mathbf{z})\tilde{\mathcal{V}}_{\frac{\mathbf{\beta}}{2} }(\mathbf{s}) \rangle_t \,\dd x \nonumber
\\
&  +\frac{\mu^2  \gamma^2  }{2}  \int_{\H_t} \int_{\H_t}  \frac{1}{(x_{k+\tilde{k}}-x)  } \Big(\frac{1}{x-x'}+\frac{1}{x-\bar{x'}}\Big)\langle  \mathcal{T}(\mathbf{x}^{(k+\tilde{k})})  \tilde{V}_\gamma(x)\tilde{V}_{\gamma}(x')\tilde{\mathcal{V}}_{\mathbf{\alpha}}(\mathbf{z})\tilde{\mathcal{V}}_{\frac{\mathbf{\beta}}{2} }(\mathbf{s}) \rangle_t \,\dd x\dd x' \nonumber
\\
&  +\frac{\mu\mu_{\partial} \gamma^2  }{2}  \int_{\H_t} \int_{\R_t}  \frac{1}{(x_{k+\tilde{k}}-x)  } \Big(\frac{1}{x-s}\Big)\langle  \mathcal{T}(\mathbf{x}^{(k+\tilde{k})})  \tilde{V}_\gamma(x)\tilde{V}_{\frac{\gamma}{2}}(s)\tilde{\mathcal{V}}_{\mathbf{\alpha}}(\mathbf{z})\tilde{\mathcal{V}}_{\frac{\mathbf{\beta}}{2} }(\mathbf{s}) \rangle_t \,\dd x \dd t\nonumber\\
&\textcolor{green}{+ {\mu  \gamma  }\lim_{\epsilon\to 0} \Big( \sum_{p=1}^n \int_{\H_t}  \frac{\alpha_p}{x_{k+\tilde{k}}-x }C_{\epsilon,0}(x,z_p) +\sum_{q=1}^m \int_{\H_t}  \frac{\beta_q}{2(x_{k+\tilde{k}}-x) }C_{\epsilon,0}(x,s_q) \Big) \langle  \mathcal{T}(\mathbf{x}^{(k+\tilde{k})})  \tilde{V}_\gamma(x)\tilde{\mathcal{V}}_{\mathbf{\alpha},\epsilon}(\mathbf{z})\tilde{\mathcal{V}}_{\frac{\mathbf{\beta}}{2},\epsilon }(\mathbf{s}) \rangle_t \,\dd x\nonumber}
\\
&\textcolor{red}{-\frac{\mu\gamma^2}{2}\int_{\H_t}\frac{1}{(x_{k+\tilde{k}}-x)}\frac{1}{x-\bar{x}}\langle   T(\mathbf{x}^{(k+\tilde{k})})  \tilde{V}_{\gamma}(x)\tilde{\mathcal{V}}_{\mathbf{\alpha}}(\mathbf{z})\tilde{\mathcal{V}}_{\frac{\mathbf{\beta}}{2} }(\mathbf{s})\rangle_t \,\dd x}
\\
=:&\textcolor{black}{C'_6(\mathbf{x},\mu,\partial)}+\textcolor{black}{C'_8(\mathbf{x},\mu,\partial)}+ \textcolor{black}{C'_{11}(\mathbf{x},\mu,\mu,\partial)}+ \textcolor{black}{C'_{11}(\mathbf{x},\mu,\mu_\partial,\partial)}+\textcolor{green}{C'_{14}(\mathbf{x},\mu,\partial)}+\textcolor{red}{T_{1}'(\mathbf{x},\partial)}\label{T4leq}
\end{align}
Similarly for $\bar{x}$ derivative, we have
\\
\begin{align*} 
 N^{\mu}(\mathbf{x},\bar{\partial}) =& - \mu  \lim_{\epsilon\to 0} \int_{\H_t} \partial_{\bar{x}}\frac{1}{x_{k+\tilde{k}}-\bar{x} }\langle   \mathcal{T}(\mathbf{x}^{(k+\tilde{k})})  \tilde{V}_{\gamma,\epsilon}(x)\tilde{\mathcal{V}}_{\mathbf{\alpha}}(\mathbf{z})\tilde{\mathcal{V}}_{\frac{\mathbf{\beta}}{2} }(\mathbf{s}) \rangle_t \,\dd x\\
=& B_{t,\epsilon}(\mathbf{x},\bar{\partial}) + \mu \gamma  \lim_{\epsilon\to 0} \int_{\H_t}  \frac{1}{x_{k+\tilde{k}}-\bar{x} } \langle   T(\mathbf{x}^{(k+\tilde{k})})  \partial_{\bar{x}}\Phi_{\epsilon}(x)\tilde{V}_{\gamma,\epsilon}(x)\tilde{\mathcal{V}}_{\mathbf{\alpha}}(\mathbf{z})\tilde{\mathcal{V}}_{\frac{\mathbf{\beta}}{2} }(\mathbf{s})\rangle_t \,\dd x\\:= &B_{t,\epsilon}(\mathbf{x},\bar{\partial}) + \tilde N^{\mu}(\mathbf{x},\bar{\partial})
  \end{align*} 
 where
 \begin{align}\label{bryterm2}
 B_{t,\epsilon}(\mathbf{x},\bar{\partial}) :=&
 +\frac{i}{2}\mu\oint_{\partial U_t}
\frac{1}{x_{k+\tilde{k}}-\bar{x}} \langle   \mathcal{T}(\mathbf{x}^{(k+\tilde{k})})  \tilde{V}_{\gamma,\epsilon}(x)\tilde{\mathcal{V}}_{\mathbf{\alpha}}(\mathbf{z})\tilde{\mathcal{V}}_{\frac{\mathbf{\beta}}{2} }(\mathbf{s})\rangle_t \,\dd x\\
&\textcolor{blue}{+\frac{i}{2}\mu\int_{\R_t}\frac{1}{x_{k+\tilde{k}}-s}\langle   T(\mathbf{u}^{(k)})  \tilde{V}_{\gamma,\epsilon}(s)\tilde{\mathcal{V}}_{\mathbf{\alpha}}(\mathbf{z})\tilde{\mathcal{V}}_{\frac{\mathbf{\beta}}{2} }(\mathbf{s}) \rangle_t \,\dd s}\\
&:=B_{t,\epsilon}^1(\mathbf{x},\bar{\partial})+\textcolor{blue}{B_{t,\epsilon}^2(\mathbf{x},\bar{\partial})}
\end{align}
Obviously, we can see that
\begin{equation}
    \textcolor{blue}{B_{t,\epsilon}^2(\mathbf{x},\partial)+B_{t,\epsilon}^2(\mathbf{x},\bar{\partial})=0}
\end{equation}
and we write
\begin{align*}
\partial_{\bar{x}}\tilde{V}_{\gamma,\epsilon}(x)=\gamma\partial_{\bar{x}}\Phi_\epsilon(x)\tilde{V}_{\gamma,\epsilon}(x).
\end{align*}

In $\tilde N^{\mu}(\mathbf{x},\partial)$ we integrate by parts the $ \partial_{\bar{x}}\Phi_\epsilon(x)$ and end up with
  \begin{align}\nonumber
\tilde N^{\mu}(\mathbf{x},\bar{\partial})
=&  -\mu Q \gamma\sum_{\ell=1}^{k+\tilde{k}-1}  \int_{\H_t}  \frac{1}{(x_{k+\tilde{k}}-\bar{x})(\bar{x}-x_\ell)^3 } \langle   \mathcal{T}(\mathbf{x}^{(\ell,k+\tilde{k})}) \tilde{V}_\gamma(x)\tilde{\mathcal{V}}_{\mathbf{\alpha}}(\mathbf{z})\tilde{\mathcal{V}}_{\frac{\mathbf{\beta}}{2} }(\mathbf{s})\rangle_t \,\dd x
\\
& +\mu  \gamma\sum_{\ell=1}^{k+\tilde{k}-1}  \int_{\H_t}  \frac{1}{(x_{k+\tilde{k}}-\bar{x})(\bar{x}-x_\ell)^2 } \langle \partial_{x_\ell}X(x_\ell)  \mathcal{T}(\mathbf{x}^{(\ell,k+\tilde{k})})  \tilde{V}_\gamma(x)\tilde{\mathcal{V}}_{\mathbf{\alpha}}(\mathbf{z})\tilde{\mathcal{V}}_{\frac{\mathbf{\beta}}{2} }(\mathbf{s}) \rangle_t \,\dd x \nonumber
\\
&  +\frac{\mu^2  \gamma^2  }{2}  \int_{\H_t} \int_{\H_t}  \frac{1}{(x_{k+\tilde{k}}-\bar{x})  } \Big(\frac{1}{\bar{x}-x'}+\frac{1}{\bar{x}-\bar{x'}}\Big)\langle  \mathcal{T}(\mathbf{x}^{(k+\tilde{k})})  \tilde{V}_\gamma(x)\tilde{V}_{\gamma}(x')\tilde{\mathcal{V}}_{\mathbf{\alpha}}(\mathbf{z})\tilde{\mathcal{V}}_{\frac{\mathbf{\beta}}{2} }(\mathbf{s}) \rangle_t \,\dd x\dd x' \nonumber
\\
&  +\frac{\mu\mu_{\partial} \gamma^2  }{2}  \int_{\H_t} \int_{\R_t}  \frac{1}{(x_{k+\tilde{k}}-\bar{x})  } \Big(\frac{1}{\bar{x}-s}\Big)\langle  \mathcal{T}(\mathbf{x}^{(k+\tilde{k})})  \tilde{V}_\gamma(x)\tilde{V}_{\frac{\gamma}{2}}(s)\tilde{\mathcal{V}}_{\mathbf{\alpha}}(\mathbf{z})\tilde{\mathcal{V}}_{\frac{\mathbf{\beta}}{2} }(\mathbf{s}) \rangle_t \,\dd x \dd t\nonumber\\
&\textcolor{green}{+ {\mu  \gamma  }\lim_{\epsilon\to 0}  \Big(\sum_{p=1}^n \int_{\H_t}  \frac{\alpha_p}{x_{k+\tilde{k}}-\bar{x} }C_{\epsilon,0}(\bar{x},z_p) +\sum_{q=1}^m \int_{\H_t}  \frac{\beta_q}{2(x_{k+\tilde{k}}-\bar{x}) }C_{\epsilon,0}(\bar{x},s_q) \Big) \langle  \mathcal{T}(\mathbf{x}^{(k+\tilde{k})})  \tilde{V}_\gamma(x)\tilde{\mathcal{V}}_{\mathbf{\alpha}}(\mathbf{z})\tilde{\mathcal{V}}_{\frac{\mathbf{\beta}}{2} }(\mathbf{s}) \rangle_t \,\dd x\nonumber}
\\
&\textcolor{red}{-\frac{\mu\gamma^2}{2}\int_{\H_t}\frac{1}{(x_{k+\tilde{k}}-\bar{x})}\frac{1}{\bar{x}-x}\langle   T(\mathbf{x}^{(k+\tilde{k})})  \tilde{V}_{\gamma}(x)\tilde{\mathcal{V}}_{\mathbf{\alpha}}(\mathbf{z})\tilde{\mathcal{V}}_{\frac{\mathbf{\beta}}{2} }(\mathbf{s})\rangle_t \,\dd x}
\\
=:&\textcolor{black}{C'_6(\mathbf{x},\mu,\bar{\partial})}+\textcolor{black}{C'_8(\mathbf{x},\mu,\bar{\partial})}+ \textcolor{black}{C'_{11}(\mathbf{x},\mu,\mu,\bar{\partial})}+ \textcolor{black}{C'_{11}(\mathbf{x},\mu,\mu_\partial,\bar{\partial})}+\textcolor{green}{C'_{14}(\mathbf{x},\mu,\bar{\partial})}+\textcolor{red}{T_{1}'(\mathbf{x},\bar{\partial})}\label{T4leq}
\end{align}
where again we took the  $\epsilon\to 0$ limit in the terms where it was obvious. In particular this identity proves that the limit on the RHS, denoted by $C'_{14}(\mathbf{u},\mathbf{z})$, exists. The numbering of these terms and the ones below will be used when comparing with \eqref{Cterms}.
Then we have 
\begin{align}
    \textcolor{red}{T'_{1}(\mathbf{x},\partial)+T'_{1}(\mathbf{x},\bar{\partial})+T_{1}(\mathbf{x})=0}
\end{align}
To deal with the $N^{\mu_{\partial}}(\mathbf{x})$, we need the following fusion estimate lemma,
\begin{lemma}
$$\lim_{s\to\infty} \frac{1}{x_{k+\tilde{k}}-s}\langle   \mathcal{T}(\mathbf{x}^{(k+\tilde{k})})  \tilde{V}_{\frac{\gamma}{2}}(s)\tilde{\mathcal{V}}_{\mathbf{\alpha}}(\mathbf{z})\tilde{\mathcal{V}}_{\frac{\mathbf{\beta}}{2} }(\mathbf{s})\rangle_t =0$$
\end{lemma}
\begin{proof}
By scaling relation \eqref{scaling}, we only need to estimate $\langle   V_{\frac{\gamma}{2}}(s)\tilde{\mathcal{V}}_{\mathbf{\alpha}}(\mathbf{z})\tilde{\mathcal{V}}_{\frac{\mathbf{\beta}}{2} }(\mathbf{s})\rangle_t$ but we know this term already goes to $0$ when $s\to \infty$
\end{proof}
We also note 
\begin{align}\label{Q_tterm}
    Q_t:=&\frac{\mu_\partial}{x_{k+\tilde{k}}-(s+e^{-t})}\langle   \mathcal{T}(\mathbf{x}^{(k+\tilde{k})})  \tilde{V}_{\frac{\gamma}{2}}(s+e^{-t})\tilde{\mathcal{V}}_{\mathbf{\alpha}}(\mathbf{z})\tilde{\mathcal{V}}_{\frac{\mathbf{\beta}}{2} }(\mathbf{s})\rangle_t\nonumber\\
    -&\frac{\mu_\partial}{x_{k+\tilde{k}}-(s-e^{-t})}\langle   \mathcal{T}(\mathbf{x}^{(k+\tilde{k})})  \tilde{V}_{\frac{\gamma}{2}}(s-e^{-t})\tilde{\mathcal{V}}_{\mathbf{\alpha}}(\mathbf{z})\tilde{\mathcal{V}}_{\frac{\mathbf{\beta}}{2} }(\mathbf{s})\rangle_t
\end{align}

Now for $N^{\mu_{\partial}}(\mathbf{x})$, we have the following,
\begin{align}\nonumber
&N^{\mu_\partial}(\mathbf{x})+Q_t=-\mu_{\partial} \int_{\R_t}\partial_s \frac{1}{x_{k+\tilde{k}}-s}\langle   \mathcal{T}(\mathbf{x}^{(k+\tilde{k})})  \tilde{V}_{\frac{\gamma}{2}}(s)\tilde{\mathcal{V}}_{\mathbf{\alpha}}(\mathbf{z})\tilde{\mathcal{V}}_{\frac{\mathbf{\beta}}{2} }(\mathbf{s})\rangle_t \,\dd s
\\
=&  -\mu_\partial Q \gamma\sum_{\ell=1}^{k+\tilde{k}-1}  \int_{\R_t}  \frac{1}{(x_{k+\tilde{k}}-s)(s-x_\ell)^3 } \langle   \mathcal{T}(\mathbf{x}^{(\ell,k+\tilde{k})}) \tilde{V}_{\frac{\gamma}{2}}(s)\tilde{\mathcal{V}}_{\mathbf{\alpha}}(\mathbf{z})\tilde{\mathcal{V}}_{\frac{\mathbf{\beta}}{2} }(\mathbf{s}) \rangle_t \,\dd x
\\
& +\mu_{\partial} \gamma\sum_{\ell=1}^{k+\tilde{k}-1}  \int_{\R_t}  \frac{1}{(x_{k+\tilde{k}}-s)(s-x_\ell)^2 } \langle \partial_{x_\ell}X(x_\ell)  \mathcal{T}(\mathbf{x}^{(\ell,k+\tilde{k})})  \tilde{V}_{\frac{\gamma}{2}}(s)\tilde{\mathcal{V}}_{\mathbf{\alpha}}(\mathbf{z})\tilde{\mathcal{V}}_{\frac{\mathbf{\beta}}{2} }(\mathbf{s}) \rangle_t \,\dd x \nonumber
\\
&  +\frac{\mu \mu_\partial  \gamma^2  }{2}  \int_{\R_t} \int_{\H_t}  \frac{1}{(x_{k+\tilde{k}}-s)  } \Big(\frac{1}{s-x}+\frac{1}{s-\bar{x}}\Big)\langle  \mathcal{T}(\mathbf{x}^{(k+\tilde{k})})  \tilde{V}_{\frac{\gamma}{2}}(s)\tilde{V}_{\gamma}(x)\tilde{\mathcal{V}}_{\mathbf{\alpha}}(\mathbf{z})\tilde{\mathcal{V}}_{\frac{\mathbf{\beta}}{2} }(\mathbf{s}) \rangle_t \,\dd x\dd s \nonumber
\\
&  +\frac{\mu_{\partial}^2 \gamma^2  }{2}  \int_{\R_t} \int_{\R_t}  \frac{1}{(x_{k+\tilde{k}}-s)  } \Big(\frac{1}{s-s'}\Big)\langle  \mathcal{T}(\mathbf{x}^{(k+\tilde{k})})  \tilde{V}_{\frac{\gamma}{2}}(s)\tilde{V}_{\frac{\gamma}{2}}(s')\tilde{\mathcal{V}}_{\mathbf{\alpha}}(\mathbf{z})\tilde{\mathcal{V}}_{\frac{\mathbf{\beta}}{2} }(\mathbf{s}) \rangle_t \,\dd s \dd s'\nonumber\\
&\textcolor{cyan}{+ {\mu_\partial  \frac{\gamma}{2}  }\lim_{\epsilon\to 0}  \Big(\sum_{p=1}^n \int_{\H_t}  \frac{\alpha_p}{x_{k+\tilde{k}}-s }C_{\epsilon,0}(s,z_p) +  \sum_{q=1}^m \int_{\H_t}  \frac{\beta_q}{2(x_{k+\tilde{k}}-s )}C_{\epsilon,0}(s,s_q) \Big)\langle  \mathcal{T}(\mathbf{x}^{(k+\tilde{k})})  \tilde{V}_{\frac{\gamma}{2}}(s)\tilde{\mathcal{V}}_{\mathbf{\alpha}}(\mathbf{z})\tilde{\mathcal{V}}_{\frac{\mathbf{\beta}}{2} }(\mathbf{s}) \rangle_t \,\dd s\nonumber}
\\
=:& \textcolor{black}{C'_6(\mathbf{x},\mu_\partial)}+\textcolor{black}{C'_8(\mathbf{x},\mu_\partial)}+ \textcolor{black}{C'_{11}(\mathbf{x},\mu,\mu_\partial)}+ \textcolor{black}{C'_{11}(\mathbf{x},\mu_\partial,\mu_\partial)}+\textcolor{cyan}{C'_{14}(\mathbf{x},\mu_\partial)}\label{T4leq}
\end{align}


 We want to compare the expression \eqref{IPPstress} to  the derivatives of the function $\langle   \mathcal{T}(\mathbf{x}^{(k+\tilde{k})})   \tilde{\mathcal{V}}_{\mathbf{\alpha}}(\mathbf{z})\tilde{\mathcal{V}}_{\frac{\mathbf{\beta}}{2} }(\mathbf{s}) \rangle_{t}$.  We have
\begin{lemma}\label{deriv2}
Let
\begin{align*}
I_\epsilon(\mathbf{x},\partial):=
\sum_{p=1}^n\frac{1}{x_{k+\tilde{k}}-z_p}&  \partial_{z_p}\langle   \mathcal{T}(\mathbf{x}^{(k+\tilde{k})}) \tilde{\mathcal{V}}_{\mathbf{\alpha},\epsilon}(\mathbf{z})\tilde{\mathcal{V}}_{\frac{\mathbf{\beta}}{2},\epsilon }(\mathbf{s}) \rangle_t .
\end{align*}  
\begin{align*}
I_\epsilon(\mathbf{x},\bar{\partial}):=
\sum_{p=1}^n\frac{1}{x_{k+\tilde{k}}-\bar{z_p}}&  \partial_{\bar{z_p}}\langle \mathcal{T}  (\mathbf{x}^{(k+\tilde{k})}) \tilde{\mathcal{V}}_{\mathbf{\alpha},\epsilon}(\mathbf{z})\tilde{\mathcal{V}}_{\frac{\mathbf{\beta}}{2},\epsilon }(\mathbf{s}) \rangle_t .
 \end{align*} 
 \begin{align*}
I_\epsilon(\mathbf{x},\beta):=
\sum_{q=1}^m\frac{1}{x_{k+\tilde{k}}-s_q}&  \partial_{s_q}\langle   \mathcal{T}(\mathbf{x}^{(k+\tilde{k})}) \tilde{\mathcal{V}}_{\mathbf{\alpha},\epsilon}(\mathbf{z})\tilde{\mathcal{V}}_{\frac{\mathbf{\beta}}{2} ,\epsilon}(\mathbf{s})\rangle_t .
  \end{align*} 
  Then all $\lim_{\epsilon\to 0}I_\epsilon:= I$
exists and defines a continuous function in $ \caO^{\rm ext}_{\H,U}$.
For $\varphi$ with compactly support in $\{\mathbf{z}\mid (\bf z,\bf s,\bf u,\bf v)\in \caO^{\rm ext}_{\H,U}\}$ and $\psi$ with compactly support in $\{\mathbf{s}\mid (\bf z,\bf s,\bf u,\bf v)\in \caO^{\rm ext}_{\H,U}\}$
we have
\begin{align}
    &\int I(x,\partial)\bar\varphi(\bf z) \dd\bf z=\int \langle   \mathcal{T}(\mathbf{x}^{(k+\tilde{k})}) \tilde{\mathcal{V}}_{\mathbf{\alpha},\epsilon}(\mathbf{z})\tilde{\mathcal{V}}_{\frac{\mathbf{\beta}}{2} ,\epsilon}(\mathbf{s})\rangle_t \hat D_{\partial}^*\bar\varphi(\bf z) \dd\bf z\\
    &\int I(x,\bar\partial)\bar\varphi(\bf z) \dd\bf z=\int \langle   \mathcal{T}(\mathbf{x}^{(k+\tilde{k})}) \tilde{\mathcal{V}}_{\mathbf{\alpha},\epsilon}(\mathbf{z})\tilde{\mathcal{V}}_{\frac{\mathbf{\beta}}{2} ,\epsilon}(\mathbf{s})\rangle_t \hat D_{\bar\partial}^*\bar\varphi(\bf z) \dd\bf z\\
    &\int I(x,\beta)\bar\psi(\bf s) \dd\bf s=\int \langle   \mathcal{T}(\mathbf{x}^{(k+\tilde{k})}) \tilde{\mathcal{V}}_{\mathbf{\alpha},\epsilon}(\mathbf{z})\tilde{\mathcal{V}}_{\frac{\mathbf{\beta}}{2} ,\epsilon}(\mathbf{s})\rangle_t \hat D_{\beta}^*\bar\psi(\bf s) \dd\bf s
    \end{align}
    where $\hat D_{\partial}=\sum_{p=1}^n\frac{1}{x_{k+\tilde{k}}-z_p}  \partial_{z_p}$, $ \hat D_{\bar\partial}=\sum_{p=1}^n\frac{1}{x_{k+\tilde{k}}-\bar{z_p}} \partial_{\bar{z_p}}$ and $\hat D_{\beta}^*\bar=\sum_{q=1}^m\frac{1}{x_{k+\tilde{k}}-s_q}  \partial_{s_q} $
\end{lemma}
\begin{proof}
We have
   \begin{align*}
I_\epsilon(\mathbf{x},\partial)=\sum_{p=1}^n\frac{\alpha_p}{x_{k+\tilde{k}}-z_p}&  \langle   \mathcal{T}(\mathbf{x}^{(k+\tilde{k})}) \partial_{z_p}\Phi_\epsilon(z_p)\tilde{\mathcal{V}}_{\mathbf{\alpha},\epsilon}(\mathbf{z})\tilde{\mathcal{V}}_{\frac{\mathbf{\beta}}{2},\epsilon }(\mathbf{s}) \rangle_t   = D_\epsilon(\mathbf{x},\partial)+L_\epsilon(\mathbf{x},\partial)
  \end{align*}  
  where we integrate by parts the $\partial_{z_p}\Phi_\epsilon(z_p)$ and $D_\epsilon({\bf x},\partial)$ collects the terms with an obvious $\epsilon\to 0$ limit :
    \begin{align*}
 D&(\mathbf{x},\partial)  =- \sum_{p=1}^n\sum_{\ell=1}^{k+\tilde{k}-1}\frac{Q\alpha_p}{(x_{k+\tilde{k}}-z_p)(z_p-x_\ell)^3}  \langle   \mathcal{T}(\mathbf{x}^{(\ell,k+\tilde{k})})  \tilde{\mathcal{V}}_{\mathbf{\alpha}}(\mathbf{z})\tilde{\mathcal{V}}_{\frac{\mathbf{\beta}}{2} }(\mathbf{s}) \rangle_t  
   \\
&+ \sum_{p=1}^n\sum_{\ell=1}^{k+\tilde{k}-1}\frac{\alpha_p}{(x_{k+\tilde{k}}-z_p)(z_p-x_\ell)^2}  \langle  \partial_{x_\ell}\Phi(x_\ell) \mathcal{T}(\mathbf{x}^{(\ell,k+\tilde{k})})  \tilde{\mathcal{V}}_{\mathbf{\alpha}}(\mathbf{z})\tilde{\mathcal{V}}_{\frac{\mathbf{\beta}}{2} }(\mathbf{s}) \rangle_t     
 \\
&+\Big(\sum_{p\not= p'=1}^n - \frac{\alpha_p\alpha_{p'}}{2}\frac{1}{(x_{k+\tilde{k}}-z_p)}(\frac{1}{z_p-z_p'} +\frac{1}{z_p-\bar{z_p'}}) +\sum_{p=1,q=1}^{n,m} - \frac{\alpha_p\beta_q}{2}\frac{1}{(x_{k+\tilde{k}}-z_p)}(\frac{1}{z_p-s_q})\Big)\langle  T(\mathbf{x}^{(k+\tilde{k})})  \tilde{\mathcal{V}}_{\mathbf{\alpha}}(\mathbf{z})\tilde{\mathcal{V}}_{\frac{\mathbf{\beta}}{2} }(\mathbf{s})\rangle_t
   \\
   &-\sum_{p=1}^n \frac{\alpha_p^2}{2}\frac{1}{(x_{k+\tilde{k}}-z_p)}\frac{1}{z_p-\bar{z_p}}\langle  T(\mathbf{x}^{(k+\tilde{k})})  \tilde{\mathcal{V}}_{\mathbf{\alpha}}(\mathbf{z})\tilde{\mathcal{V}}_{\frac{\mathbf{\beta}}{2} }(\mathbf{s}) \rangle_t\\
:=&D_5(\mathbf{x},\partial)+D_9(\mathbf{x},\partial)+D_{10}(\mathbf{x},\partial)+T_2(\mathbf{x},\partial)
  \end{align*}  
whereas 
\begin{align*}
 L_\epsilon(\mathbf{x},\partial)
&=\textcolor{green}{-\mu\gamma\sum_{p=1}^n  
\alpha_p\int_{\H_t}
\frac{1}{x_{k+\tilde{k}}-z_p}C_{\epsilon,0}(z_p,x)  \langle  T(\mathbf{x}^{(k+\tilde{k})}) \tilde{V}_\gamma(x)\tilde{\mathcal{V}}_{\mathbf{\alpha},\epsilon}(\mathbf{z})\tilde{\mathcal{V}}_{\frac{\mathbf{\beta}}{2} ,\epsilon}(\mathbf{s}) \rangle_t     \,\dd x   .}\\
&\textcolor{cyan}{-\mu_{\partial}\frac{\gamma}{2}\sum_{p=1}^n 
\alpha_p\int_{\R_t}
\frac{1}{x_{k+\tilde{k}}-z_p}C_{\epsilon,0}(z_p,s)  \langle  T(\mathbf{x}^{(k+\tilde{k})}) \tilde{V}_\frac{\gamma}{2}(s)\tilde{\mathcal{V}}_{\mathbf{\alpha},\epsilon}(\mathbf{z})\tilde{\mathcal{V}}_{\frac{\mathbf{\beta}}{2} ,\epsilon}(\mathbf{s}) \rangle_t     \,\dd s} \\
  &:=\textcolor{green}{ L_\epsilon(\mathbf{x},\mu,\partial)}+ \textcolor{cyan}{L_\epsilon(\mathbf{x},\mu_\partial,\partial)}\\
  \end{align*}  \\
And, similarly, we derive the formula for $\bar{z_p}$
  \begin{align*}
I_\epsilon(\mathbf{x},\bar{\partial})=\sum_{p=1}^n\frac{\alpha_p}{x_{k+\tilde{k}}-\bar{z_p}}&  \langle   \mathcal{T}(\mathbf{x}^{(k+\tilde{k})}) \partial_{\bar{z_p}}\Phi_\epsilon(z_p)\tilde{\mathcal{V}}_{\mathbf{\alpha},\epsilon}(\mathbf{z})\tilde{\mathcal{V}}_{\frac{\mathbf{\beta}}{2} ,\epsilon}(\mathbf{s})\rangle_t   = D_\epsilon(\mathbf{x},\bar{\partial})+L_\epsilon(\mathbf{x},\bar{\partial})
  \end{align*}  
  where we integrate by parts the $\partial_{\bar{z_p}}\Phi_\epsilon(z_p)$ and $D_\epsilon({\bf x},\bar{\partial})$ collects the terms with an obvious $\epsilon\to 0$ limit:
    \begin{align*}
 D(\mathbf{x},\bar{\partial})  =&- \sum_{p=1}^n\sum_{\ell=1}^{k+\tilde{k}-1}\frac{Q\alpha_p}{(x_{k+\tilde{k}}-\bar{z_p})(\bar{z_p}-x_\ell)^3}  \langle   \mathcal{T}(\mathbf{x}^{(\ell,k+\tilde{k})})  \tilde{\mathcal{V}}_{\mathbf{\alpha}}(\mathbf{z})\tilde{\mathcal{V}}_{\frac{\mathbf{\beta}}{2} }(\mathbf{s}) \rangle_t   
   \\
&+ \sum_{p=1}^n\sum_{\ell=1}^{k+\tilde{k}-1}\frac{\alpha_p}{(x_{k+\tilde{k}}-\bar{z_p})(\bar{z_p}-x_\ell)^2}  \langle  \partial_{x_\ell}\Phi(x_\ell) \mathcal{T}(\mathbf{x}^{(\ell,k+\tilde{k})})  \tilde{\mathcal{V}}_{\mathbf{\alpha}}(\mathbf{z})\tilde{\mathcal{V}}_{\frac{\mathbf{\beta}}{2} }(\mathbf{s}) \rangle_t     
 \\
&-\Big(\sum_{p\not= p'=1}^n  \frac{\alpha_p\alpha_{p'}}{2}\frac{1}{(x_{k+\tilde{k}}-\bar{z_p})}(\frac{1}{\bar{z_p}-z_p'} +\frac{1}{\bar{z_p}-\bar{z_p'}}) +\sum_{p=1,q=1}^{n,m} \frac{\alpha_p\beta_q}{2}\frac{1}{(x_{k+\tilde{k}}-\bar{z_p})}(\frac{1}{\bar{z_p}-s_q})\Big)\langle  \mathcal{T}(\mathbf{x}^{(k+\tilde{k})})  \tilde{\mathcal{V}}_{\mathbf{\alpha}}(\mathbf{z})\tilde{\mathcal{V}}_{\frac{\mathbf{\beta}}{2} }(\mathbf{s}) \rangle_t
   \\
   &-\sum_{p=1}^n \frac{\alpha_p^2}{2}\frac{1}{(x_{k+\tilde{k}}-\bar{z_p})}\frac{1}{\bar{z_p}-z_p}\langle  T(\mathbf{x}^{(k+\tilde{k})})  \tilde{\mathcal{V}}_{\mathbf{\alpha}}(\mathbf{z})\tilde{\mathcal{V}}_{\frac{\mathbf{\beta}}{2} }(\mathbf{s}) \rangle_t\\
:=&D_5(\mathbf{x},\bar{\partial})+D_9(\mathbf{x},\bar{\partial})+D_{10}(\mathbf{x},\bar{\partial})+T_2(\mathbf{x},\bar{\partial})
  \end{align*}  
whereas 
\begin{align*}
 L_\epsilon(\mathbf{x},\bar{\partial})
&=\textcolor{green}{-\mu\gamma\sum_{p=1}^n  
\alpha_p\int_{\H_t}
\frac{1}{x_{k+\tilde{k}}-z_p}C_{\epsilon,0}(z_p,x)  \langle  \mathcal{T}(\mathbf{x}^{(k+\tilde{k})}) \tilde{V}_\gamma(x)\tilde{\mathcal{V}}_{\mathbf{\alpha},\epsilon}(\mathbf{z})\tilde{\mathcal{V}}_{\frac{\mathbf{\beta}}{2} ,\epsilon}(\mathbf{s}) \rangle_t     \,\dd x   }\\
&\textcolor{cyan}{-\mu_{\partial}\frac{\gamma}{2}\sum_{p=1}^n 
\alpha_p\int_{\R_t}
\frac{1}{x_{k+\tilde{k}}-z_p}C_{\epsilon,0}(z_p,s)  \langle  \mathcal{T}(\mathbf{x}^{(k+\tilde{k})}) \tilde{V}_\frac{\gamma}{2}(s)\tilde{\mathcal{V}}_{\mathbf{\alpha},\epsilon}(\mathbf{z})\tilde{\mathcal{V}}_{\frac{\mathbf{\beta}}{2} ,\epsilon}(\mathbf{s}) \rangle_t     \,\dd s} \\
  &:=\textcolor{green}{L_\epsilon(\mathbf{x},\mu,\bar{\partial})} + \textcolor{cyan}{L_\epsilon(\mathbf{x},\mu_\partial,\bar{\partial})}\\
  \end{align*}  \\
  And, similarly, we derive the formula for $s_q$
  \begin{align*}
I_\epsilon(\mathbf{x},\beta)=\sum_{q=1}^m\frac{\beta_q}{2(x_{k+\tilde{k}}-s_q)}&  \langle   \mathcal{T}(\mathbf{x}^{(k+\tilde{k})}) \partial_{s_q}\Phi_\epsilon(s_q)\tilde{\mathcal{V}}_{\mathbf{\alpha},\epsilon}(\mathbf{z})\tilde{\mathcal{V}}_{\frac{\mathbf{\beta}}{2} ,\epsilon}(\mathbf{s}) \rangle_t   = D_\epsilon(\mathbf{x},\beta)+L_\epsilon(\mathbf{x},\beta)
  \end{align*}  
  where we integrate by parts the $\partial_{s_q}\Phi_\epsilon(s_q)$ and $D_\epsilon({\bf x},\beta)$ collects the terms with an obvious $\epsilon\to 0$ limit:
    \begin{align*}
 D(\mathbf{x},\beta)  =&- \sum_{q=1}^m\sum_{\ell=1}^{k+\tilde{k}-1}\frac{Q\beta_q}{(x_{k+\tilde{k}}-s_q)(s_q-x_\ell)^3}  \langle   \mathcal{T}(\mathbf{x}^{(\ell,k+\tilde{k})})  \tilde{\mathcal{V}}_{\mathbf{\alpha}}(\mathbf{z})\tilde{\mathcal{V}}_{\frac{\mathbf{\beta}}{2} }(\mathbf{s}) \rangle_t   
   \\
&+ \sum_{q=1}^m\sum_{\ell=1}^{k+\tilde{k}-1}\frac{\beta_q}{(x_{k+\tilde{k}}-s_q)(s_q-x_\ell)^2}  \langle  \partial_{x_\ell}\Phi(x_\ell) \mathcal{T}(\mathbf{x}^{(\ell,k+\tilde{k})})  \tilde{\mathcal{V}}_{\mathbf{\alpha}}(\mathbf{z})\tilde{\mathcal{V}}_{\frac{\mathbf{\beta}}{2} }(\mathbf{s}) \rangle_t     
 \\
& -\Big(\sum_{p=1,q=1}^{n,m} \frac{\alpha_p\beta_q}{2}\frac{1}{(x_{k+\tilde{k}}-s_q)}(\frac{1}{s_q-z_p} +\frac{1}{s_q-\bar{z_p}}) +\sum_{q\not q'=1}^{m} \frac{\beta_q\beta_{q'}}{2}\frac{1}{(x_{k+\tilde{k}}-s_q)}(\frac{1}{s_q-s_{q'}})\Big) \langle \mathcal{T} (\mathbf{x}^{(k+\tilde{k})})  \tilde{\mathcal{V}}_{\mathbf{\alpha}}(\mathbf{z})\tilde{\mathcal{V}}_{\frac{\mathbf{\beta}}{2} }(\mathbf{s}) \rangle_t
   \\
:=&D_5(\mathbf{x},\beta)+D_9(\mathbf{x},\beta)+D_{10}(\mathbf{x},\beta)
  \end{align*}  
whereas 
\begin{align*}
 L_\epsilon(\mathbf{x},\beta)
&=\textcolor{green}{-\mu\gamma\sum_{q=1}^m  
\frac{\beta_q}{2}\int_{\H_t}
\frac{1}{x_{k+\tilde{k}}-s_q}C_{\epsilon,0}(s_q,x)  \langle  \mathcal{T}(\mathbf{x}^{(k+\tilde{k})}) \tilde{V}_\gamma(x)\tilde{\mathcal{V}}_{\mathbf{\alpha},\epsilon}(\mathbf{z})\tilde{\mathcal{V}}_{\frac{\mathbf{\beta}}{2} ,\epsilon}(\mathbf{s})\rangle_t     \,\dd x}   \\
&\textcolor{cyan}{-\mu_{\partial}\frac{\gamma}{2}\sum_{q=1}^m 
\frac{\beta_q}{2}\int_{\R_t}
\frac{1}{x_{k+\tilde{k}}-s_q}C_{\epsilon,0}(s_q,s)  \langle  \mathcal{T}(\mathbf{x}^{(k+\tilde{k})}) \tilde{V}_\frac{\gamma}{2}(s)\tilde{\mathcal{V}}_{\mathbf{\alpha},\epsilon}(\mathbf{z})\tilde{\mathcal{V}}_{\frac{\mathbf{\beta}}{2} ,\epsilon}(\mathbf{s}) \rangle_t     \,\dd s }\\
  &:= \textcolor{green}{L_\epsilon(\mathbf{x},\mu,\beta)}+ \textcolor{cyan}{L_\epsilon(\mathbf{x},\mu_\partial,\beta)}\\
  \end{align*}  \\
Since $C_{0,0}(z_p,x) =-\frac{1}{2}\Big(\frac{1}{z_p-x}+\frac{1}{z_p-\bar{x}}\Big)$  and since it is not clear that $\frac{1}{z_p-x} \langle  \mathcal{T}(\mathbf{u}^{(k+\tilde{k})})  V_\gamma(x)\prod_{i=1}^nV_{\alpha_i }(z_i) \rangle_t    $ is integrable the $\epsilon\to 0$ limit of $L_\epsilon$ is problematic. However, we can compare it with the term $C'_{14}$ in \eqref{T4leq} and $C_{14}$.  Since 
\begin{align}
    &-\frac{1}{u_k-x}(\frac{1}{x-z_p}+\frac{1}{x-\bar{z_p}})-\frac{1}{u_k-\bar{x}}(\frac{1}{\bar{x}-z_p}+\frac{1}{\bar{x}-\bar{z_p}})+(\frac{1}{u_k-z_p}+\frac{1}{u_k-\bar{z_p}})(\frac{1}{u_k-x}+\frac{1}{u_k-\bar{x}})\nonumber\\=&\frac{1}{u_k-x}(\frac{x-u_k}{(u_k-z_p)(x-z_p)}+\frac{x-u_k}{(u_k-\bar{z_p})(x-\bar{z_p})})+\frac{1}{u_k-\bar{x}}(\frac{\bar{x}-u_k}{(u_k-z_p)(\bar{x}-z_p)}+\frac{\bar{x}-u_k}{(u_k-\bar{z_p})(\bar{x}-\bar{z_p})})\nonumber\\=&\frac{1}{u_k-z_p}(\frac{1}{z_p-x}+\frac{1}{z_p-\bar{x}})+\frac{1}{u_k-\bar{z_p}}(\frac{1}{\bar{z_p}-x}+\frac{1}{\bar{z_p}-\bar{x}})\nonumber\\
\end{align}
we conclude that $ L_\epsilon$ converges:
\begin{align*}
\textcolor{green}{\lim_{\epsilon\to 0} \Big(L_\epsilon(\mathbf{x},\mu,\partial)+L_\epsilon(\mathbf{x},\mu,\bar{\partial})+(L_\epsilon(\mathbf{x},\mu,\beta)\Big)
=C_{14}(\mathbf{x},\mu)+C'_{14}(\mathbf{x},\mu,\partial)+C'_{14}(\mathbf{x},\mu,\bar{\partial})}.
\end{align*}  
\begin{align*}
\textcolor{cyan}{\lim_{\epsilon\to 0} \Big(L_\epsilon(\mathbf{x},\mu_\partial,\partial)+L_\epsilon(\mathbf{x},\mu_\partial,\bar{\partial})+(L_\epsilon(\mathbf{x},\mu_{\partial},\beta)\Big)
=C_{14}(\mathbf{x},\mu_\partial)+C'_{14}(\mathbf{x},\mu_\partial).}
\end{align*}  
Here we also note that 
\begin{align}
M_2(\mathbf{x},\alpha)=T_2(\mathbf{x},\partial)+T_2(\mathbf{x},\bar{\partial})
\end{align}
\end{proof}
Now we turn to the derivatives of the SET-insertions. Let 
\begin{equation}\label{Sterm}
S(\mathbf{x}):=\sum_{\ell=1}^{k+\tilde{k}-1}\frac{1}{x_{k+\tilde{k}}-x_{\ell}}\partial_{x_{\ell}}\langle   \mathcal{T}(\mathbf{x}^{(k+\tilde{k})}) \tilde{\mathcal{V}}_{\mathbf{\alpha}}(\mathbf{z})\tilde{\mathcal{V}}_{\frac{\mathbf{\beta}}{2} }(\mathbf{s}) \rangle_{\mathcal{D}}.
\end{equation}
Then we claim that
\begin{align*}
S(\mathbf{x})=\sum_{i=1}^{}S_i(x)
\end{align*}
 with the terms $S_i$ given by
\begin{align*}
\textcolor{black}{S_1(\mathbf{x})}:=& - \sum_{\ell\not=\ell'=1}^{k+\tilde{k}-1}\frac{12 Q^2}{(x_{k+\tilde{k}}-x_{\ell})(x_\ell-x_{\ell'})^5} \langle   \mathcal{T}(\mathbf{x}^{(k+\tilde{k},\ell,\ell')})  \tilde{\mathcal{V}}_{\mathbf{\alpha}}(\mathbf{z})\tilde{\mathcal{V}}_{\frac{\mathbf{\beta}}{2} }(\mathbf{s}) \rangle_{t}  
\\
\textcolor{black}{S_{3}(\mathbf{x})}:=&  -6Q  \sum_{\ell\not=\ell'=1}^{k+\tilde{k}-1}   \frac{ 1}{(x_{k+\tilde{k}}-x_{\ell})(x_{\ell}- x_{\ell'})^4} \langle   \partial_{x_{\ell}}\Phi(x_{\ell}) \mathcal{T}(\mathbf{x}^{(k+\tilde{k} ,\ell,\ell')})    \tilde{\mathcal{V}}_{\mathbf{\alpha}}(\mathbf{z})\tilde{\mathcal{V}}_{\frac{\mathbf{\beta}}{2} }(\mathbf{s}) \rangle_{t} 
\\
\textcolor{black}{S_{5}(\mathbf{x})}:=& - \sum_{\ell=1}^{k+\tilde{k}-1} \sum_{p=1}^n    \frac{ Q\alpha_p}{(x_{k+\tilde{k}}-x_{\ell})}(\frac{1}{(x_\ell-z_p)^3}+\frac{1}{(x_\ell-\bar{z_p})^3}) \langle  \mathcal{T}(\mathbf{x}^{(k+\tilde{k},\ell )})   \tilde{\mathcal{V}}_{\mathbf{\alpha}}(\mathbf{z})\tilde{\mathcal{V}}_{\frac{\mathbf{\beta}}{2} }(\mathbf{s}) \rangle_{t}
\\
&- \sum_{\ell=1}^{k+\tilde{k}-1} \sum_{q=1}^m    \frac{ Q\beta_q}{(x_{k+\tilde{k}}-x_{\ell})}(\frac{1}{(x_\ell-s_q)^3}) \langle  \mathcal{T}(\mathbf{x}^{(k+\tilde{k},\ell )})   \tilde{\mathcal{V}}_{\mathbf{\alpha}}(\mathbf{z})\tilde{\mathcal{V}}_{\frac{\mathbf{\beta}}{2} }(\mathbf{s}) \rangle_{t}
\\
\textcolor{black}{S_{6}(\mathbf{x},\mu)}:=&  Q\mu\gamma  \sum_{\ell=1}^{k+\tilde{k}-1}  \int_{\H_t} \frac{ 1}{(x_{k+\tilde{k}}-x_{\ell})} (\frac{1}{(x_{\ell}- x)^3}+\frac{1}{(x_\ell-\bar{x})^3})\langle  \mathcal{T}(\mathbf{x}^{(k+\tilde{k} ,\ell)})   \tilde{V}_\gamma(x)\tilde{\mathcal{V}}_{\mathbf{\alpha}}(\mathbf{z})\tilde{\mathcal{V}}_{\frac{\mathbf{\beta}}{2} }(\mathbf{s})\rangle_{t}\,\dd x
\\
\textcolor{black}{S_{6}(\mathbf{x},\mu_\partial)}:=&  Q\mu_\partial\gamma  \sum_{\ell=1}^{k+\tilde{k}-1}  \int_{\R_t} \frac{ 1}{(x_{k+\tilde{k}}-x_{\ell})} (\frac{1}{(x_{\ell}- s)^3}\langle  \mathcal{T}(\mathbf{x}^{(k+\tilde{k} ,\ell)})   \tilde{V}_{\frac{\gamma}{2}}(s)\tilde{\mathcal{V}}_{\mathbf{\alpha}}(\mathbf{z})\tilde{\mathcal{V}}_{\frac{\mathbf{\beta}}{2} }(\mathbf{s}) \rangle_{t}\,\dd s
\\
\textcolor{black}{S_7(\mathbf{x})  }:=& 6Q \sum_{\ell\not=\ell'=1}^{k+\tilde{k}-1}\frac{1}{(x_{k+\tilde{k}}-x_{\ell})(x_\ell-x_{\ell'})^4} \langle  \partial_{x_\ell'}\Phi(x_{\ell'}) \mathcal{T}(\mathbf{x}^{(k+\tilde{k},\ell,\ell')})  \tilde{\mathcal{V}}_{\mathbf{\alpha}}(\mathbf{z})\tilde{\mathcal{V}}_{\frac{\mathbf{\beta}}{2} }(\mathbf{s}) \rangle_{t}
\\
\textcolor{black}{S_{8}(\mathbf{x},\mu)}:=& + \mu\gamma \sum_{\ell=1}^{k+\tilde{k}-1}  \int_{\H_t} \frac{ 1}{(x_{k+\tilde{k}}-x_{\ell})}(\frac{1}{(x_{\ell}- x)^2}+ \frac{1}{(x_{\ell}- \bar{x})^2})\langle   \partial_{x_\ell}\Phi(x_{\ell} )\mathcal{T}(\mathbf{x}^{(k+\tilde{k},\ell )})   \tilde{V}_\gamma(x)\tilde{\mathcal{V}}_{\mathbf{\alpha}}(\mathbf{z})\tilde{\mathcal{V}}_{\frac{\mathbf{\beta}}{2} }(\mathbf{s}) \rangle_{t}\,\dd x
\\
\textcolor{black}{S_{8}(\mathbf{x},\mu_\partial)}:=& + \mu_\partial\gamma \sum_{\ell=1}^{k+\tilde{k}-1}  \int_{\R_t} \frac{ 1}{(x_{k+\tilde{k}}-x_{\ell})}(\frac{1}{(x_{\ell}- s)^2})\langle   \partial_{x_\ell}\Phi(x_{\ell} )\mathcal{T}(\mathbf{x}^{(k+\tilde{k},\ell )})   \tilde{V}_{\frac{\gamma}{2}}(s)\tilde{\mathcal{V}}_{\mathbf{\alpha}}(\mathbf{z})\tilde{\mathcal{V}}_{\frac{\mathbf{\beta}}{2} }(\mathbf{s}) \rangle_{t}\,\dd s
\\
\textcolor{black}{S_{9}(\mathbf{x})}:=& -  \sum_{\ell=1}^{k+\tilde{k}-1}\sum_{p=1}^n  \frac{\alpha_p}{(x_{k+\tilde{k}}-x_{\ell})}(\frac{1}{(x_{\ell}-z_p)^2}+\frac{1}{(x_{\ell}-\bar{z_p})^2}) \langle  \partial_{x_{\ell}}\Phi(x_{\ell} ) T(\mathbf{x}^{(k+\tilde{k},\ell  )})   \tilde{\mathcal{V}}_{\mathbf{\alpha}}(\mathbf{z})\tilde{\mathcal{V}}_{\frac{\mathbf{\beta}}{2} }(\mathbf{s}) \rangle_{t} 
\\
&-  \sum_{\ell=1}^{k+\tilde{k}-1}\sum_{q=1}^m  \frac{\beta_q}{(x_{k+\tilde{k}}-x_{\ell})}(\frac{1}{(x_{\ell}-s_q)^2}) \langle  \partial_{x_{\ell}}\Phi(x_{\ell} ) T(\mathbf{x}^{(k+\tilde{k},\ell  )})    \tilde{\mathcal{V}}_{\mathbf{\alpha}}(\mathbf{z})\tilde{\mathcal{V}}_{\frac{\mathbf{\beta}}{2} }(\mathbf{s})\rangle_{t} \\
\textcolor{black}{S_{12}(\mathbf{x})}:=&  4 \sum_{\ell\not=\ell'=1}^{k+\tilde{k}-1}   \frac{ 1}{(x_{k+\tilde{k}}-x_{\ell})( x_\ell-x_{\ell'})^3} \langle  \partial_{x_\ell}\Phi(x_{\ell} )\partial_{x_{\ell'}}\Phi(x_{\ell'} )\mathcal{T}(\mathbf{x}^{(k+\tilde{k},\ell,\ell' )})  \tilde{\mathcal{V}}_{\mathbf{\alpha}}(\mathbf{z})\tilde{\mathcal{V}}_{\frac{\mathbf{\beta}}{2} }(\mathbf{s})\rangle_{t} 
\end{align*}
Again, the numbering will be used to compare with \eqref{Cterms}. Also, to establish this formula, we can regularize the SET-insertion $T(u_\ell)$, differentiate it, then use Gaussian integration by parts and then pass to the limit as $\epsilon\to 0$. Notice that the convergence of all these terms is obvious as variables $u_\ell$ for $\ell=1,\dots,k+\tilde{k}-1$ belong to $\mathcal{D}$. The same strategy can be applied to establish that
\begin{equation}\label{Lterm}
L(\mathbf{x}):=\sum_{\ell=1}^{k+\tilde{k}-1}\frac{2}{(x_{k+\tilde{k}}-x_{\ell})^2} \langle   \mathcal{T}(\mathbf{x}^{(k+\tilde{k})})    \tilde{\mathcal{V}}_{\mathbf{\alpha}}(\mathbf{z})\tilde{\mathcal{V}}_{\frac{\mathbf{\beta}}{2} }(\mathbf{s}) \rangle_{t}
\end{equation}
can be written as  $L(\mathbf{x} )=\sum_{i=1}^{}L_i(\mathbf{x})$ with the terms $L_i$'s given by
\begin{align*}
\textcolor{black}{ L_1(\mathbf{x})} :=& \sum_{\ell\not=\ell'=1}^{k+\tilde{k}-1}\frac{6Q^2}{(x_{k+\tilde{k}}-x_\ell)^2(x_\ell-x_{\ell'})^4} \langle  \mathcal{T} (\mathbf{x}^{(k+\tilde{k},\ell,\ell')}) \tilde{\mathcal{V}}_{\mathbf{\alpha}}(\mathbf{z})\tilde{\mathcal{V}}_{\frac{\mathbf{\beta}}{2} }(\mathbf{s}) \rangle_{t}  
\\
\textcolor{black}{L_3(\mathbf{x}) }:=& -2 Q  \sum_{\ell\not=\ell'=1}^{k+\tilde{k}-1} \frac{1}{(x_{k+\tilde{k}}-x_\ell)^2(x_\ell-x_{\ell'})^3} \langle      \partial_{x_\ell}\Phi(x_{\ell} )   \mathcal{T}(\mathbf{x}^{(k+\tilde{k},\ell ,\ell')})    \tilde{\mathcal{V}}_{\mathbf{\alpha}}(\mathbf{z})\tilde{\mathcal{V}}_{\frac{\mathbf{\beta}}{2} }(\mathbf{s}) \rangle_{t}   
\\
\textcolor{black}{L_5(\mathbf{x}) }:=&\sum_{\ell=1}^{k+\tilde{k}-1} \sum_{p=1}^n \frac{Q\alpha_p}{(x_{k+\tilde{k}}-x_{\ell})^2}  (\frac{1}{(x_\ell-z_p)^2} +\frac{1}{(x_\ell-\bar{z_p})^2})\langle      \mathcal{T}(\mathbf{x}^{(k+\tilde{k},\ell )})   \tilde{\mathcal{V}}_{\mathbf{\alpha}}(\mathbf{z})\tilde{\mathcal{V}}_{\frac{\mathbf{\beta}}{2} }(\mathbf{s}) \rangle_{t}  
\\
&+\sum_{\ell=1}^{k+\tilde{k}-1} \sum_{q=1}^m \frac{Q\beta_q}{(x_{k+\tilde{k}}-x_{\ell})^2}  (\frac{1}{(x_\ell-s_q)^2} )\langle      \mathcal{T}(\mathbf{x}^{(k+\tilde{k},\ell )})   \tilde{\mathcal{V}}_{\mathbf{\alpha}}(\mathbf{z})\tilde{\mathcal{V}}_{\frac{\mathbf{\beta}}{2} }(\mathbf{s}) \rangle_{t}  
\\
\textcolor{black}{L_{6}(\mathbf{x} ,\mu)} :=& -\mu\gamma Q\sum_{\ell=1}^{k+\tilde{k}-1} \int_{\H_t} \frac{1}{(x_{k+\tilde{k}}-x_\ell)^2} (\frac{1}{(x_\ell-x)^2}+\frac{1}{(x_\ell-\bar{x})^2})\langle      \mathcal{T}(\mathbf{x}^{(k+\tilde{k},\ell )})  \tilde{V}_\gamma(x)\tilde{\mathcal{V}}_{\mathbf{\alpha}}(\mathbf{z})\tilde{\mathcal{V}}_{\frac{\mathbf{\beta}}{2} }(\mathbf{s}) \rangle_{t}  \,\dd x  
\\
\textcolor{black}{L_{6}(\mathbf{x} ,\mu_\partial)} :=& -\mu_\partial\gamma Q\sum_{\ell=1}^{k+\tilde{k}-1} \int_{\R_t} \frac{1}{(x_{k+\tilde{k}}-x_\ell)^2} (\frac{1}{(x_\ell-s)^2})\langle      \mathcal{T}(\mathbf{x}^{(k+\tilde{k},\ell )})  \tilde{V}_{\frac{\gamma}{2}}(s)\tilde{\mathcal{V}}_{\mathbf{\alpha}}(\mathbf{z})\tilde{\mathcal{V}}_{\frac{\mathbf{\beta}}{2} }(\mathbf{s}) \rangle_{t}  \,\dd s  
\\
\textcolor{black}{ L_7(\mathbf{x})} :=&-4Q \sum_{\ell\not=\ell'=1}^{k+\tilde{k}-1}\frac{1}{(x_{k+\tilde{k}}-x_\ell)^2(x_\ell-x_{\ell'})^3}     \langle     \partial_{x_\ell'}\Phi(x_{\ell'} )\mathcal{T}(\mathbf{x}^{(k+\tilde{k},\ell,\ell')})   \tilde{\mathcal{V}}_{\mathbf{\alpha}}(\mathbf{z})\tilde{\mathcal{V}}_{\frac{\mathbf{\beta}}{2} }(\mathbf{s})\rangle_{t}  
\\
\textcolor{black}{L_{8}(\mathbf{x},\mu )} :=&-\mu\gamma   \sum_{\ell=1}^{k+\tilde{k}-1}\int_{\H_t} \frac{1}{(x_{k+\tilde{k}}-x_{\ell})^2}(\frac{1}{(x_{\ell}-x)}+\frac{1}{(x_{\ell}-\bar{x})}) \langle      \partial_{x_\ell}\Phi(x_{\ell} )    \mathcal{T}(\mathbf{x}^{(k+\tilde{k} ,\ell)})   \tilde{V}_\gamma(x)\tilde{\mathcal{V}}_{\mathbf{\alpha}}(\mathbf{z})\tilde{\mathcal{V}}_{\frac{\mathbf{\beta}}{2} }(\mathbf{s})\rangle_{t}  \,\dd x
\\
\textcolor{black}{L_{8}(\mathbf{x},\mu_\partial )} :=&-\mu_\partial\gamma   \sum_{\ell=1}^{k+\tilde{k}-1}\int_{\R_t} \frac{1}{(x_{k+\tilde{k}}-x_{\ell})^2}(\frac{1}{(x_{\ell}-s)}) \langle      \partial_{x_\ell}\Phi(x_{\ell} )    \mathcal{T}(\mathbf{x}^{(k+\tilde{k} ,\ell)})   \tilde{V}_{\frac{\gamma}{2}}(s)\tilde{\mathcal{V}}_{\mathbf{\alpha}}(\mathbf{z})\tilde{\mathcal{V}}_{\frac{\mathbf{\beta}}{2} }(\mathbf{s}) \rangle_{t}  \,\dd s
\\
\textcolor{black}{L_{9}(\mathbf{x})}:=& \sum_{\ell=1}^{k+\tilde{k}-1} \sum_{p=1}^n \frac{ \alpha_p}{(x_{k+\tilde{k}}-x_{\ell})^2}  ( \frac{1}{(x_\ell-z_p) }+\frac{1}{(x_\ell-\bar{z_p}) })\langle         \partial_{x_\ell}\Phi(x_{\ell} )      \mathcal{T} (\mathbf{x}^{(k+\tilde{k} ,\ell  )})    \tilde{\mathcal{V}}_{\mathbf{\alpha}}(\mathbf{z})\tilde{\mathcal{V}}_{\frac{\mathbf{\beta}}{2} }(\mathbf{s})\rangle_{t}    
\\
&+\sum_{\ell=1}^{k+\tilde{k}-1} \sum_{q=1}^m \frac{ \beta_q}{(x_{k+\tilde{k}}-x_{\ell})^2}  ( \frac{1}{(x_\ell-s_q) }) \langle         \partial_{x_\ell}\Phi(x_{\ell} )      \mathcal{T} (\mathbf{x}^{(k+\tilde{k} ,\ell  )})    \tilde{\mathcal{V}}_{\mathbf{\alpha}}(\mathbf{z})\tilde{\mathcal{V}}_{\frac{\mathbf{\beta}}{2} }(\mathbf{s}) \rangle_{t}    
\\
\textcolor{black}{L_{12}(\mathbf{x} ) }:=&- 2\sum_{\ell\not=\ell'=1}^{k+\tilde{k}-1} \frac{1}{(x_{k+\tilde{k}}-x_{\ell})^2(x_\ell-x_{\ell'})^2} \langle      \partial_{x_\ell}\Phi(x_{\ell} )  \partial_{x_{\ell'}}\Phi(x_{\ell'} )   \mathcal{T}(\mathbf{x}^{(k+\tilde{k},\ell,\ell' )})   \tilde{\mathcal{V}}_{\mathbf{\alpha}}(\mathbf{z})\tilde{\mathcal{V}}_{\frac{\mathbf{\beta}}{2} }(\mathbf{s}) \rangle_{t}  .
\end{align*}

\paragraph{Ward algebra.}
Now we are going to show that the appropriate combination of all these expressions combines to produce Ward's identity. Let us consider the expression
\begin{align*} 
K(\mathbf{u},{\bf z}):=&\langle   \mathcal{T}(\mathbf{x}) \tilde{\mathcal{V}}_{\mathbf{\alpha}}(\mathbf{z})\tilde{\mathcal{V}}_{\frac{\mathbf{\beta}}{2} }(\mathbf{s})\rangle_{t}  -I(\mathbf{x})- M(\mathbf{x})- T(\mathbf{x})-S(\mathbf{x} )-L(\mathbf{x} )  
\\
=&C(\mathbf{x})   + N(\mathbf{x})    -I(\mathbf{x})-S(\mathbf{x})-L(\mathbf{x})  .
\end{align*}
{\bf We will use the notation $C_k$ to represent the sum of all possible $C_k(\mathbf{x},\mu,\mu_\partial,\partial,\bar{\partial})$ terms.} and similarly for $C'_k$, $D_k$, $S_k$ and $L_k$. \\
 Also, we have obtained the relation
\begin{align}
N -I &=B+C'_6+C'_8+C'_{11}   -D_5-D_9-D_{10}-C_{14},\label{N-I}
\end{align}
in such a way that $K$ can be rewritten as
\begin{align*} 
K =& A_1+A_3+A_5+A_6+A_7+A_8+A_9+A_{10}+A_{11}+A_{12}+B
\end{align*}
with

 \begin{align}
A_1:=&C_1+C_2-S_1-L_1\label{theone}\\
A_3:=&C_3+C_4/2-S_3/2-S_7/2-L_3/2-L_7/2\label{thethree}\\
A_5:=&C_5+C_{15}-D_5-S_5-L_5\label{thefive}\\
 A_6:=& C_6+C_{13}-C'_6-S_6-L_6\label{thesix}\\
 A_7:=&C_7+C_4/2-S_3/2-S_7/2-L_3/2-L_7/2\label{theseven}\\
A_8:=&C_{8}+C'_{8}-S_{8}-L_{8} \label{theeight}\\
A_9:=& C_9-D_9-S_9-L_9 \label{thenine}\\
A_{10}:=& C_{10}-D_{10}\\
A_{11}:=&C_{11}+C'_{11}\\
A_{12}:= &C_{12}-S_{12}-L_{12}\label{fuck!!!!}.
\end{align}
Finally, we claim that all the $A_i$'s vanish.  Indeed, this is straightforward for $A_{10}, A_{11}$: it comes from the relation 
$$\frac{1}{(u_k-z_p)(u_k-z_{p'})}=\frac{1}{z_p-z_{p'}}\Big(\frac{1}{u_k-z_p}-\frac{1}{u_k-z_{p'}}\Big)$$ and a re-indexation of the double sum for $A_{10}$.

All the other terms results from  algebraic identities:  for \eqref{theone}, we use
\begin{align*}
\frac{1}{(z-u_\ell)^3 (z-u_{\ell'})^3}=&\frac{1}{(z-u_\ell)^3 (u_\ell-u_{\ell'})^3}-\frac{3}{(z-u_\ell)^2 (u_\ell-u_{\ell'})^4}+\frac{6}{(z-u_\ell) (u_\ell-u_{\ell'})^5}
\\ 
&-\frac{1}{(z-u_{\ell'})^3 (u_\ell-u_{\ell'})^3}-\frac{3}{(z-u_{\ell'})^2 (u_\ell-u_{\ell'})^4}-\frac{6}{(z-u_\ell) (u_{\ell'}-u_{\ell})^5}.
\end{align*}
For \eqref{thethree} and \eqref{theseven} we use 
\begin{align*}\frac{1}{(z-u'_{\ell'})^3(z-u'_{r'})^2}=&\frac{1}{(z-u'_{\ell'})^3(u'_{\ell'}-u'_{r'})^2}-\frac{2}{(z-u'_{\ell'})^2(u'_{\ell'}-u'_{r'})^3} 
\\
&+\frac{3}{(z-u'_{\ell'})(u'_{\ell'}-u'_{r'})^4} +\frac{1}{(z-u'_{r'})^2(u'_{r'}-u'_{\ell'})^3}- \frac{3}{(z-u'_{r'}) (u'_{r'}-u'_{\ell'})^4}.
\end{align*} 
For \eqref{thefive} and \eqref{thesix} we use
$$  \frac{1}{(u_k-u_\ell)^3(u_k-x)}= \frac{1}{(u_k-u_\ell)^3(u_\ell-x)} - \frac{1}{(u_k-u_\ell)^2(x-u_\ell)^2}+\frac{1}{(u_k-u_\ell)(u_\ell-x)^3}-\frac{1}{(u_k-x)(u_\ell-x)^3}. $$
For \eqref{theeight} and \eqref{thenine} we use
$$ \frac{1}{(u_k-u_\ell)^2(u_k-x)}= \frac{1}{(u_k-u_\ell)^2(u_\ell-x)} - \frac{1}{(u_k-u_\ell)(x-u_\ell)^2}+\frac{1}{(x-u_\ell)^2(u_k-x)} .$$
For \eqref{fuck!!!!} we use the identity
$$\frac{1}{(z-u_\ell)^2(z-u'_{\ell'})^2}=\frac{1}{(z-u_\ell)^2(u_\ell-u'_{\ell'})^2} -\frac{2}{(z-u_\ell) (u_\ell-u_{\ell'})^3}+\frac{1}{(u_\ell-u_{\ell'})^2(z-u'_{\ell'})^2}-\frac{2}{(u'_{\ell'}-u_\ell)^3(z-u'_{\ell'}) }.$$
This concludes the proof of Proposition \ref{WardTH}.\qed


\section{The Liouville vertex operator and identification of torus 1-point conformal block}\label{Identification}
In this section, we identify some already exisiting notations of torus 1-point conformal block
\begin{itemize}
    \item 1. The torus conformal block constructed by path integral method and analytic continuation in \cite{GKRV21};
    \item 2. The torus conformal block constructed by GMC integral in \cite{GRSS};
    \item 3. The torus conformal block formulated by Liouville vertex operator in \cite{Ca};
    \item 4. Nekrasov partition function in \cite{AGT} times a factor.
\end{itemize}
The fact $(2)\Longleftrightarrow(4)$ is the main theorem of \cite{GRSS}. $(3)\Longleftrightarrow(4)$ is the so-called AGT conjecture and proved in \cite{Neg}. Finally, $(1)\Longleftrightarrow(3)$ follows from the identity \eqref{torusward} (send $t\to \infty$) and the definition of Liouville vertex operator in \cite[Section 2.2]{Ca}. The detail of the argument is actually contained in proposition \eqref{0coefficient}.

\hspace{10 cm}

\end{document}